\documentclass{amsart}
\usepackage[margin=1.25in]{geometry}
\usepackage{lipsum}
\usepackage[justification=centering]{caption}
\usepackage{amsmath}
\usepackage{amssymb}
\usepackage{algorithm,algpseudocode, float}
\usepackage{mathtools}
\algrenewcommand\textproc{}

\usepackage{graphicx}
\usepackage{tikz}
\usepackage{color}
\usepackage[all]{xy}
\usepackage{amsthm}
\usepackage{verbatim}
\usepackage[normalem]{ulem}
\usepackage{textcomp}

\usepackage{listings}
\usepackage{enumitem}
\usepackage{indentfirst}

\usetikzlibrary{arrows,angles, shapes,positioning}
\usetikzlibrary{decorations.markings}
\tikzstyle arrowstyle=[scale=1]
\tikzstyle directed=[postaction={decorate,decoration={markings,
    mark=at position .55 with {\arrow[arrowstyle]{stealth}}}}]
\tikzstyle ddirected=[postaction={decorate,decoration={markings,
    mark=at position .45 with {\arrow[arrowstyle]{stealth}},
    mark=at position .55 with {\arrow[arrowstyle]{stealth}}}}]
\tikzstyle reverse directed=[postaction={decorate,decoration={markings,
    mark=at position .55 with {\arrowreversed[arrowstyle]{stealth};}}}]
\tikzstyle reverse ddirected=[postaction={decorate,decoration={markings,
    mark=at position .55 with {\arrowreversed[arrowstyle]{stealth};},
    mark=at position .65 with {\arrowreversed[arrowstyle]{stealth};}}}]

\usepackage[colorlinks,linkcolor=blue,citecolor=blue,pdfstartview=FitH, backref=page]{hyperref}

\theoremstyle{plain}

\newtheorem{para}{}[section]
\newtheorem{theorem}[para]{Theorem}
\newtheorem{proposition}[para]{Proposition}
\newtheorem{lemma}[para]{Lemma}
\newtheorem{cl}[para]{Claim}

\newtheorem{corollary}[para]{Corollary}

\newtheorem{fact}[para]{Fact}

\newtheorem*{question}{Question}

\theoremstyle{remark}

\newtheorem{remark}[para]{Remark}

\theoremstyle{definition}

\newtheorem{defn}[para]{Definition}

\newtheorem{conjecture}[para]{Conjecture}

\begin{document}

\title{Hidden symmetries and Dehn surgery on tetrahedral links}
\begin{abstract}
Motivated by a question of Neumann and Reid, we study whether Dehn fillings on all but one cusp of a hyperbolic link complement can produce infinite families of knot complements with hidden symmetries which geometrically converge to the original link complement. We prove several results relating the existence of such an infinite family of knot complements with hidden symmetries to the existence of certain symmetries of the horoball packings associated to the original link. Using these results, we develop an algorithm which when run on SnapPy can test when such symmetries do not exist. We then use this SnapPy code and two utilities from \cite{orbcenpract} to show that for any given link in the tetrahedral census of Fominykh-Garoufalidis-Goerner-Tarkaev-Vesnin, no such family of Dehn fillings exists. We establish the same result for two infinite families of cyclic covers of the Berge manifold and the $6^2_2$-complement as well.
\end{abstract}

\author{Priyadip Mondal}
\address{Department of Mathematics, Rutgers University-New Brunswick} \email{pm868@scarletmail.rutgers.edu}

\maketitle


\section{Introduction}
Interest in understanding which knot complements have hidden symmetries can be traced back to Neumann and Reid's seminal paper \cite{NeRe} from 1992. A \textit{hidden symmetry} of a hyperbolic $3$-manifold $M$ is an isometry between two finite covers of $M$ that is not a lift of any self-isometry of $M$. 

It is well known that arithmetic hyperbolic $3$-manifolds have infinitely many hidden symmetries (see \cite[Corollary 8.4.5]{MacRe}), but the figure eight knot is the only arithmetic hyperbolic knot by Reid \cite[Theorem 2]{Reid_arith}, shifting our focus to the non-arithemetic hyperbolic knots. Neumann and Reid \cite{NeRe} explained why the two dodecahedral knot complements of Aitchison and Rubinstein \cite{AiRu}, which are non-arithmetic, have hidden symmetries. As a follow-up, they asked in \cite[Question 1, \S 9]{NeRe} if these are the only hyperbolic knot complements with hidden symmetries. This question had been guiding the contemporary research on hidden symmetries. There are work in the literature showing that different well-known infinite families of hyperbolic knot complements admit no hidden symmetries, e.g., non-arithmetic two bridge knot complements by Reid and Walsh \cite{ReWa}, $(-2,3,n)$ hyperbolic pretzel knot complements by Macasieb and Mattman \cite{MacMat}. 

Indeed, attacking this question in its whole breadth is challenging. A weaker conjecture was proposed in \cite[Conjecture 0.1]{CDM} which states that given a positive real number $v$, there are at-most finitely many hyperbolic knot complements with volume less than $v$ which admit hidden symmetries. So, by Thurston's work \cite[Theorem 6.5.6]{Thurs}, to investigate this conjecture one would study the hyperbolic knot complements obtained by Dehn filling all but one cusp of a hyperbolic link which \textit{geometrically converge} to that link complement. The goal of understanding the existence of hidden symmetries of such families of hyperbolic knot complements has drawn significant recent attention, eg, in \cite{Hoff_Comm}, \cite{Hoff_smallknot}, \cite{CDM}, \cite{HMW}, \cite{CDHMMWIMRN}. 

In Theorem \ref{tetracompthm} of this paper, we prove the following. 

\newcommand\tetcompthm{Let $M$ be one of the tetrahedral homology link complements with two or more cusps in the orientable tetrahedral census given by \cite{FGGTV} and $K_0$  a cusp of $M$. 
Then a family of hyperbolic knot complements obtained by Dehn filling all cusps of $M$ but $K_0$ and geometrically converging to $M$ can have at-most finitely many elements with hidden symmetries. }
\theoremstyle{plain}
\newtheorem*{compthm}{Theorem \ref{tetracompthm}}
\begin{compthm}\tetcompthm\end{compthm}

A hyperbolic $3$-manifold is called \textit{tetrahedral} if it has a decomposition into regular ideal tetrahedra and it is a \textit{homology link complement} if it is isometric to a link complement in a homology sphere. In \cite{FGGTV}, Fominykh, Garoufalidis, Goerner, Tarkaev and Vesnin provided a census of tetrahedral manifolds. Their census includes all the orientable tetrahedral manifolds made up of $25$ or fewer tetrahedra and all the non-orientable tetrahedral manifolds made up of $21$ or fewer tetrahedra. 

Now, \cite[Theorem 2.4, Proposition 2.7]{NeRe} implies that all \textit{cusp shapes} of the \textit{tetrahedral manifolds} belong to $\mathbb{Q}(\sqrt{-3})$. This along with \cite[Corollary 1.4]{CDM} (or \cite[Theorem 2.6]{CDHMMWIMRN}) makes the tetrahedral links ideal candidates for our search of hidden symmetries in infinite families of knot complements obtained from Dehn filling all but one cusp of a link complement that geometrically converge to that link complement. 

Even though it is difficult to identify which tetrahedral manifolds are link complements, by \cite[Lemma 6.1]{FGGTV}, it is quite straightforward to identify the tetrahedral homology links in the orientable census of \cite{FGGTV} by looking at their first integral homology groups. Since the (orientable) tetrahedral census of \cite{FGGTV} can be accessed from SnapPy \cite{snappy}, we compute their first homology groups on SnapPy \cite{snappy} and use \cite[Lemma 6.1]{FGGTV} to get hold of all the $882$ tetrahedral homology links from the census \cite{FGGTV}. We should also keep in mind that if Dehn filling all but one cusp of a hyperbolic $3$-manifold produces a hyperbolic knot complement, the original hyperbolic $3$-manifold has to be a link complement and therefore a homolgy link complement as well. To establish Theorem \ref{tetracompthm}, we draw extensively on the symmetries of the horoball packings of $\mathbb{H}^3$ that map to the maximal cusp neighborhoods of these tetrahedral homology links.

We also consider the infinite families of tetrahedral links $B^{0,n}$ and $(6^2_2)^{0,n}$ whose complements are the cyclic covers of respectively the Berge manifold (see the second picture in \cite[Figure 3]{FGGTV}) and the complement of the $6^2_2$ link from \cite[Appendix C]{Rolfsen} (see the third picture in \cite[Figure 3]{FGGTV}). When $n$ is a multiple of $7$, $B^{0,n}$ has eight components and its complement covers the complement of the eight component link shown in Figure \ref{berge7fold}, otherwise it has two components. On the other hand, when $n$ is a multiple of $3$, $(6^2_2)^{0,n}$ has four components and its complement covers the complement of the four component link identified as \texttt{L12n2208} from the Hoste-Thistlewaite census on SnapPy \cite{snappy}, otherwise $(6^2_2)^{0,n}$ has two components. We prove the following two results. 

\newcommand\bergecyclic{Let $n \in \mathbb{N}$. Then, a family of hyperbolic knot complements obtained by Dehn filling all but a fixed component of $B^{0,n}$ and geometrically converging to $\mathbb{S}^3-B^{0,n}$ can only have at-most finitely many elements with hidden symmetries. 
}
\theoremstyle{plain}
\newtheorem*{Bergecyclicthm}{Theorem \ref{Bergecyclic}}
\begin{Bergecyclicthm}
\bergecyclic
\end{Bergecyclicthm}

\newcommand\sixtwotwocyclic{For each $n \in \mathbb{N}$, a family of hyperbolic knot complements obtained by Dehn filling all but a fixed component of $(6^2_2)^{0,n}$ and geometrically converging to $\mathbb{S}^3-(6^2_2)^{0,n}$ can only have at-most finitely many elements with hidden symmetries. 
}
\theoremstyle{plain}
\newtheorem*{622cyclicthm}{Theorem \ref{622cyclicresult}}
\begin{622cyclicthm}
\sixtwotwocyclic
\end{622cyclicthm}

Theorem \ref{Bergecyclic} follows quite easily by analyzing the horoball packings associated with the links $B^{0,n}$. Theorem \ref{622cyclicresult} however requires other subtle observations and techniques addressed at the end of the introduction.

We now give a brief overview of the contents of the paper. In Section \ref{cuspsection}, \ref{hidsymsection} and \ref{packingbackground}, we set up the background that we would need for the rest of the paper. In Section \ref{cuspsection}, we review the definitions of cusps, Dehn fillings and geometric convergence. In Section \ref{hidsymsection}, we recall the definitions and the known results from the literature regarding hidden symmetries. In Section \ref{packingbackground}, we go over the definitions and terminologies related to the horoball packings of $\mathbb{H}^3$ and give a brief account on how one can work with the horoball packings in the $3$-manifold software SnapPy \cite{snappy}. 

Before we delve into our review of Section \ref{hscp}, we recall that Neumann and Reid \cite[Proposition 9.1]{NeRe} characterized the non-arithmetic hyperbolic knot complements with hidden symmetries as the ones that non-normally cover an orbifold with a rigid cusp.  This fact was  extended to the Dehn filling scenario by \cite[Theorem 2.6]{CDHMMWIMRN}): if there is an infinite family of hyperbolic knot complements with hidden symmetries obtained by Dehn filling all but one component $K$ of a hyperbolic link $L$ and geometrically converging to $\mathbb{S}^3-L$, then $\mathbb{S}^3-L$ covers an orbifold $O$ with exactly one rigid cusp and at-least one smooth cusp so that the only cusp of $L$ that maps to the rigid cusp of $O$ is the cusp corresponding to the $K$-component.

In Section \ref{hscp}, we focus on retrieving more geometric information from \cite[Theorem 2.6]{CDHMMWIMRN}. With this aim, in the following proposition we translate \cite[Theorem 2.6]{CDHMMWIMRN} in the semantics of the symmetries of the horoball packing of $\mathbb{H}^3$ that map to the maximal cusp neighborhood corresponding to the un-filled component of $L$. 

\newcommand\wallthm{For a hyperbolic link $L$ with components $K_1, \dots, K_n$, where $n \ge 2$, if there is an infinite family $\mathcal{F}$ of hyperbolic knot complements with hidden symmetries obtained by Dehn filling all cusps of $L$ but the $K_1$-cusp which geometrically converges to $\mathbb{S}^3-L$, then the symmetry group of the $K_1$-circle packing contains the wallpaper group $W_{(3,3,3)}$ such that  for any $j \in \{2, \dots, n\}$, the stabilizer of the horocenter of each $K_j$-horoball in $W_{(3,3,3)}$ is trivial and the co-area of the (maximal) translational subgroup of $W_{(3,3,3)}$ is $$\frac{\text{maximal cusp area of the }K_1\text{-cusp}}{4m}$$ for some $m\in \mathbb N$.}
\theoremstyle{plain}
\newtheorem*{wall333thm}{Theorem \ref{hexlatt}}
\begin{wall333thm}\label{introwallthm}
\wallthm
\end{wall333thm}

The above result is built upon Neil Hoffman's observation which we prove in Theorem \ref{2comp_sym} of this paper. In addition, Hoffman's degree formulae (\cite[Part 1 of Lemma 5.5]{Hoff_smallknot} and \cite[Theorem 1.2]{Hoff_cusp}) play a vital role in the proving the above result. 

Using properties of the wallpaper group $W_{(3,3,3)}$ we can apply Theorem \ref{hexlatt} to obtain Corollary \ref{neardistinctrot} which loosely says the following: if there is an infinite family of hyperbolic knot complements with hidden symmetries obtained by Dehn filling all components of a hyperbolic link $L$ but component $K$ which geometrically converge to $\mathbb{S}^3-L$, then for each circle in the $K$-circle packing, the $K$-circle packing has two distinct order $3$ symmetries whose fixed points are ``close" to the center of that circle and not the horocenter of a horoball corresponding to a filled component of $L$. We use this result in Section \ref{hsalgorithm} to develop an algorithm, which when run as a Python code on SnapPy \cite{snappy}, would detect when such close symmetries do not exist.

We begin Section \ref{coversection} by noting in Proposition \ref{8v0prop} that \cite[Theorem 2.6]{CDHMMWIMRN} can also be applied to obtain a lower bound on the volume of the link complements whose Dehn fillings we should study. This can be seen as an extension of \cite[Theorem 4.1]{CDHMMWIMRN} for link complements. Hoffman's degree formula (\cite{Hoff_smallknot} and \cite{Hoff_cusp}) and Adams' classification of low volume multi-cusped orbifolds \cite{Adams_multi} are also crucial for this result.

\newcommand\eightvprop{Let $L$ be a hyperbolic link with two or more components whose volume is less than or equal to $8v_0$ where $v_0$ is the volume of a regular ideal tetrahedron.  Then for any family $\mathcal{F}$ of hyperbolic knot complements obtained by Dehn filling all but one component of $L$ and geometrically converging to $\mathbb{S}^3-L$,  $\mathcal{F}$ has at-most finitely many elements with hidden symmetries. }
\theoremstyle{plain}
\newtheorem*{8v0thm}{Proposition \ref{8v0prop}}
\begin{8v0thm}
\eightvprop
\end{8v0thm}

Another question that \cite[Theorem 2.6]{CDHMMWIMRN} poses is which orbifolds $O$ with exactly one rigid cusp and at-least one smooth cusp could appear as the one laid out in \cite[Theorem 2.6]{CDHMMWIMRN}. Not many orbifolds with exactly one rigid cusp and at-least one smooth cusp are known in the literature. In his thesis, Tyler Gaona \cite{GTthesis} studied an orbifold of volume $\frac{5v_0}{6}$ with exactly one $(2,3,6)$ cusp and one pillowcase cusp, which we will denote in this paper by $O_{(2,3,6), (2,2,2,2)}$. Gaona showed in \cite[Theorem 3.2]{GTthesis} that this is the smallest volume hyperbolic $3$-orbifold with one $(2,3,6)$-cusp and one smooth cusp. Application of the cusp killing homomorphism (see \cite{Hoff_smallknot}) rules out this orbfifold as the one whose Dehn filling on the smooth cusp could produce an orbifold covered by a knot complement. In fact, Adams' classification \cite{Adams_multi}, \cite[Theorem 2.6]{CDHMMWIMRN} and ideas from \cite{GoHeHo} corroborate the following stronger implication.

\newcommand\coversmallestmulticuspthm{Let $c$ and $c'$ be two symmetric cusps ($c$ can be same as $c'$) of a hyperbolic link $L$. Suppose $\mathbb{S}^3-L$ covers $O_{(2,3,6), (2,2,2,2)}$ in a way such that the only cusp of $\mathbb{S}^3-L$ that maps to the $(2,3,6)$ cusp $c_{(2,3,6)}$ of $O_{(2,3,6), (2,2,2,2)}$ is $c$. Then 
\begin{enumerate}
\item the symmetry group of the $c'$-circle packing contains a $W_{(2,3,6)}$ group of symmetries whose elements of order $3$ or $6$ do not fix the horocenter of any $K$-horoball where $K$ is any cusp of $L$ different from $c'$,
\item for any family $\mathcal{F}$ of hyperbolic knot complements obtained by Dehn filling all cusps of $L$ but $c'$ and geometrically converging to $\mathbb{S}^3-L$,  $\mathcal{F}$ has at-most finitely many elements with hidden symmetries. 
\end{enumerate}}
\theoremstyle{plain}
\newtheorem*{coversmallestmulticusp}{Proposition \ref{coversmallestmulticusp}}
\begin{coversmallestmulticusp}
\coversmallestmulticuspthm
\end{coversmallestmulticusp}

In Section \ref{coversection}, we also show (in the proposition stated below) that the converse of Theorem \ref{hexlatt} is not true, i.e., existence of $W_{(3,3,3)}$ wallpaper groups satisfying the properties laid out in Theorem \ref{hexlatt} does not imply that $\mathcal{F}$ can have infinitely many knot complements with hidden symmetries. We show that extra ``bad" order 6 symmetries outside $W_{(3,3,3)}$ may obstruct $\mathcal{F}$ to have infinitely many elements with hidden symmetries.

\newcommand\twosixcuspthm{Let $c$ be a cusp of a hyperbolic link $L$ with two or more components whose volume is less than or equal to $24 v_0$, where $v_0$ is the volume of a regular ideal tetrahedron. Suppose $\mathbb{S}^3-L$ covers a hyperbolic orbifold $O'$ with at-least two $(2,3,6)$-cusps such that $c$ covers a $(2,3,6)$ cusp $c'$ of $O'$ and that the only cusp of $\mathbb{S}^3-L$ that covers $c'$ is $c$. Then for any family $\mathcal{F}$ of hyperbolic knot complements obtained by Dehn filling all cusps of $L$ but $c$ and geometrically converging to $\mathbb{S}^3-L$,  $\mathcal{F}$ has at-most finitely many elements with hidden symmetries. }
\theoremstyle{plain}
\newtheorem*{twosixcusp}{Proposition \ref{two236cusps}}
\begin{twosixcusp}
\twosixcuspthm
\end{twosixcusp}

To prove the above proposition, we need to use the tools employed in proving Theorem \ref{hexlatt} along with consequences of cusp killing homomorphism \cite{Hoff_smallknot}, Adams' classification of low volume multi-cusped orbifolds \cite{Adams_multi} and discreteness of symmetry groups of the horoball packings of $\mathbb{H}^3$ related to the maximal cusps from Goodman-Heard-Hodgson \cite{GoHeHo}.

In Section \ref{hsalgorithm}, we use Corollary \ref{neardistinctrot} from Section \ref{hscp} to build an algorithm which tests when the $K_1$-circle packing of $\mathbb{C}$ cannot have two distinct order $3$ rotational symmetries neither fixing the centers of other $K_j$-horoballs and satisfying the distance bounds given in that corollary. In particular, in Proposition \ref{codeprop}, we show that if Dehn filling all but $K_1$-cusp of a hyperbolic link $L$ produces an infinitely family of knot complements with hidden symmetries which geometrically converge to $\mathbb{S}^3-L$, then given the center $\mathbf{c}_0$ of a circle in the $K_1$-circle packing, there is a finite set $\mathcal{E}_{strong}$ in $\mathbb{C}$ where we would see the centers of two distinct order $3$ rotations satisfying the distance bounds in Corollary \ref{neardistinctrot}. The code that we write chooses such a $\mathbf{c}_0$ and rules out cases when the corresponding $\mathcal{E}_{strong}$ can't have two such distinct order $3$ rotations which do not fix the centers of the other $K_j$-horoballs. 

In Section \ref{tlcom}, we discuss the tetrahedral manifolds from the census of \cite{FGGTV} and prove Theorem \ref{tetracompthm}. In order to prove Theorem \ref{tetracompthm}, we first run the code from Section \ref{hsalgorithm} on SnapPy \cite{snappy} for the $882$ tetrahedral homology link complements (with two or more cusps) in the census \cite{FGGTV}. This gives us a set $\mathcal{E}$ consisting of $86$ exceptional pairs $(M,i)$, where $M$ is one such tetrahedral homology link complement and $i$ one of its cusps, such that for $M$, our code does not guarantee non-existence of $W_{(3,3,3)}$ wallpaper group of symmetries for the $i$-cusp circle packing of $\mathbb{C}$ satisfying the properties given in Theorem \ref{hexlatt}. $\mathcal{E}$ can be split into four mutually disjoint proper subsets $\mathcal{E}_1, \mathcal{E}_2, \mathcal{E}_3$ and $\mathcal{E}_4$. We argue that even for each such $86$ pairs $(M,i)$, Dehn filling on all cusps but cusp $i$ of $M$ does not produce an infinite family of hyperbolic knot complements with hidden symmetries which geometrically converge to $M$. 

For the pairs in $\mathcal{E}_3$, $i$-cusp maximal horoball packing of $\mathbb{H}^3$ does not have the required $W_{(3,3,3)}$ wallpaper group of symmetries even though the corresponding circle packing of the complex plane does. On the other hand, for each pair in $\mathcal{E}_2$, the $i$-cusp maximal horoball packing has a symmetry of order $6$ joining the $i$-cusp and a different cusp. So, this case too is taken care of by Corollary \ref{badorder6sym}. 

For the pairs $(M,i)$ in $\mathcal{E}_1$ and $\mathcal{E}_4$, we use the existence of ``good" covers from $M$ onto the orbifold $O_{(2,3,6), (2,2,2,2)}$. In particular, using utilities \texttt{SigToSeq.py} and \texttt{TestForCovers.py} of \cite{orbcenpract} we argue that there is covering map from $M$ onto $O_{(2,3,6), (2,2,2,2)}$ such that the only cusp of $M$ that maps to the $(2,3,6)$ cusp of $O_{(2,3,6), (2,2,2,2)}$ via this cover is (symmetric to) $i$. We then use Proposition \ref{coversmallestmulticusp} to complete the argument for these two remaining cases as well.

In Section \ref{cycliccovers}, we prove Theorem \ref{Bergecyclic} and Theorem \ref{622cyclicresult}. After careful observation of the horoball packings of $\mathbb{H}^3$ that map to the concerned maximal cusp neighborhoods associated with the $B^{0,n}$ complement, we can use Theorem \ref{hexlatt} to prove Theorem \ref{Bergecyclic}. However, Theorem \ref{hexlatt} alone can't take care of the general $(6^2_2)^{0,n}$ case as the concerned horoball packings of $\mathbb{H}^3$ corresponding to the maximal cusp neighborhoods of the $(6^2_2)^{0,n}$ complement will have $W_{(3,3,3)}$ wallpaper group of symmetries that Theorem \ref{hexlatt} prescribes. Nevertheless, the analogous result for $(6^2_2)^{0,n}$ holds as well. 

In addition to Theorem \ref{hexlatt}, proof of Theorem \ref{622cyclicresult} uses additional observations on the concerned horoball packings of $\mathbb{H}^3$ associated with the $(6^2_2)^{0,n}$ complement and the discreteness of their symmetry groups following Goodman, Heard and Hodgson \cite[Lemma 2.1 and (Proof of) Lemma 2.2]{GoHeHo}. Moreover, we use applications of the cusp killing homomorphism from \cite{Hoff_smallknot} and the classification of the small volume multi-cusped orbifolds given by Adams \cite{Adams_multi}. We would also need to use the utilities \texttt{TestForCovers.py} and  \texttt{SigToSeq.py} from \cite{orbcenpract}.

\subsection*{Terminology and convention}

Throughout this paper, we will use to term \textit{hyperbolic $3$-orbifold} to mean a smooth orientable $3$-orbifold orbifold-diffeomorphic to $\mathbb{H}^3/\Gamma$ where $\Gamma$ is a finite co-volume Kleinian group (i.e. a discrete subgroup of $\operatorname{PSL}(2,\mathbb{C})$ whose fundamental domains in $\mathbb{H}^3$ have finite volume) such that $\infty$ is a parabolic fixed point of $\Gamma$. 

\subsection*{A note on the computations} 

The codes and data associated with this article can be accessed from the GitHub repository \cite{tetra_code}. The codes are written in Python. The codes that need to use SnapPy \cite{snappy} import SnapPy as a Python module\footnote{See \url{https://snappy.computop.org/installing.html}} and can be run as normal python files. On the other hand, the codes that use Regina \cite{regina} needs to be run in \texttt{regina-python}\footnote{See \url{https://regina-normal.github.io/docs/man-regina-python.html}}. The other modules, files and data that the codes import are detailed when they are discussed in the article. 

\subsection*{Acknowledgement}
This article is an extension of parts of author's PhD thesis. The author expresses his gratitude to Jason DeBlois for many valuable suggestions. The author would also like to thank Neil Hoffman for communicating Theorem \ref{2comp_sym} in a SQuaRE program of American Institute of Mathematics (AIM). The author conveys his thanks to AIM for hosting the said SQuaRE program, to NSF for their support for the AIM SQuaREs and to his fellow SQuaRE team members Eric Chesebro, Michelle Chu, Jason DeBlois, Neil Hoffman and Genevieve Walsh for many helpful discussions. The author also thanks Hongbin Sun for many helpful suggestions.

\newcommand\censuspaperlinks{If $L$ is one of the $25$ tetrahedral links with more than one component that are explicitly listed in the Fominykh-Garoufalidis-Goerner-Tarkaev-Vesnin paper \cite{FGGTV}, then for any cusp $c$ of $L$, a family of hyperbolic knot complements obtained by Dehn filling all cusps of $L$ but $c$ and geometrically converging to $\mathbb{S}^3-L$ cannot have infinitely many members with hidden symmetries.}

\section{Cusps, Dehn filling and Geometric Convergence}\label{cuspsection}
Let $\Gamma$ be a finite co-volume Kleinian group. Consider the finite volume hyperbolic $3$-orbifold $O=\mathbb{H}^3/\Gamma$. We know that the non-compact ends of $O$ referred to as the \textit{cusp ends} of $O$ (or simply the \textit{cusps} of $O$) can be written as $B/W$ where $B$ is a horoball in $\mathbb{H}^3$ and $W$ is an oriented wallpaper group such that $W$ is the stabilizer of the horocenter of $B$ in $\Gamma$ and $B/W$ embeds into $O$ (see \cite[Corollary 2.2]{DunMeyer}). Now, $\partial B$ inherits a Euclidean metric from the hyperbolic metric on $\mathbb{H}^3$ and the wallpaper group $W$ acts as (oriented) Euclidean isometries on $\partial B$. This makes $\partial B/W$ an (oriented) Euclidean $2$-orbifold (see \cite[Section 13.3]{Thurs} for a discussion on $2$-orbifolds). We refer to $\partial B/W$ as the \textit{cusp cross-section} of the cusp $B/W$. There are five (isomorphic classes) of oriented Euclidean $2$-orbifolds - $\mathbb{T}^2$ (aka torus), $\mathbb{S}^2(2,2,2,2)$ (aka \textit{pillowcase}), $\mathbb{S}^2(3,3,3)$, $\mathbb{S}^2(2,3,6)$ and $\mathbb{S}^2(2,4,4)$. We refer to the torus and pillowcase cusps as the \textit{smooth cusps} and $\mathbb{S}^2(3,3,3)$, $\mathbb{S}^2(2,3,6)$, $\mathbb{S}^2(2,4,4)$ as the \textit{rigid cusps}. We make the following definition. 

\begin{defn}
Two cusps $c$ and $c'$ of a hyperbolic $3$-orbifold are said to be \textit{symmetric} if there is a self-isometry of $O$ sending them to each other.

\end{defn}

\subsection{Review of the wallpaper groups}
We will now recall a couple of crucial facts about wallpaper groups from \cite{Sch} - (i). the \textit{translational subgroup} of a wallpaper group (i.e. the subgroup solely containing all the translational isometries) is generated by two linearly independent translations and (ii). two wallpaper groups are \textit{equivalent} iff there is an isomorphism from one to the other mapping the translational subgroup of the former onto that of the later. 
We will denote the wallpaper groups corresponding to the rigid cusps $\mathbb{S}^2(3,3,3)$ and $\mathbb{S}^2(2,3,6)$ as respectively $W_{(3,3,3)}$ and $W_{(2,3,6)}$. We will use $\Lambda_{(3,3,3)}$ (respectively, $\Lambda_{(2,3,6)}$) to denote the translational subgroup of $W_{(3,3,3)}$ (respecitively, $W_{(2,3,6)}$). We now describe $W_{(3,3,3)}$ and $W_{(2,3,6)}$  up-to equivalency. Let $r_{n,\mathbf{p}}$ denote a counter-clockwise rotation of angle $\frac{2\pi}{n}$ around the point $\mathbf{p}$ in $\mathbb{C}$. Then, $W_{(2,3,6)}$  is equivalent to the group $W^{s}_{(2,3,6)}$ generated by the order $6$ rotation $r_{6,(0,0)}$ and the order $3$ rotation $r_{3,(\frac{1}{2},\frac{1}{2\sqrt{3}})}$. On the other hand, $W_{(3,3,3)}$  is equivalent to the group $W^{s}_{(3,3,3)}$ generated by the order $3$ rotations $r_{3,(0,0)}$ and $r_{3,(\frac{1}{2},\frac{1}{2\sqrt{3}})}$. Our choice of generators here is similar to that in Hoffman's treatment in \cite[Section 2]{Hoff_cusp} for $\mathbb{S}^2(2,3,6)$ and $\mathbb{S}^2(2,4,4)$ (the generating rotations there are clockwise though).

For $\mathbf{p} \in \mathbb{C}$, let $t_{\mathbf{p}}$ denote translational isometry sending $\mathbf{x}$ to $\mathbf{p}+\mathbf{x}$ for $\mathbf{x} \in \mathbb{C}$. Note that $t_{(1,0)}=r_{3,(\frac{1}{2},\frac{1}{2\sqrt{3}})} \circ r^{-1}_{3,(0,0)}$ and $t_{(\frac{1}{2},\frac{\sqrt{3}}{2})}=r^{-1}_{3,(\frac{1}{2},\frac{1}{2\sqrt{3}})} \circ r_{3,(0,0)}$. The group generated by $t_{(1,0)}$ and $t_{(\frac{1}{2},\frac{\sqrt{3}}{2})}$ is the translational subgroup for both $W^{s}_{(2,3,6)}$ and $W^{s}_{(3,3,3)}$.
Using the fact that $r_{3,(0,0)}=r^2_{6,(0,0)}$, from now on, we will identify $W_{(3,3,3)}$ as an index $2$ subgroup of $W_{(2,3,6)}$ such that $\Lambda_{(3,3,3)}=\Lambda_{(2,3,6)}$. We note here that $\Lambda_{(3,3,3)}$ (which is $\Lambda_{(2,3,6)}$) has index $3$ in $W_{(3,3,3)}$ and index $6$ in $W_{(2,3,6)}$. 

Let us implicitly fix an equivalent map (of wallpaper groups) from $W^{s}_{(3,3,3)}$to $W_{(3,3,3)}$ and denote the image of $w \in W^{s}_{(3,3,3)}$ in $W_{(3,3,3)}$ as $\widetilde{w}$. Figure \ref{W333fd} shows the fundamental domains $D_{\Lambda_{(3,3,3)}}$ and $D_{W_{(3,3,3)}}$ of $\Lambda_{(3,3,3)}$ and $W_{(3,3,3)}$ respectively. Both $D_{\Lambda_{(3,3,3)}}$ and $D_{W_{(3,3,3)}}$ are hexagonal rhombuses with sets of vertices $\{\mathbf{v}_1, \mathbf{v}_2, \mathbf{v}_3, \mathbf{v}_4\}$ and $\{\mathbf{v}_2, \mathbf{v}_6, \mathbf{v}_4, \mathbf{v}_5\}$ respectively. We should note that $r_{3,\mathbf{v}_i} \in W_{(3,3,3)}$ for each $i =1,\dots, 6$. In particular, we can write 
$r_{3,\mathbf{v}_1}=\widetilde{r_{3,(0,0)}},r_{3,\mathbf{v}_5}=\widetilde{r_{3,(\frac{1}{2},\frac{1}{2\sqrt{3}})}}$.

\begin{figure}
\centering 
\captionsetup{justification=centering}
\begin{tikzpicture}[scale=3]
\draw (0,0)--(.5,.867);
\draw (0,0)--(1,0);
\draw (.5,.867)--(1.5,.867);
\draw (1,0)--(1.5,.867);
\draw (.5,.29)--(.5,.867);
\draw (.5,.29)--(1,0);
\draw (1,0)--(1,.578);
\draw (1,.578)--(.5,.867);
\node at (-.1,-.1){$\mathbf{v}_1$};
\node at (1.1,-.1){$\mathbf{v}_2$};
\node at (1.6,.877){$\mathbf{v}_3$};
\node at (.4,.877){$\mathbf{v}_4$};
\node at (.4,.23){$\mathbf{v}_5$};
\node at (1.1,.588){$\mathbf{v}_6$};
\end{tikzpicture}

\caption{A fundamental domain of $W_{(3,3,3)}$ contained in a fundamental domain of $\Lambda_{(3,3,3)}$.}
\label{W333fd}
\end{figure}
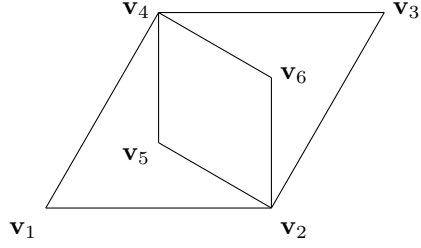
\subsection{Cusp moduli and cusp field}
Given an ordered pair $(\mathbf{m}, \mathbf{l})$ of linearly independent vectors in $\mathbb{R}^2$, we would define $\operatorname{sign}(\mathbf{m}, \mathbf{l})$ as the sign of the determinant of the $ 2 \times 2$ matrix $\begin{pmatrix}\mathbf{m} & \mathbf{l} \end{pmatrix}$. Let $c$ be a torus cusp of a hyperbolic $3$-orbifold $O$. Then $\partial c\cong \partial B/W$ where $\partial B$ is a horosphere and $W=\left \langle t_{\mathbf{m}}, t_{\mathbf{l}} \right \rangle$ is a lattice. Then the complex number $\left(\frac{\mathbf{l}}{\mathbf{m}}\right)^{\operatorname{sign}(\mathbf{m}, \mathbf{l})}$ is said to be a \textit{cusp moduli} of $c$. Note that a cusp moduli always belongs to the upper-half plane. If  $\operatorname{sign}(\mathbf{m}, \mathbf{l})$ is positive, we say $(\mathbf{m}, \mathbf{l})$ is an \textit{oriented pair} of generators for $c$. 

Now, given a cusp $c$ of a hyperbolic $3$-manifold $M$, the \textit{cusp field associated to cusp $c$} of $M$ is defined to be $\mathbb{Q}(\tau_c)$ where $\tau_c$ is a cusp moduli of cusp $c$. Note that this definition is independent of the choice of the cusp moduli. The cusp field of a hyperbolic $3$-manifold $M$ is defined to be the smallest field containing all the cusp fields associated to all the cusps of $M$. 

\subsection{Dehn filling and geometric convergence}
We will now define what \textit{Dehn filling} on a smooth cusp means (see \cite[Section 4]{DunMeyer}, \cite[Section 2.5]{BMP}, \cite[Section 2.1]{CDHMMWIMRN}). We first define some notations and terminologies. Let $m$ denotes a curve on $\mathbb{S}^1\times \mathbb{S}^1=\partial (\mathbb{S}^1 \times \mathbb{D}^2)$ which bounds a disk in $\mathbb{S}^1 \times \mathbb{D}^2$ and $r_{d}$ denotes the order $d$ rotation of $\mathbb{S}^1 \times \mathbb{D}^2$ around the core curve $\mathbb{S}^1 \times \{0\}$ where $d$ is a positive integer. Denote the quotient of $\mathbb{S}^1 \times \mathbb{D}^2$ by the action of $r_{d}$ by $(\mathbb{S}^1 \times \mathbb{D}^2)_d$. Topologically, $(\mathbb{S}^1 \times \mathbb{D}^2)_d$ is a solid torus but it has an orbifold structure and it contains cone points of order $d$ (when $d$ is $1$, the cone points are not really cone points and orbifold structure is the trivial manifold structure). Let $m_d$ denotes the image of $m$ in $(\mathbb{S}^1 \times \mathbb{D}^2)_d$. 

Now, $(\mathbb{S}^1 \times \mathbb{D}^2)_d$ has an order $2$ involution $i_2$ which fixes four points in $\partial\left((\mathbb{S}^1 \times \mathbb{D}^2)_d\right)-m_d$. We refer to the quotient of $(\mathbb{S}^1 \times \mathbb{D}^2)_d$ by $i_2$ as an order $d$ \textit{orbi-pillow} and denote it as $P_d$. Note that the boundary $\partial P_d$ of the orbifold $P_d$ is the pillowcase orbifold $\mathbb{S}^2(2,2,2,2)$. Denote the image of $m_d$ in $P_d$ as $m_d^{pillow}$. 

We first define Dehn filling on a torus cusp. Let $c$ be a torus cusp of an orbifold $O$. Choose a generating pair $(m,l)$ of $H_1(\partial c)$. Let $(p,q)$ be a pair of integers and $d=gcd(p,q)$. Let $h$ be a homeomorphism from 
$\partial ((\mathbb{S}^1 \times \mathbb{D}^2)_d)$, which is topologically a torus, onto $\partial c$ so that $h(m_d)=\frac{p}{d}m +\frac{q}{d} l$. Then we say that the orbifold $\left(O-int(c)\right)\cup_{h} (\mathbb{S}^1 \times \mathbb{D}^2)_d$ is \textit{obtained by $(p,q)$-Dehn filling on cusp $c$ of $O$}. See Figure 1. (a) of \cite{CDHMMWIMRN}. 

Now, let $c$ be a pillowcase cusp of an orbifold $O$. Then there is a quotient map $\phi: \mathbb{S}^1\times \mathbb{S}^1\to \partial c$. Choose generators $m$ and $l$ of $H_1(\mathbb{S}^1\times \mathbb{S}^1)$ so that $\phi(\frac{p}{d}m +\frac{q}{d} l)$ does not contain any cone points. Denote $\phi(\frac{p}{d}m +\frac{q}{d} l)$ by $\gamma$. Let $h$ be a homeomorphism from $\partial P_d$ onto $\partial c$ such that $h(m_d^{pillow})=\gamma$. Then we say that the orbifold $\left(O-int(c)\right)\cup_{h} P_d$ is \textit{obtained by $(p,q)$-Dehn filling on cusp $c$ of $O$}. See Figure 1. (b) of \cite{CDHMMWIMRN}. 

In this paper, we are interested in the Dehn filling on the cusps of hyperbolic link complements. Suppose $L=K_1 \cup \ldots \cup K_n$ is an $n$-component hyperbolic link. Then, for each $j\in \{1,2,\dots,n\}$, we will refer to the cusp of $\mathbb{S}^3-L$ corresponding to the $K_j$-component as the \textit{$K_j$-cusp}. Given some $j\in \{1,2,\dots,n\}$, if we don't fill the $K_j$-cusp of $\mathbb{S}^3-L$, (using the standard convention) we will say that the filling coefficient at the $K_j$-cusp is $\infty$.  We recall from Thurston's Dehn surgery theorem \cite[Theorem 5.8.2]{Thurs} that there is a compact set $C$ in $(\mathbb{S}^1)^{n}$ (which depends on the hyperbolic link $L$) such that whenever $(p_1,q_1, p_2,q_2, \dots, p_n,q_n)$ is outside $C$, the manifold obtained by $(p_j,q_j)$-Dehn filling the $K_j$-cusp for each $j \in \{1, \dots, n\}$ is hyperbolic. We also recall that \cite[Theorem 5.3]{DunMeyer} shows that the same result holds in general for the orbifolds obtained by Dehn filling the smooth cusps of a hyperbolic $3$-orbifold. 
\begin{defn}
Let $L=K_1 \cup \ldots \cup K_n$ be an $n$-component hyperbolic link where $n$ is greater than $1$. Let $\mathcal{F}=\left \{\mathbb{S}^3-K'_{\infty, (p^i_2, q^i_2), \dots, (p^i_n, q^i_n)}\right \}_{i\in \mathbb{N}}$ be a family of hyperbolic knot complements such that for each $i$, the knot complement $\mathbb{S}^3-K'_{\infty, (p^i_2, q^i_2), \dots, (p^i_n, q^i_n)}$ is obtained by $(p^i_j, q^i_j)$-Dehn filling the $K_j$-component of $L$ for each $j$ in $\left \{2, \dots, n\right \}$. If for each $j$ in $\left \{2, \dots, n\right \}$, $(p^i_j, q^i_j) \to \infty$ as $i \to \infty$, then we say that the family $\mathcal{F}$ geometrically converges to $\mathbb{S}^3-L$. If $\mathcal{F}$ geometrically converges to $\mathbb{S}^3-L$, we also say $\left \{K'_{\infty, (p^i_2, q^i_2), \dots, (p^i_n, q^i_n)}\right \}_{i\in \mathbb{N}}$ geometrically converges to $L$. 
\end{defn}
We remark that geometric convergence in general is a much broader concept. See \cite[Chapter 4]{Marden_hm} for details. We end this section with the following fact. 
\begin{fact}\label{symmfact}
Let $c$ and $c'$ be two symmetric cusps ($c$ and $c'$ can be same) of a hyperbolic link complement $\mathbb{S}^3-L$, i.e. there is a self-isometry of $\mathbb{S}^3-L$ exchanging the cusps. Then for each family $\mathcal{F}$ of hyperbolic knot complements obtained by Dehn filling all cusps of $\mathbb{S}^3-L$ but $c$ and geometrically converging to $\mathbb{S}^3-L$, there exists a family $\mathcal{F}'$ of hyperbolic knot complements obtained by Dehn filling all cusps of $\mathbb{S}^3-L$ but $c'$ and geometrically converging to $\mathbb{S}^3-L$ such that each member of $\mathcal{F}$ is isometric to a member of $\mathcal{F}'$ and vice versa.
\end{fact}
\section{Hidden symmetries and covering maps to orbifolds}\label{hidsymsection}
In this section we recall some basic facts and important results on hidden symmetries. We refer the readers to \cite{NeRe} for more details on hidden symmetries. Let $M$ be a finite volume hyperbolic $3$-manifold. Denote its fundamental group by $\Gamma$. Suppose $g$ is an isometry between two finite-index covers of $M$. Then, $g$ is said to be a \textit{hidden symmetry} of $M$ if it is not a lift of any self-isometry of $M$. It should be noted that isometries between finite covers of $M$ can be realized as elements of the \textit{commensurator} of $\Gamma$ in $\operatorname{PSL}(2,\mathbb{C})$
$$\operatorname{Comm}(\Gamma)=\{g \in \operatorname{PSL}(2,\mathbb{C}): [\Gamma:\Gamma \cap g \Gamma g^{-1}]<\infty,  [g\Gamma g^{-1}:\Gamma \cap g \Gamma g^{-1}]<\infty\}.$$

On the other hand, if such an element $g \in \operatorname{Comm}(\Gamma)$ is a lift of a self-isometry of $M$, then $g \in \operatorname{N}(\Gamma)$. So, the existence of hidden symmetries of $M$ is equivalent to the existence of more than one cosets of $\operatorname{N}(\Gamma)$ in $\operatorname{Comm}(\Gamma)$. It is well known (see \cite[Corollary 8.4.5]{MacRe}) that for an arithmetic Kleinian group $\Gamma$, $\operatorname{Comm}(\Gamma)$ is dense in $\operatorname{PSL}(2,\mathbb{C})$ (Margulis \cite{Mar} showed that the converse is true as well). This guarantees the existence of hidden symmetries in arithmetic hyperbolic $3$-manifolds. In particular, the figure-eight knot complement, which is the only arithmetic hyperbolic knot complement by Reid \cite[Theorem 2]{Reid_arith}, has hidden symmetries.

The only non-arithmetic hyperbolic knot complements that are known to admit hidden symmetries are the two dodecahedral knots of Aitchison and Rubinstein \cite{AiRu}. The existence of hidden symmetries in these two knot complements is argued in Neumann and Reid \cite{NeRe} (after the proof of Proposition 9.1 on p. 307) using the following fundamental characterization from their paper. 
\begin{proposition}[Neumann-Reid, \cite{NeRe}, Proposition 9.1] \label{hscharac}
A non-arithmetic hyperbolic knot complement has hidden symmetries if and only if it non-normally covers an orbifold with a rigid cusp. 
\end{proposition}
We remark that the figure-eight knot complement covers the rigid cusped orbifold $\mathbb{H}^3/\operatorname{PGL}(2,O_3)$ which has a $(2,3,6)$-cusp. The fact below which is an immediate consequence of the proposition above is very important in detecting probable candidates for hyperbolic knot complements with hidden symmetries. 
\begin{fact}[Corollary 2.2, \cite{ReWa}] \label{cusp3i}
If a hyperbolic knot complement has hidden symmetries, then its cusp field is $Q(i)$ or $Q(i\sqrt{3})$. 
\end{fact}
Proposition \ref{hscharac} directs the attention to one-rigid cusped orbifolds and their possible (knot complement) covers. Understanding the knot complement covers of rigid cusp orbifolds is a difficult task. However, in a recent result, Hoffman \cite{Hoff_cusp} sheds more light on this. He proved the following. 
\begin{theorem}[Hoffman, Theorem 1.1, \cite{Hoff_cusp}]\label{hoff_cusptype}
A hyperbolic knot complement cannot cover an orbifold with a $(2,4,4)$-cusp. 
\end{theorem}
The above theorem helps streamline the study of hyperbolic knot complements with hidden symmetries to a great degree. We will use this theorem multiple times in our paper. Note that Theorem \ref{hoff_cusptype} also simplifies Fact \ref{cusp3i} into the following fact. 
\begin{fact}
If a hyperbolic knot complement has hidden symmetries, then its cusp field is $Q(i\sqrt{3})$. 
\end{fact}
The lack of examples of hyperbolic knot complements with hidden symmetries led Neumann and Reid to ask the following question which is pivotal in the study of hyperbolic knot complements with hidden symmetries. 
\begin{question}[Question 1, \S 9, \cite{NeRe}]
Is there any hyperbolic knot other than the dodecahedral knots of Aitchison and Rubinstein \cite{AiRu} and the figure-eight knot whose complement has hidden symmetries?
\end{question}

This question seems very challenging and there are no tools or techniques in the literature that provides a good understanding on how to attack this question. In order to get a better grasp of the above Neumann-Reid question, the following weaker conjecture was stated in \cite{CDM}. 

\begin{conjecture}[Conjecture 0.1, \cite{CDM}]
Given a positive number $v$, there are at-most finitely many hyperbolic knot complements admitting hidden symmetries with volume less than $v$. 
\end{conjecture}
This conjectures motivates us to investigate whether a family of hyperbolic knot complements with a volume bound can have infinitely many elements with hidden symmetries. So, in light of Thurston's Dehn surgery theorem \cite[Theorem 5.8.2]{Thurs} (also see \cite[Theorem 5.3]{DunMeyer}) and \cite [Theorem 6.4]{DunMeyer} (see \cite[Theorem 6.5.6]{Thurs} as well), one should focus on the knot complements obtained by Dehn filling all but one cusp of a hyperbolic $3$-manifold with two or more cusps. The following result which was proved in \cite{CDHMMWIMRN} is an extension of the Neumann and Reid characterization from Proposition \ref{hscharac} in this Dehn filling set up.
\begin{theorem}[\cite{CDHMMWIMRN}, Theorem 2.6] \label{CDHMMWorb}
Let $M$ be a hyperbolic $3$-manifold with $n$-cusps where $n \ge 2$. Let $\{M_{\infty, (p_2^i,q_2^i), \dots, (p_n^i, q_n^i)}\}_i$ be a family of hyperbolic orbifolds obtained by Dehn filling all but one cusp of $M$ such that for each $j$ in $\{2, \dots, n\}$, the filling coefficients $(p_j^i,q_j^i) \rightarrow \infty$ as $i \rightarrow \infty$. Suppose that for each $i$, there exists an orbifold cover $\phi_i:M_{\infty, (p_2^i,q_2^i), \dots, (p_n^i, q_n^i)} \rightarrow O_i$ where $O_i$ is a rigid cusped orbifold. Then, there exists a hyperbolic $3$-orbifold $O$ with exactly one rigid cusp and at-least one smooth cusp, and, an orbifold cover $\phi: M \rightarrow O$ such that the only cusp of $M$ which covers the rigid cusp of $O$ is the un-filled cusp and after taking a subsequence we can further say that, 
\begin{enumerate}
\item for each $i$, the orbifold $O_i$ is obtained by some Dehn fillings on all the smooth cusps of $O$, 
\item each $\phi_i$ has the same degree, which is equal to the degree of $\phi$. 
\end{enumerate}
\end{theorem}

The form in which the above result is stated here gives less information than how it is stated in \cite[Theorem 2.6]{CDHMMWIMRN}. The original statement of Part 2 of the above result in \cite[Theorem 2.6]{CDHMMWIMRN} is more involved and technical. But, we will only use the form we state here in the proof of our results in this paper. We make a note that the proof of \cite[Theorem 2.6]{CDHMMWIMRN} builds on machineries developed in \cite{HMW}.

\section{Horoball packings and circle packings associated with an orbifold}\label{packingbackground}
 In this section, we go over the notions related to the horoball packings of $\mathbb{H}^3$ that we need for the rest of the paper. Let $O=\mathbb{H}^3/\Gamma$ be a finite volume hyperbolic $3$-orbifold with cusps such that $\infty$ is a parabolic fixed point of $\Gamma$. Let $c$ be a cusp of $O$. A \textit{horoball packing of $\mathbb{H}^3$} is a union of horoballs such that any two such horoballs are \textit{interior-disjoint}, i.e., they do not intersect each other in the interior. 

Recall from Section \ref{cuspsection} that $c$ is of the form $B/W$ where $B$ is a horoball in $\mathbb{H}^3$ and $W$ the stabilizer of the horocenter of $B$ in $\Gamma$. Then, all the horoballs obtained by translating $B$ by the elements of $\Gamma$ project down to cusp $c$ as well and their union is the pre-image of $c$. 
The fact that the projection of $B$ to $O$ factors through an embedding of $B/W$ implies that the distinct $\Gamma$-translations of $B$ are interior-disjoint from $B$, and no two translates intersect in the interior unless they are equal. So, the pre-image of $c$ becomes a horoball packing of $\mathbb{H}^3$ and this collection of horoballs is $\Gamma$-invariant since they are all $\Gamma$-translates of $B$. Now, if we take multiple cusp neighborhoods of $O$ interior-disjoint from each other in similar manner, then, the pre-image of this union of mutually interior-disjoint cusp neighborhoods is also a horoball packing of $\mathbb{H}^3$, which too is $\Gamma$-invariant. So, we can conclude the following fact. 
  \begin{fact}
 The pre-image in $\mathbb{H}^3$ of a union of mutually interior-disjoint cusp neighborhoods of $O$ is a $\Gamma$-invariant horoball packing of $\mathbb{H}^3$. 
 \end{fact}
 The \textit{$c$-maximal horoball packing of $\mathbb{H}^3$} is a horoball packing of $\mathbb{H}^3$ that maps to a maximal cusp neighborhood of cusp $c$. For the rest of the paper, we will always assume that the horoball at $\infty$ in a $c$-maximal horoball packing of $\mathbb{H}^3$ maps to the $c$-cusp neighborhood. An \textit{$O$-horoball packing of $\mathbb{H}^3$} is a horoball packing $\mathcal{H}$ of $\mathbb{H}^3$ that maps to a union of mutually interior-disjoint cusp neighborhoods of the cusps of $O$ such that for each cusp of $O$ there is a horoball in $\mathcal{H}$ mapping to a neighborhood of that cusp. We will refer to a horoball in an $O$-horoball packing as a \textit{$c$-horoball} if it maps to a cusp neighborhood of $c$. An \textit{$O$-maximal horoball packing of $\mathbb{H}^3$} is an $O$-horoball packing of $\mathbb{H}^3$ which is not contained in a bigger $O$-horoball packing of $\mathbb{H}^3$. A \textit{$(c,O)$-maximal horoball packing of $\mathbb{H}^3$} is a $O$-maximal horoball packing $\mathcal{H}$ of $\mathbb{H}^3$ such that the union of the $c$-horoballs in $\mathcal{H}$ forms the $c$-maximal horoball packing of $\mathbb{H}^3$ (and so, the horoball at $\infty$ in $\mathcal{H}$ maps to the $c$-cusp neighborhood).

Given an $O$-horoball packing $\mathcal{H}$ of $\mathbb{H}^3$ for an orbifold $O=\mathbb{H}^3/\Gamma$, there is a hyperbolic element $g$ in $\operatorname{PSL}(2,\mathbb{C})$ such that $g(H_{\infty})$ is the horoball centered at $\infty$ at height $1$ where $H_{\infty} \in \mathcal{H}$ is the horoball centered at $\infty$. 
Now, $g(\mathcal{H})$ is an $O_g$-horoball packing of $\mathbb{H}^3$ where $O_g=\mathbb{H}^3/g\Gamma g^{-1}$. But, $O$ and $O_g$ are isometric. So, in light of this fact we will always assume that for an orbifold $O$, the horoball centered at $\infty$ of an $O$-horoball packing is at height $1$. We will use the term \text{full sized horoballs} to refer to the horoballs centered at points in $\mathbb{C}$ that have diameter $1$. So, the full sized horoballs are tangent to the horoball at height $1$. 

Let $\mathcal{H}$ be a horoball packing of $\mathbb{H}^3$ contained an $O$-maximal horoball packing of $\mathbb{H}^3$. 
The circle packing of $\mathbb{C}$ obtained by projecting down the boundaries of the full-sized horoballs of $\mathcal{H}$ onto $\mathbb{C}$ is called the \textit{$\mathcal{H}$-circle packing of $\mathbb{C}$}. 
For a cusp $c'$ of $O$, a circle $C$ in the $\mathcal{H}$-circle packing is called a \textit{$c'$-circle} if the horoball in $\mathcal{H}$ projecting down to the disk bounded by $C$ maps to a cusp neighborhood of $c'$ in $O$. If $\mathcal{H}_c$ denotes the $c$-maximal horoball packing of $\mathbb{H}^3$ (so the horoball centered at $\infty$ projects down to a cusp neighborhood of cusp $c$), then, we will refer to the $\mathcal{H}_c$-circle packing as the \textit{$c$-circle packing of $\mathbb{C}$}. So, the elements of the $c$-circle packing of $\mathbb{C}$ are the boundaries of the projections of the full-sized $c$-horoballs in $\mathcal{H}_c$ onto the complex plane. 

Given a horoball packing $\mathcal{H}$ of $\mathbb{H}^3$ contained an $O$-maximal horoball packing of $\mathbb{H}^3$, a symmetry of the $\mathcal{H}$-circle packing is a self-automorphism of the circle packing that sends the $c'$-circles to the $c'$-circles for each cusp $c'$ of $O$.

\begin{remark}
For a hyperbolic link $L$ with components $K_1, \dots, K_n$, we will use the term \textit{$K_j$-circle packing} to mean the $K_j$-cusp circle packing, where $j \in \{1, \dots, n\}$. 
\end{remark}

Let us consider the $c$-maximal horoball packing $\mathcal{H}_c$ of $\mathbb{H}^3$ for $O$. So, the horoball in $\mathcal{H}_c$ centered at $\infty$ projects down to a $c$-cusp neighborhood. The volume of this cusp neighborhood is called the \textit{(maximal) cusp volume} of $c$. $\operatorname{Stab}(\infty)$ in $\Gamma$ acts on the $c$-circle packing as a wallpaper group $W$. The area of the fundamental domains of $W$ is called the \textit{(maximal) cusp area} of $c$. When $O$ is a hyperbolic manifold, $W$ is a lattice wallpaper group and a fundamental domain of $W$ is called a \textit{cusp parallelogram} of $c$. We will use the following fact later.
\begin{fact}
For a cusp $c$ of an orbifold $O$, the cusp area of $c$ is twice the cusp volume of $c$. 
\end{fact}

\subsection{Cusp neighborhoods in SnapPy}

We note here that the $3$-manifold software SnapPy \cite{snappy} denotes the cusps of a hyperbolic $3$-manifold by indices staring from $0$. On SnapPy, one can view the horoball packing of $\mathbb{H}^3$ that come from the cusp neighborhoods of a hyperbolic $3$-manifold. This view of horoball packing is from $\infty$ above and one has the option to select which cusp of the hyperbolic $3$-manifold is to be arranged at $\infty$. The SnapPy horoballs are labelled by different colors. Each color corresponds to a cusp. For example, if a hyperbolic $3$-manifold has seven cusps, then the horoballs projecting down to cusp $0$ is labelled red, cusp $1$ blue, cusp $2$ green, cusp $3$ light blue, cusp $4$ magenta, cusp $5$ yellow, cusp $6$ orange and cusp $7$ purple. One can also (equivariantly) maximize the horoballs of a single cusp until they touch one of their own tangentially. Maximizing (the horoballs corresponding to the) different cusps gives us different pictures of horoball packings. Now, we will discuss some examples of horoball packings using pictures from SnapPy. 

\begin{remark}
In this paper, given a SnapPy manifold $M$ and a cusp $c$ of $M$, pictures of the $c$-maximal or a $(c,M)$-maximal horoball packing of $\mathbb{H}^3$ obtained from SnapPy \cite{snappy} has a $c$-horoball at $\infty$ (eye). 
\end{remark}
\subsection{Examples}

\begin{figure}

\centering 
\captionsetup{justification=centering}
\includegraphics[scale=.14]{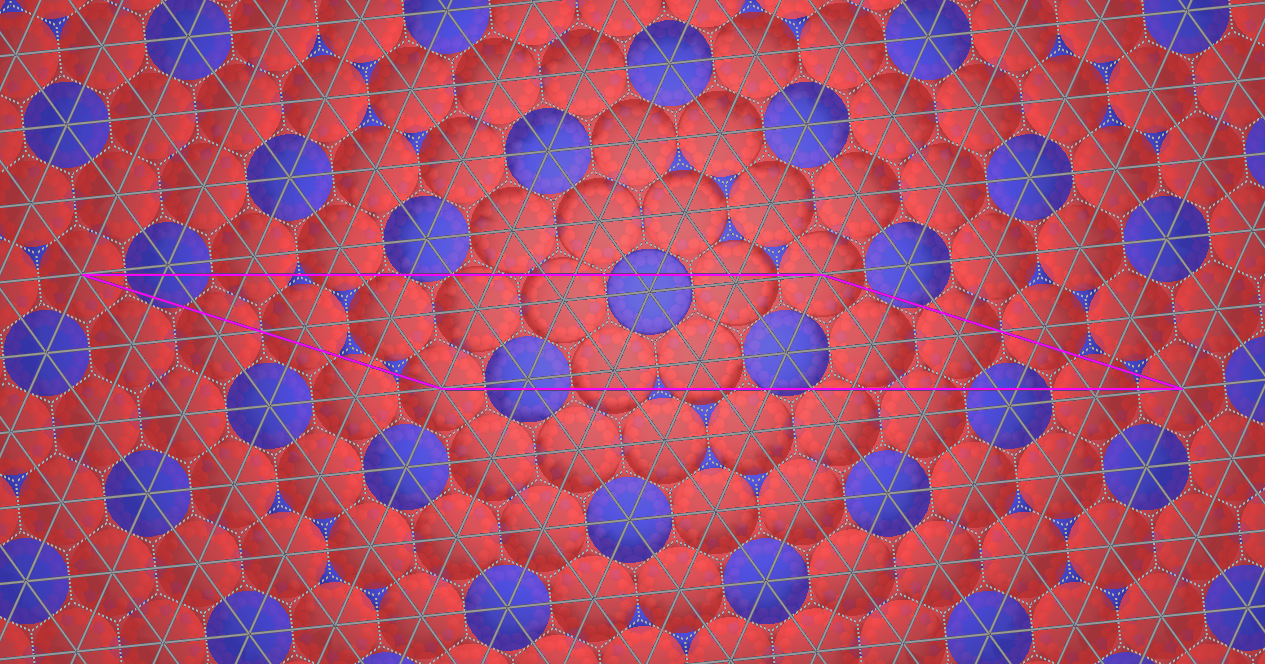} \includegraphics[scale=.14]{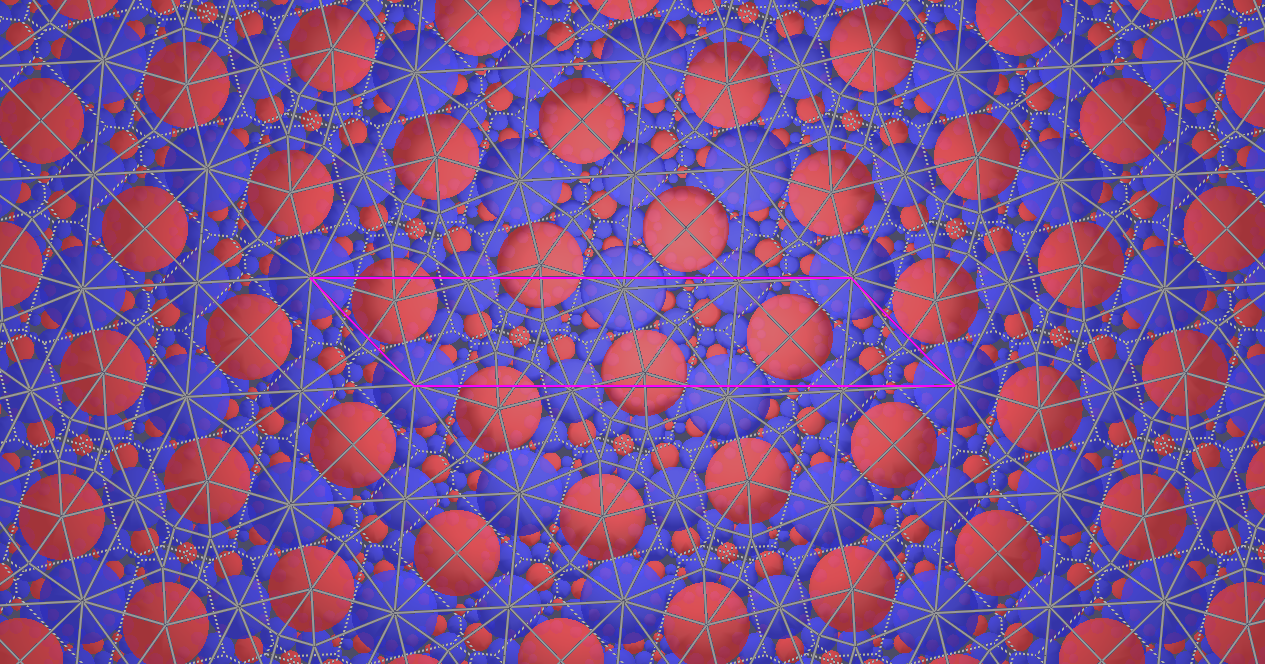}
\caption{Left: $(0, \mathbb{S}^3-\texttt{L14n24613})$-maximal horoball packing of $\mathbb{H}^3$, 
Right: $(1, \mathbb{S}^3-\texttt{L14n24613})$-maximal horoball packing of $\mathbb{H}^3$  (pictures obtained from SnapPy \cite{snappy}).}
\label{L14n24613}
\end{figure}

Figure \ref{L14n24613} shows two $(\mathbb{S}^3-\texttt{L14n24613})$-maximal horoball packings of $\mathbb{H}^3$. \texttt{L14n24613} is a $2$-component hyperbolic link from the Hoste-Thistlewaite census (refer to Figure 3 in \cite{FGGTV} for a picture of the link). The cusps of \texttt{L14n24613} are labelled in red (cusp $0$) and blue (cusp $1$) in Figure \ref{L14n24613}. Left of Figure \ref{L14n24613} shows how the $(0, \mathbb{S}^3-\texttt{L14n24613})$-maximal horoball packing of $\mathbb{H}^3$ looks like from $\infty$, where $\infty$ corresponds to cusp $0$, i.e., the horoball centered at $\infty$ maps down to a cusp neighborhood of cusp $0$. Here the horoball centered at $\infty$ (the horizontal plane) is at Euclidean height $1$. We obtained this maximal horoball packing of $\mathbb{H}^3$ by first (equivariantly) maximizing the horoballs corresponding to cusp $0$, and then,  (equivariantly) maximizing the horoballs corresponding to cusp $1$. Similarly, the picture in the right of Figure \ref{L14n24613} shows how the $(1, \mathbb{S}^3-\texttt{L14n24613})$-maximal horoball packing of $\mathbb{H}^3$ looks like from $\infty$, where $\infty$ corresponds to the cusp $1$. This is obtained by first maximizing the blue horoballs and then the red horoballs. The parallelograms drawn in pink in both pictures are the cusp parallelograms. 

Note that if we only consider the full sized red (respectively, blue) horoballs in the picture in the left (respectively, right) of Figure \ref{L14n24613}, it will give us the $0$-cusp (respectively, $1$-cusp)-circle packing of $\mathbb{C}$.

\section{Hidden symmetry and symmetries of circle packings of $\mathbb{C}$}\label{hscp}
In this section, we extract geometric information from Theorem \ref{CDHMMWorb} in the set up of symmetries of the corresponding horoball packings of $\mathbb{H}^3$. As we will see in the rest of the paper, we can capitalize on this geometric translation of Theorem \ref{CDHMMWorb} to a far greater degree. We begin with a key observation of Neil Hoffman on two component hyperbolic links. 

\subsection{Hoffman's observation}
The following result was communicated by Neil Hoffman in an AIM SQuaRE program in 2019. It aligns with an approach taken by Millichap \cite[Proof of Proposition 7.5]{Millichap} (cf. \cite[Lemma 3.2]{HMW}). We include a proof below. 

\begin{theorem} \label{2comp_sym}
Let $L$ be a $2$-component hyperbolic link with components $K_1$ and $K_2$. Let $\mathcal{F}$ be a family of hyperbolic knot complements obtained from Dehn filling the $K_2$-cusp of $L$ which geometrically converges to $\mathbb{S}^3-L$. If infinitely many elements of $\mathcal{F}$ have hidden symmetries then for any $(\mathbb{S}^3-L)$-maximal horoball packing $\mathcal{H}$ of $\mathbb{H}^3$ containing a $K_1$-horoball at $\infty$, the $\mathcal{H}$-circle packing of $\mathbb{C}$ has a rotational symmetry of order $3$ or $6$ whose fixed point is not the horocenter of any $K_2$-horoball.   
\end{theorem}

\begin{proof}
Let $\mathcal{F}=\left \{\mathbb{S}^3-K'_{\infty, (p^i,q^i)}\right \}_{i\in \mathbb{N}}$ where the knot complement $\mathbb{S}^3-K'_{\infty, (p^i,q^i)}$ is obtained by $(p^i,q^i)$-Dehn filling on the $K_2$-cusp of $L$. By the Neumann-Reid characterization in Proposition \ref{hscharac}, each $\mathbb{S}^3-K'_{\infty, (p^i,q^i)}$ having a hidden symmetry will cover an orbifold $O_i$ with a rigid cusp. Since the knot complements in $\mathcal{F}$ geometrically converge to $\mathbb{S}^3-L$, Theorem \ref{CDHMMWorb} implies that $\mathbb{S}^3-L$ covers a $2$-cusped orbifold $O$ with one rigid cusp and one smooth cusp such that the rigid cusp is covered only by the $K_1$-cusp of $\mathbb{S}^3-L$ and after taking a subsequence, we may assume that each $O_i$ is obtained by Dehn filling the smooth cusp of $O$.

Let us denote the rigid cusp of $O$ by $c$. Since the $K_1$-cusp covers $c$ and the $K_2$-cusp covers the smooth cusp of $O$, $\mathcal{H}$ is also an $O$-maximal horoball packing of $\mathbb{H}^3$. The elements of $\pi_1^{Orb}(O)$ act as symmetries of any $O$-maximal horoball packing of $\mathbb{H}^3$ and so of $\mathcal{H}$. Note that the cusp of $O$ at $\infty$ is $c$ and the elements of $\operatorname{Stab}(\infty)$ in $\pi_1^{Orb}(O)$ act as Euclidean isometries of $\mathbb{C}$.  So, since $\operatorname{Stab}(\infty)$, a subgroup of $\pi_1^{Orb}(O)$, acts as a group of symmetries of $\mathcal{H}$, $\operatorname{Stab}(\infty)$ acts as a group of symmetries of the $\mathcal{H}$-circle packing of $\mathbb{C}$. Now, $O_i$ is covered by the knot complement $\mathbb{S}^3-K'_{\infty, (p^i,q^i)}$. So, by Theorem \ref{hoff_cusptype}, we see that cusp $c$ is $(3,3,3)$ or $(2,3,6)$. So, $\operatorname{Stab}(\infty)$ in $\pi_1^{Orb}(O)$ is $W_{(3,3,3)}$ or $W_{(2,3,6)}$. So, there is an order $3$ or $6$ rotational symmetry $r$ of the $\mathcal{H}$-circle packing of $\mathbb{C}$. 

Now, $r$ is an elliptic element in $\pi_1^{Orb}(O)$ of order $3$ or $6$ which fixes all the points in the geodesic joining its fixed point in $\mathbb{C}$ and $\infty$. So, if the fixed point of $r$ in $\mathbb{C}$ is the center of a $K_2$-horoball, then that would imply that the cross-section of the cusp of $O$ to which that $K_2$-horoball maps has a cone point of order $3$ or $6$ contradicting that it is a smooth cusp of $O$. This completes the proof. 
\end{proof}
\subsection{Hyperbolic link with three or more components}
In general, when a hyperbolic $3$-manifold $M$ covers a hyperbolic three orbiold $O$ and two or more cusps of $M$ map to the same cusp of $O$, the $M$-maximal horoball packing of $\mathbb{H}^3$ might not be an $O$-maximal packing of $\mathbb{H}^3$. This obstructs us in extending Theorem \ref{2comp_sym} for links with three or more components as in the current form. We therefore use the one-to-one nature between the un-filled cusps in the orbifold covering map provided by Theorem \ref{CDHMMWorb} and prove the following result.    

\begin{theorem}\label{hexlatt}
Let $L$ be a hyperbolic link with $n$ components $K_1, K_2, \ldots, K_n$, where $n \ge 2$. If there is an infinite family $\mathcal{F}$ of hyperbolic knot complements with hidden symmetries obtained from Dehn filling all cusps of $L$ but the $K_1$-cusp which geometrically converges to $\mathbb{S}^3-L$, then the following results hold: 
\begin{enumerate}
\item Symmetry group of the $K_1$-circle packing contains a wallpaper group $W$ such that $W$ is $W_{(3,3,3)}$ or $W_{(2,3,6)}$ and the order $3$ or $6$ (if any) elements of $W$ do not fix the horocenter of any $K_j$-horoball for $j \in \{2, \dots, n\}$. Consequently, the $K_1$-circle packing has an order $3$ or $6$ rotational symmetry that does not fix the horocenter of any $K_j$-horoball for $j \in \{2, \dots, n\}$. \label{hexsym}
\item Co-area of the (maximal) translational subgroup of $W=$ $$\frac{\text{maximal cusp area of the }K_1\text{-cusp}}{4m}$$ for some $m\in \mathbb N$. \label{hexarea}
\end{enumerate}
\end{theorem}
\begin{proof}
We can use Proposition \ref{hscharac}, Theorem \ref{CDHMMWorb} and Theorem \ref{hoff_cusptype} to obtain a proof of Part \ref{hexsym} which follows similarly as that of Theorem \ref{2comp_sym}. The only difference is that unlike the $n=2$ case, for the general $n\ge 2$ case, an $(\mathbb{S}^3-L)$-maximal horoball packing might not be an $O$-maximal horoball packing for the orbifold $O$ from Theorem \ref{CDHMMWorb} since more than one cusp of $L$ might map to a single smooth cusp of $O$. But, since only the $K_1$-cusp of $L$ covers the rigid cusp, say $c_{rigid}$, of $O$, the $K_1$-maximal horoball packing is also a $c_{rigid}$-maximal horoball packing. So, arranging the cusp at $\infty$ to be $c_{rigid}$, one can proceed similarly as in the proof of Theorem \ref{2comp_sym} and take $\operatorname{Stab}(\infty)$ in $\pi_1^{Orb}(O)$ as $W$, which is either $W_{(3,3,3)}$ or $W_{(2,3,6)}$, and acts as a group symmetries of the $K_1$-circle packing of $\mathbb C$. Furthermore, for $j=2,\dots,n$, the order $3$ or order $6$ (if any) elements of $W$ can't fix the center of a $K_j$-horoball, otherwise, it would violate the fact that the cusp of $O$ that the $K_j$-cusp of $L$ covers is smooth, where $j \in \{2, \dots, n\}$. So, we are done with Part \ref{hexsym}. 

Before we prove Part \ref{hexarea}, we recall the following two degree formulae of Hoffman which we will use in the proof. 
\begin{theorem}[Hoffman, Part 1 of Lemma 5.5, \cite{Hoff_smallknot}]\label{Hoff_333}
If $\phi$ is an orbifold covering from a manifold $M$ to a $(3,3,3)$-cusped orbifold $O$ such that $M$ is covered by a hyperbolic knot complement , then the index of $\phi$ is $12m$ for some $m \in \mathbb{N}$. 
\end{theorem}
\begin{theorem}[Hoffman, Theorem 1.2, \cite{Hoff_cusp}]\label{Hoff_236}
For an orbifold covering $\phi$ from a manifold $M$ to a $(2,3,6)$-cusped orbifold $O$ where $M$ is covered by a hyperbolic knot complement, the index of $\phi$ is $24m$ for some $m \in \mathbb{N}$. 
\end{theorem}

Let $\Lambda$ denote the (maximal) translational subgroup of $W$. We note that the maximal cusp area of $c_{rigid}$ is equal to the co-area of $W$. If $c_{rigid}$ is  $(3,3,3)$, then using the degree formula in Theorem \ref{CDHMMWorb} and applying Theorem \ref{Hoff_333} we get, $\operatorname{degree}(\mathbb{S}^3-L \rightarrow O)= 12m$ for some $m \in \mathbb N$. Now, the only cusp of $L$ that covers $c_{rigid}$ is the $K_1$-cusp. So, the maximal cusp area of $c_{rigid}$ is 
$$\frac{\text{maximal cusp area of the }K_1\text{-cusp}}{12m}.$$ 
But, this is same as $\frac{\text{co-area of }\Lambda}{3}$. This implies 
$$\text{co-area of }\Lambda =\frac{\text{maximal cusp area of the }K_1\text{-cusp}}{4m}.$$ 
If the rigid cusp of $O$ is $(2,3,6)$, then the degree formula in Theorem \ref{CDHMMWorb} and Theorem \ref{Hoff_236} implies that $\operatorname{degree}(\mathbb{S}^3-L \rightarrow O)=24m$ for some $m \in \mathbb N$. So, 
$$\frac{\text{maximal cusp area of the }K_1\text{-cusp}}{24m}=\text{maximal cusp area of }c_{rigid}=\frac{\text{co-area of }\Lambda}{6}.$$
Thus, co-area of $\Lambda$=$\frac{\text{maximal cusp area of the }K_1\text{-cusp}}{4m}$ in this case as well. This ends the proof. 
 \end{proof}
 \begin{remark}\label{hexlattrem}
 Note that the wallpaper group $W_{(2,3,6)}$ contains the wallpaper group $W_{(3,3,3)}$ such that both  $W_{(2,3,6)}$ and $W_{(3,3,3)}$ have the same translational subgroup. Also, the non-identity elements of $W_{(3,3,3)}$ having fixed points are the order $3$ rotations. So, Theorem \ref{hexlatt} can be reformulated as the version mentioned in the introduction. 
  \end{remark}
 We emphasize the following crucial fact which we need later. 
 \begin{fact}\label{hexsymmhoro}
 We note from the proof of Theorem \ref{hexlatt} that the elements of the $W_{(2,3,6)}$ or $W_{(3,3,3)}$ wallpaper group of symmetries of the $K_1$-circle packing are also symmetries of the $K_1$-maximal horoball packing $\mathcal{H}$ of $\mathbb{H}^3$ (with a $K_1$-horoball at $\infty$). So, in $\mathcal{H}$, they must send a non-full sized $K_1$-horoball to another $K_1$-horoball of the same (Euclidean) size.
 \end{fact}
 
We described Theorem \ref{hexlatt} in terms of the $K_1$-circle packing since it will be easier for us to code the theorem into an algorithm and use the code for a large collection of links.  We develop the coding mechanism in Section \ref{hsalgorithm}. 
\subsection{Examples}\label{hexthmeg}
\begin{figure}
\centering 
\captionsetup{justification=centering}
\includegraphics[scale=.14]{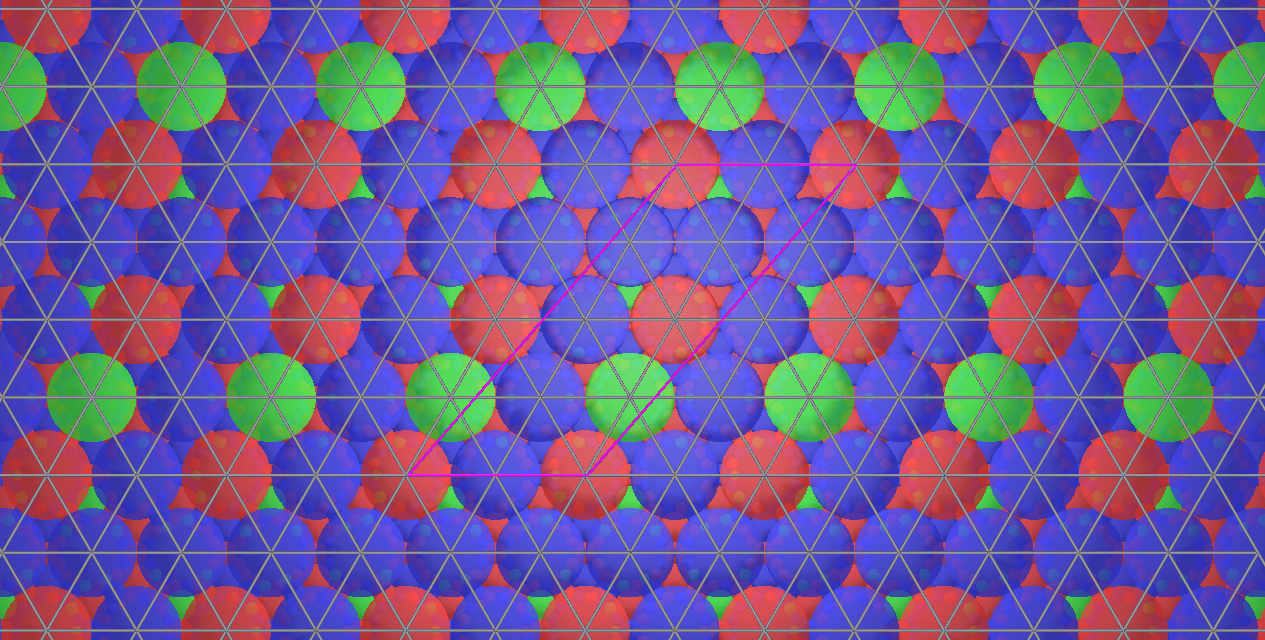}
\caption{$(0, \mathbb{S}^3-\texttt{L11n354})$-maximal horoball packing of $\mathbb{H}^3$ (picture obtained from SnapPy \cite{snappy}).}
\label{L11n354red}
\end{figure}
We now discuss some examples of circle packings coming from the links \texttt{L11n354}, \texttt{L10n101}, \texttt{L8a20} and \texttt{L14n24613} from the Hoste-Thistlewaite census (these links are \textit{tetrahedral links}, see Section \ref{tlcom}), and comment on whether the required order $3$ or $6$ symmetries from Theorem \ref{hexlatt} exist for them. See Figure 3 and Figure 4 in \cite{FGGTV} for pictures of these links. 

Figure \ref{L11n354red} shows a $(0, \mathbb{S}^3-\texttt{L11n354})$-maximal horoball packing of $\mathbb{H}^3$ where $\infty$ corresponds to cusp $0$. We can see that $0$-cusp circle packing of $\mathbb{C}$ in Figure \ref{L11n354red} has no order $3$ or $6$ rotational symmetry. 

\begin{figure}
\centering
\captionsetup{justification=centering}
 \includegraphics[scale=.125]{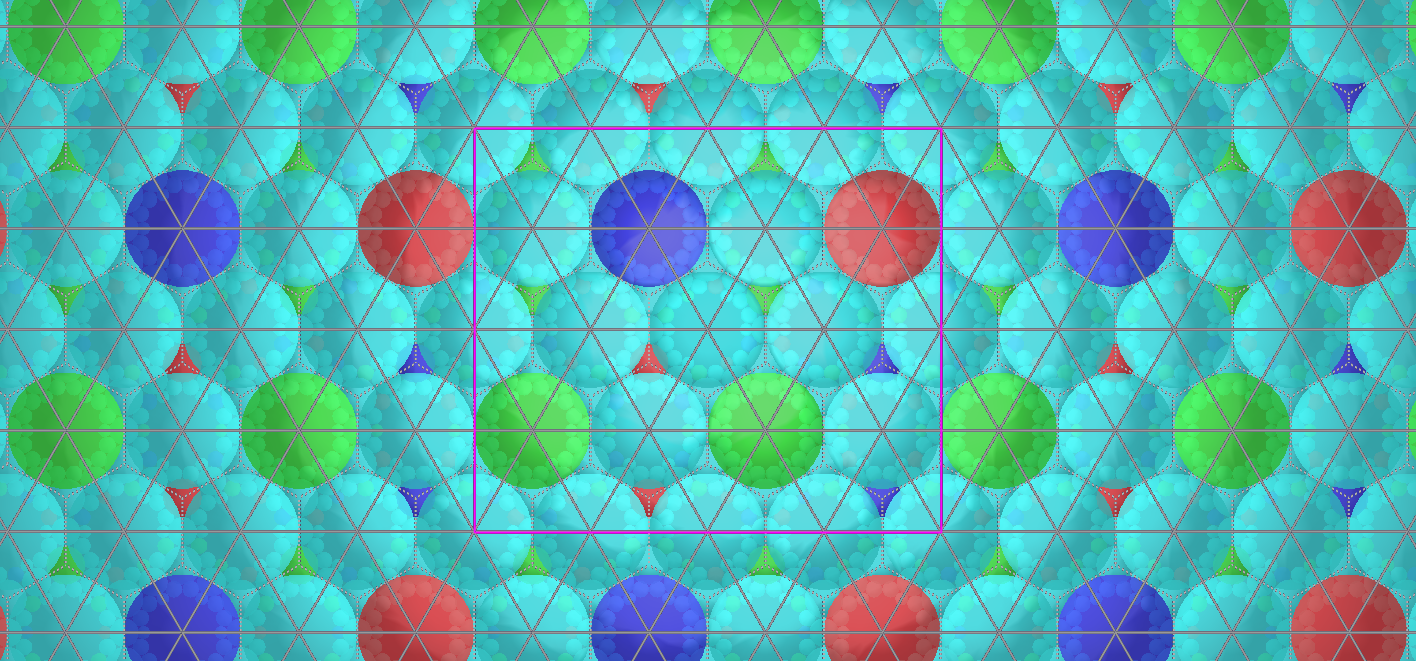}
\caption{$(3, \mathbb{S}^3-\texttt{L10n101})$-maximal horoball packing of $\mathbb{H}^3$ (picture obtained from SnapPy \cite{snappy}).}
\label{L10n101_3}
\end{figure}
A $(3, \mathbb{S}^3-\texttt{L10n101})$-maximal horoball packing of $\mathbb{H}^3$ is shown in Figure \ref{L10n101_3} where $\infty$ is at cusp $3$. The $3$-cusp circle packing in Figure \ref{L10n101_3} has rotational symmetries of order $6$ (and so of order $3$ as well). But, note that these symmetries fix the centers of red, blue or green horoballs. 

\begin{figure}
\centering
\captionsetup{justification=centering}
\includegraphics[scale=.14]{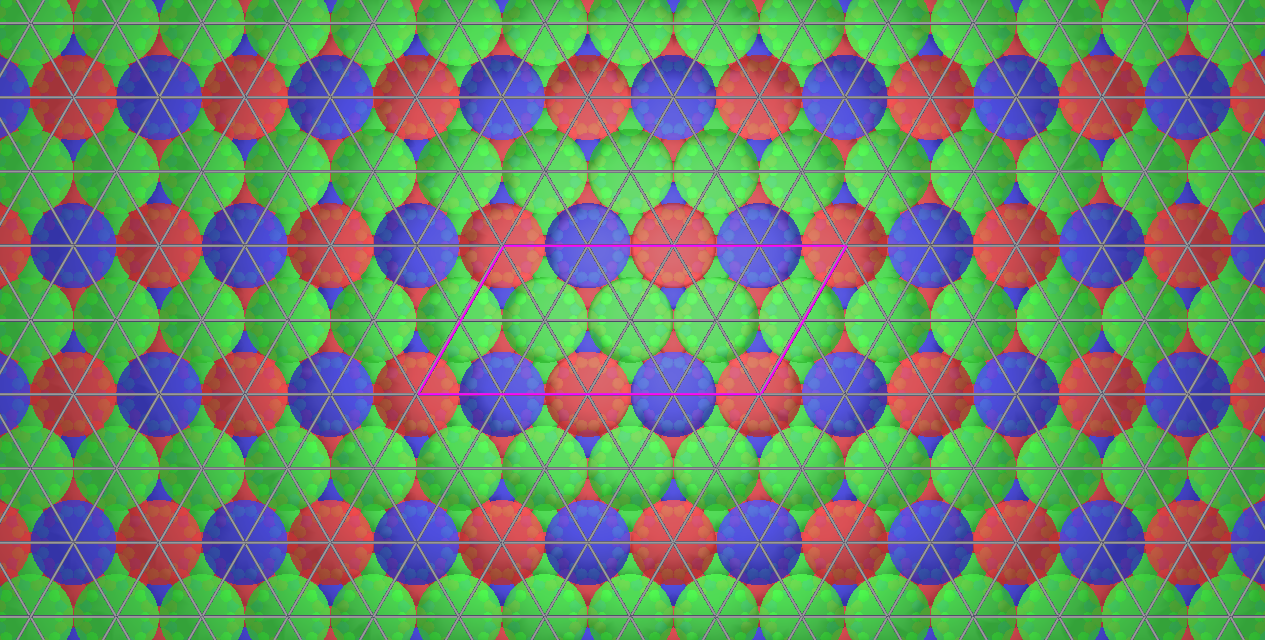}
\caption{$(0, \mathbb{S}^3-\texttt{L8a20})$-maximal horoball packing of $\mathbb{H}^3$ (picture obtained from SnapPy \cite{snappy}).}
\label{L8a20red}
\end{figure}
If we analyze Figure \ref{L8a20red}, which depicts the view of a $(0, \mathbb{S}^3-\texttt{L8a20})$-maximal horoball packing of $\mathbb{H}^3$ from cusp $0$ at $\infty$, we see that the $0$-cusp circle packing of $\mathbb{C}$ has order $3$ as well as order $6$ rotational symmetries which do not fix the centers of horoballs corresponding to the blue or green cusp. One can see that the (maximal) cusp area of cusp $0$ in $\texttt{L8a20}$ is $16 a$ where $a=\frac{\sqrt{3}}{4}$. But, the $0$-cusp circle packing does not have a hexagonal lattice symmetry of co-area less than or equal to $4a$ and so cannot have a hexagonal lattice symmetry of co-area of the form $\frac{16a}{4m}$ for some $m \in \mathbb{N}$. 

If we consider the horoball packings that we have seen in Figure \ref{L14n24613}, we see that neither the $0$-cusp circle packing of $\mathbb{C}$ in left of Figure \ref{L14n24613} nor the $1$-cusp circle packing of $\mathbb{C}$ has any order $3$ rotational symmetry. 

\subsection{Two corollaries}
One can derive Corollary \ref{nearrot} and \ref{neardistinctrot} below from Theorem \ref{hexlatt}. Corollary \ref{neardistinctrot} will play a vital role in the development of the algorithm discussed in Section \ref{hsalgorithm}. 
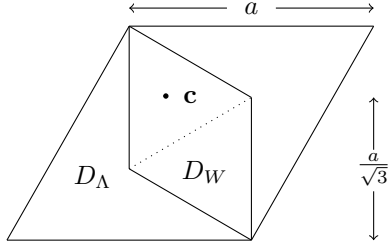
\begin{figure}
\centering 
\captionsetup{justification=centering}
\begin{tikzpicture}[scale=3.25]
\draw (0,0)--(1,0);
\draw (0,0)--(.5,.876);
\draw (.5, .876)--(1.5, .876);
\draw (1,0)--(1.5, .876);
\draw (1,0)--(.5, .292);
\draw (.5, .292)--(.5, .876);
\draw (1,0)--(1, .584);
\draw (1, .584)--(.5, .876);
\draw[dotted] (.5, .292)--(1, .584);
\filldraw[black] (.65,.59) circle (.25pt);

\node at (.75, .584){$\mathbf{c}$};
\node at (.81, .28){$D_{W}$};
\node at (.35, .26){$D_{\Lambda}$};

\draw [<-] (.5, .95)--(.9, .95);
\node at (1, .95){$a$};
\draw [->] (1.1, .95)--(1.5, .95);
\draw [<-] (1.5, .584)--(1.5, .37);
\node at (1.5, .29){$\frac{a}{\sqrt{3}}$};
\draw [->] (1.5, .2)--(1.5, 0);

\end{tikzpicture}

\caption{Order $3$ rotational symmetries with fixed point near $\mathbf{c}$.}
\label{coropic}
\end{figure}
\begin{corollary}\label{nearrot}
Let $L=K_1\sqcup K_2\ldots \sqcup K_n$ be an $n$-component hyperbolic link where $n \ge 2$. Let $\mathcal{F}$ be an infinite family of hyperbolic knot complements with hidden symmetries obtained by Dehn filling all cusps of $L$ but the $K_1$-cusp and geometrically converging to $\mathbb{S}^3-L$. Then for each circle $C$ in the $K_1$-circle packing of $\mathbb{C}$, there is an order $3$ or $6$ rotational symmetry $r$ of the $K_1$-circle packing of $\mathbb{C}$ that does not fix the horocenter of any $K_j$-horoball for $j \in \{2,\ldots,n\}$ such that the distance between the center of $C$ and the fixed point of $r$ is less than or equal to $\frac{\sqrt{\text{maximal cusp area of the }K_1\text{-cusp}}}{\sqrt{2}3^{\frac{5}{4}}}$. 
\end{corollary}
\begin{proof}
Since $\mathcal{F}$ has infinitely many hyperbolic knot complements with hidden symmetries which geometrically converge to $\mathbb{S}^3-L$, Theorem \ref{hexlatt} tells us that the symmetry group of the $K_1$-circle packing of $\mathbb{C}$ contains a wallpaper group $W$ that is either $W_{(3,3,3)}$ or $W_{(2,3,6)}$ and satisfies the properties laid out in that theorem. 
Let $\Lambda$ be the (maximal) translational subgroup of $W$ and $\mathbf{c}$ be the center of $C$. Then $\mathbf{c}\in D_W \subset D_{\Lambda}$ where $D_{\Lambda}$ and $D_W$ are some fundamental domains of $\Lambda$ and $W$ respectively. $D_{\Lambda}$ is the region bounded by a hexagonal rhombus. If the length of the sides of $D_{\Lambda}$ is $a$, then by Part \ref{hexarea} of Theorem \ref{hexlatt}, 
\begin{equation*}
\frac{\sqrt{3} a^2}{2}=2 \left(\frac{\sqrt{3} a^2}{4}\right) \le \frac{\text{maximal cusp area of the }K_1\text{-cusp}}{4}.
\end{equation*}
This means $a\le \frac{\sqrt{\text{maximal cusp area of the }K_1\text{-cusp}}}{\sqrt{2}3^{\frac{1}{4}}}$. Now, the region $D_W$ is also bounded by a hexagonal rhombus such that the length of the sides of $D_W$  is $\frac{2}{3}\frac{\sqrt{3}a}{2}=\frac{a}{\sqrt{3}}$ (see Figure \ref{coropic}). 
So, given any point in $D_W$, there is at-least one vertex of $D_W$ whose distance from that point is less than or equal to $\frac{2}{3} \frac{\sqrt{3}}{2}\frac{a}{\sqrt{3}}=\frac{a}{3}$ since this given point would belong to one of the equilateral triangles divided by one of $D_W$'s diagonal as shown in Figure \ref{coropic}. But, $\frac{a}{3} \le \frac{\sqrt{\text{maximal cusp area of the }K_1\text{-cusp}}}{\sqrt{2}3^{\frac{5}{4}}}$ and each vertex of $D_W$ is the fixed point of a rotation $r$ in $W$ of order $3$ or $6$. Moreover, $r$ does not fix the center of a $K_j$-horoball for $j \in \{2, \dots, n\}$ . That completes the proof.

\end{proof}
\begin{remark}\label{order3rotation}
Note that the fixed point of $r$ in the above corollary is always the fixed point of an order $3$ rotational symmetry of the $K_1$-circle packing of $\mathbb{C}$ (see Remark \ref{hexlattrem}). 
\end{remark}
We could argue similarly as in the proof above to get slightly different distance bounds as in the following result. Recall from Section \ref{cuspsection} that $r_{n,\mathbf{x}}$ denotes order $n$ rotational symmetry of the complex plane with fixed point $\mathbf{x}$. 
\begin{corollary}\label{neardistinctrot}
Let $L=K_1\sqcup K_2\ldots \sqcup K_n$ be an $n$-component hyperbolic link where $n \ge 2$. Let $\mathcal{F}$ be an infinite family of hyperbolic knot complements with hidden symmetries obtained by Dehn filling all cusps of $L$ but the $K_1$-cusp and geometrically converging to $\mathbb{S}^3-L$. Let $\mathbf{c}$ denote the center of a circle $C$ in the $K_1$-circle packing of $\mathbb{C}$. Then the $K_1$-circle packing of $\mathbb{C}$ has symmetries $r_{3,\mathbf{p_1}}$ and $r_{3,\mathbf{p_2}}$ neither fixing the horocenter of any $K_j$-horoball for $j \in \{2,\ldots,n\}$ such that $\mathbf{p_1}\ne\mathbf{p_2}$ and $0 <\lVert \mathbf{c}-\mathbf{p}_m\rVert\le \frac{\sqrt{\text{maximal cusp area of the }K_1\text{-cusp}}}{\sqrt{2}3^{\frac{3}{4}}}$ for each $m=1,2$. 
\end{corollary}

\section{Hidden symmetries, covering maps and horoball packings}\label{coversection}
In this section, we explore some further implications of Theorem \ref{CDHMMWorb}. We begin with a proposition which is a Dehn filling version of \cite[Theorem 4.1]{CDHMMWIMRN} for link complements. \cite[Theorem 4.1]{CDHMMWIMRN} says that if a manifold with volume less than or equal to $6v_0$ is covered by a hyperbolic knot complement with hidden symmetries, then the knot is the figure eight knot. The following proposition says that this volume bound can be extended to $8v_0$ in the Dehn filling set up of a link complement. 

\begin{proposition}\label{8v0prop}
\eightvprop
\end{proposition}
\begin{proof}
Let $L$ be a hyperbolic link with two or more components such that the volume of $\mathbb{S}^3-L$ is less than or equal to $8v_0$. Suppose for contradiction we have an infinite family $\mathcal{F}$ of hyperbolic knot complements with hidden symmetries obtained from Dehn filling all but a fixed cusp of $L$ which geometrically converges to $\mathbb{S}^3-L$. If we index the knots whose complements belong to $\mathcal{F}$ by $K'_i$ (i.e. $\mathcal{F}=\left \{\mathbb{S}^3-K'_i\right \}_i$),  then, by the Neumann-Reid characterization in Proposition \ref{hscharac} and Hoffman's result in Theorem \ref{hoff_cusptype}, we can see that the complement of each $K'_i$ covers an orbifold $O_i$ with a $(3,3,3)$ or $(2,3,6)$ cusp via some covering map $\phi_i: \mathbb{S}^3-K'_i \rightarrow O_i$. Now, we can use Theorem \ref{CDHMMWorb} to see that there is an orbifold covering $\phi:\mathbb{S}^3-L \rightarrow O$ where orbifold $O$ has exactly one rigid cusp $c_{rigid}$ and at-least one smooth cusp such that $c_{rigid}$ is only covered by the un-filled cusp of $L$ and after taking a subsequence one can say that each $O_i$ is obtained by Dehn filling the smooth cusps of $O$ and degree of each $\phi_i$ is equal to the degree of $\phi$. 

Now if $c_{rigid}$ is $(3,3,3)$ (respectively $(2,3,6)$), then, by Hoffman's degree formula in Theorem \ref{Hoff_333} (respectively Theorem \ref{Hoff_236}), we see that the degree of $\phi$ is $12m$ (respectively $24m$) for some $m \in \mathbb{N}$. Then, $O$ is a multi-cusped orbifold with a $(3,3,3)$ cusp (respectively $(2,3,6)$ cusp) and volume less than $\frac{8v_0}{12m} \le \frac{2v_0}{3}$ (respectively, $\frac{8v_0}{24m} \le \frac{v_0}{3}$). Then, \cite[Lemma 2.3]{Adams_multi} (respectively, \cite[Lemma 2.2]{Adams_multi}) implies that $O$ is a two cusped orbifold with two $(3,3,3)$ cusps (respectively, two $(2,3,6)$ cusps). This contradicts that $O$ has at-least one smooth cusp. So, we conclude that no such family $\mathcal{F}$ can exist, which completes the proof. 
\end{proof}

Theorem \ref{CDHMMWorb} also alludes us to investigate the hyperbolic orbifolds with exactly one rigid cusp and at-least one smooth cusp, and the one rigid-cusped orbifolds obtained by Dehn filling all of their smooth cusps. Unfortunately, we do not know many orbifolds with exactly one rigid cusp and at-least one smooth cusp. However, it is prudent to study the ones that are low-volume as we could potentially leverage volume obstructions in the study of their Dehn fillings. Tyler Gaona in his PhD thesis \cite{GTthesis} examined a $2$-cusped orbifold of volume $\frac{5v_0}{6}$ with one $(2,3,6)$ cusp and one pillowcase cusp. Its isotropy graph can be seen from the picture in the right in \cite[Figure 5]{GTthesis}. In this paper we will refer to this orbifold as $O_{(2,3,6),(2,2,2,2)}$. In fact, Gaona proved in \cite[Theorem 3.2]{GTthesis} that $O_{(2,3,6),(2,2,2,2)}$ has the least volume amongst the hyperbolic $3$-orbifolds with one $(2,3,6)$-cusp and one smooth cusp. Proposition \ref{coversmallestmulticusp} below shows that covering $O_{(2,3,6),(2,2,2,2)}$ with only one cusp mapping to the $(2,3,6)$ cusp of $O_{(2,3,6),(2,2,2,2)}$ obstructs a link complement from being a good candidate for finding infinitely many knot complements with hidden symmetries via (appropriate) Dehn fillings. 

We first need the following lemma which is set in the context of (Dehn fillings of) orbifolds and can be obtained by combining Hoffman's work from \cite[Section 2.3]{Hoff_smallknot}, \cite[Proof of Part 1 of Lemma 5.5]{Hoff_smallknot} and  \cite[Theorem 1.2]{Hoff_cusp}. Hence, we attribute it to Hoffman.

\begin{lemma}[Hoffman]\label{isovertex}
Let $O$ be a hyperbolic $3$-orbifold with exactly one rigid cusp and at-least one smooth cusp. Suppose there exists a one-cusped hyperbolic $3$-orbifold obtained by Dehn filling all the smooth cusps of $O$ that is covered by a hyperbolic knot complement. Then the isotropy graph of $O$ has a finite vertex $p$ of the type  $(2,3,3)$ or $(2,3,4)$ or $(2,3,5)$, i.e., the isotropy group $G_p$ of $p$ in $\pi_1^{Orb}(O)$ is either $A_4$ or $S_4$ or $A_5$.
\end{lemma}

\begin{proof}

\begin{figure}
\centering 
\captionsetup{justification=centering}
\includegraphics[scale=.5]{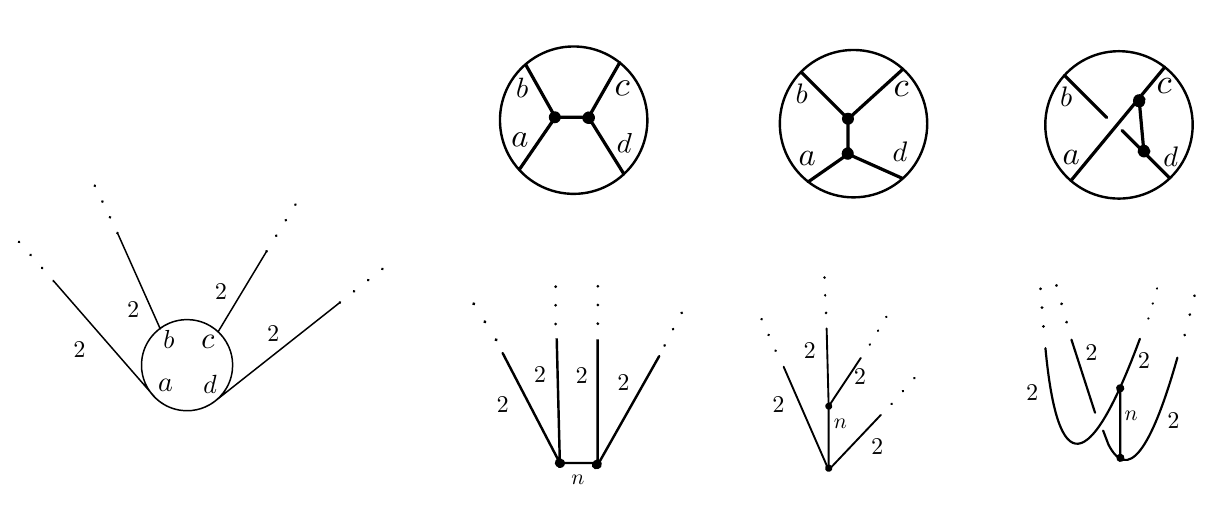}
\caption{First: Neighborhood of a pillowcase cusp in the isotropy graph, Second: Type I filling (top), Local picture of the isotropy graph after type I filling (bottom), Third: Type II filling (top), Local picture of the isotropy graph after type II filling (bottom), Fourth: Type III filling (top), Local picture of the isotropy graph after type III filling (bottom).}
\label{pillowcasefilling}
\end{figure}

Let $O'$ be a one-cusped hyperbolic $3$-orbifold obtained by Dehn filling all the smooth cusps of $O$ that is covered by a hyperbolic knot complement $\mathbb{S}^3-K$. Hoffman's result in Theorem \ref{hoff_cusptype} implies that the cusp of $O'$ is of type $(3,3,3)$ or $(2,3,6)$. If $O'$ has a $(2,3,6)$-cusp, then \cite[Theorem 1.2]{Hoff_cusp} implies that $\mathbb{S}^3-K$ also covers the double cover of $O'$ which has a $(3,3,3)$ cusp. Define $O''$ to be this double cover when $O'$ has a $(2,3,6)$ cusp, and $O'$ when $O'$ has a $(3,3,3)$ cusp. Since the cusp of $O''$ is $(3,3,3)$ in either case and $O''$ is covered by a hyperbolic knot complement, the argument using cusp killing homomorphism (see \cite[Section 2.3]{Hoff_smallknot}) given in \cite[Proof of Part 1 of Lemma 5.5]{Hoff_smallknot} tells us that the isotropy graph of $O''$ will have a finite vertex $p''$ of type $(2,3,3)$ or $(2,3,4)$ or $(2,3,5)$ connected to the $(3,3,3)$-cusp by an edge $e''$ of label $3$.  

First consider the case when $O'$ has a $(2,3,6)$-cusp. Since neither of $A_4$, $S_4$ and $A_5$ is a subgroup of a finite dihedral group, image of $p''$ in $O'$ is also a finite vertex of the type $(2,3,3)$ or $(2,3,4)$ or $(2,3,5)$. Moreover, the image of $e''$ in isotropy graph of $O'$ joins the cusp of $O'$ and the image of $p''$  via a collection of edges whose labels are multiples of $3$. So, if there are finite vertices in the image of $e''$ other than the image of $p''$, they would be of the type $(2,3,3)$. Let $p'$ denote the first finite vertex in the image of $e''$ which is joined to the cusp of $O'$ (note that $p'$ could be the image of $p''$). 

When $O'$ has $(3,3,3)$-cusp, take $p'=p''$. This means, $O'$ always has a finite vertex $p'$ of the type $(2,3,3)$ or $(2,3,4)$ or $(2,3,5)$ which is connected to its cusp by an edge of its isotropy graph. We observe that the isotropy graph of $O'$ is obtained from the isotropy graph of $O$ by replacing the neighborhoods of all the pillowcase cusp vertices (if any) in one of the three ways as shown in Figure \ref{pillowcasefilling}. In particular, if all the smooth cusps of $O$ are torus cusps, then the isotropy graph of $O'$ is same as the isotropy graph of $O$. In all of these cases of Figure \ref{pillowcasefilling}, existence of a finite vertex $p'$ in the isotropy graph of $O'$ implies that the isotropy graph of $O$ has a finite vertex $p$ of the type  $(2,3,3)$ or $(2,3,4)$ or $(2,3,5)$, i.e., the isotropy group $G_p$ of $p$ in $\pi_1^{Orb}(O)$ is either $A_4$ or $S_4$ or $A_5$.

\end{proof}

\begin{proposition}\label{coversmallestmulticusp}
\coversmallestmulticuspthm
\end{proposition}
\begin{proof}
To prove the first part of the proposition, we observe that since $c$ and $c'$ are symmetric, it is enough to prove that the $c$-circle packing has a $W_{(2,3,6)}$ group of symmetries whose elements of order $3$ or $6$ do not fix the center of a $K$-horoball where $K$ is a cusp of $L$ different from $c$. Now, $c$ is the only cusp of $L$ covering $c_{(2,3,6)}$ and the other cusp of $O_{(2,3,6), (2,2,2,2)}$ is a smooth pillowcase cusp. Hence, we can argue as in the proof of Theorem \ref{2comp_sym} and Theorem \ref{hexlatt} to see that the peripheral subgroup of $\pi_1^{Orb}(O_{(2,3,6), (2,2,2,2)})$ corresponding $c_{(2,3,6)}$ is $W_{(2,3,6)}$ which acts as a group of symmetries on the $c_{(2,3,6)}$-circle packing (which is same as the $c$-circle packing) of $\mathbb{C}$ and its order $3$ or $6$ elements do not fix the centers of horoballs corresponding to the cusps of $L$ different from $c$. 

We now prove the second part. As in the first part, since $c$ and $c'$ are symmetric, it is enough to show that each family $\mathcal{F}$ of hyperbolic knot complements obtained by Dehn filling all cusps of $L$ but $c$ and geometrically converging to $\mathbb{S}^3-L$ contains at-most finitely many elements with hidden symmetries (see Fact \ref{symmfact}). Assume for contradiction that such a family $\mathcal{F}$ contains infinitely many knot complements $\{\mathbb{S}^3-K'_i\}_i$ with hidden symmetries. 

As in the proof of Theorem \ref{2comp_sym}, Theorem \ref{hexlatt} and Proposition \ref{8v0prop}, we can then use Theorem \ref{CDHMMWorb} along with the Neumann and Reid characterization in Proposition \ref{hscharac} and Hoffman's result in Theorem \ref{hoff_cusptype} to see that there exists orbifold covering maps $\phi_i: \mathbb{S}^3-K'_i \rightarrow O_i$ and $\phi:\mathbb{S}^3-L \rightarrow O$ where $O$ is an orbifold with exactly one rigid cusp $c_{rigid}$ of type $(3,3,3)$ or $(2,3,6)$ and at-least one smooth cusp such that $c_{rigid}$ is only covered by cusp $c$ of $L$, and after taking a subsequence, the orbifolds $O_i$ are obtained by Dehn filling the smooth cusps of $O$ and furthermore, degree of $\phi_i$ is equal to the degree of $\phi$ for all $i$. 

Let $\mathcal{H}$ denote the $c$-maximal horoball packing of $\mathbb{H}^3$. Let $\Gamma$ and $\Gamma_O$ denote the orbifold fundamental groups of $O_{(2,3,6), (2,2,2,2)}$ and $O$ respectively. Note that $\mathcal{H}$ is both the $c_{rigid}$- and the $c_{(2,3,6)}$- maximal horoball packings of $\mathbb{H}^3$ and that both $\Gamma$ and $\Gamma_O$ are subgroups of the symmetry group of $\mathcal{H}$. Let 
$$\Gamma_1=\left \langle \Gamma, \Gamma_O \right \rangle.$$
Then, we can apply \cite[Lemma 2.1 and Proof of Lemma 2.2]{GoHeHo} to see that $\Gamma_1$ is a Kleinian group. (We emphasize here that we need to use a slightly
tweaked version of \cite[Lemma 2.2]{GoHeHo} via an identical proof therein to see that the parabolic fixed points corresponding to cusp $c$ span $\mathbb{H}^3$.)

 Now since $\Gamma$ and $\Gamma_O$ are commensurable, they have the same parabolic fixed points. Now, the sets of parabolic fixed points corresponding to $c_{(2,3,6)} $ and $c_{rigid}$ are same as this is the set of parabolic fixed points of $c$ as well. Call this set $P_r$. Since both $\Gamma$ and $\Gamma_O$ have only one rigid cusp, the total set of parabolic fixed points corresponding to all the smooth cusps of $O_{(2,3,6)}$ and $O$ are same as well. Call this set $P_s$. Since both $P_r$ and $P_s$ are invariant under the action of both $\Gamma$ and $\Gamma_O$, they are invariant under the action of $\Gamma_1$. So, $\mathbb{H}^3/\Gamma_1$  has at-least two cusps. This means that $\mathbb{H}^3/\Gamma_1$ has exactly two cusps since $\mathbb{H}^3/\Gamma_1$ is covered by the two-cusped orbifold $O_{(2,3,6), (2,2,2,2)}$. Also note that the cusp of $\mathbb{H}^3/\Gamma_1$ covered by $c_{(2,3,6)}$ is a $(2,3,6)$ cusp. 

Since $O_i$'s are covered by the knot complements $\mathbb{S}^3-K'_i$, Lemma \ref{isovertex} implies that the isotropy graph of any orbifold that $O$ covers has a finite vertex of type $(2,3,3)$ or $(2,3,4)$ or $(2,3,5)$. We now observe from \cite[Figure 5]{GTthesis} that the isotropy graph of $O_{(2,3,6), (2,2,2,2)}$ does not contain any finite vertex of type $(2,3,3)$ or $(2,3,4)$ or $(2,3,5)$. So, $O_{(2,3,6), (2,2,2,2)}$ is not covered by $O$. So, $\Gamma_O$ is not a subgroup of $\Gamma$ and consequently, $\Gamma_1$ properly contains $\Gamma$. So, since the volume of $O_{(2,3,6), (2,2,2,2)}$ is $\frac{10v_0}{12}$, the volume of $\mathbb{H}^3/\Gamma_1$ is equal to $\frac{10v_0}{12m}$ for some integer $m \ge 2$. Since $\mathbb{H}^3/\Gamma_1$ is a two cusped orbifold with a $(2,3,6)$ cusp, \cite[Lemma 2.2]{Adams_multi} implies that $\mathbb{H}^3/\Gamma_1$ is one of the unique two cusped orbifolds with a $(2,3,6)$ cusp of volume $\frac{v_0}{3}$ or $\frac{5v_0}{12}$. Since $m$ is an integer, the volume of $\mathbb{H}^3/\Gamma_1$ must be  $\frac{5v_0}{12}$, i.e., $m=2$. Hence, $\mathbb{H}^3/\Gamma_1$ is the unique two-cusped orbifold of volume $\frac{5v_0}{12}$ with two $(2,3,6)$ cusps given in  Adams \cite{Adams_multi}. Since $\mathbb{H}^3/\Gamma_1$ is covered by the two cusped orbifold $O_{(2,3,6), (2,2,2,2)}$ with a $(2,2,2,2)$ cusp, $m$ has to be at-least be $3$. This contradicts that $m=2$. So, $\mathcal{F}$ could have at-most finitely many knot complements with hidden symmetries. 
\end{proof} 

In the remaining of this section, we show that the converse of Theorem \ref{hexlatt} is not true. In particular, we will see that for a cusp of a hyperbolic link with two or more components and volume less than or equal to $24v_0$, the existence of some order $6$ symmetries of the full sized horoballs corresponding to that cusp is actually the crux of the issue. We lay out this phenomenon in Proposition \ref{two236cusps} and Corollary \ref{badorder6sym}. One can see some explicit examples (of tetrahedral homology links) pertaining to this situation later in Subsection \ref{snappyrun}. 
Before we state and prove Proposition \ref{two236cusps}, we record a lemma that concerns a class of orbifolds covering the unique multi-cusped orbifolds of volume $\frac{v_0}{3}$ and $\frac{v_0}{2}$ with a $(2,3,6)$ cusp discussed in Adams \cite{Adams_multi}. 

\begin{lemma}\label{coveradams}
Let $O$ be a hyperbolic $3$-orbifold with exactly one rigid cusp and at-least one smooth cusp such that $O$ covers one of the unique multi-cusped orbifolds with a $(2,3,6)$ cusp of volume $\frac{v_0}{3}$ and  $\frac{v_0}{2}$ given in Adams \cite{Adams_multi}. Then no one-cusped hyperbolic $3$-orbifold obtained by Dehn filling all the smooth cusps of $O$ is covered by a hyperbolic knot complement. 
\end{lemma}
\begin{proof}
Let $O_{\frac{v_0}{3}}$ and $O_{\frac{v_0}{2}}$ denote the unique multi-cusped orbifolds with a $(2,3,6)$ cusp of volume respectively $\frac{v_0}{3}$ and $\frac{v_0}{2}$ given in Adams \cite{Adams_multi}. Suppose for contradiction there is a one-cusped orbifold $O'$ obtained by Dehn filling all the smooth cusps of $O$ such that $O'$ is covered by a hyperbolic knot complement $\mathbb{S}^3-K$. Then, by Lemma \ref{isovertex}, the isotropy graph of $O$ has a finite vertex $p$ of the type  $(2,3,3)$ or $(2,3,4)$ or $(2,3,5)$ and the isotropy group $G_p$ of $p$ in $\pi_1^{Orb}(O)$ is either $A_4$ or $S_4$ or $A_5$. So, the image of $p$ in the isotropy graph of $O_{\frac{v_0}{3}}$ or $O_{\frac{v_0}{2}}$ is also a finite vertex of the type $(2,3,3)$ or $(2,3,4)$ or $(2,3,5)$ since neither of $A_4$, $S_4$ and $A_5$ is a subgroup of a finite dihedral group. This is a contradiction to the fact that the isotropy graph of $O_{\frac{v_0}{3}}$ (respectively, $O_{\frac{v_0}{2}}$) has only finite vertices of the type $(2,2,3)$ (respectively, $(2,2,6)$), which can be seen from Figure \ref{O4andO6_1isopic} taken from \cite{orbcenpract}. 
\end{proof}

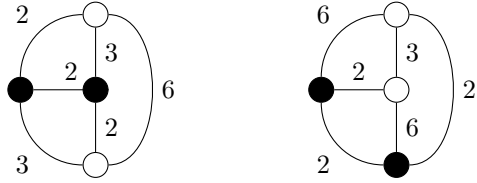
\begin{figure}
\centering 
\captionsetup{justification=centering}
\begin{tikzpicture}
\node at (-1,0)(finiteleft)[circle, draw, fill=black!100]{};

\node at (0,0)(finiteright)[circle, draw, fill=black!100]{};
\node at (0,1)(cuspup)[circle, draw]{};
\node at (0,-1)(cuspdown) [circle, draw]{};

\draw [-] (cuspup.west) [out=180,in=90] to node[auto,swap]{2} (finiteleft.north);
\draw [-] (finiteleft.east) to node[auto]{2}(finiteright.east);
\draw [-] (finiteright.south) to node[auto]{2} (cuspdown.north);
\draw [-] (cuspup.east) [out=0,in=0] to node[auto]{6} (cuspdown.east);
\draw [-] (finiteleft.south) [out=270,in=180] to node[auto,swap]{3} (cuspdown.west);
\draw [-] (cuspup.south) to node[auto]{3} (finiteright.north);

\begin{scope}[xshift=4cm]
\node at (-1,0)(finiteleft)[circle, draw, fill=black!100]{};

\node at (0,0)(cuspdown)[circle, draw]{};
\node at (0,1)(cuspup)[circle, draw]{};
\node at (0,-1)(finitedown) [circle, draw, fill=black!100]{};

\draw [-] (cuspup.west) [out=180,in=90] to node[auto,swap]{6} (finiteleft.north);
\draw [-] (finiteleft.east) to node[auto]{2}(cuspdown.west);
\draw [-] (cuspdown.south) to node[auto]{6} (finitedown.north);
\draw [-] (cuspup.east) [out=0,in=0] to node[auto]{2} (finitedown.east);
\draw [-] (finiteleft.south) [out=270,in=180] to node[auto,swap]{2} (finitedown.west);
\draw [-] (cuspup.south) to node[auto]{3} (cuspdown.north);
\end{scope}
\end{tikzpicture}
\caption{Isotropy graphs of $O_{\frac{v_0}{3}}$ (left) and $O_{\frac{v_0}{2}}$ (right) from \cite{Adams_multi} (both pictures reproduced from \cite{orbcenpract}).}

\label{O4andO6_1isopic}

\end{figure}
We are now ready to state and prove Proposition \ref{two236cusps}. 
\begin{proposition}\label{two236cusps}
\twosixcuspthm

\end{proposition}
\begin{proof}
Assume for contradiction that $\mathcal{F}$ has an infinite subfamily of knot complements $\{\mathbb{S}^3-K'_i\}_i$ with hidden symmetries. We proceed similarly as in the proof of Theorem \ref{2comp_sym}, Theorem \ref{hexlatt}, Proposition \ref{8v0prop} and Proposition \ref{coversmallestmulticusp}, and use Theorem \ref{CDHMMWorb} along with the Neumann and Reid characterization in Proposition \ref{hscharac} and Hoffman's result in Theorem \ref{hoff_cusptype} to see the existence of the covering maps $\phi_i: \mathbb{S}^3-K'_i \rightarrow O_i$ and $\phi:\mathbb{S}^3-L \rightarrow O$ where $O$ is an orbifold with exactly one rigid cusp $c_{rigid}$ and at-least one smooth cusp so that $c_{rigid}$ is either $(3,3,3)$ or $(2,3,6)$ and the only cusp of $L$ that covers $c_{rigid}$ is $c$. Moreover, after taking a subsequence, the orbifolds $O_i$ are all obtained by Dehn filling all the smooth cusps of $O$ and for all $i$, the degree of $\phi_i$ is equal to that of $\phi$.

By Hoffman's degree formulae in Theorem \ref{Hoff_333} and Theorem \ref{Hoff_236}, $\operatorname{degree}(\phi)=12m$ when $c_{rigid}$ is $(3,3,3)$ and $24m$ when $c_{rigid}$ is $(2,3,6)$, where $m$ is a positive integer. 

Let $\mathcal{H}$ be the $c$-maximal horoball packing of $\mathbb{H}^3$. Suppose $\Gamma_O$ and $\Gamma_{O'}$ denote the fundamental groups of $O$ and $O'$ respectively. Arguing similarly as in the proof of Proposition \ref{coversmallestmulticusp}, we can use  \cite[Lemma 2.1 and Proof of Lemma 2.2]{GoHeHo} to see that $\Gamma_1=\left \langle  \Gamma_O, \Gamma_{O'}\right  \rangle $ is a Kleinian group. Let $O_1=\mathbb{H}^3/\Gamma_1$. Since the only cusp of $\mathbb{S}^3-L$ that maps to $c'$ is $c$, we can again argue as in proof of Proposition \ref{coversmallestmulticusp} to see that both $c_{rigid}$ and $c'$ maps to the same cusp of $O_1$, say $c_1$ and that the only cusp of $O$ (respectively, $O'$) that maps to $c_1$ is $c_{rigid}$ (respectively, $c'$). Now, since $c'$ is $(2,3,6)$, $c_1$ is $(2,3,6)$ as well.  Let $c_2$ denote the cusp of $O_1$ where another $(2,3,6)$ cusp of $O'$ maps to. So, $c_2$ is $(2,3,6)$ too. Now, any cusp $c_s$ of $O$ that maps to $c_2$ is smooth and hence, at the cusp level, $\operatorname{degree}(c_s\to c_2)$ is a multiple of $3$. This means $\operatorname{degree}(O\to O_1)$ is also a multiple of $3$. We now consider the following two cases. 

At first, we consider the case when $c_{rigid}$ is $(3,3,3)$, then, at the cusp level, $\operatorname{degree}(c_{rigid} \to c_1)$ is a multiple of $2$. Since $c_{rigid}$ is the only cusp of $O$ that maps to $c_1$, this implies that $\operatorname{degree}(O\to O_1)$ is a multiple of $2$ and consequently, a multiple of $6$. Since  $\operatorname{degree}(\phi)=12m$ where $m \in \mathbb{N}$, we can conclude that $\operatorname{degree}(\mathbb{S}^3-L \to O_1)$ is a multiple of $72$. 

The second case is when $c_{rigid}$ is $(2,3,6)$. Then, since $\operatorname{degree}(\phi)=24m$ where $m \in \mathbb{N}$, we see that $\operatorname{degree}(\mathbb{S}^3-L \to O_1)$ in this case is also a multiple of $72$. 

So, in both cases, volume of $O_1$ is less than or equal to $\frac{24 v_0}{72}=\frac{v_0}{3}$. Since $O_1$ is a multi-cusped orbifold with at-least one $(2,3.6)$ cusp, Adams' result 
\cite[Lemma 2.2]{Adams_multi} implies that $O_1$ is the unique multi-cusped orbifold with two $(2, 3, 6)$ cusps of volume $\frac{v_0}{3}$ given in Adams \cite{Adams_multi}. But, $O$ covers $O_1$. This contradicts Lemma \ref{coveradams}. Hence, we are done. 

\end{proof}

We end this section with the following corollary. We will use this corollary later in Subsection \ref{snappyrun} to prove Theorem \ref{tetracompthm}. 

\begin{corollary}\label{badorder6sym}
Let $c$ and $\tilde{c}$ be two distinct cusps of a hyperbolic link $L$ of volume less than or equal to $24 v_0$ where $v_0$ is the volume of a regular ideal tetrahedron. Let $\mathcal{H}$ denote the $c$-maximal horoball packing of $\mathbb{H}^3$ whose horoball at $\infty$ maps to cusp $c$. Suppose there is an order $6$ rotational symmetry $r$ of $\mathcal{H}$ whose axis is the geodesic joining $\infty$ and the horocenter of a $\tilde{c}$-horoball. Then for any family $\mathcal{F}$ of hyperbolic knot complements obtained by Dehn filling all cusps of $L$ but $c$ and geometrically converging to $\mathbb{S}^3-L$,  $\mathcal{F}$ has at-most finitely many elements with hidden symmetries. 
\end{corollary}
\begin{proof}
Let $\Gamma$ be the fundamental group of $\mathbb{S}^3-L$. Consider $\Gamma'=\left \langle \Gamma, r\right \rangle$. As in the proof of Proposition \ref{coversmallestmulticusp}, we use \cite[Lemma 2.1 and Proof of Lemma 2.2]{GoHeHo} to argue that $\Gamma'$ is a Kleinian group. Let $O'$ denote the hyperbolic $3$-orbifold $\mathbb{H}^3/\Gamma'$, and $c'$ and $\tilde{c}'$ the cusps of $O'$ that respectively $c$ and $\tilde{c}$ maps to. Note that the set of parabolic fixed points of $\Gamma$ is same as that of $\Gamma'$. Let $S_c$ denote the parabolic fixed points of $\Gamma$ corresponding to cusp $c$ and $S_{o}$ denote the rest of the parabolic fixed points of $\Gamma$. Since $r$ is a symmetry of $\mathcal{H}$, $S_c$  and $S_o$ are both invariant under $r$ and consequently $\Gamma'$ as well. This means $O'$ has at-least two cusps ($c'$ and $\tilde{c}'$) and the only cusp of $\mathbb{S}^3-L$ that maps to $c'$ is $c$. Furthermore, $r$ has order $6$ and it fixes the horocenters of a $c$-horoball and a $\tilde{c}$-horoball. So, both $c'$ and $\tilde{c}'$ are $(2,3,6)$ as their cross-sections will have cone points of order $6$. Now, the corollary follows from Proposition \ref{two236cusps}.

\end{proof}

\section{An algorithm for testing symmetries}\label{hsalgorithm}
We are interested in understanding whether there exist a hyperbolic link (with two or more components) such that the Dehn fillings on all but one component of that link produces infinitely many hyperbolic knot complements with hidden symmetries geometrically converging to the complement of the link. Our approach in this pursuit is to use Theorem \ref{hexlatt} to exclude as many hyperbolic links as possible until we find a potential candidate for such phenomenon. In this section, we describe an algorithm which can be implemented in SnapPy \cite{snappy} to determine for which cusps the corresponding circle packings of $\mathbb{C}$ do not have the required symmetry prescribed by Theorem \ref{hexlatt}.

Unless otherwise specified, throughout this section, $L$ will be a hyperbolic link with components $K_1, \dots, K_n$ where $n \ge 2$. Let $\mathcal{C}$ denote the collection of the centers of all the circles in the $K_1$-circle packing and $P$ denote a cusp parallelogram of the $K_1$-cusp. Observe that a rotational symmetry of $\mathbb{C}$ is a symmetry of the $K_1$-circle packing if and only if its a symmetry $\mathcal{C}$. We recall from Section \ref{cuspsection} that $r_{n,\mathbf{p}}$ denotes the order $n$ counter-clockwise rotation of $\mathbb{C}$ around $\mathbf{p}$. Now, for an $r_{n,\mathbf{p}}$ of $\mathcal{C}$, there exists $g$ in $\operatorname{Stab}(\infty)$ in $\pi_1(\mathbb{S}^3-L)$ such that $g(\mathbf{p}) \in P$. Now, since $g$ is also symmetry of $\mathcal{C}$, so is $gr_{n,\mathbf{p}}g^{-1}=r_{n,g(\mathbf{p})}$, which has order $n$. Since $g^{-1} \in \pi_1(\mathbb{S}^3-L)$, it sends the horocenter of a $K_j$-horoball to that of another $K_j$-horoball. This means if $\mathbf{p}$ is not the horocenter of a $K_j$-horoball where $j \in \{2, \dots, n\}$, so not is $g(\mathbf{p})$. We write this observation in the following fact.
\begin{fact}
If the $K_1$-circle packing has some order $n$ rotational symmetry $r$, then for each cusp parallelogram $P$ of the $K_1$-cusp, the $K_1$-circle packing has an order $n$ rotational symmetry $r'$ whose fixed point lies inside $P$. Furthermore, if $r$ does not fix the horocenter of a $K_j$-horoball for $j \in \{2, \dots, n\}$, then so does not $r'$. 
\label{cuspparasym}
\end{fact}

\subsection{General Algorithm}

 The cusp parallelograms that SnapPy \cite{snappy} determines always have horizontal sides. Since we will implement our results in a SnapPy/Python code, we are going to henceforth assume that two parallel sides of $P$ are horizontal. Let $\mathbf{l}$ be the complex number (with imaginary part $0$) that represents the vector these horizontal sides of $P$ determine from left to right and $\mathbf{m}$ the complex number that represents the vector the other two parallel sides of $P$ determine from bottom left to top left. Note that $\mathbf{l}$ is the length of the horizontal sides of $P$.   

Let $\mathcal{C}_P$ be the set of points in $\mathcal{C}$ that lies in $P$. Fix $\mathbf{c_0}\in \mathcal{C}_P$. Denote the angle between $\mathbf{m}$ and $\mathbf{l}$ by $\theta$. We take $d$ to be the maximum of $\lVert \mathbf{m} \rVert$, $\lVert \mathbf{l} \rVert$, $\lVert \mathbf{m}+\mathbf{l} \rVert$ and $\lVert \mathbf{m}-\mathbf{l} \rVert$. By $\mathbb{R}_{>0}$, we will denote the set of positive real numbers. We now define three functions $k_{\mathbf{h}}, k, k_{\mathbf{l}}: \mathbb{R}_{>0} \to \mathbb{R}_{>0}$ as

\begin{align*}
k_{\mathbf{h}}(x)&= \left\lceil \frac{(x+1)d+1}{\lVert \mathbf{m} \rVert |\sin \theta|}  \right \rceil, \\
k(x)&=\left \lceil \frac{k_{\mathbf{h}} (x) \lVert \mathbf{m} \rVert |\cos \theta| }{\lVert \mathbf{l} \rVert }\right \rceil, \\
k_\mathbf{l}(x)&=\left \lceil \frac{(x+1)d+1}{\lVert \mathbf{l} \rVert}\right \rceil
\end{align*}

where $x \in \mathbb{R}_{>0}$. 
We need one more definition before we proceed further. For each $x$ in $\mathbb{R}_{>0}$, we define,  
\begin{align*}
\mathcal{C}_P(x)&=\left\{ \mathbf{c}+p\, \mathbf{m}+q\, \mathbf{l}:\mathbf{c} \in \mathcal{C}_P, p,q \in \mathbb{Z}\text{ and } |p|\le k_{\mathbf{h}}(x), |q| \le k(x)+k_\mathbf{l}(x) \right \}.
\end{align*}

Observe that for each $x$ in $\mathbb{R}_{>0}$, elements of $\mathcal{C}_P(x) $ are certain translations of the elements of $\mathcal{C}_P$. The following lemma will be crucial in the  proof of Proposition \ref{codeprop}. 

\begin{lemma}\label{circleradius}
For each $x$ in $\mathbb{R}_{>0}$, $\mathcal{C}_P(x)$ contains all elements of $\mathcal{C}$ that are within distance $xd$ from $\mathbf{c_0}$. \end{lemma}
\begin{proof}

\begin{figure}
\centering 
\captionsetup{justification=centering}
\begin{tikzpicture}[scale=.8]
\draw[fill=none](0,0) circle (2.8);

\filldraw[color=red!60!black!70,thick,fill=red!30!black!20](0,0)--(0,2)--(1,2)--(1,0)--(0,0);
\node at (.35,1.3){$R$};

\filldraw[color=magenta,thick,fill=magenta!20](0,0)--(1.25,2)--(2.25,2)--(1,0)--(0,0);
\node at (1.25,1.3){$P$};

\draw (-3,0)--(3,0);
\draw (-3,2)--(3,2);
\draw (-3,4)--(3,4);
\draw (-3,-2)--(3,-2);
\draw (-3,-4)--(3,-4);
\draw (-3,-4)--(-3,4);
\draw (-2,-4)--(-2,4);
\draw (-1,-4)--(-1,4);
\draw (0,-4)--(0,4);
\draw (1,-4)--(1,4);
\draw (2,-4)--(2,4);
\draw (3,-4)--(3,4);

\draw[color=magenta, thick] (0,0)--(2.5,4)--(3.5,4)--(1,0)--(0,0);
\node at (3.25,5){$P_{(x+1)d}^{+}$};
\node at (-2.25,-4.5){$P_{(x+1)d}^{-}$};

\draw (-1,0)--(1.5,4);
\draw (-2,0)--(.5,4);
\draw (-3,0)--(-.5,4);

\draw[color=green!60!black!80, thick] (0,0)--(-2.5,-4)--(-1.5,-4)--(1,0)--(0,0);

\draw(2,0)--(-.5,-4);
\draw(3,0)--(.5,-4);
\draw(4,0)--(1.5,-4);
\draw(4,0)--(3,0);

\draw[color=red!60!black!70,thick](0,0)--(0,4)--(1,4)--(1,0)--(0,0);
\node at (.5,5){$R_{(x+1)d}^{+}$};
\node at (.5,-4.5){$R_{(x+1)d}^{-}$};

\draw[color=red,thick](0,0)--(0,-4)--(1,-4)--(1,0)--(0,0);

\draw [dotted,->,thick] (0,-.15)--(1,-.15);
\draw [dotted,->,thick] (-.15,0)--(-.15,2);
\node at (-.32,1){$\mathbf{h}$};
\node at (.5,-.32){$\mathbf{l}$};

\begin{scope}[xshift=-.1cm]
\draw [dotted,->,thick] (0,0)--(1.25,2);
\node at (.22,.72){$\mathbf{m}$};
\end{scope}

\begin{scope}[yshift=.1cm]
\draw[<->](0,4)--(2.5,4);
\node at (1.5,4.2){$k_{\mathbf{h}}(x) \lVert \mathbf{m} \rVert \lvert \cos (\theta)\rvert$};
\end{scope}
\draw[color=magenta] (1,0) coordinate (A) -- (0,0) coordinate (B) -- (1.25,2) coordinate (C)
pic[draw, ->, black, angle radius=.35cm]{angle=A--B--C};

\node at(.5,.4){$\theta$};

\begin{scope}[xshift=10cm]

\draw[fill=none](0,0) circle (2.8);

\filldraw[color=red!60!black!70,thick,fill=red!30!black!20](0,0)--(0,2)--(1,2)--(1,0)--(0,0);
\node at (.35,1.3){$R$};

\filldraw[color=magenta,thick,fill=magenta!20](1,0)--(-.75,2)--(-1.75,2)--(0,0)--(1,0);
\node at (-.6,1.3){$P$};

\draw (-3,0)--(3,0);
\draw (-3,2)--(3,2);
\draw (-3,4)--(3,4);
\draw (-3,-2)--(3,-2);
\draw (-3,-4)--(3,-4);
\draw (-3,-4)--(-3,4);
\draw (-2,-4)--(-2,4);
\draw (-1,-4)--(-1,4);
\draw (0,-4)--(0,4);
\draw (1,-4)--(1,4);
\draw (2,-4)--(2,4);
\draw (3,-4)--(3,4);

\draw[color=magenta, thick] (0,0)--(-3.5,4)--(-2.5,4)--(1,0)--(0,0);
\node at (-3,5){$P_{(x+1)d}^{+}$};
\node at (4,-4.5){$P_{(x+1)d}^{-}$};

\draw (2,0)--(-1.5,4);
\draw (3,0)--(-.5,4);
\draw (4,0)--(.5,4);
\draw (5,0)--(1.5,4);
\draw (5,0)--(3,0);

\draw[color=green!60!black!80, thick] (0,0)--(3.5,-4)--(4.5,-4)--(1,0)--(0,0);

\draw(-1,0)--(2.5,-4);
\draw(-2,0)--(1.5,-4);
\draw(-3,0)--(.5,-4);
\draw(-4,0)--(-.5,-4);
\draw(-4,0)--(-3,-0);

\draw(3,-4)--(3.5,-4);

\draw[color=red!60!black!70,thick](0,0)--(0,4)--(1,4)--(1,0)--(0,0);
\node at (.5,5){$R_{(x+1)d}^{+}$};
\node at (.5,-4.5){$R_{(x+1)d}^{-}$};

\draw[color=red,thick](0,0)--(0,-4)--(1,-4)--(1,0)--(0,0);

\draw [dotted,->,thick] (0,-.15)--(1,-.15);
\draw [dotted,->,thick] (-.15,0)--(-.15,2);
\node at (-.32,1){$\mathbf{h}$};
\node at (.5,-.32){$\mathbf{l}$};

\begin{scope}[xshift=-.13cm]
\draw [dotted,->,thick] (0,0)--(-1.75,2);
\node at (-1.1,.72){$\mathbf{m}$};
\end{scope}

\begin{scope}[yshift=.1cm]
\draw[<->](0,4)--(-3.5,4);
\node at (-2,4.2){$k_{\mathbf{h}}(x) \lVert \mathbf{m} \rVert \lvert \cos (\theta)\rvert$};
\end{scope}
\draw[color=magenta] (1,0) coordinate (A) -- (0,0) coordinate (B) -- (-1.75,2) coordinate (C)
pic[draw, ->, black, angle radius=.35cm]{angle=A--B--C};

\node at(.5,.4){$\theta$};

\end{scope}
\end{tikzpicture}
\caption{$\mathcal{C}_P(x)$ contains $\mathcal{C} \cap D_{(x+1)d}$ - Left: $\theta$ is actue, Right: $\theta$ is obtuse.}
\label{obtactlemma}

\end{figure}
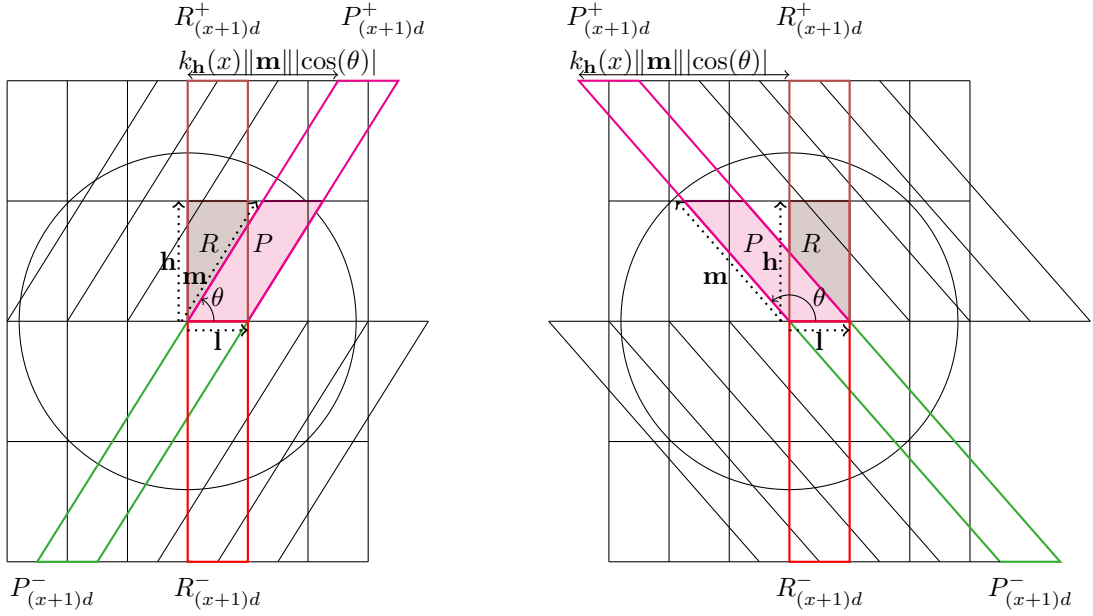

Let $x \in \mathbb{R}_{>0}$. Suppose $\mathbf{v_0}$ denotes the left bottom vertex of $P$. 
Let us denote the disk of radius $s$ centered $\mathbf{v_0}$ by $D_s$. We should note that the distance between any two points in $P$ is less than or equal to $d$. So, $\lVert \mathbf{c_0}-\mathbf{v_0} \rVert \le d$.

Let $\mathbf{c} \in \mathcal{C}$ such that $\lVert \mathbf{c}-\mathbf{c_0} \rVert \le xd$. We have, 
$$\lVert \mathbf{c}-\mathbf{v_0} \rVert \le \lVert \mathbf{c}-\mathbf{c_0} \rVert + \lVert \mathbf{c_0}-\mathbf{v_0} \rVert \le xd+d=(x+1)d.$$

So, if we prove that $\mathcal{C}_P(x)$ contains $\mathcal{C} \cap D_{(x+1)d}$,  it will imply $\mathcal{C}_P(x)$ contains all elements of $\mathcal{C}$ that are within distance $xd$ from $\mathbf{c}_0$. 
 
We define $\mathbf{h}$ to be the vector equal to $\lVert \mathbf{m} \rVert | \sin \theta | \,\mathbf{i}$ with initial point $\mathbf{v_0}$. Note that $\mathbf{h}$ is the vector starting from $\mathbf{v_0}$ determined by the height of $P$. Let $R$ be the rectangle determined by $\mathbf{h}$ and $\mathbf{l}$ with the bottom left vertex as $\mathbf{v_0}$. We define, 
$$R_{(x+1)d}^{+}=\bigcup\limits_{j=0}^{k_{\mathbf{h}}(x)-1} R+j\, \mathbf{h} \quad \text { and } \quad R_{(x+1)d}^{-}=\bigcup\limits_{j=-k_{\mathbf{h}}(x)}^{-1} R+j\, \mathbf{h}.
$$ 
Note that $R_{(x+1)d}^{+}$ and $R_{(x+1)d}^{-}$ respectively are the vertical strips (and so in the direction of $\mathbf{h}$) over and below $R$. 

We similarly define, 
$$P_{(x+1)d}^{+}=\bigcup\limits_{j=0}^{k_{\mathbf{h}}(x)-1} P+j\, \mathbf{m}\quad \text{ and } \quad P_{(x+1)d}^{-}=\bigcup\limits_{j=-k_{\mathbf{h}}(x)}^{-1} P+j\, \mathbf{m}.
$$
Note that $P_{(x+1)d}^{+}$ (respectively, $P_{(x+1)d}^{-}$) is the slanted strip (in the direction of $\mathbf{m}$) based over (respectively, below) $P$. 

Now,
$$R_{(x+1)d}^{+} \subset \bigcup\limits_{j=-k(x)}^{k(x)} P_{(x+1)d}^{+} +j\, \mathbf{l} \quad \text{ and } \quad R_{(x+1)d}^{-} \subset \bigcup\limits_{j=-k(x)}^{k(x)} P_{(x+1)d}^{-} +j\, \mathbf{l}.
$$ 
(See Figure \ref{obtactlemma} for a pictorial explanation.)

If we define, $P_{(x+1)d}= P_{(x+1)d}^{+} \cup P_{(x+1)d}^{-}$, then, $P_{(x+1)d}=\bigcup\limits_{j=-k_{\mathbf{h}}(x)}^{k_{\mathbf{h}}(x)-1} P+j\, \mathbf{m}$. This would mean that both $R_{(x+1)d}^{+}$ and $R_{(x+1)d}^{-}$ are contained in  $\bigcup\limits_{j=-k(x)}^{k(x)} P_{(x+1)d} +j\, \mathbf{l}$.  So, $R_{(x+1)d}=R_{(x+1)d}^{+}  \cup R_{(x+1)d}^{-}$ is also contained in  $\bigcup\limits_{j=-k(x)}^{k(x)} P_{(x+1)d} +j\, \mathbf{l}$. 

Now, we note that $D_{(x+1)d} \subset  \bigcup\limits_{n=-k_\mathbf{l}(x)}^{k_\mathbf{l}(x)} R_{(x+1)d} +n\, \mathbf{l}$ (See Figure \ref{obtactlemma}). This means 
\begin{align*}
D_{(x+1)d} &\subset  \bigcup\limits_{n=-k_\mathbf{l}(x)}^{k_\mathbf{l}(x)}  ( \bigcup\limits_{j=-k(x)}^{k(x)} P_{(x+1)d} +j\, \mathbf{l} )+ n\, \mathbf{l}\\
&=\bigcup\limits_{n=-(k(x)+k_\mathbf{l}(x))}^{k(x)+k_\mathbf{l}(x)} P_{(x+1)d} +n\, \mathbf{l} \\
&= \bigcup\limits_{n=-(k(x)+k_\mathbf{l}(x))}^{k(x)+k_\mathbf{l}(x)} (\bigcup\limits_{j=-k_{\mathbf{h}}(x)}^{k_{\mathbf{h}}(x)-1} P+j\, \mathbf{m})+ n\, \mathbf{l}\\
&\subset \bigcup\limits_{n=-(k(x)+k_\mathbf{l}(x))}^{k(x)+k_\mathbf{l}(x)} \bigcup\limits_{j=-k_{\mathbf{h}}(x)}^{k_{\mathbf{h}}(x)} P+j\, \mathbf{m}+ n\, \mathbf{l}.
\end{align*}

The intersection of the right most union above with $\mathcal{C}$ is $\mathcal{C}_P(x)$. Therefore, $\mathcal{C} \cap D_{(x+1)d}$ is contained in $\mathcal{C}_P(x)$.

\end{proof}
We now make a note of the following fact. 
\begin{fact}\label{equilatfact}
Let $r$ be an order $3$ symmetry of $\mathcal{C}$. Then the orbits of the action of the group $\{1,r,r^2\}$ on $\mathcal{C}$ have either $1$ element or $3$ elements. If this action has an orbit with one element, then it is unique and the element in that orbit is the fixed point of $r$. On the other hand, the elements in an orbit with $3$ elements form an equilateral triangle. 
\end{fact}
Now, if $x$, $y$ and $z$ are the vertices of an equilateral triangle $\Delta$, then the centroid of $\Delta$ is fixed by any symmetry of $\Delta$. With this in mind and motivated by Corollary \ref{nearrot} and \ref{neardistinctrot}, we define the sets
\begin{align*}
\mathcal{E}_{weak}&=\left \{\mathbf{c_0}\right\} \bigcup \left\{\frac{\mathbf{x}+\mathbf{y}+\mathbf{c_0}}{3}: \mathbf{x}, \mathbf{y} \in \mathcal{C}_P\left(\frac{1}{3}\right), 0  < \lVert \mathbf{x}-\mathbf{c_0} \rVert=\lVert \mathbf{y}-\mathbf{c_0} \rVert \le \frac{\sqrt{\operatorname{Area}(P)}}{\sqrt{2} 3^{\frac{3}{4}}},\right. \\
 \qquad &\left.  \frac{\lVert (\mathbf{x}-\mathbf{c_0}).(\mathbf{y}-\mathbf{c_0}) \rVert}{ \lVert \mathbf{x}-\mathbf{c_0} \rVert \lVert \mathbf{y}-\mathbf{c_0} \rVert}=\frac{1}{2}\right\}, \text{ and }, \\
\mathcal{E}_{strong} &=\left\{\frac{\mathbf{x}+\mathbf{y}+\mathbf{c_0}}{3}: \mathbf{x}, \mathbf{y} \in \mathcal{C}_P\left(\frac{4}{7}\right), 0  < \lVert \mathbf{x}-\mathbf{c_0} \rVert=\lVert \mathbf{y}-\mathbf{c_0} \rVert \le \frac{\sqrt{\operatorname{Area}(P)}}{\sqrt{2} 3^{\frac{1}{4}}},\right. \\
 \qquad &\left.  \frac{\lVert (\mathbf{x}-\mathbf{c_0}).(\mathbf{y}-\mathbf{c_0}) \rVert}{ \lVert \mathbf{x}-\mathbf{c_0} \rVert \lVert \mathbf{y}-\mathbf{c_0} \rVert}=\frac{1}{2}\right\}.
\end{align*}
Observe that the point $\mathbf{c_0}$ as well as the centroids of the equilateral triangles with one vertex at $\mathbf{c_0}$ and two other vertices from $\mathcal{C}_P\left(\frac{1}{3}\right)$, and length of its sides less than or equal to $\frac{\sqrt{\operatorname{Area}(P)}}{\sqrt{2} 3^{\frac{3}{4}}}$ belong to $\mathcal{E}_{weak}$. Similarly, the centroids of the equilateral triangles with one vertex at $\mathbf{c_0}$ and two other vertices from $\mathcal{C}_P\left(\frac{4}{7}\right)$, and length of its sides less than or equal to $\frac{\sqrt{\operatorname{Area}(P)}}{\sqrt{2} 3^{\frac{1}{4}}}$ belong to $\mathcal{E}_{strong}$.

Recall from Section \ref{cuspsection} that $r_{n,\mathbf{x}}$ denotes the counterclockwise rotation of order $n$ around $\mathbf{x} \in \mathbb{C}$. We have the following proposition. 

\begin{proposition}\label{codeprop}
If there exists an infinite family $\mathcal{F}$ of hyperbolic knot complements with hidden symmetries obtained by Dehn filling all cusps of $L$ but the $K_1$-cusp which geometrically converges to $\mathbb{S}^3-L$, then, 
\begin{enumerate}
\item there exists $\mathbf{e}\in \mathcal{E}_{weak}$ such that $\lVert \mathbf{e}-\mathbf{c_0} \rVert <\frac{d}{5}$, both $r_{3,\mathbf{e}}\left(\mathcal{C}_P\left(\frac{1}{3}\right)\right)$ and $r_{3,\mathbf{e}}^2\left(\mathcal{C}_P\left(\frac{1}{3}\right)\right)$ are contained in $\mathcal{C}_P\left(\frac{7}{5}+k_{\mathbf{h}}(\frac{1}{3})+k(\frac{1}{3})+k_{\mathbf{l}}(\frac{1}{3})\right)$ and $\mathbf{e}$ is not the horocenter of any $K_j$-horoball where $j \in \{2,3, \dots, n\}$, and, 

\item there exists $\mathbf{e}_1, \mathbf{e}_2 \in \mathcal{E}_{strong}$ such that $\mathbf{e}_1\ne \mathbf{e}_2$, and for each $i \in \{1,2\}$, $0<\lVert \mathbf{e}_i-\mathbf{c_0} \rVert <\frac{d}{3}$, both $r_{3,\mathbf{e}_i}\left(\mathcal{C}_P\left(\frac{4}{7}\right)\right)$ and $r_{3,\mathbf{e}_i}^2\left(\mathcal{C}_P\left(\frac{4}{7}\right)\right)$ are contained in $\mathcal{C}_P\left(\frac{5}{3}+k_{\mathbf{h}}(\frac{4}{7})+k(\frac{4}{7})+k_{\mathbf{l}}(\frac{4}{7})\right)$ and neither of $\mathbf{e}_1$ and $\mathbf{e}_2$ is the horocenter of any $K_j$-horoball where $j \in \{2,3, \dots, n\}$. 
\end{enumerate}
\end{proposition}
\begin{proof}
By Corollary \ref{nearrot} (and Remark \ref{order3rotation}), and Corollary \ref{neardistinctrot}, there exists order $3$ symmetries $r_{3,\mathbf{e}}$, $r_{3,\mathbf{e}_1}$ and $r_{3,\mathbf{e}_2}$ of the $K_1$-circle packing of $\mathbb{C}$ such that $\mathbf{e}_1\ne \mathbf{e}_2$, 
$\lVert \mathbf{e}-\mathbf{c_0} \rVert  \le \frac{\sqrt{\operatorname{Area}(P)}}{\sqrt{2}3^{\frac{5}{4}}}$ and $0< \lVert \mathbf{e_i}-\mathbf{c_0} \rVert \le \frac{\sqrt{\operatorname{Area}(P)}}{\sqrt{2}3^{\frac{3}{4}}}$ for  $i=1,2$, where $\mathbf{e}, \mathbf{e}_1$ and $\mathbf{e}_2$ belong to $\mathbb{C}$. 

We will first prove the first part. Now, either $\mathbf{e}=\mathbf{c_0}$ or $\mathbf{c_0}$, $r_{3,\mathbf{e}}(\mathbf{c_0})$ and $r_{3,\mathbf{e}}^2(\mathbf{c_0})$ are the three distinct vertices of an equilateral triangle centered at $\mathbf{e}$. So, 
$$\lVert r_{3,\mathbf{e}}(\mathbf{c_0})-\mathbf{c_0} \rVert =\lVert r_{3,\mathbf{e}}^2(\mathbf{c_0})-\mathbf{c_0} \rVert =\sqrt{3}\lVert \mathbf{e}-\mathbf{c_0}  \rVert \le \frac{\sqrt{\operatorname{Area}(P)}}{\sqrt{2}3^{\frac{3}{4}}} \le \frac{\sqrt{\lVert \mathbf{m} \rVert \lVert \mathbf{l} \rVert}}{\sqrt{2}3^{\frac{3}{4}}}<\frac{d}{3}.
$$
Now by Lemma \ref{circleradius}, we have, $r_{3,\mathbf{e}}(\mathbf{c_0})$ and $r_{3,\mathbf{e}}^2(\mathbf{c_0})$ lie in $\mathcal{C}_P\left(\frac{1}{3}\right)$. So, $\mathbf{e} \in \mathcal{E}_{weak}$.  

Let $\mathbf{c'} \in \mathcal{C}_P\left(\frac{1}{3}\right)$. Then, $\mathbf{c'}=\mathbf{c}+p\,\mathbf{m}+q\,\mathbf{l}$ for some $\mathbf{c}  \in \mathcal{C}_P$, and $p$, $q$ $\in \mathbb{Z}$ such that $|p|\le k_{\mathbf{h}}\left(\frac{1}{3}\right)$ and $|q| \le k\left(\frac{1}{3}\right)+k_{\mathbf{l}}\left(\frac{1}{3}\right)$. So, we have,  
\begin{align*}
\lVert \mathbf{c'}-\mathbf{e} \rVert & \le \lVert \mathbf{c'}-\mathbf{c_0} \rVert +\lVert \mathbf{c_0} - \mathbf{e}\rVert \\
&\le \lVert \mathbf{c}-\mathbf{c_0}\rVert +|p| \lVert \mathbf{m}\rVert+|q| \lVert \mathbf{l} \rVert +\lVert \mathbf{c_0}-\mathbf{e}\rVert \\
&< d+k_{\mathbf{h}}\left(\frac{1}{3}\right) d+\left(k\left(\frac{1}{3}\right)+k_{\mathbf{l}}\left(\frac{1}{3}\right)\right) d+\frac{d}{5}\\
&=d\left(\frac{6}{5}+k_{\mathbf{h}}\left(\frac{1}{3}\right)+k\left(\frac{1}{3}\right)+k_{\mathbf{l}}\left(\frac{1}{3}\right)\right).
\end{align*}
Here, note that $\lVert \mathbf{c_0}-\mathbf{e} \rVert \le \frac{\sqrt{\operatorname{Area}(P)}}{\sqrt{2}3^{\frac{5}{4}}}\le \frac{\sqrt{\lVert m \rVert \lVert l \rVert}}{\sqrt{2}3^{\frac{5}{4}}}< \frac{d}{5}$. 

Since, $r_{3,\mathbf{e}}$ is an order $3$ rotation around $\mathbf{e}$, we have $\lVert r_{3,\mathbf{e}}(\mathbf{c'})-\mathbf{e}\rVert =\lVert r_{3,\mathbf{e}}^2(\mathbf{c'})=\mathbf{e}\rVert=\lVert \mathbf{c'}-\mathbf{e}\rVert$, which is less than $d\left(\frac{6}{5}+k_{\mathbf{h}}\left(\frac{1}{3}\right)+k\left(\frac{1}{3}\right)+k_{\mathbf{l}}\left(\frac{1}{3}\right)\right)$. So, 
\begin{align*}
\lVert r_{3,\mathbf{e}}(\mathbf{c'})-\mathbf{c_0}\rVert &\le \lVert r_{3,\mathbf{e}}(\mathbf{c'})-\mathbf{e}\rVert +\lVert \mathbf{e}-\mathbf{c_0}\rVert \\
& < d\left(\frac{6}{5}+k_{\mathbf{h}}\left(\frac{1}{3}\right)+k\left(\frac{1}{3}\right)+k_{\mathbf{l}}\left(\frac{1}{3}\right)\right)+ \frac{d}{5}\\
&=d\left(\frac{7}{5}+k_{\mathbf{h}}\left(\frac{1}{3}\right)+k\left(\frac{1}{3}\right)+k_{\mathbf{l}}\left(\frac{1}{3}\right)\right).
\end{align*}
Similarly, $\lVert r_{3,\mathbf{e}}^2(\mathbf{c'})-\mathbf{c_0}\rVert < d\left(\frac{7}{5}+k_{\mathbf{h}}\left(\frac{1}{3}\right)+k\left(\frac{1}{3}\right)+k_{\mathbf{l}}\left(\frac{1}{3}\right)\right)$. So, by Lemma \ref{circleradius}, 
$$r_{3,\mathbf{e}}(\mathbf{c'}), r_{3,\mathbf{e}}^2(\mathbf{c'}) \in \mathcal{C}_P\left(\frac{7}{5}+k_{\mathbf{h}}\left(\frac{1}{3}\right)+k\left(\frac{1}{3}\right)+k_{\mathbf{l}}\left(\frac{1}{3}\right)\right).$$ This completes the proof of the first part. 

Now, for the second part, we argue similarly as in the first part. In particular, for $i=1,2$, we get, 
$$\lVert r_{3,\mathbf{e}_i}(\mathbf{c_0})-\mathbf{c_0} \rVert =\lVert r_{3,\mathbf{e}_i}^2(\mathbf{c_0})-\mathbf{c_0} \rVert =\sqrt{3}\lVert \mathbf{e}_i-\mathbf{c_0}  \rVert \le \frac{\sqrt{\operatorname{Area}(P)}}{\sqrt{2}3^{\frac{1}{4}}} \le \frac{\sqrt{\lVert m \rVert \lVert l \rVert}}{\sqrt{2}3^{\frac{1}{4}}}<\frac{4d}{7}.
$$
Lemma \ref{circleradius} then implies that $\mathbf{e}_1, \mathbf{e}_2 \in \mathcal{E}_{strong}$. 

We observe that $\lVert \mathbf{c}_0-\mathbf{e}_i\rVert < \frac{d}{3}$ for $i=1,2$. Then, arguing again as in the first part, we can see that, for any $\mathbf{c}' \in \mathcal{C}_P(\frac{4}{7})$, 
\begin{align*}
\lVert \mathbf{c'}-\mathbf{e}_i\rVert &< d+k_{\mathbf{h}} \left(\frac{4}{7}\right)d+\left(k \left(\frac{4}{7}\right)+k_{\mathbf{l}}\left(\frac{4}{7}\right) \right) d+\frac{d}{3}\\
&=d\left(\frac{4}{3}+k_{\mathbf{h}}\left(\frac{4}{7}\right)+k\left(\frac{4}{7}\right)+k_{\mathbf{l}}\left(\frac{4}{7}\right)\right).
\end{align*}
Consequently, 
\begin{align*}
\lVert r_{3,\mathbf{e}_i}^{\pm1}(\mathbf{c'})-\mathbf{c_0}\rVert & \le \lVert r_{3,\mathbf{e}_i}^{\pm 1}(\mathbf{c'})-\mathbf{e}_i\rVert +\lVert \mathbf{e}_i-\mathbf{c_0}\rVert \\
& < d\left(\frac{4}{3}+k_{\mathbf{h}}\left(\frac{4}{7}\right)+k\left(\frac{4}{7}\right)+k_{\mathbf{l}}\left(\frac{4}{7}\right)\right)+ \frac{d}{3}\\
&=d\left(\frac{5}{3}+k_{\mathbf{h}}\left(\frac{4}{7}\right)+k \left(\frac{4}{7}\right)+k_{\mathbf{l}}\left(\frac{4}{7}\right)\right).
\end{align*}
Lemma \ref{circleradius} now completes the proof of the second part. 
\end{proof}

We end this section with the following proposition. 

\begin{proposition}\label{codeprop20v0}

Suppose the volume of $\mathbb{S}^3-L$ is $20v_0$ where $v_0$ is the volume of a regular ideal tetrahedron. Let $\phi: \mathbb{S}^3-L \to O_{(2,3,6),(2,2,2,2)}$ be an orbifold cover such that the only cusp of $\mathbb{S}^3-L$ that maps to the $(2,3,6)$ cusp of $O_{(2,3,6),(2,2,2,2)}$ is the $K_1$-cusp. Then, there exists $\mathbf{e}_1, \mathbf{e}_2 \in \mathcal{E}_{strong}$ such that $\mathbf{e}_1\ne \mathbf{e}_2$, and for each $i \in \{1,2\}$, $0<\lVert \mathbf{e}_i-\mathbf{c_0} \rVert <\frac{d}{3}$, both $r_{3,\mathbf{e}_i}\left(\mathcal{C}_P\left(\frac{4}{7}\right)\right)$ and $r_{3,\mathbf{e}_i}^2\left(\mathcal{C}_P\left(\frac{4}{7}\right)\right)$ are contained in $\mathcal{C}_P\left(\frac{5}{3}+k_{\mathbf{h}}(\frac{4}{7})+k(\frac{4}{7})+k_{\mathbf{l}}(\frac{4}{7})\right)$ and neither of $\mathbf{e}_1$ and $\mathbf{e}_2$ is the horocenter of any $K_j$-horoball where $j \in \{2,3, \dots, n\}$. 

\end{proposition}

\begin{proof}

Let $c_{(2,3,6)}$ denote the $(2,3,6)$ cusp of $O_{(2,3,6),(2,2,2,2)}$. We note that the index of $\phi$ is $24$. By (the proof of) Part 1 of Proposition \ref{coversmallestmulticusp}, the symmetry group of the $K_1$-circle packing contains $W_{(2,3,6)}$ group of symmetries which is the peripheral subgroup of $\pi_1^{Orb}(O_{(2,3,6),(2,2,2,2)})$ corresponding to the cusp $c_{(2,3,6)}$. A similar line of argument as in the proof of Part 2 of Theorem \ref{hexlatt} shows that the co-area of the maximal translational subgroup $\Lambda_{(2,3,6)}$ of $W_{(2,3,6)}$ is $\frac{\text{maximal cusp area of the }{K_1\text{-cusp}}}{4}$. Following the arguments as in Corollary \ref{nearrot} (and Remark \ref{order3rotation}), we can obtain a similar conclusion as in Corollary \ref{neardistinctrot}: the $K_1$-circle packing of $\mathbb{C}$ has symmetries $r_{3,\mathbf{e}_1}$ and $r_{3,\mathbf{e}_2}$ neither fixing the center of a $K_j$-horoball for $j \in \{2,\ldots,n\}$ such that $\mathbf{e}_1\ne\mathbf{e}_2$ and $0 <\lVert \mathbf{c}_0-\mathbf{e}_m\rVert \le \frac{\sqrt{\text{maximal cusp area of the }K_1\text{-cusp}}}{\sqrt{2}3^{\frac{3}{4}}}$ for each $m=1,2$. We can now imitate the proof of Proposition \ref{codeprop} to conclude the result. 
\end{proof}

\begin{remark}\label{codeprop20v0remark}
Note that $r_{3,\mathbf{e}_1}$ and $r_{3,\mathbf{e}_2}$ in Proposition \ref{codeprop20v0} are actually symmetries of the $K_1$-horoball packing of $\mathbb{H}^3$ with a $K_1$-horoball at $\infty$ (see proofs of Proposition \ref{coversmallestmulticusp} and Theorem \ref{hexlatt}). 
\end{remark}

\subsection{Implementation of the algorithm in SnapPy}\label{hscodesubs}
We now use Proposition \ref{codeprop} (and Proposition \ref{codeprop20v0}) to set up a code on SnapPy \cite{snappy}. We describe the pseudocode in Algorithm \ref{twocentroidalg}. The code is available in the Python file \texttt{TwoCentroids.py} from \cite{tetra_code}. In Algorithm \ref{twocentroidalg}, we will use $\mathcal{C}'_P$ and $\mathcal{C}''_P$ to denote $\mathcal{C}_P(\frac{4}{7})$ and $\mathcal{C}_P\left(\frac{5}{3}+k_{\mathbf{h}}(\frac{4}{7}) +k(\frac{4}{7})+k_{\mathbf{l}}(\frac{4}{7})\right)$ respectively. 

The function \texttt{Free\_rot\_strng} in  Algorithm \ref{twocentroidalg} takes two input - a SnapPy manifold \texttt{M} and a cusp index \texttt{i} of \texttt{M}. The function \texttt{Free\_rot\_strng(M,i)} maximizes the $i$-cusp horoball packing of $\mathbb{H}^3$ in SnapPy with an $i$-horoball at $\infty$, considers the center $\mathbf{c}_0$ of a full-sized $i$-horoball in the cusp parallelogram $P$ for this maximized $i$-cusp neighborhood given by SnapPy, and tries to find two distinct complex numbers $\mathbf{e}_1$ and $\mathbf{e}_2$ in $\mathcal{E}_{strong}$ so that for each $i\in\{1,2\}$,  $0<\lVert \mathbf{e}_i-\mathbf{c_0} \rVert <\frac{d}{3}$ and both $r_{3,\mathbf{e}_i}\left(\mathcal{C}'_P\right)$, $r_{3,\mathbf{e}_i}^2\left(\mathcal{C}'_P\right)$ are contained in $\mathcal{C}''_P$, and neither of $\mathbf{e}_1$ and $\mathbf{e}_2$ belongs to the set \texttt{H\_diff} of the horocenters of a large number of non $i$-horoballs. The other two functions \texttt{Include} and \texttt{Centroid\_strng} are used inside the  \texttt{Free\_rot\_strng} function. 

The \texttt{Include} function checks whether a floating point (complex) number $\mathbf{x}$ is within $0.005$ distance of an element of the Python set \texttt{H}, i.e. whether $\mathbf{x}$ belongs to \texttt{H}. For the $i$-maximized horoball packing, the function \texttt{Centroid\_strng} is used inside the \texttt{Free\_rot\_strng} function to get hold of the Python set where we would find $\mathbf{e}_1$ and $\mathbf{e}_2$, i.e., the Python set representing $\left \{\mathbf{x} \in \mathcal{E}_{strong}-\texttt{H\_diff}: 0<\lVert \mathbf{x}-\mathbf{c_0} \rVert <\frac{d}{3}\right\}$.

\begin{algorithm}
\caption{An algorithm to check non-existence of required order $3$ symmetries}
\begin{algorithmic}
\Function{Include}{complex number $\mathbf{x}$, Python set H}
\State r $\gets$ False
\For {$\mathbf{y} \in$ H}
\If{$\lVert x - y \rVert <0.005$}
\State r$\gets$ True
\State break
\EndIf
\EndFor
\State \Return r
\EndFunction
\State 
\Function{Centroid\_strng}{complex number $\mathbf{h}$,  Python set S, Python set G, float v, float d}
\State L $\gets$ list version of S
\State E $\gets \{\}$
\For {$\mathbf{x} \in $ L}
\If {$0.005<\lVert \mathbf{x} -\mathbf{h}\rVert <\frac{\sqrt{v}}{3^{\frac{1}{4}}}+0.005$}
\For {$\mathbf{y} \in$ L with index $>$ index of $\mathbf{x}$ in L}
\If {$\left \lvert \lVert \mathbf{x}-\mathbf{h}\rVert-\lVert \mathbf{y}-\mathbf{h}\rVert \right\rvert <0.005$ and 
$\left \lvert \frac{\lvert (\mathbf{x}-\mathbf{h}).(\mathbf{y}-\mathbf{h})\rvert}{\lVert \mathbf{x}-\mathbf{h}\rVert \lVert \mathbf{y}-\mathbf{h}\rVert}-\frac{1}{2}\right\rvert<0.005$ and \\ $\hspace*{6.8em}$ $ \lVert \mathbf{h}-\frac{\mathbf{x}+\mathbf{y}+\mathbf{h}}{3}\rVert<\frac{d}{3}+0.005$ and Include($\frac{\mathbf{x}+\mathbf{y}+\mathbf{h}}{3}$,G)=False}
\State add $\frac{\mathbf{x}+\mathbf{y}+\mathbf{h}}{3}$ to E
\EndIf
\EndFor
\EndIf
\EndFor
\State \Return E
\EndFunction
\State
\Function{Free\_rot\_strng}{SnapPy manifold $M$, cusp $i$}
  \State Maximize the $i$-cusp neighborhood of $M$ and put the $i$-cusp at the eye at $\infty$
 \State Get $\mathbf{m}, \mathbf{l}, d, \mathcal{C}_P, k_{\mathbf{h}}(\frac{4}{7}), k(\frac{4}{7}), k_{\mathbf{l}}(\frac{4}{7}), \mathcal{C}'_P, \mathcal{C}''_P$  
 \State vol $\gets$ maximum $i$-cusp volume
  \State h $\gets \text{first element of the list of version of $\mathcal{C}_P$}$ \Comment{h is $\mathbf{c}_0$}

  \State H\_diff\_cen $\gets \{\text{centers of non $i$-horoballs of height $\ge 0.1$ in the $i$-cusp paralleogram}\}$
  \State H\_diff $\gets \{\mathbf{x}+p \mathbf{m}+q \mathbf{l}: \mathbf{x} \in \text{H\_diff\_cen},p,q \in \mathbb{Z}, \lvert p \rvert \le k_{\mathbf{h}}(\frac{4}{7}), \lvert q \rvert \le  k(\frac{4}{7})+k_{\mathbf{l}}(\frac{4}{7})\}$
\State exist\_free\_rot $\gets$ False 
\State H\_free\_rot $\gets$ $\{\}$

 \For {$\mathbf{x} \in$ Centroid\_strng(h,  $\mathcal{C}'_P$, H\_diff, vol,$d$)}
  \State r $\gets$ True
   \For {$\mathbf{y} \in  \mathcal{C}'_P$}
  \If{either of Include($r_{3,\mathbf{x}}(\mathbf{y})$,$\mathcal{C}''_P$)  or Include($r^2_{3,\mathbf{x}}(\mathbf{y})$, $\mathcal{C}''_P$) is False}
  \State r $\gets$ False
  \State break 
	     \EndIf
  	    \EndFor
	    \If {r is True and Include($\mathbf{x}$, H\_free\_rot)=False}
	    \State add $\mathbf{x}$ to H\_free\_rot
	    \EndIf
	    \If {H\_free\_rot has more then one element}
	    \State exist\_free\_rot $\gets$ True
	    \State break 
	    \EndIf
            \EndFor
            \State \Return exist\_free\_rot
\EndFunction
\end{algorithmic}
\label{twocentroidalg}
\end{algorithm}

\begin{remark}
We make a note of the following regarding Algorithm \ref{twocentroidalg} and the code \texttt{TwoCentroids.py}. 
\begin{itemize}
\item We imported SnapPy \cite{snappy} as a module inside \texttt{TwoCentroids.py}. We also imported the \texttt{sqrt} and \texttt{ceil} functions from the \texttt{math} module inside  \texttt{TwoCentroids.py}.

\item Observe that we use this code to rule out cases, i.e., if \texttt{Free\_rot\_strng(M, i)} returns \texttt{False} for a cusp $i$ of manifold $M$, we conclude that the $i$-circle packing does not have two required order $3$ symmetries that Proposition \ref{codeprop} and Proposition \ref{codeprop20v0} mandate. Since Python floating point numbers are not exact, in order to reduce approximation error, we have added ``0.005" to the bounds of some of the conditionals in the code.

\item In the actual code \texttt{TwoCentroids.py} (in \cite{tetra_code}), in order to get the full sized $i$-horoballs in the cusp parallelogram (i.e. horoballs with centers in $\mathcal{C}_P$), we have set the \texttt{cutoff} in the definition of the Python set \texttt{H\_circ\_cen} as $0.9$ instead of $1$ (\texttt{cutoff} in SnapPy \cite{snappy} means lower cutoff of the diameter of the horoballs) to reduce the error of missing out a full sized $i$-horoball. 
\item In order to get $r_{3,\mathbf{x}}(\mathbf{y})$ and $r^2_{3,\mathbf{x}}(\mathbf{y})$, we have written two functions \texttt{R\_1} and \texttt{R\_2} in the Python file \texttt{TwoCentroids.py} (in \cite{tetra_code}) both of which take two input $\mathbf{x}$ and $\mathbf{y}$. 
\item We wrote a function \texttt{Cos\_angle} in \texttt{TwoCentroids.py} (in \cite{tetra_code}) in order to shorten the expression $\frac{\lvert \mathbf{u}.\mathbf{v}\rvert}{\lVert \mathbf{u}\rVert \lVert \mathbf{v}\rVert}$, i.e., \texttt{Cos\_angle} returns this expression when $\mathbf{u}$ and $\mathbf{v}$ are passed as its two input. 
 
\item We also note in Algorithm \ref{twocentroidalg} that we set $0.1$ as \texttt{cutoff} in the definition of the Python set \texttt{H\_diff\_cen} which returns the centers of a list of non $i$-horoballs. Since we are trying to rule out the cases where the fixed points of the order $3$ rotational symmetries happen to be the centers of non $i$-horoballs, taking a smaller cutoff in the definition of the Python set \texttt{H\_diff\_cen} would (potentially) rule out more cases, but it will increase computation time. 
\end{itemize}
\end{remark}

\section{Tetrahedral links}\label{tlcom}
In this section, we will apply our results and the SnapPy code from the previous section to certain elements of a family of links called tetrahedral links. We explain below why this family of links are natural candidates to analyze in our study of knot complement with hidden symmetries. 

A hyperbolic manifold $M$ is called a \textit{tetrahedral manifold} if it can be triangulated into regular ideal hyperbolic tetrahedra. A hyperbolic link whose complement is a tetrahedral manifold is referred to as a \textit{tetrahedral link}. For example, the figure eight knot is a tetrahedral knot since its complement decomposes into two regular ideal tetrahedra. Fominykh, Garoufalidis, Goerner, Tarkaev and Vesnin \cite{FGGTV} found a census consisting of the orientable tetrahedral manifolds with decomposition into 25 or fewer (regular ideal) tetrahedra and the non-orientable tetrahedral manifolds with decomposition into 21 or fewer (regular ideal) tetrahedra. 

All the \textit{shape parameters} (see \cite[Section 2.2]{NeRe}, \cite[Chapter 4]{Thurs}) of a regular ideal tetrahedron are $\frac{1+i\sqrt{3}}{2}$. So, \cite[Theorem 2.4]{NeRe} implies that the invariant trace field of a tetrahedral manifold is $\mathbb{Q}(i\sqrt{3})$. So, the cusp field associated to a cusp of a tetrahedral manifold is $\mathbb{Q}(i\sqrt{3})$ as well since by \cite[Proposition 2.7]{NeRe} it lies in the invariant trace field  and cusp moduli are non-real complex numbers. Now, \cite[Corollary 1.4]{CDM} (or Theorem \ref{CDHMMWorb}) implies that if Dehn filling all but one component of a link $L$ produces an infinitely family of hyperbolic knot complements with hidden symmetries which geometrically converges to $\mathbb{S}^3-L$, then the cusp field of the non-filled cusp of $L$ must be $\mathbb{Q}(i)$ or $\mathbb{Q}(i\sqrt{3})$. Since the cusp field of all the cusps of a tetrahedral link is $\mathbb{Q}(i\sqrt{3})$, tetrahedral links are natural candidates for testing whether by Dehn filling all but one cusp we can obtain an infinitely family of hyperbolic knot complements with hidden symmetries and geometrically converging to the original link complement.

\subsection{Tetrahedral homology link complements}
Fominykh, Garoufalidis, Goerner, Tarkaev and Vesnin \cite{FGGTV} also gave an explicit list of tetrahedral links (see Figure 3, 4 and 5 in their paper). Their list consists of the figure eight knot and 25 other tetrahedral links all with two or more components. One of these links has 20 and the rest has 12 or fewer (regular ideal) tetrahedra in their triangulations. We note that the links \texttt{L14n24613}, \texttt{L11n354}, \texttt{L10a157} and \texttt{L8a20} of the Hoste-Thistlewaite census discussed in Section \ref{packingbackground} and \ref{hscp} belong to this list of $25$ links with two or more components from \cite{FGGTV}. The number of total link complements in their orientable tetrahedral census would be far bigger (to the author's knowledge this number is not known) and it is not an easy task to find out whether an orientable tetrahedral manifold is a link complement or not. But, we should note that tetrahedral links are links in the integral homology spheres and from the following result from \cite{FGGTV} one can easily check when a cusped hyperbolic $3$-manifold is a link complement in an integral homology sphere.

\begin{proposition}[Lemma 6.1, \cite{FGGTV}]\label{hlinkprop}
A (finite-volume) cusped hyperbolic $3$-manifold $M$ is a link complement in an integral homology sphere if and only if the first homology group $H_1(M,\mathbb{Z})$ of $M$ is torsion free and $rank(H_1(M,\mathbb{Z}))$ is equal to the number of cusps of $M$. 
\end{proposition}
The orientable tetrahedral census of \cite{FGGTV} is embedded in SnapPy \cite{snappy}. We use the above proposition to write the functions \texttt{Hlgy\_tor} and \texttt{Link\_HS} on SnapPy \cite{snappy} and extract a Python list \texttt{K\_hlgy\_lk} on SnapPy \cite{snappy} which consists of all the homology links in the embedded orientable tetrahedral census of \cite{FGGTV} with more than one cusp (and $25$ or fewer regular ideal tetrahedra in their regular ideal tetrahedral decompositions). In addition, we also write the \texttt{Cusp\_list} function SnapPy \cite{snappy} which extracts a maximal list of pairwise non-symmetric cusps of a SnapPy manifold. The pseudocode for \texttt{Hlgy\_tor}, \texttt{Link\_HS},  \texttt{K\_hlgy\_lk} and \texttt{Cusp\_list}  is given in Algorithm \ref{K_hlgyLk algorithm}. The code is available in the Python file \texttt{Tetrahedral\_hlgy\_link.py} from \cite{tetra_code}. We imported SnapPy \cite{snappy} as a module inside \texttt{Tetrahedral\_hlgy\_link.py}.

\begin{algorithm}
\caption{Extracting the homology links in the orientable tetrahedral census of \cite{FGGTV} and their maximal lists of pair-wise non-symmetric cusps}
\begin{algorithmic}
\Function {Hlgy\_tor}{SnapPy manifold M}
\State r $\gets$ False
\If {the first elementary divisor of the first homology group of M in SnapPy $\ne 0$}
\State r $\gets$ True
\EndIf
\State \Return r 
\EndFunction
\State
\Function{Link\_HS}{SnapPy manifold M}
\State  r $\gets$ False
\If {Hlgy\_tor(M)=False}
\If {rank of the first homology group of M $=$ number of cusps of M}
\State r $\gets$ True
\EndIf
\EndIf
\State \Return r
\EndFunction
\State
\State K\_hlgy\_lk$\gets$ empty list
\State
\For {$i \in \{1, \dots, 25\}$}
\For {$M\in$ orientable tetrahedral census of \cite{FGGTV} in SnapPy with $i$ regular ideal tetrahedra}
\If{Link\_HS($M$)=True}
\If {$M$ has more than one cusp}
\State add $M$ to the list K\_hlgy\_lk
\EndIf
\EndIf
\EndFor
\EndFor

\State

\Function {Cusp\_list}{SnapPy manifold M}
\State L $\gets \{0, \dots, \text{number of cusps of M}-1\}$ as a list
\State K $\gets$ empty list
\For {each self-isometry $g$ of M}
\For {each $i\in$ L}
\If{$g(i)>i$ and $g(g(i))=i$}
\State add $g(i)$ to K
\EndIf
\EndFor
\EndFor
\State \Return L minus the members of K
\EndFunction
\end{algorithmic}
\label{K_hlgyLk algorithm}
\end{algorithm}

In particular, this \texttt{K\_hlgy\_lk} contains (may be strictly) all the tetrahedral links (with two or more components) in the orientable tetrahedral census of \cite{FGGTV} (i.e. they have 25 or fewer regular ideal tetrahedra in their triangulations). One can see (for example by typing \texttt{len(K\_hlgy\_lk)} in SnapPy/Python) that \texttt{K\_hlgy\_lk} contains 882 elements. 

The \texttt{Cusp\_list} function which when run on SnapPy \cite{snappy} returns a maximal list of (labels of) cusps of its argument \texttt{M} (a SnapPy manifold) with the property that no two elements in this returned list are exchanged by a self-isometry of \texttt{M}. 

We note that if a Dehn filling on all but one cusp of homology link complement gives a hyperbolic knot complement, the homology link complement must be a link complement. Our goal is to understand whether for each element $M$ of \texttt{K\_hlgy\_lk}, an infinite family of hyperbolic knot complements obtained from Dehn filling all but one cusp of $M$ and geometrically converging to $M$ can have hidden symmetries. We will use Theorem \ref{hexlatt} to find out whether for a given $M$ in \texttt{K\_hlgy\_lk} and a cusp $c$ of $M$, the $c$-circle packing of $\mathbb{C}$ can have the required order $3$ rotational symmetries. We will apply the algorithm from Subsection \ref{hscodesubs} on the elements of \texttt{K\_hlgy\_lk} for this objective. 

We recall from Fact \ref{symmfact} that if there is a self-isometry of a hyperbolic link complement $\mathbb{S}^3-L$ exchanging two of its cusps $c_1$ and $c_2$, then, the members of a family of hyperbolic knot complements obtained from Dehn filling all cusps of $\mathbb{S}^3-L$  but $c_1$ and geometrically converging to $\mathbb{S}^3-L$  are each isometric to a member of a family of hyperbolic knot complements obtained from Dehn filling all cusps of $\mathbb{S}^3-L$  but $c_2$ and geometrically converging to $\mathbb{S}^3-L$ . Motivated by this we are interested in the equivalence relation on the set of cusps of a manifold $M$ such that two cusps of $M$ belong to an equivalence class if and only if they are symmetric. Keeping this in mind, given a homology link complement $M$ in \texttt{K\_hlgy\_lk}, we will run the \texttt{Free\_rot\_strng} function from Subsection \ref{hscodesubs} on SnapPy \cite{snappy} only for a single representative of each of these equivalence classes, i.e. only for cusps of $M$ that appear in \texttt{Cusp\_list(M)}, in order to reduce computation time.  We introduce the following notation before we end this subsection.

 \begin{defn}
We define $\mathcal{C}(M)$ to be the set of numbers that belong to the Python list that the \texttt{Cusp\_list} function returns on SnapPy \cite{snappy} when the manifold $M$ is taken as an input. 
\end{defn}

\subsection{Running the SnapPy code from Subsection \ref{hscodesubs}} \label{snappyrun}

We recall from the algorithm in Subsection \ref{hscodesubs} that the \texttt{Free\_rot\_strng} function takes two arguments. For a pair of arguments \texttt{(M,i)} where \texttt{M} is a SnapPy manifold with $n$ cusps and \texttt{i} an integer in \texttt{[0, ... , n-1]} (a Python list), if \texttt{Free\_rot\_strng(M,i)} returns \texttt{False}, then, the symmetry group of the $i$-circle packing of $\mathbb{C}$ does not contain required $W_{(3,3,3)}$ group of symmetries as in Theorem \ref{hexlatt} . Now, one can check on Python/SnapPy that the set $\left\{(M,i): M \in \texttt{K\_hlgy\_lk}, i \in \mathcal{C}(M)\right\}$ has $3026$ elements. 

We run the \texttt{Free\_rot\_strng} function on SnapPy for all such $3026$ pairs, i.e., for all pairs $(M,i)$ such that $M \in \texttt{K\_hlgy\_lk}$ and $i\in \mathcal{C}(M)$, via \texttt{Compute.py} in \cite{tetra_code}. 
Define
\begin{align*}
\mathcal{N} &= \left \{(M,i): M \in \texttt{K\_hlgy\_lk}, i \in \mathcal{C}(M) \text{ and } \texttt{Free\_rot\_strng(M,i)}=\texttt{False} \right \}, \\
\mathcal{E} &= \left \{(M,i): M \in \texttt{K\_hlgy\_lk}, i \in \mathcal{C}(M) \text{ and } \texttt{Free\_rot\_strng(M,i)}=\texttt{True} \right \}.
\end{align*}

For an \texttt{(M,i)} in $\mathcal{N}$,  the symmetry group of $i$-circle packing of $\mathbb{C}$ does not contain $W_{(2,3,6)}$ or $W_{(3,3,3)}$ satisfying the conditions of Theorem \ref{hexlatt} whereas for an \texttt{(M,i)} in $\mathcal{E}$  we cannot conclude anything as such for the $i$-circle packing of $\mathbb{C}$. 

We import \texttt{TwoCentroids}, \texttt{Tetrahedral\_hlgy\_link} and \texttt{json} as Python modules in \texttt{Compute.py} (the code is available at \cite{tetra_code}).  \texttt{Compute.py} exports a list representing $\mathcal{E}$ in the \texttt{Excep\_tuple.json} file (can also be accessed from \cite{tetra_code}). The elements of this list are $4$-tuples each representing a member $(M,i)\in \mathcal{E}$. The $4$-tuple corresponding an $(M,i)\in \mathcal{E}$ is of the form $(j, \text{namestring}, i, \text{sigstring})$ where $j$ is the index of $M$ in $\texttt{K\_hlgy\_lk}$, $\text{namestring}$ is the name of $M$ (in the notation of \cite{FGGTV}) as a Python string and \text{sigstring} is the Python string that is the default isomorphism signature of $M$ in SnapPy, i.e. the isomorphism signature for the default tetrahedral decomposition of \texttt{Manifold(namestring)} in SnapPy \cite{snappy} (\texttt{Manifold(namestring).triangulation\_isosig(decorated=False)} in SnapPy \cite{snappy} would return sigstring).

We see that there are $86$ pairs $(M,i)$ in $\mathcal{E}$. We list the members of $\mathcal{E}$ in Appendix \ref{allEi}. So, Theorem \ref{hexlatt} and Proposition \ref{codeprop} implies that for each $(M,i) \in \left\{(M,i): M \in \texttt{K\_hlgy\_lk}, i \in \mathcal{C}(M)\right\}-\mathcal{E}$, any family of hyperbolic knot complements obtained from Dehn filling all cusps of $M$ but cusp $i$ and geometrically converging to $M$ can have at-most finitely many elements with hidden symmetries.

For an $M$ in \texttt{K\_hlgy\_lk}, we define
$$\mathcal{C}_{\mathcal{E}}(M)=\left\{ i \in \mathcal{C}(M): (M,i) \in \mathcal{E}\right\}.$$

From Appendix \ref{allEi}, one can see that we have split the set $\mathcal{E}$ into four mutually disjoint subsets $\mathcal{E}_1$, $\mathcal{E}_2$, $\mathcal{E}_3$ and $\mathcal{E}_4$ so that $\bigcup\limits_{i=1}^{4} \mathcal{E}_i=\mathcal{E}$. We delay the discussion on how we get this split (i.e. how we get $\mathcal{E}_1$ and $\mathcal{E}_4$) to Subsection \ref{E1E4} as it would require us to know the notion of orbifold triangulation from \cite{orbcenpract} and \cite{orbcentheory}, which we review in Subsection \ref{censusutilities}. In Subsection \ref{E1E4}, we also analyze the elements of $\mathcal{E}_1$ and $\mathcal{E}_4$ (we need the utilities \texttt{SigToSeq.py} and \texttt{TestForCovers.py} of \cite{orbcenpract}, reviewed in Subsection \ref{censusutilities} and available from \cite{orbtricode}, for this analysis). 
In the rest of this subsection, we discuss the pairs in $\mathcal{E}_2$ and $\mathcal{E}_3$, i.e., pairs in $\mathcal{E}-\left(\mathcal{E}_1 \cup \mathcal{E}_4\right)$.

At first, we focus on the elements of $\mathcal{E}_2$. The $i$-maximal horoball packings of $\mathbb{H}^3$ for $M$ with an $i$-horoball at $\infty$ where $(M,i) \in \mathcal{E}_2$ are shown in Figures \ref{bad6sym_3}, \ref{bad6sym_4}, \ref{bad6sym_22}, \ref{bad6sym_58}, \ref{bad6sym_74}, \ref{bad6sym_79}, \ref{bad6sym_98}, \ref{bad6sym_127}, \ref{bad6sym_145}, \ref{bad6sym_219}, \ref{bad6sym_245}, \ref{bad6sym_299}, \ref{bad6sym_307}, \ref{bad6sym_425}, \ref{bad6sym_633}, \ref{bad6sym_634}, \ref{bad6sym_636},  \ref{bad6sym_645}, \ref{bad6sym_681}, \ref{bad6sym_682}, \ref{bad6sym_683}, \ref{bad6sym_684} and \ref{bad6sym_821} . From these figures, we can see that each such horoball packing of $\mathbb{H}^3$ has an order $6$ rotational symmetry fixing a geodesic joining $\infty$ and the center of a $\tilde{c}$-horoball of $M$ where cusp $\tilde{c}$ of $M$ is different from $i$. Since the homology link complements in \texttt{K\_hlgy\_lk} that appear in $\mathcal{E}_2$ have volume less than or equal to $24 v_0$, Corollary \ref{badorder6sym} then would imply that for each $(M,i) \in \mathcal{E}_2$, a family of hyperbolic knot complements obtained from Dehn filling all cusps of $M$ but cusp $i$ and geometrically converging to $M$ can have at-most finitely many elements with hidden symmetries. 

We now consider the elements of $\mathcal{E}_3$. Figures \ref{fullhoronosym_160}, \ref{fullhoronosym_167}, \ref{fullhoronosym_172}, \ref{fullhoronosym_175}, \ref{fullhoronosym_301}, \ref{fullhoronosym_331}, \ref{fullhoronosym_333}, \ref{fullhoronosym_334}, \ref{fullhoronosym_335} and \ref{fullhoronosym_336} record the $i$-maximal horoball packings of $\mathbb{H}^3$ for $M$ with an $i$-horoball at $\infty$ where $(M,i) \in \mathcal{E}_3$. In each of these figures, we can see that $i$-maximal horoball packings do not have any order $3$ rotational symmetry fixing $\infty$ (even though the $i$-circle packings do). So, Fact \ref{hexsymmhoro} implies that for each $(M,i) \in \mathcal{E}_3$, a family of hyperbolic knot complements obtained from Dehn filling all cusps of $M$ but cusp $i$ and geometrically converging to $M$ cannot have infinitely many elements with hidden symmetries. 


\begin{figure} 
\centering 
\captionsetup{justification=centering}
\includegraphics[scale=.1]{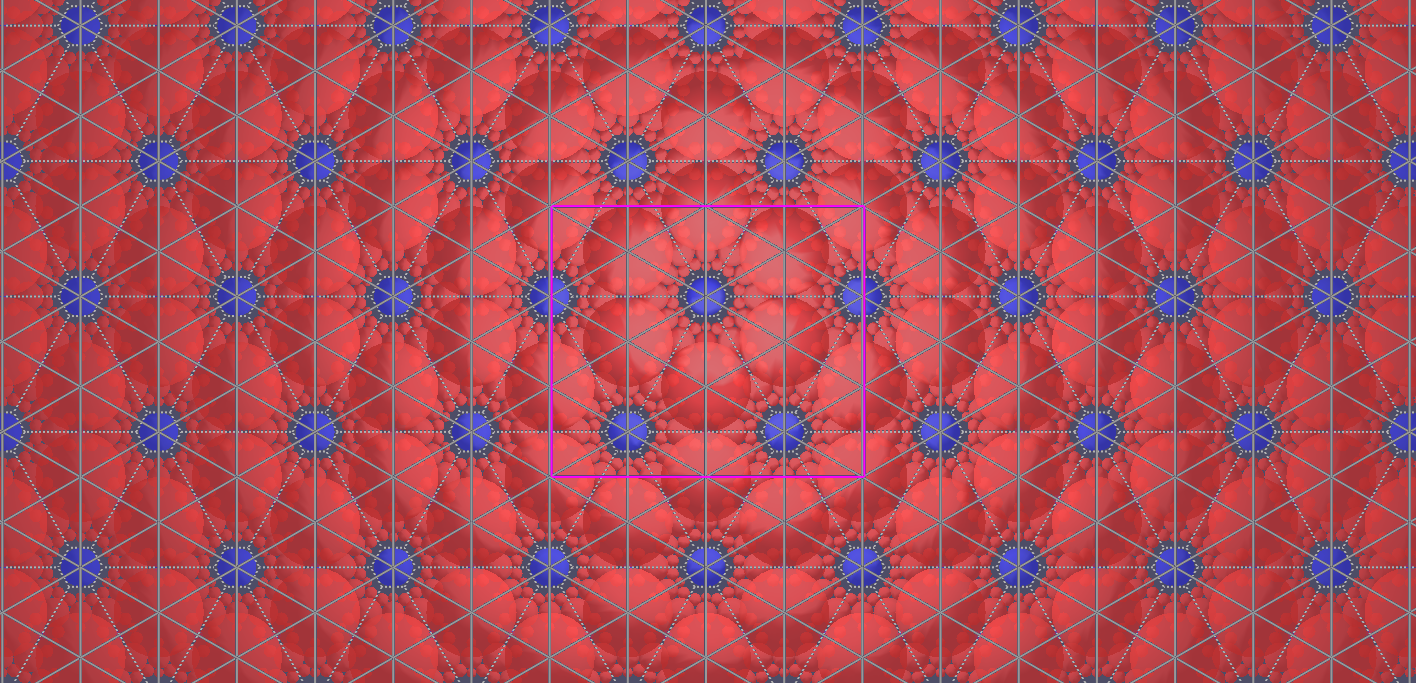} \includegraphics[scale=.1]{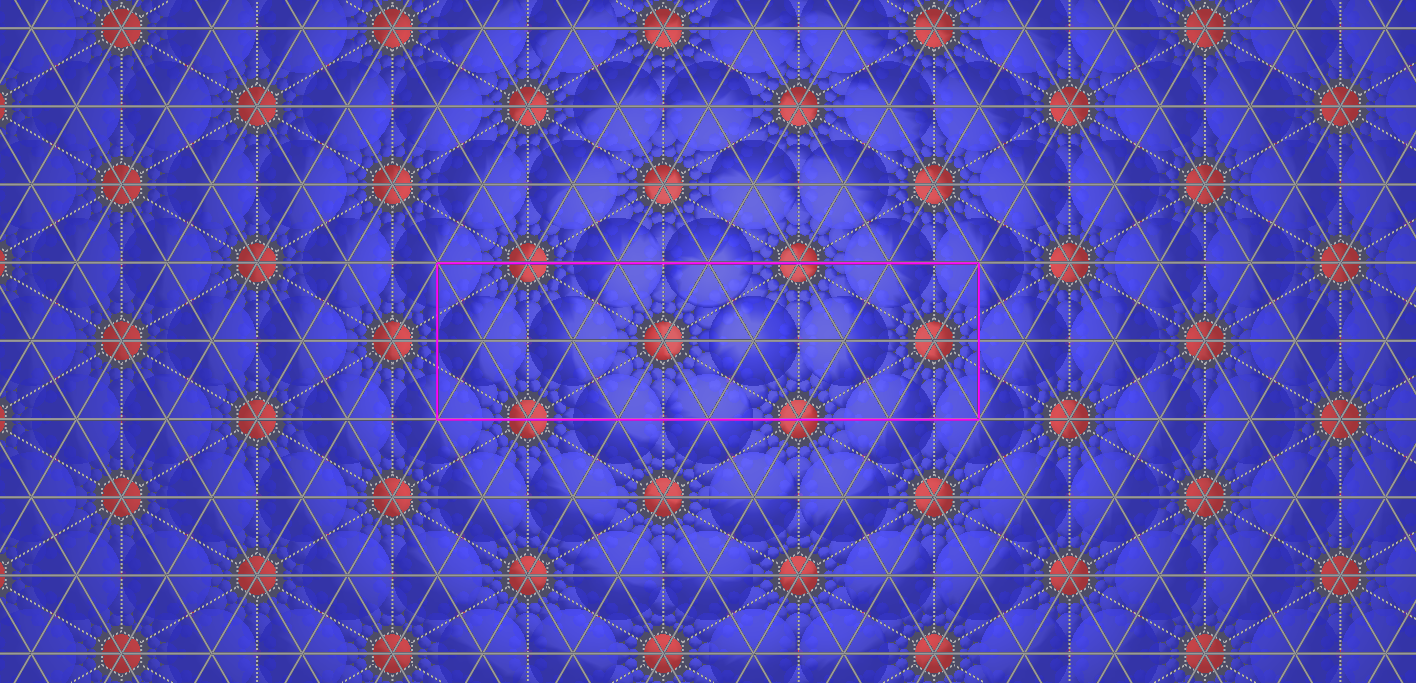} 
\caption{Horoball packings of $\mathbb{H}^3$ for \texttt{otet08\_00002} (\texttt{K\_hlgy\_lk}[3]) ; \\ Left: $0$-cusp maximal horoball packing, Right: $1$-cusp maximal horoball packing  (pictures obtained from SnapPy \cite{snappy}).}
\label{bad6sym_3}
\end{figure}

\begin{figure}
\centering 
\captionsetup{justification=centering}
\includegraphics[scale=.1]{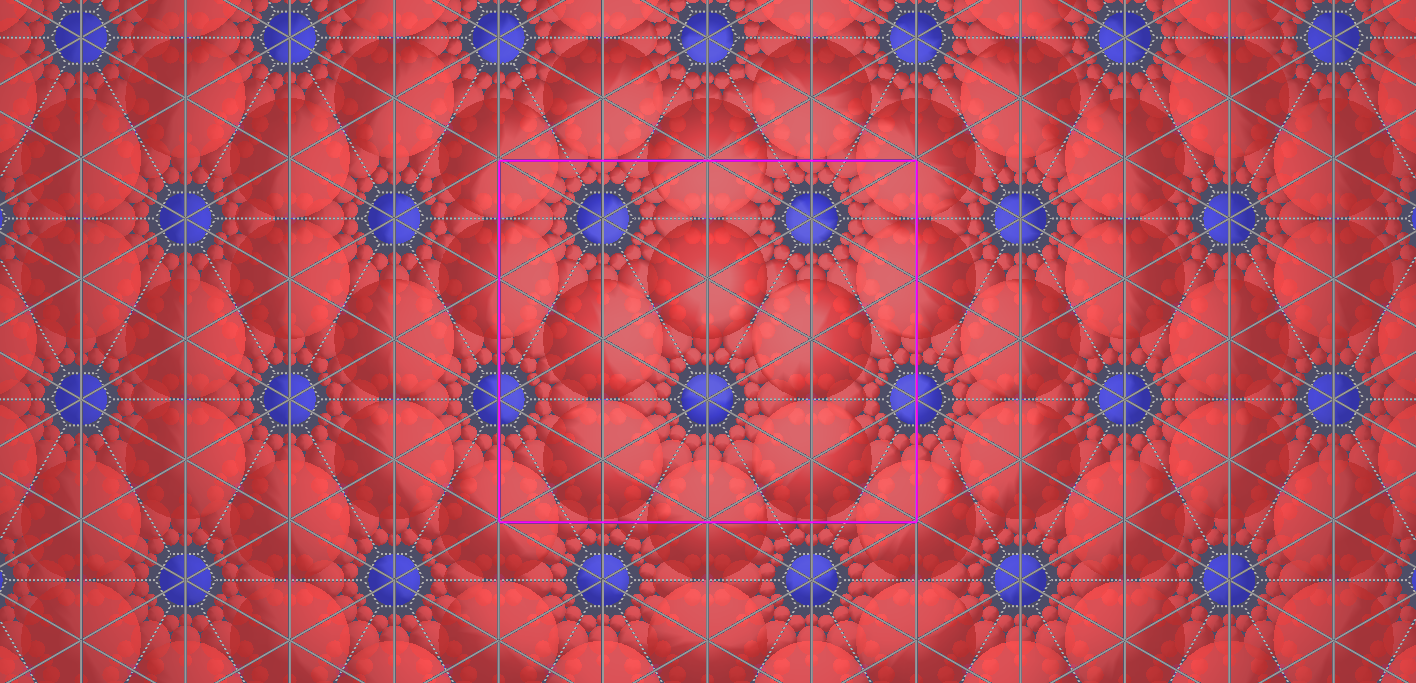}
\caption{ $0$-cusp maximal horoball packing of $\mathbb{H}^3$ for \texttt{otet08\_00003} (\texttt{K\_hlgy\_lk}[4])  (picture obtained from SnapPy \cite{snappy}).}
\label{bad6sym_4}
\end{figure}

\begin{figure}
\centering 
\captionsetup{justification=centering}
\includegraphics[scale=.1]{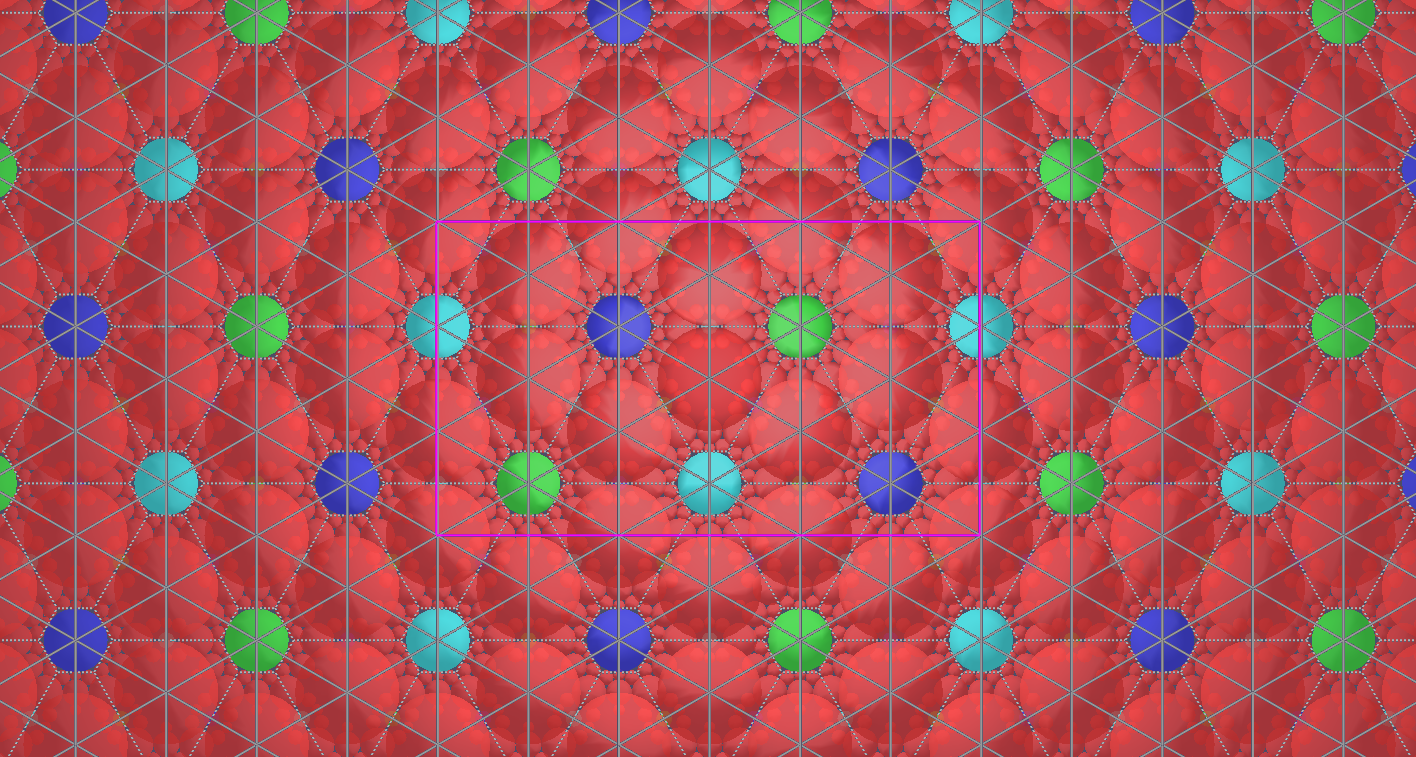}
\caption{ $0$-cusp maximal horoball packing of $\mathbb{H}^3$ for \texttt{otet12\_00009} (\texttt{K\_hlgy\_lk}[22]) (picture obtained from SnapPy \cite{snappy}).}
\label{bad6sym_22}
\end{figure}

\begin{figure}
\centering 
\captionsetup{justification=centering}
\includegraphics[scale=.1]{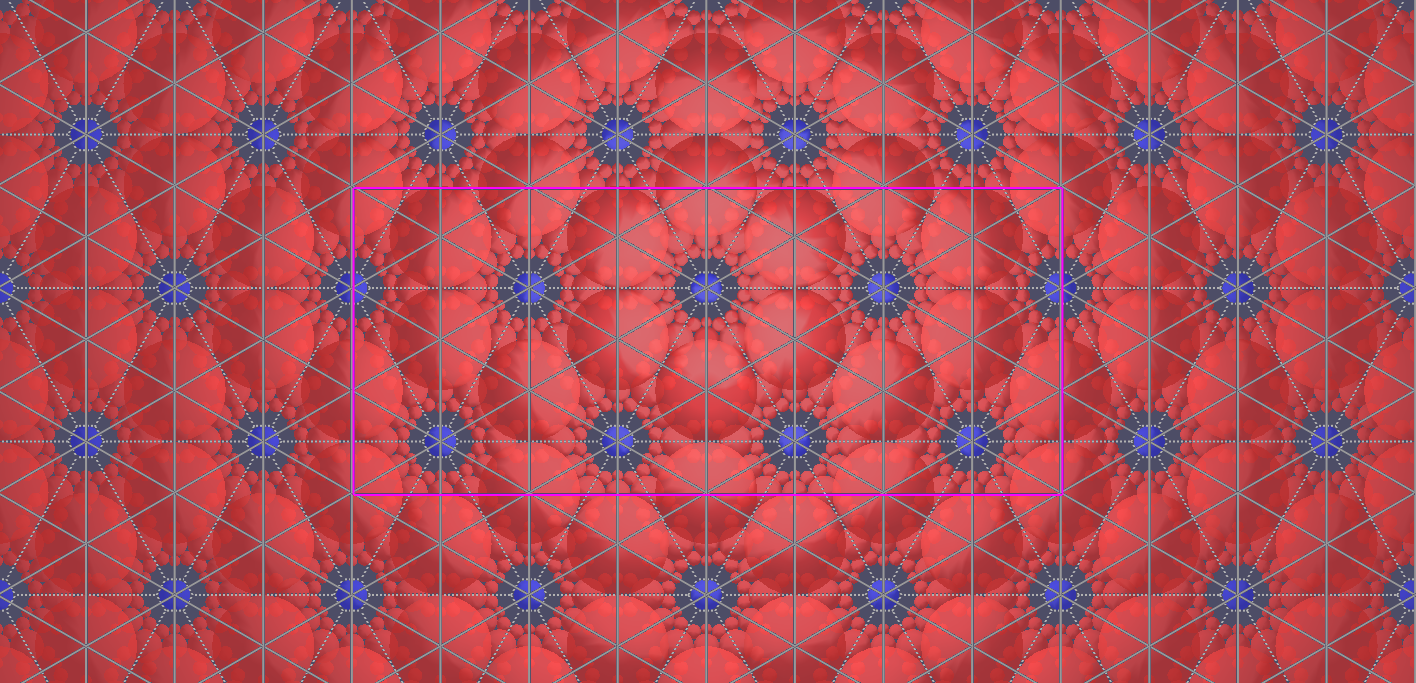} \includegraphics[scale=.1]{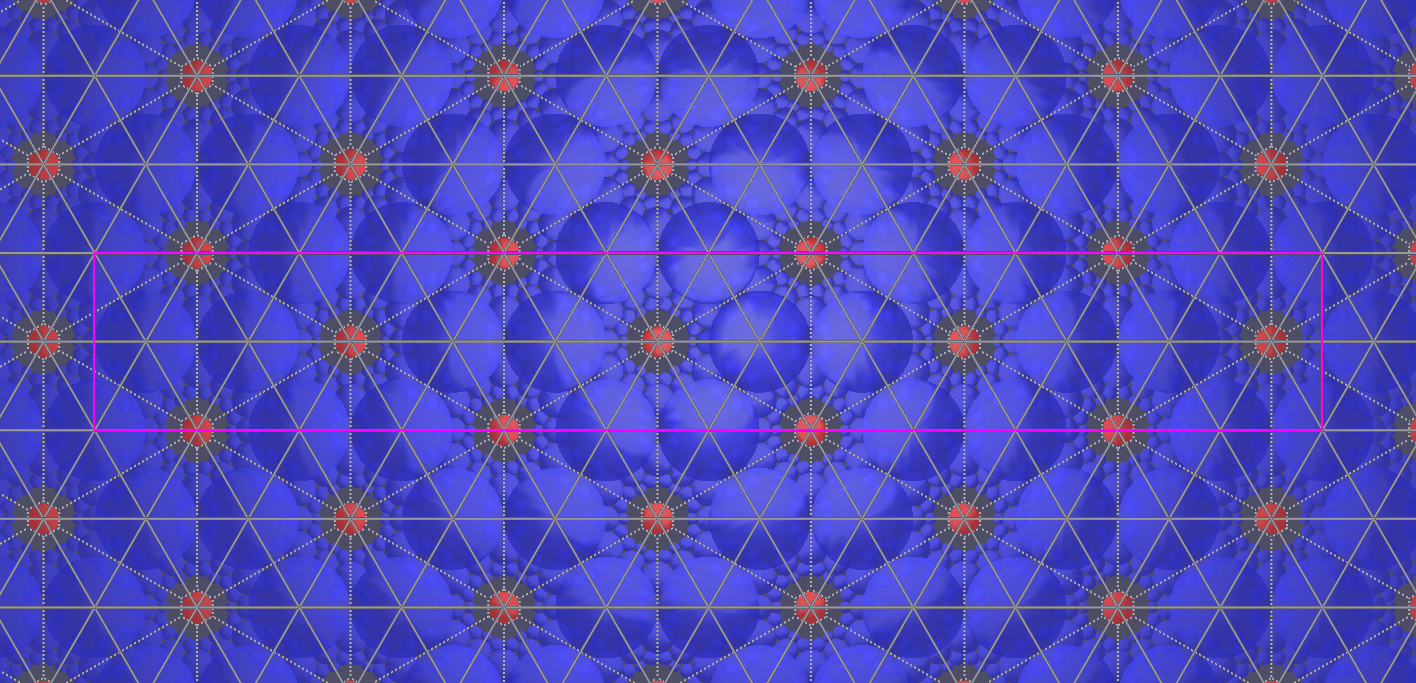} 
\caption{Horoball packings of $\mathbb{H}^3$ for \texttt{otet16\_00013} (\texttt{K\_hlgy\_lk}[58]); \\ Left: $0$-cusp maximal horoball packing, Right: $1$-cusp maximal horoball packing  (pictures obtained from SnapPy \cite{snappy}).}
\label{bad6sym_58}
\end{figure}

\begin{figure}
\centering 
\captionsetup{justification=centering}
\includegraphics[scale=.1]{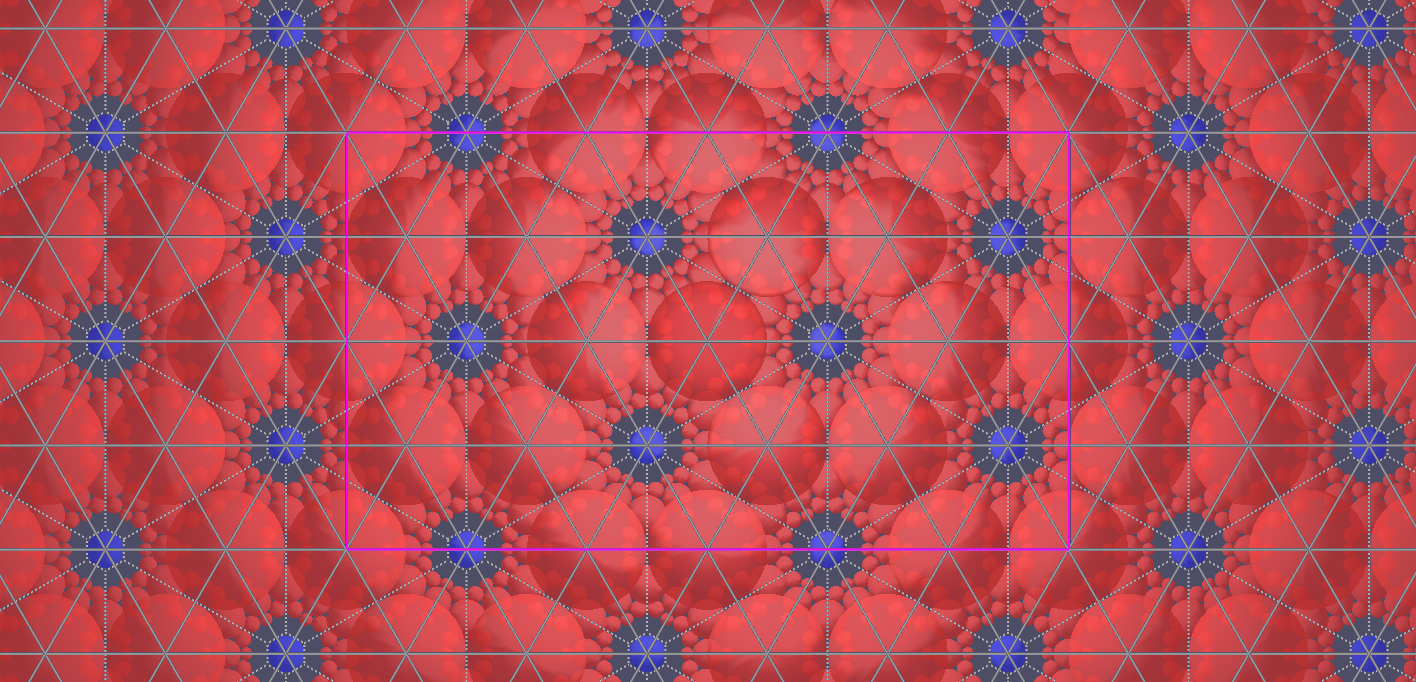} \includegraphics[scale=.1]{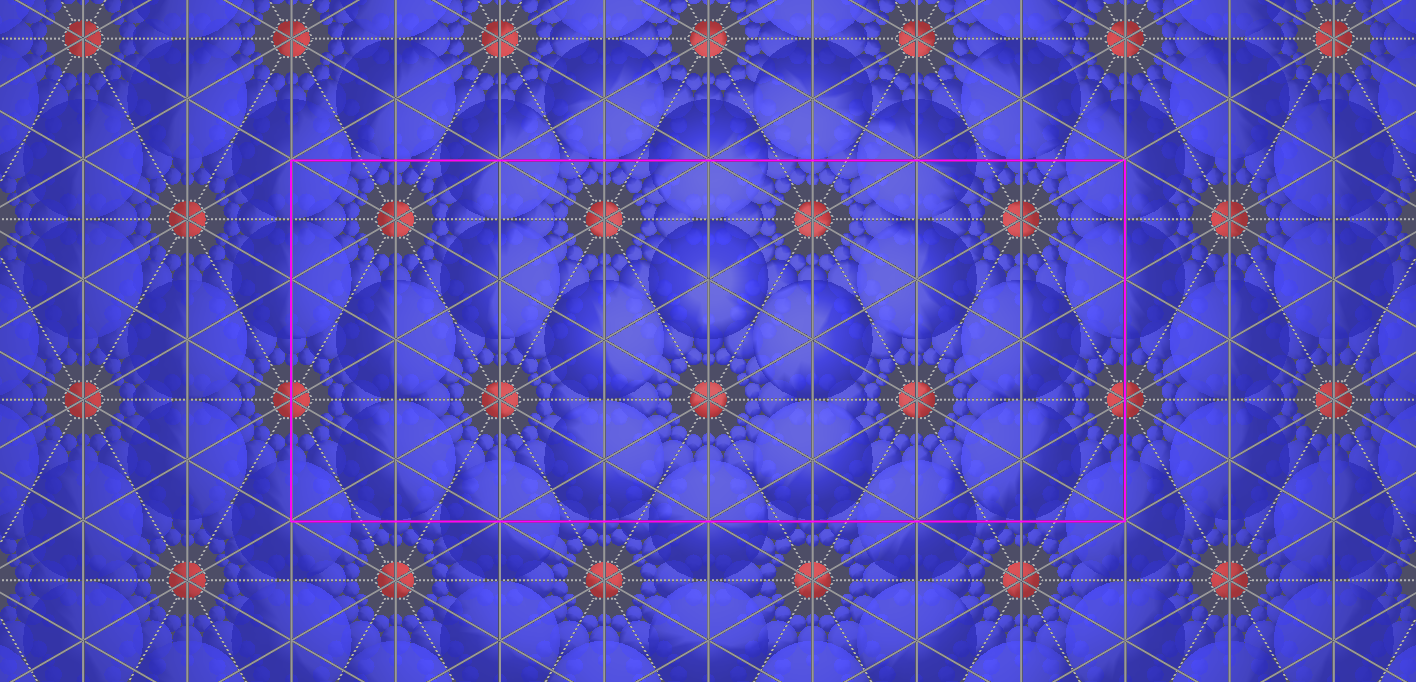} 
\caption{Horoball packings of $\mathbb{H}^3$ for \texttt{otet16\_00058} (\texttt{K\_hlgy\_lk}[74]); \\ Left: $0$-cusp maximal horoball packing, Right: $1$-cusp maximal horoball packing  (pictures obtained from SnapPy \cite{snappy}).}
\label{bad6sym_74}
\end{figure}

\begin{figure}
\centering 
\captionsetup{justification=centering}

\includegraphics[scale=.1]{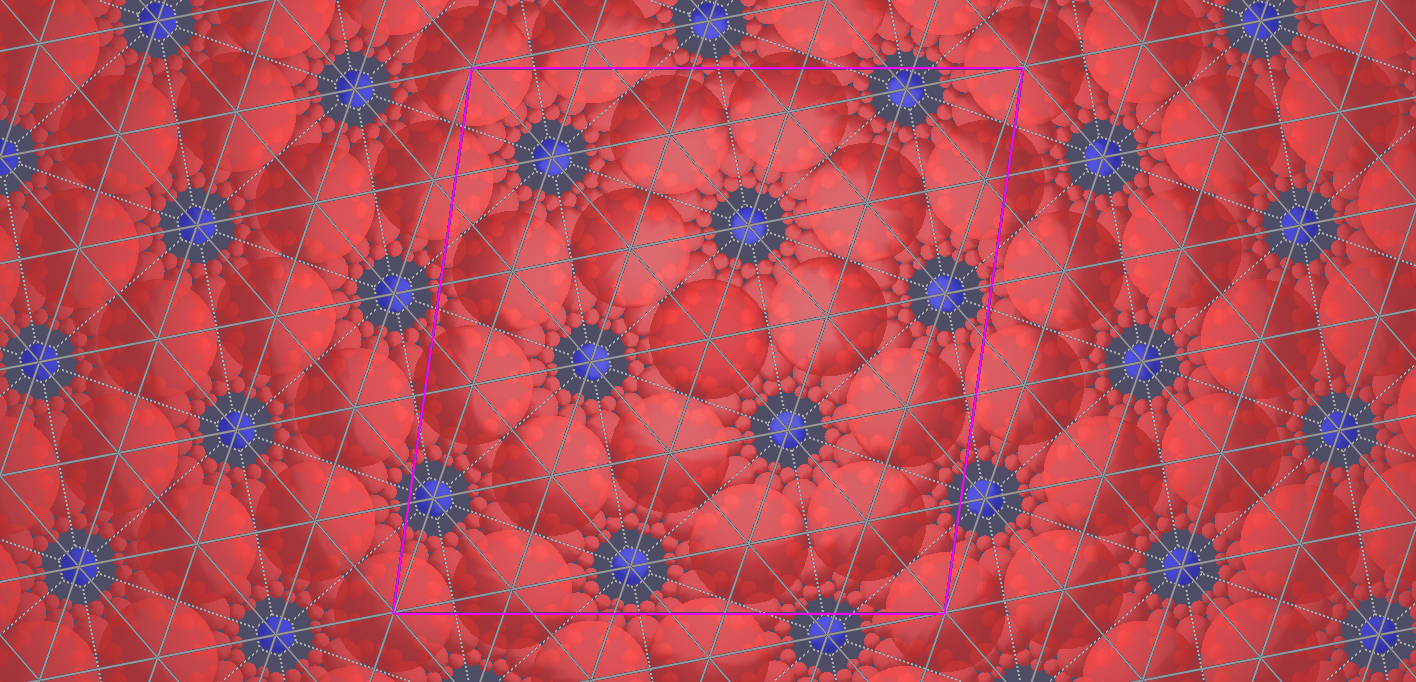}
\caption{ $0$-cusp maximal horoball packing of $\mathbb{H}^3$ for \texttt{otet16\_00090} (\texttt{K\_hlgy\_lk}[79]) (picture obtained from SnapPy \cite{snappy}).}
\label{bad6sym_79}
\end{figure}

\begin{figure}
\centering 
\captionsetup{justification=centering}
\includegraphics[scale=.1]{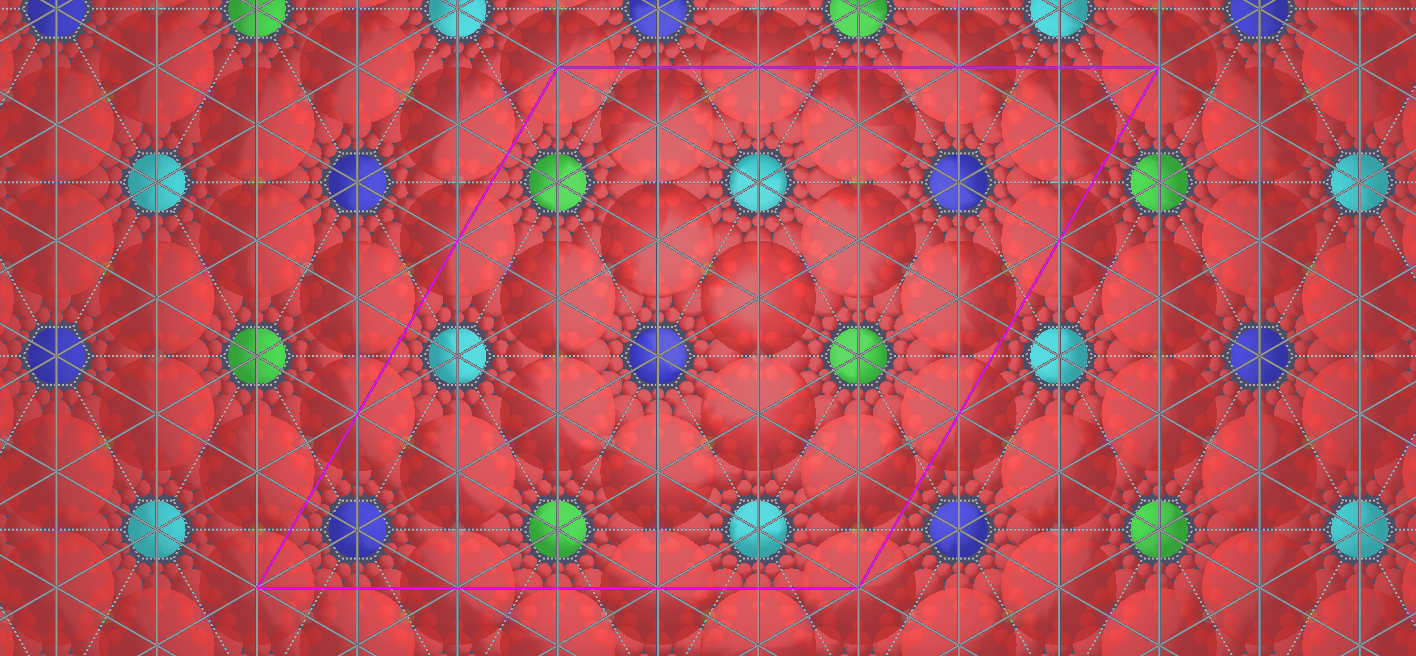} \includegraphics[scale=.1]{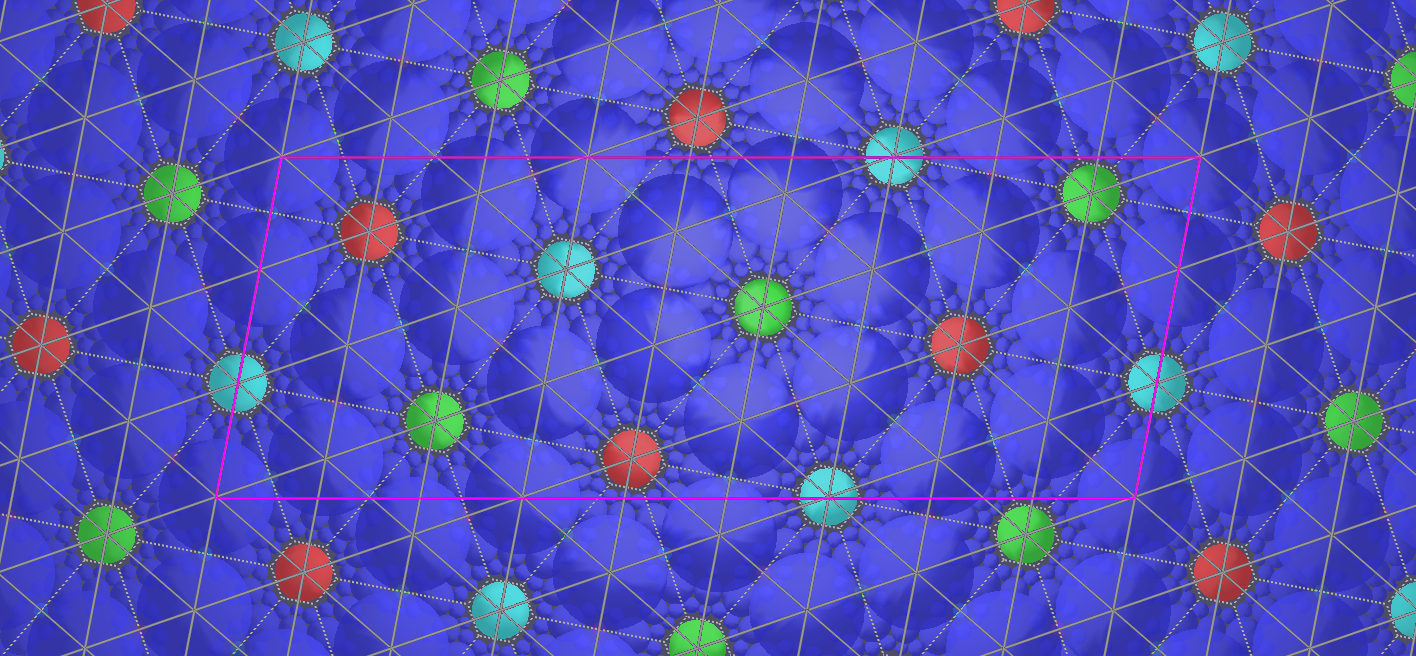} 
\caption{Horoball packings of $\mathbb{H}^3$ for \texttt{otet18\_00028} (\texttt{K\_hlgy\_lk}[98]); \\ Left: $0$-cusp maximal horoball packing, Right: $1$-cusp maximal horoball packing  (pictures obtained from SnapPy \cite{snappy}).}
\label{bad6sym_98}
\end{figure}

\begin{figure}
\centering 
\captionsetup{justification=centering}
\includegraphics[scale=.1]{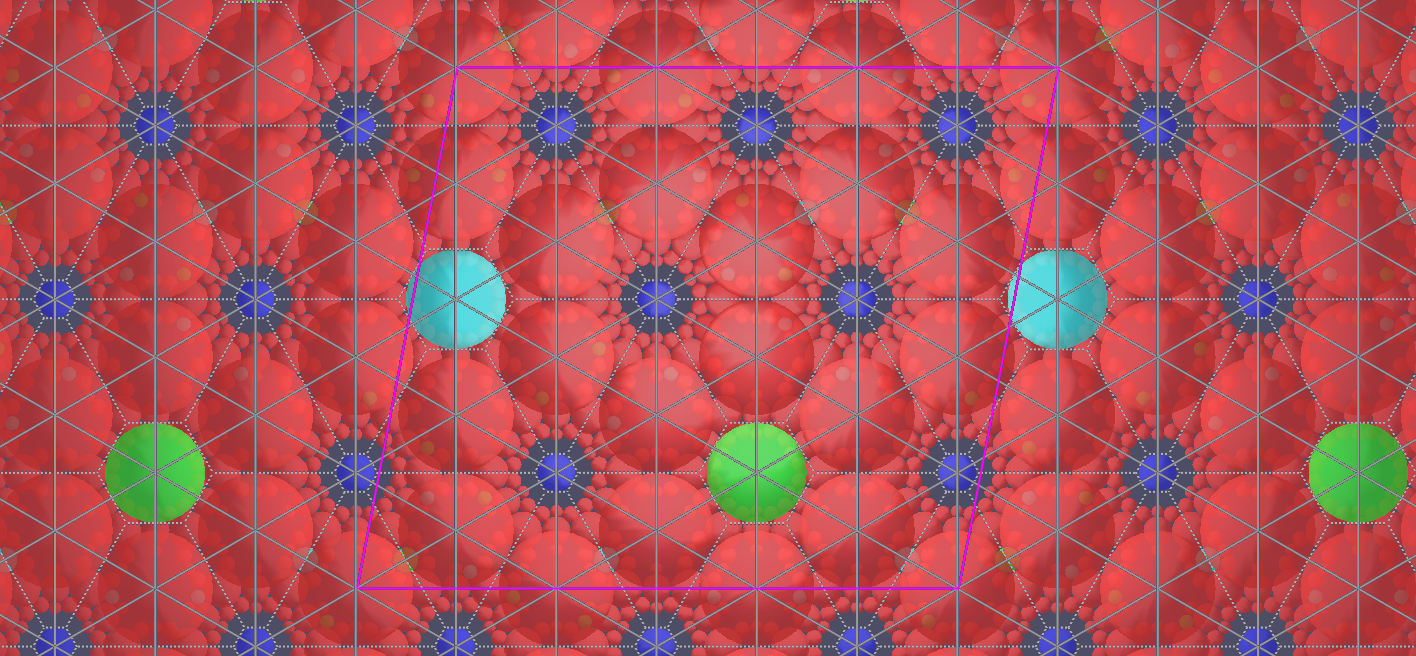}
\caption{ $0$-cusp maximal horoball packing of $\mathbb{H}^3$ for \texttt{otet18\_00104} (\texttt{K\_hlgy\_lk}[127]) (picture obtained from SnapPy \cite{snappy}).}
\label{bad6sym_127}
\end{figure}

\begin{figure}
\centering 
\captionsetup{justification=centering}
\includegraphics[scale=.1]{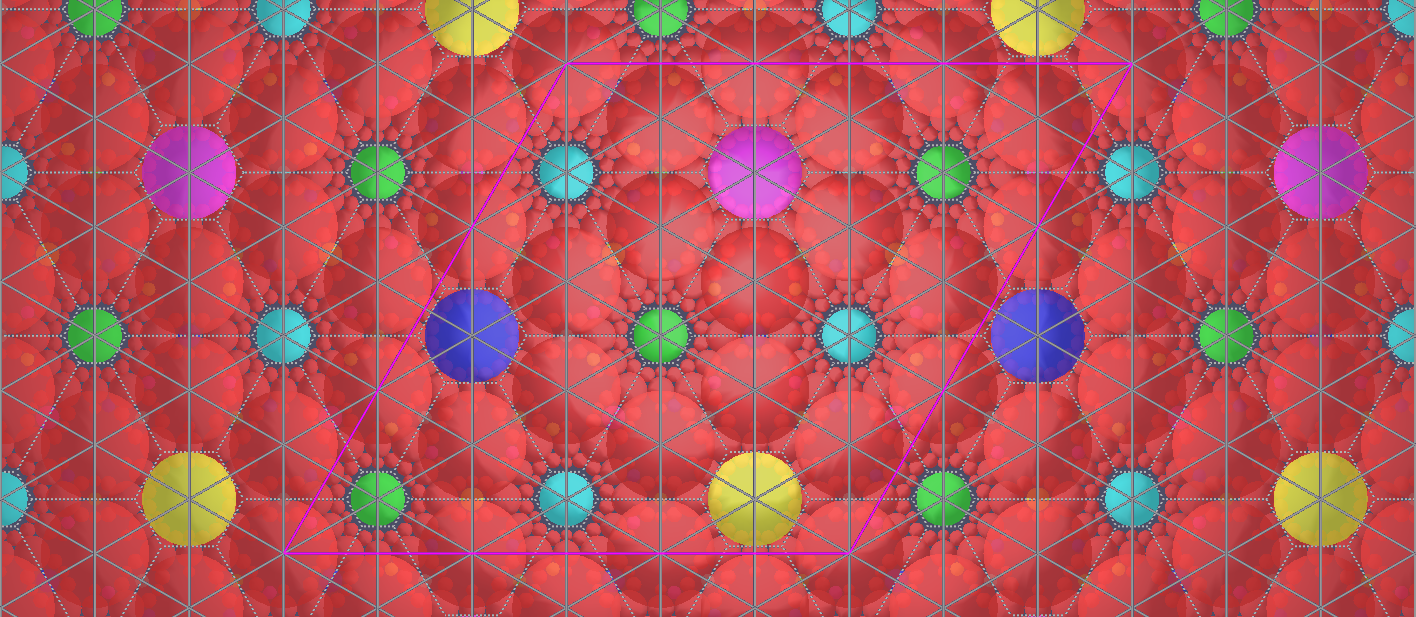}
\caption{ $0$-cusp maximal horoball packing of $\mathbb{H}^3$ for \texttt{otet18\_00171} (\texttt{K\_hlgy\_lk}[145]) (picture obtained from SnapPy \cite{snappy}).}
\label{bad6sym_145}
\end{figure}

\begin{figure}
\centering 
\captionsetup{justification=centering}
\includegraphics[scale=.1]{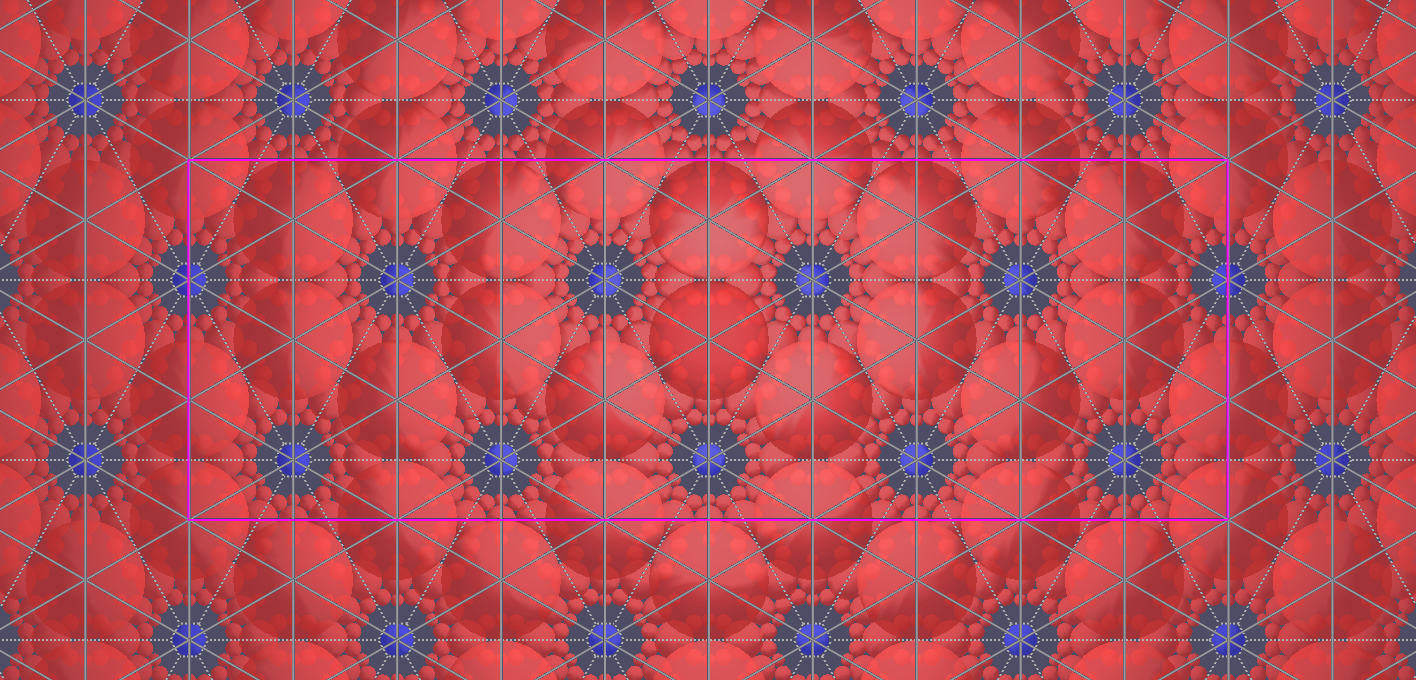} \includegraphics[scale=.1]{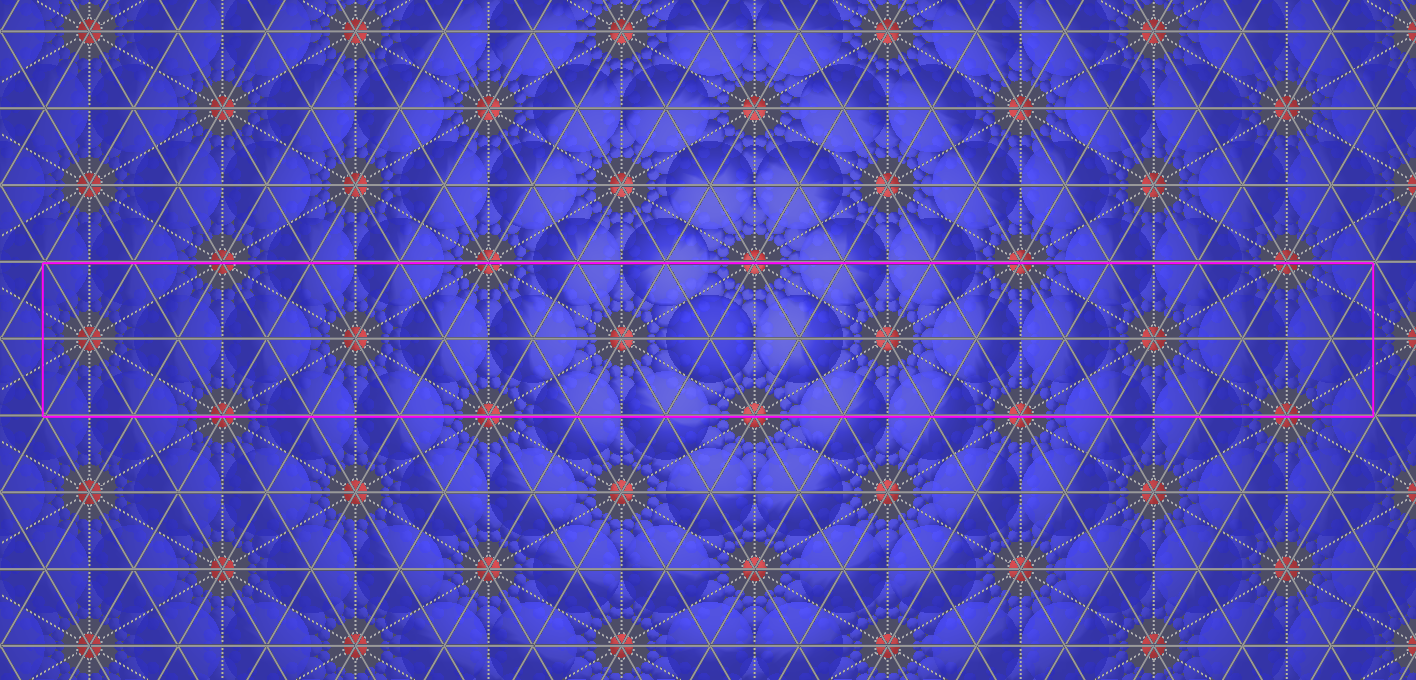} 
\caption{Horoball packings of $\mathbb{H}^3$ for \texttt{otet20\_00443} (\texttt{K\_hlgy\_lk}[219]) ; \\ Left: $0$-cusp maximal horoball packing, Right: $1$-cusp maximal horoball packing  (pictures obtained from SnapPy \cite{snappy}).}
\label{bad6sym_219}
\end{figure}

\begin{figure}
\centering 
\captionsetup{justification=centering}
\includegraphics[scale=.1]{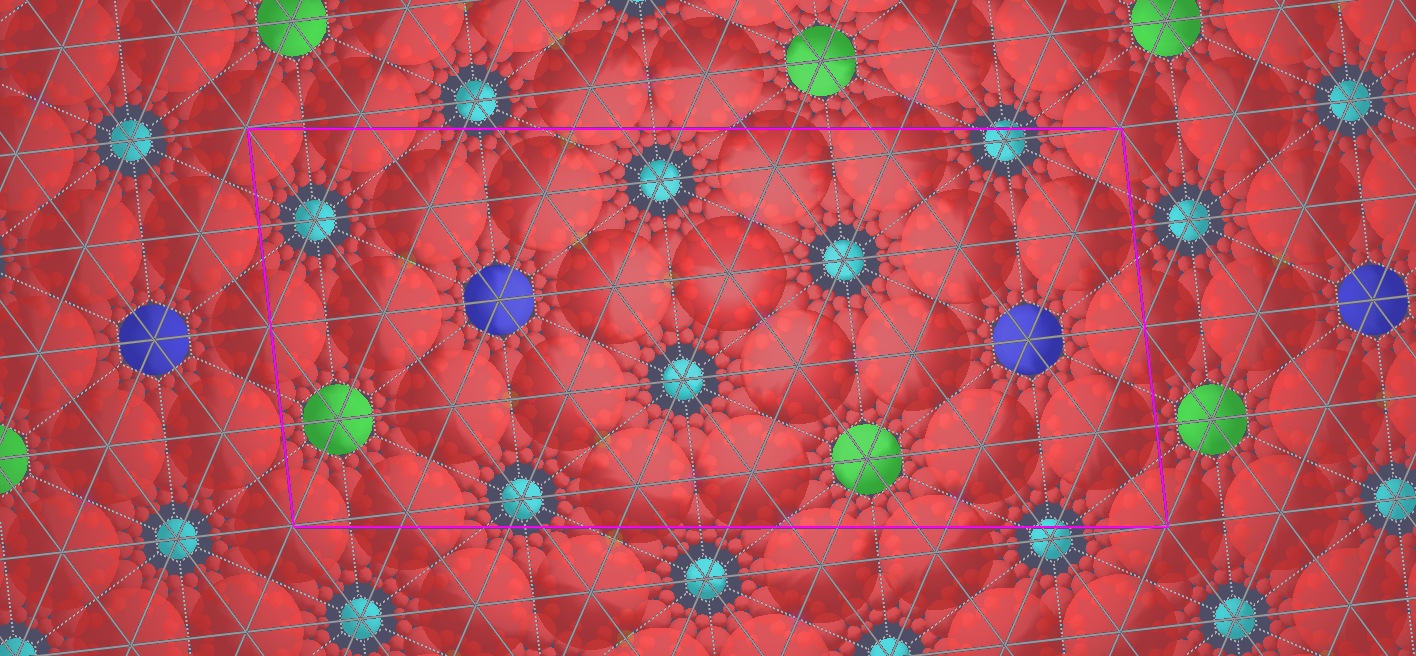}
\caption{ $0$-cusp maximal horoball packing of $\mathbb{H}^3$ for \texttt{otet20\_00543} (\texttt{K\_hlgy\_lk}[245]) (picture obtained from SnapPy \cite{snappy}).}
\label{bad6sym_245}
\end{figure}

\begin{figure}
\centering 
\captionsetup{justification=centering}
\includegraphics[scale=.1]{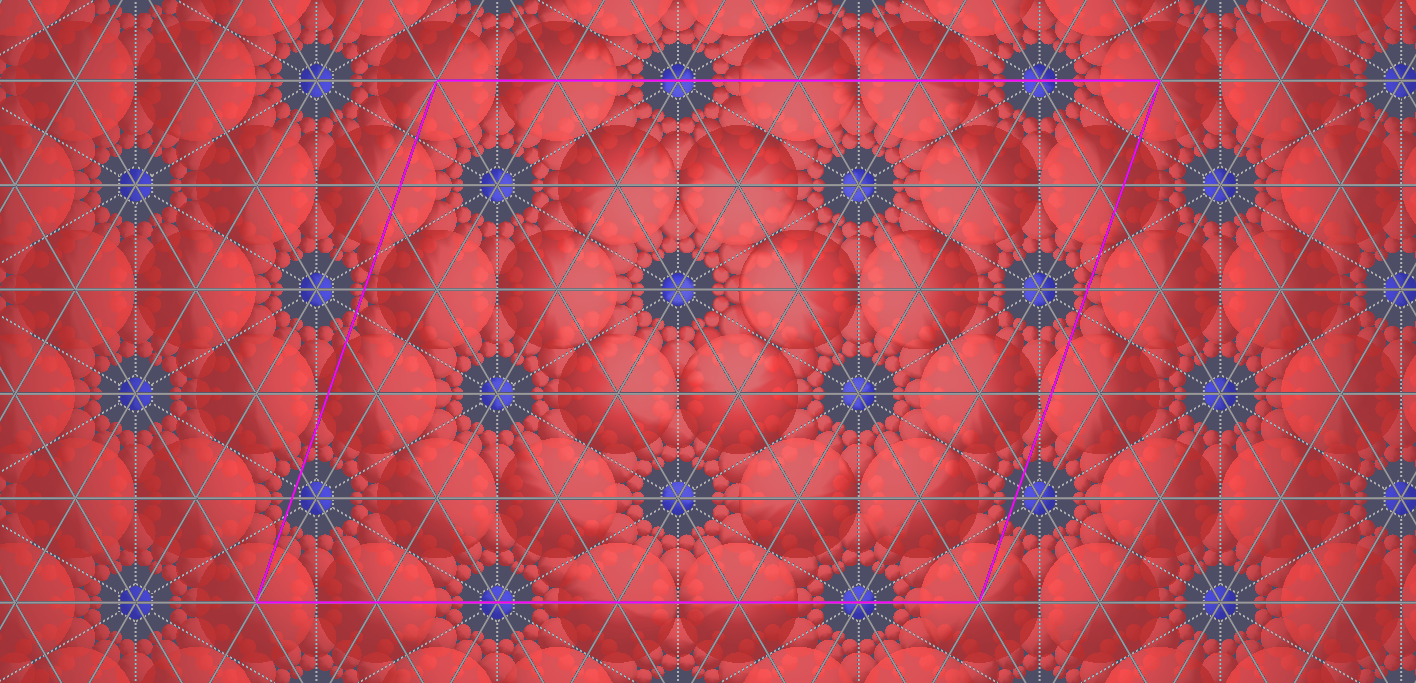}
\caption{ $0$-cusp maximal horoball packing of $\mathbb{H}^3$ for \texttt{otet20\_00762} (\texttt{K\_hlgy\_lk}[299]) (picture obtained from SnapPy \cite{snappy}).}
\label{bad6sym_299}
\end{figure}

\begin{figure}
\centering 
\captionsetup{justification=centering}
\includegraphics[scale=.1]{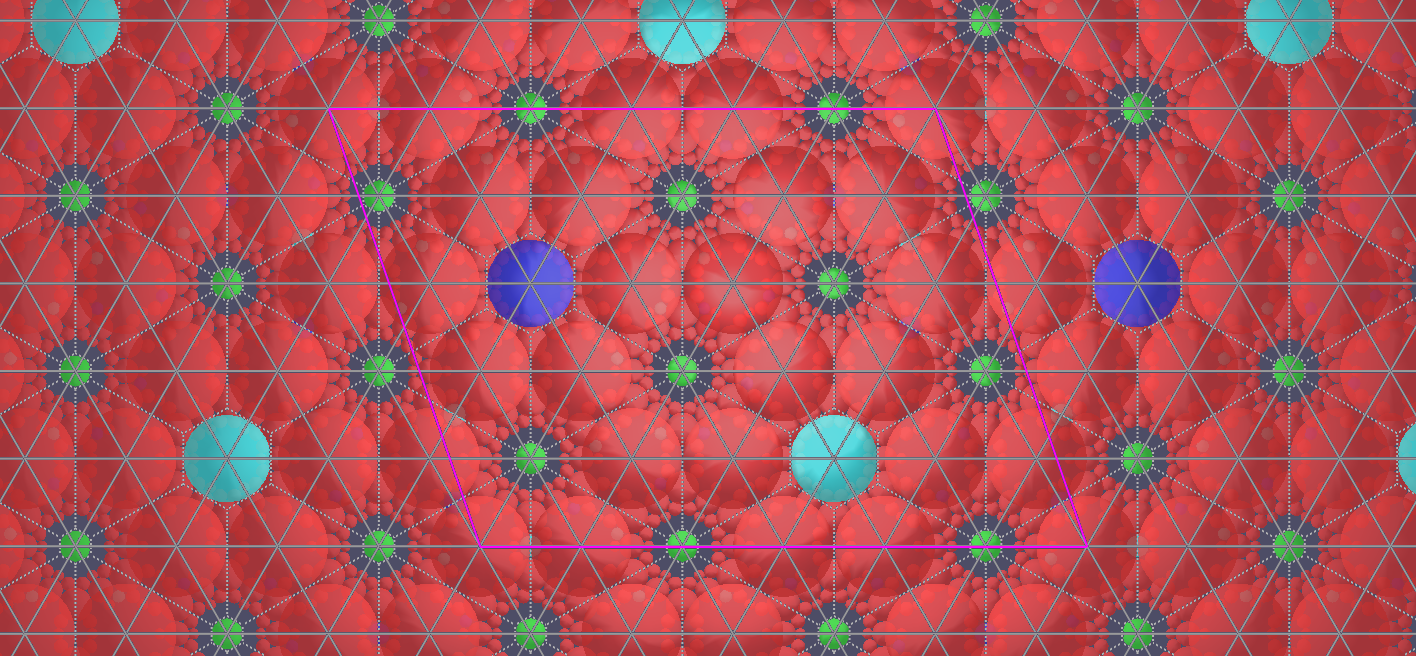}
\caption{ $0$-cusp maximal horoball packing of $\mathbb{H}^3$ for \texttt{otet20\_00793} (\texttt{K\_hlgy\_lk}[307]) (picture obtained from SnapPy \cite{snappy}).}
\label{bad6sym_307}
\end{figure}

\begin{figure}
\centering 
\captionsetup{justification=centering}
\includegraphics[scale=.1]{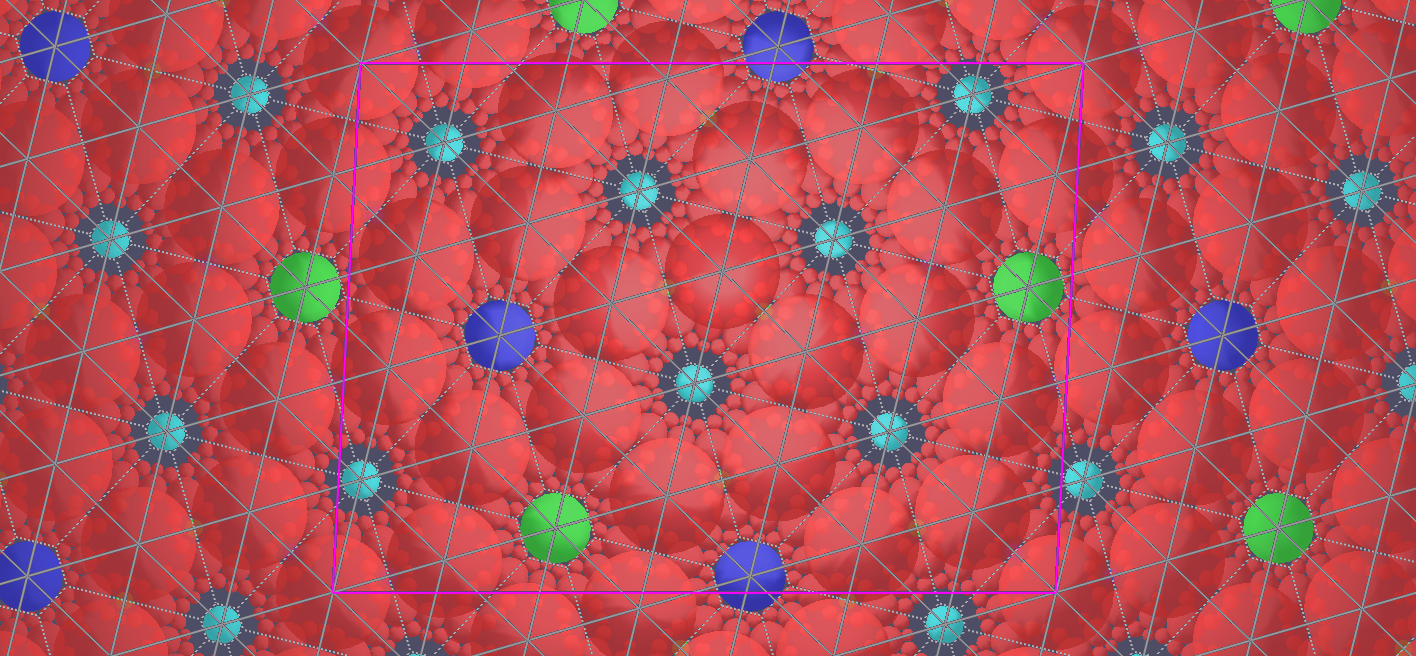}
\caption{ $0$-cusp maximal horoball packing of $\mathbb{H}^3$ for \texttt{otet22\_00130} (\texttt{K\_hlgy\_lk}[425]) (picture obtained from SnapPy \cite{snappy}).}
\label{bad6sym_425}
\end{figure}

\begin{figure}
\centering 
\captionsetup{justification=centering}
\includegraphics[scale=.1]{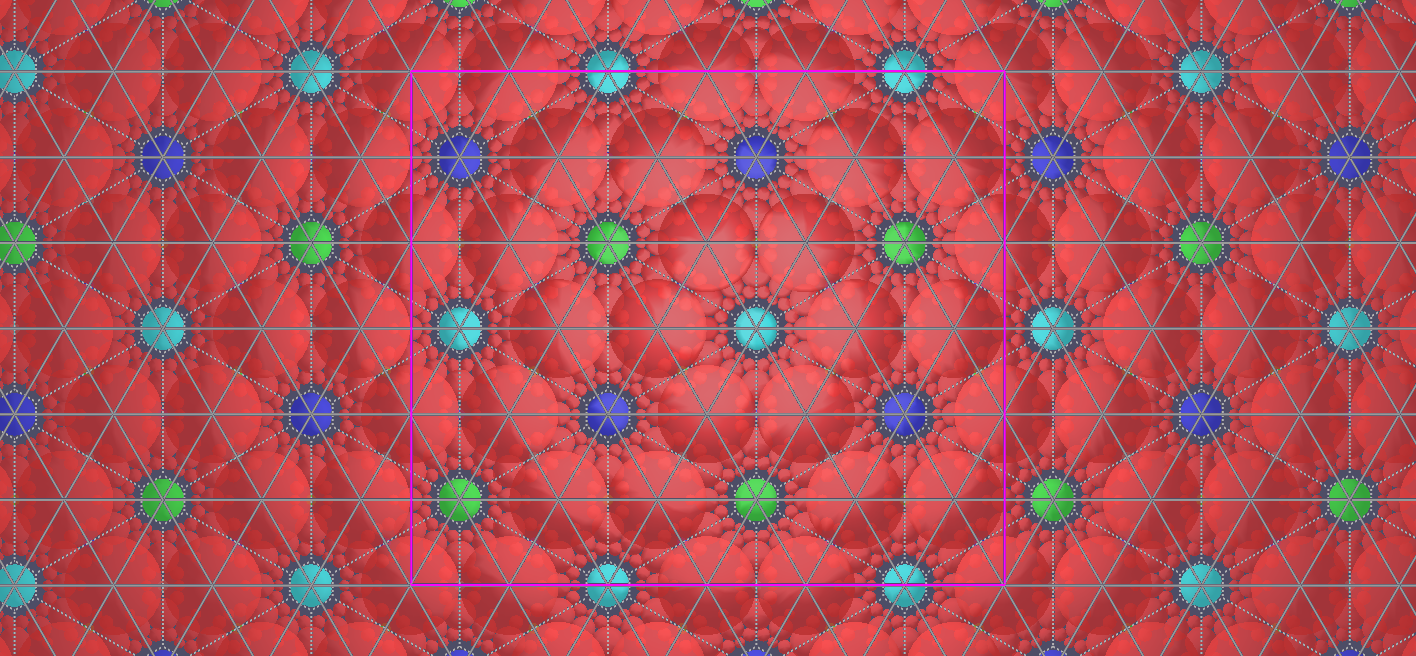} \includegraphics[scale=.1]{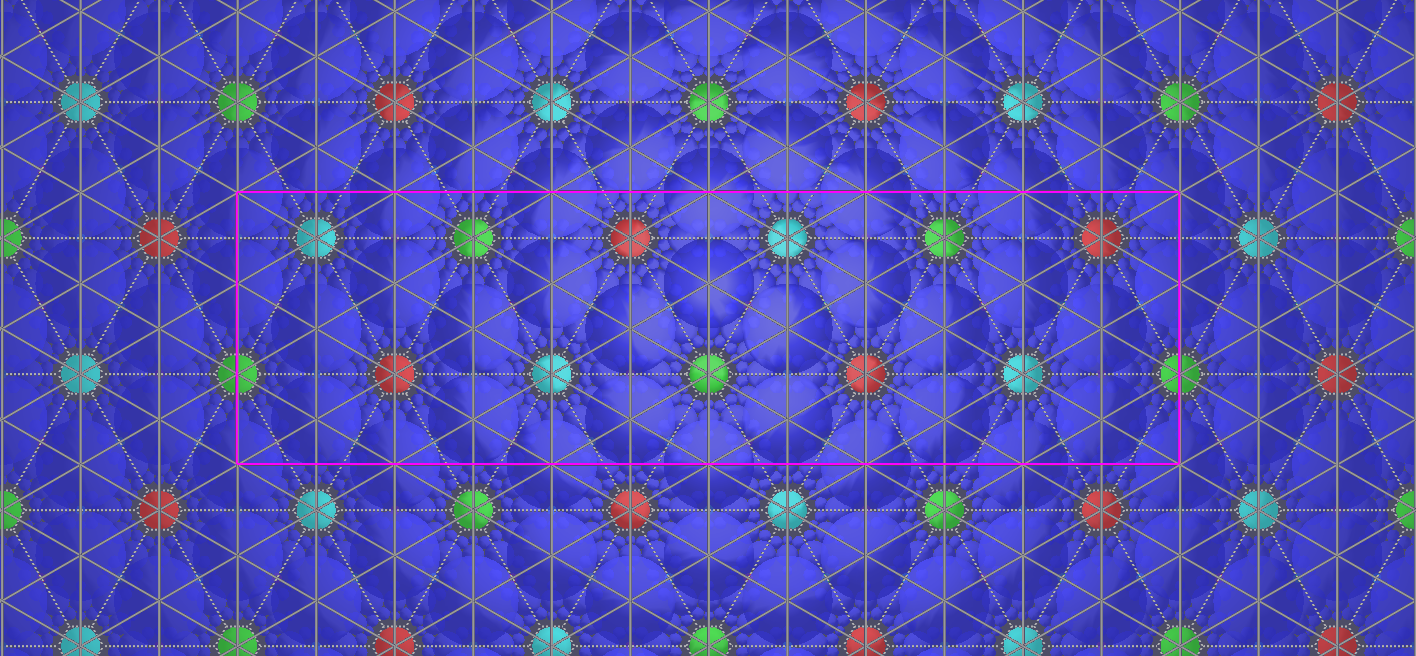} 
\caption{Horoball packings of $\mathbb{H}^3$ for \texttt{otet24\_00259} (\texttt{K\_hlgy\_lk}[633]); \\ Left: $0$-cusp maximal horoball packing, Right: $1$-cusp maximal horoball packing  (pictures obtained from SnapPy \cite{snappy}).}
\label{bad6sym_633}
\end{figure}

\begin{figure}
\centering 
\captionsetup{justification=centering}
\includegraphics[scale=.1]{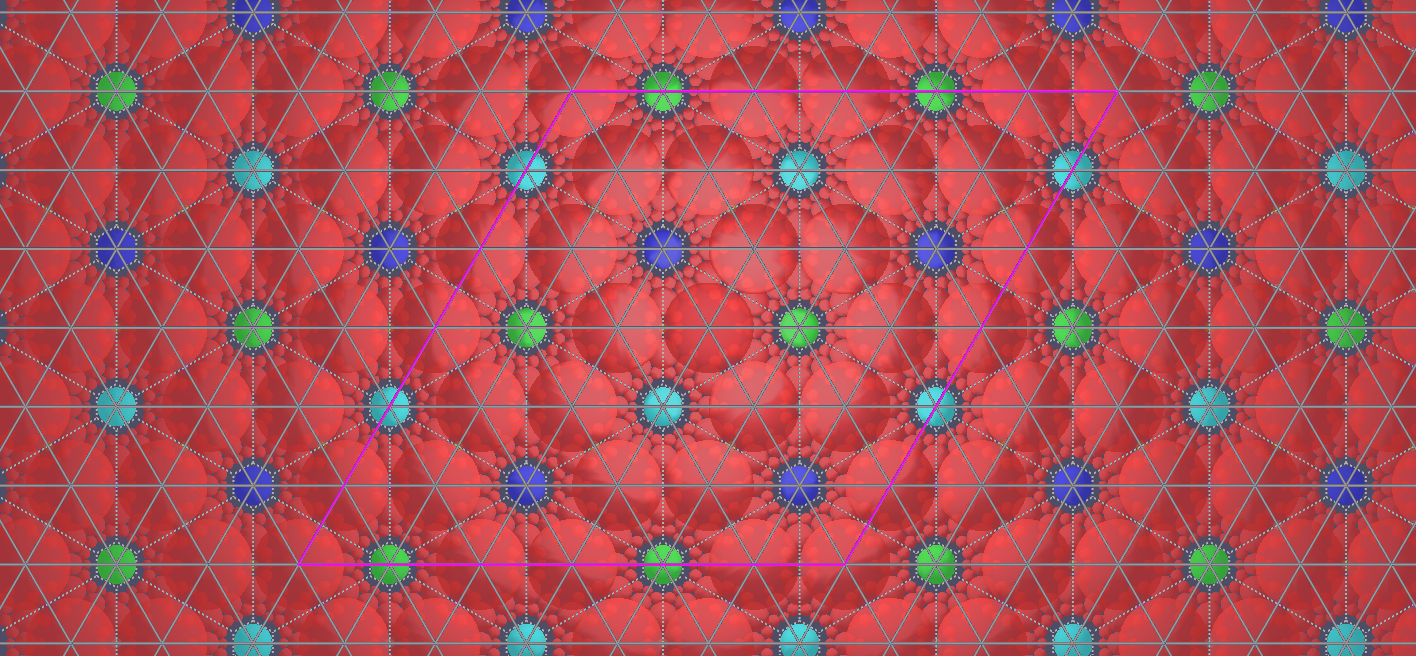} \includegraphics[scale=.1]{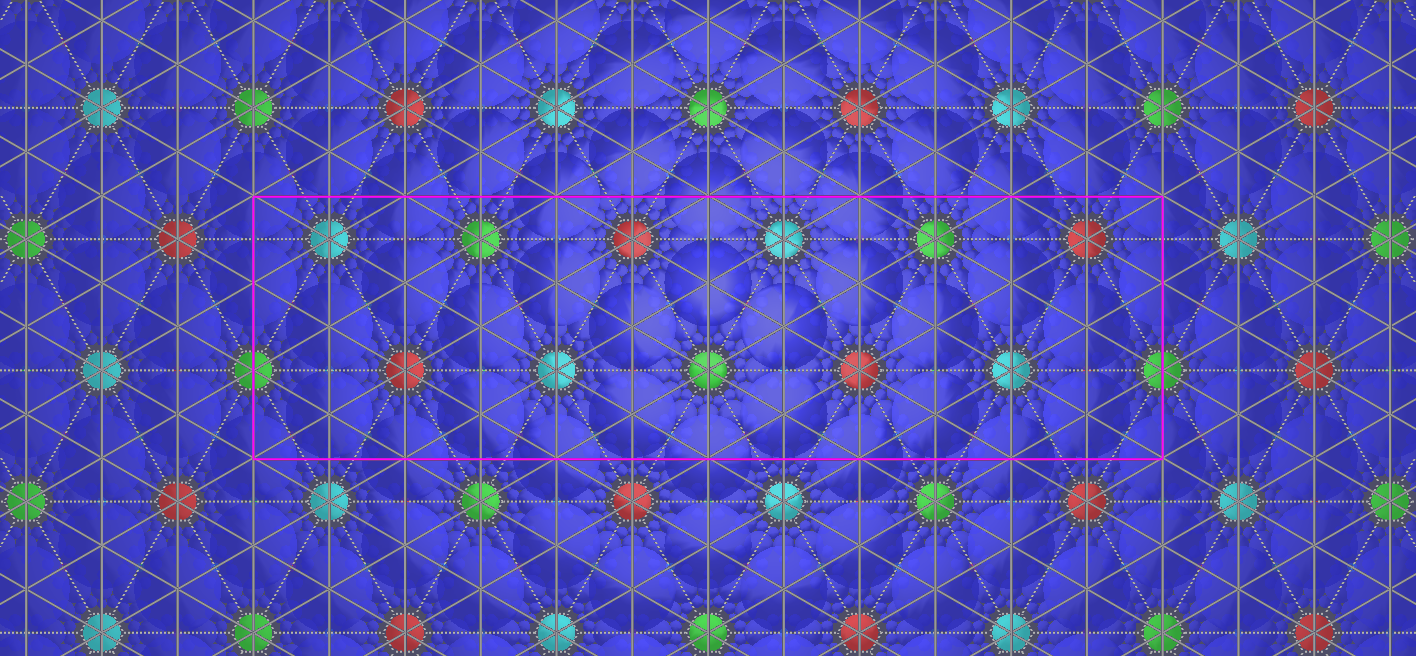} 
\caption{Horoball packings of $\mathbb{H}^3$ for \texttt{otet24\_00260} (\texttt{K\_hlgy\_lk}[634]); \\ Left: $0$-cusp maximal horoball packing, Right: $1$-cusp maximal horoball packing  (pictures obtained from SnapPy \cite{snappy}).}
\label{bad6sym_634}
\end{figure}

\begin{figure}
\centering 
\captionsetup{justification=centering}
\includegraphics[scale=.1]{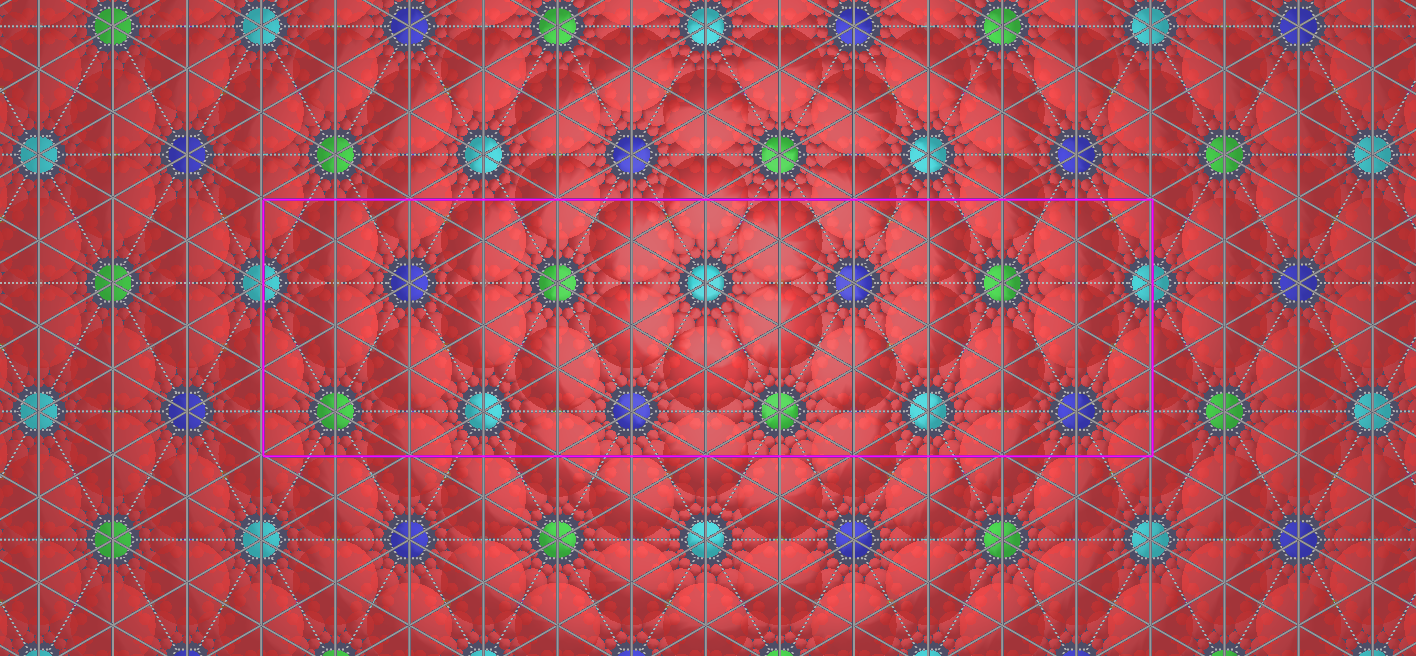}
\caption{ $0$-cusp maximal horoball packing of $\mathbb{H}^3$ for \texttt{otet24\_00263} (\texttt{K\_hlgy\_lk}[636]) (picture obtained from SnapPy \cite{snappy}).}
\label{bad6sym_636}
\end{figure}

\begin{figure}
\centering 
\captionsetup{justification=centering}
\includegraphics[scale=.1]{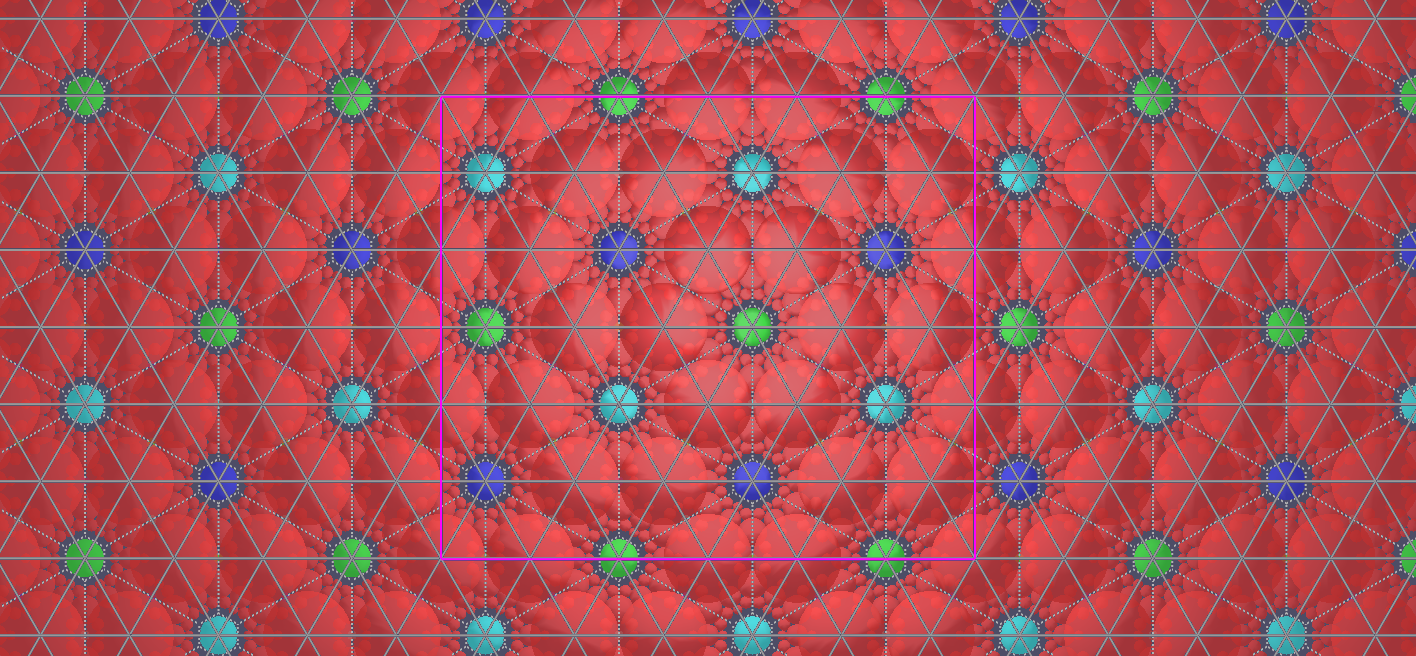} \includegraphics[scale=.1]{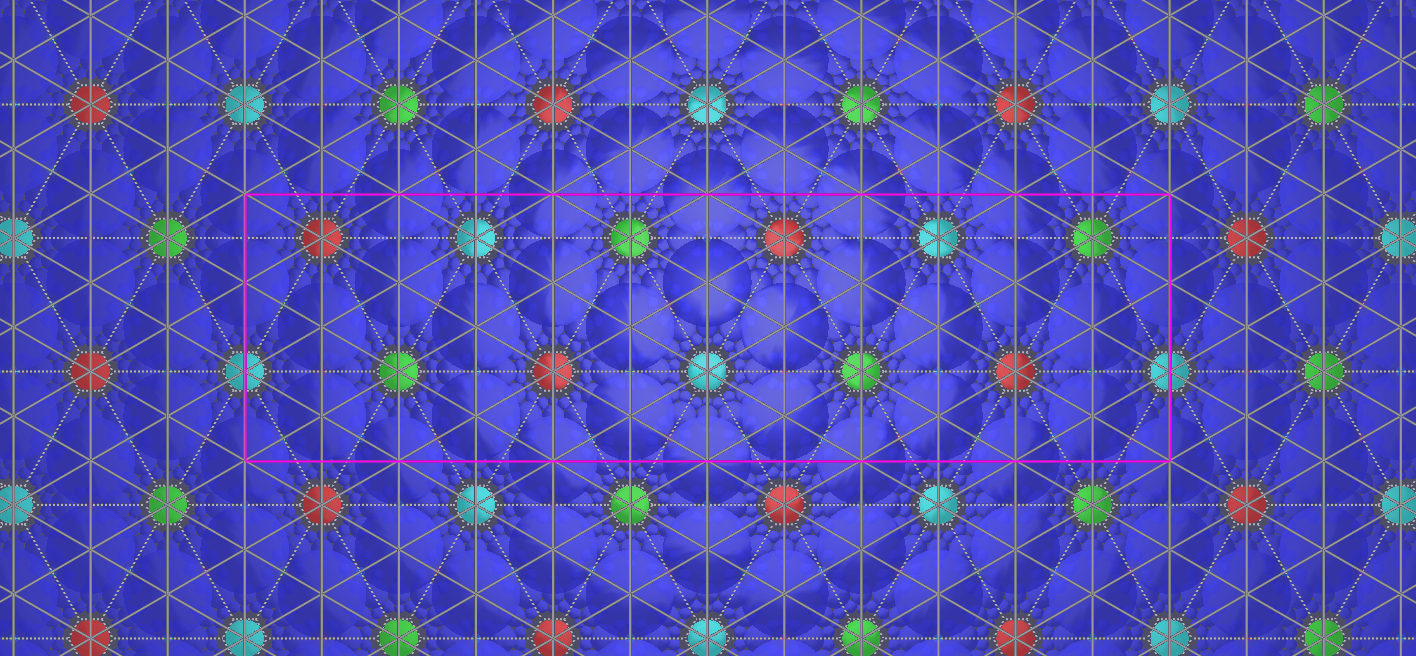} 
\caption{Horoball packings of $\mathbb{H}^3$ for \texttt{otet24\_00290} (\texttt{K\_hlgy\_lk}[645]) ; \\ Left: $0$-cusp maximal horoball packing, Right: $1$-cusp maximal horoball packing  (pictures obtained from SnapPy \cite{snappy}).}
\label{bad6sym_645}
\end{figure}

\begin{figure}
\centering 
\captionsetup{justification=centering}
\includegraphics[scale=.1]{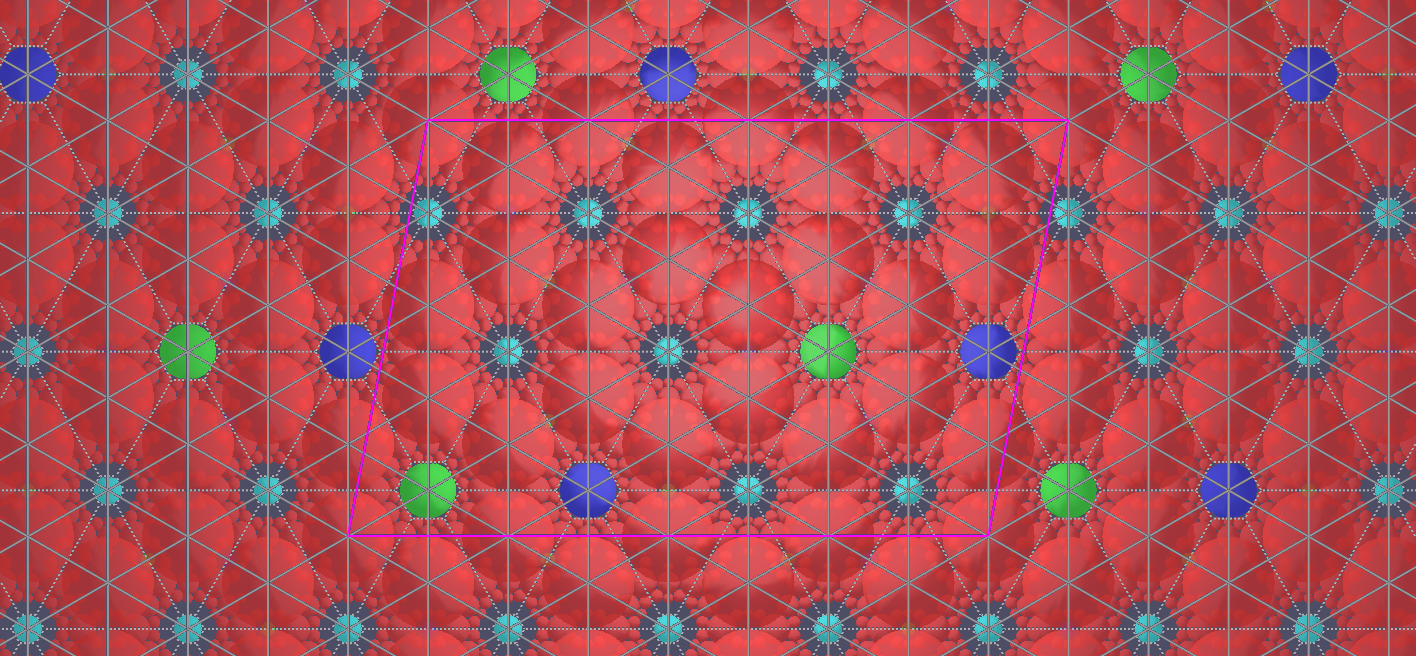}
\caption{ $0$-cusp maximal horoball packing of $\mathbb{H}^3$ for \texttt{otet24\_00396} (\texttt{K\_hlgy\_lk}[681])  (picture obtained from SnapPy \cite{snappy}).}
\label{bad6sym_681}
\end{figure}

\begin{figure}
\centering 
\captionsetup{justification=centering}
\includegraphics[scale=.1]{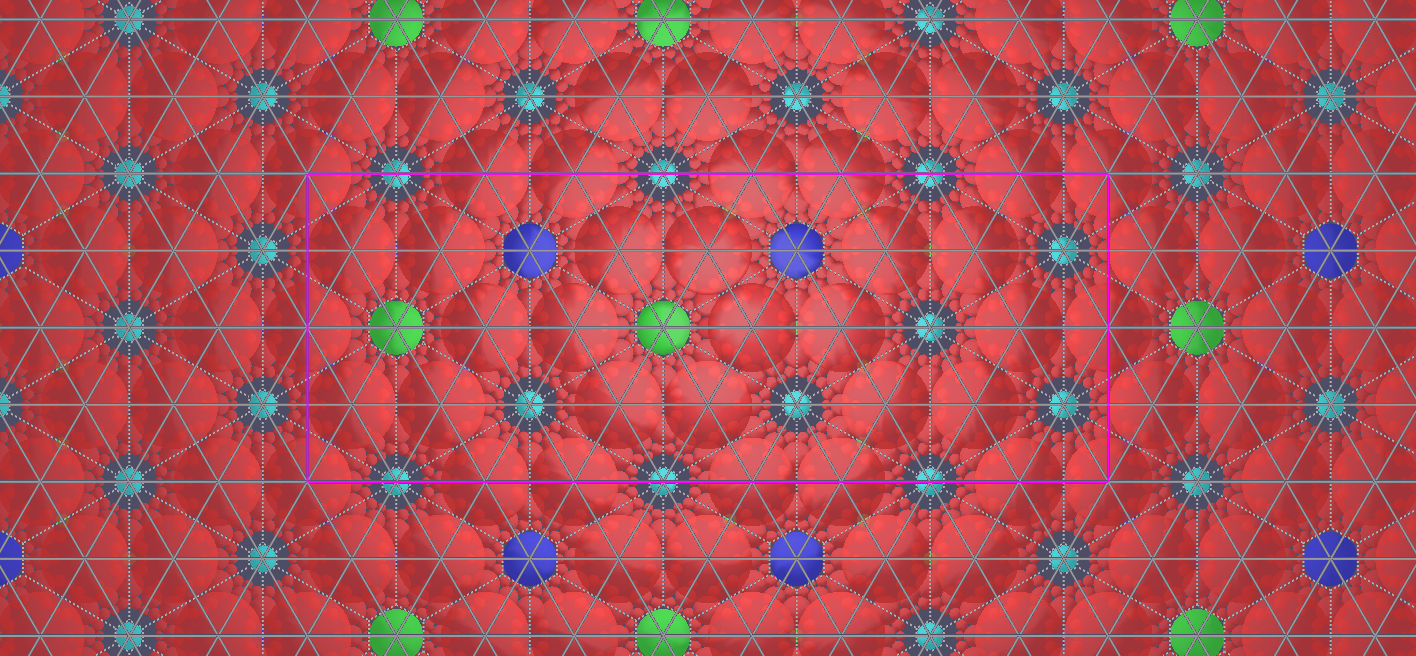}
\caption{ $0$-cusp maximal horoball packing of $\mathbb{H}^3$ for \texttt{otet24\_00398} (\texttt{K\_hlgy\_lk}[682])  (picture obtained from SnapPy \cite{snappy}).}
\label{bad6sym_682}
\end{figure}

\begin{figure}
\centering 
\captionsetup{justification=centering}
\includegraphics[scale=.1]{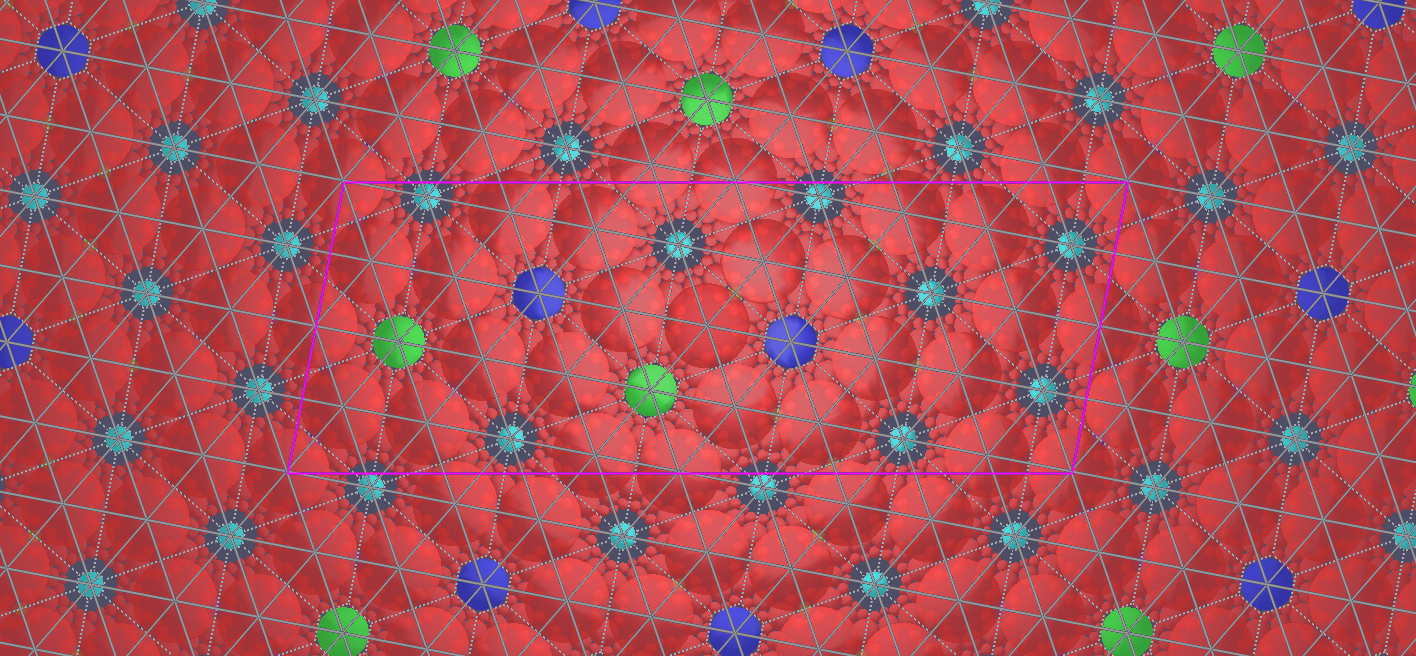}
\caption{ $0$-cusp maximal horoball packing of $\mathbb{H}^3$ for \texttt{otet24\_00399} (\texttt{K\_hlgy\_lk}[683])  (picture obtained from SnapPy \cite{snappy}).}
\label{bad6sym_683}
\end{figure}

\begin{figure}
\centering 
\captionsetup{justification=centering}
\includegraphics[scale=.1]{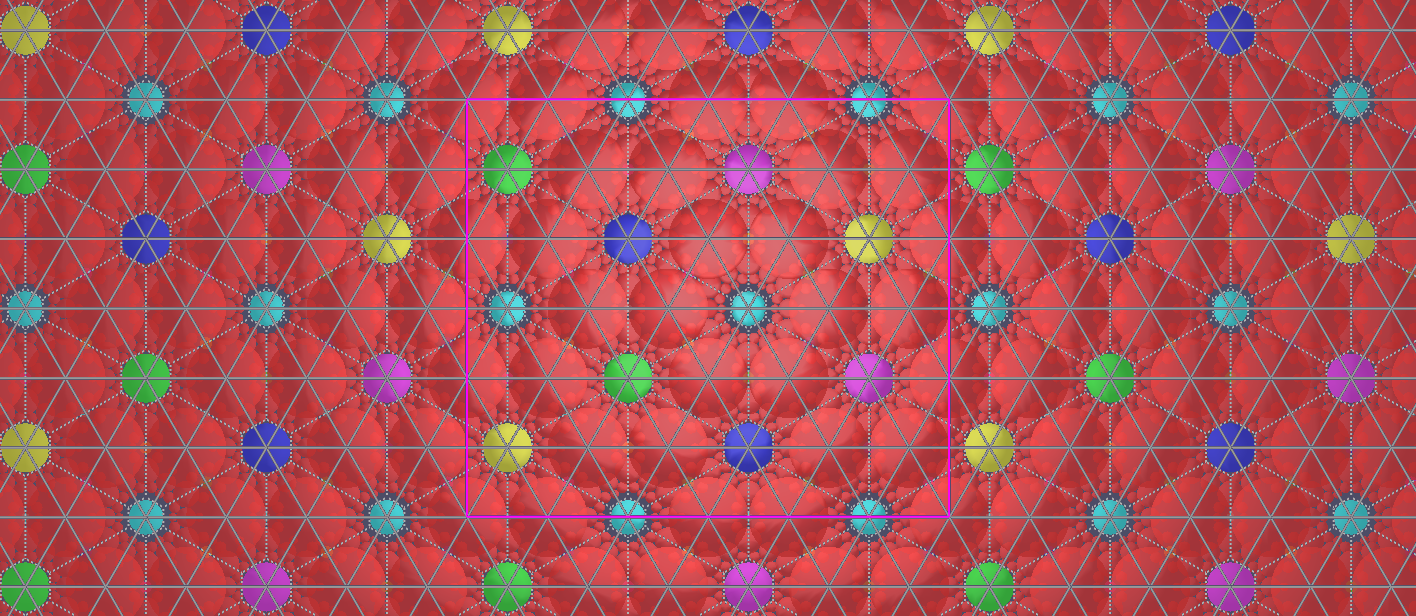}
\caption{ $0$-cusp maximal horoball packing of $\mathbb{H}^3$ for \texttt{otet24\_00401} (\texttt{K\_hlgy\_lk}[684])  (picture obtained from SnapPy \cite{snappy}).}
\label{bad6sym_684}
\end{figure}

\begin{figure}
\centering 
\captionsetup{justification=centering}
\includegraphics[scale=.1]{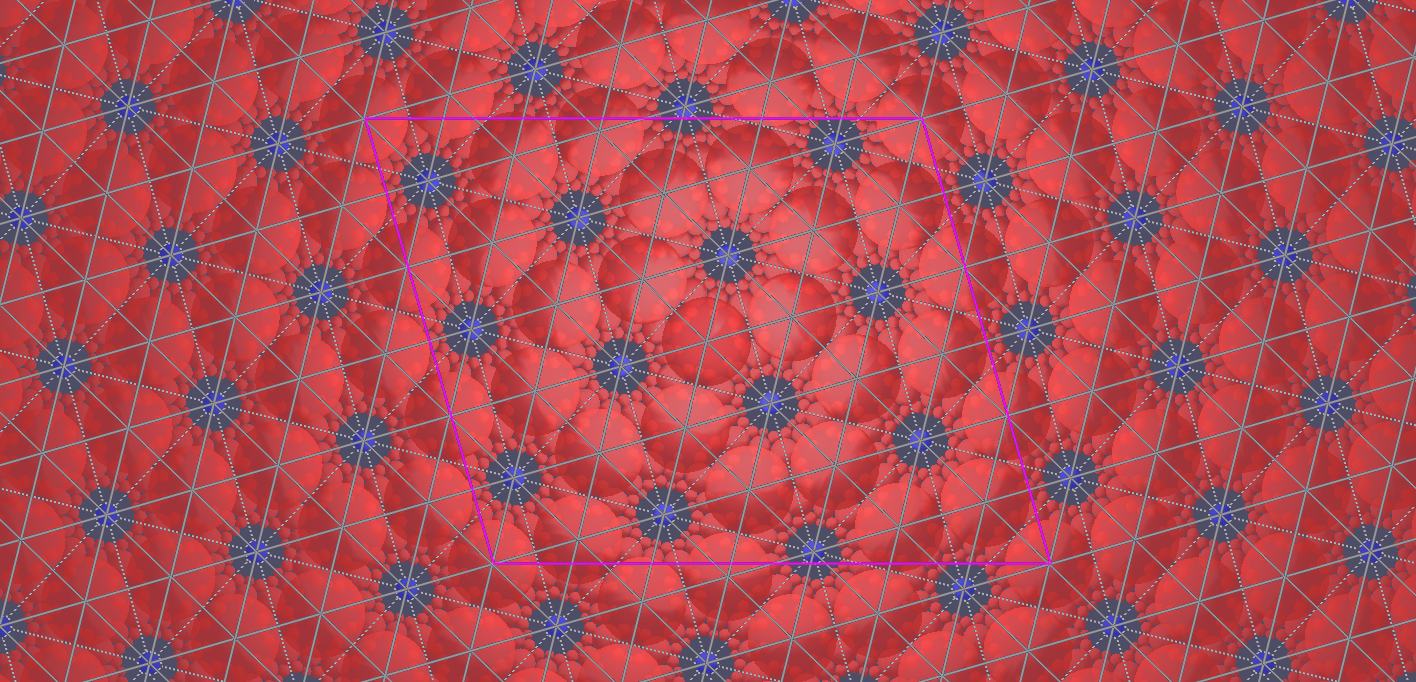}
\caption{ $0$-cusp maximal horoball packing of $\mathbb{H}^3$ for \texttt{otet24\_00979} (\texttt{K\_hlgy\_lk}[821])  (picture obtained from SnapPy \cite{snappy}).}
\label{bad6sym_821}
\end{figure}

\begin{figure}
\centering 
\captionsetup{justification=centering}
\includegraphics[scale=.1]{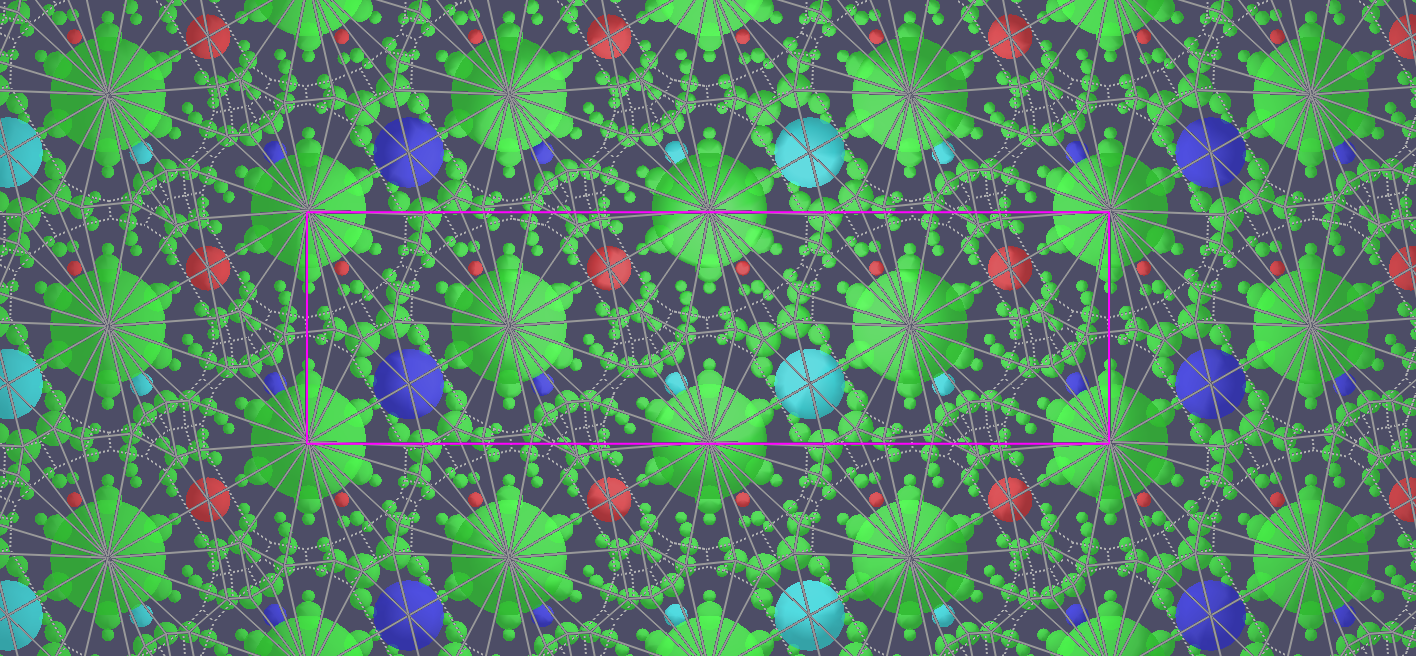}
\caption{ $2$-cusp maximal horoball packing of $\mathbb{H}^3$ for \texttt{otet20\_00062} (\texttt{K\_hlgy\_lk}[160])  (picture obtained from SnapPy \cite{snappy}).}
\label{fullhoronosym_160}
\end{figure}

\begin{figure}
\centering 
\captionsetup{justification=centering}
\includegraphics[scale=.1]{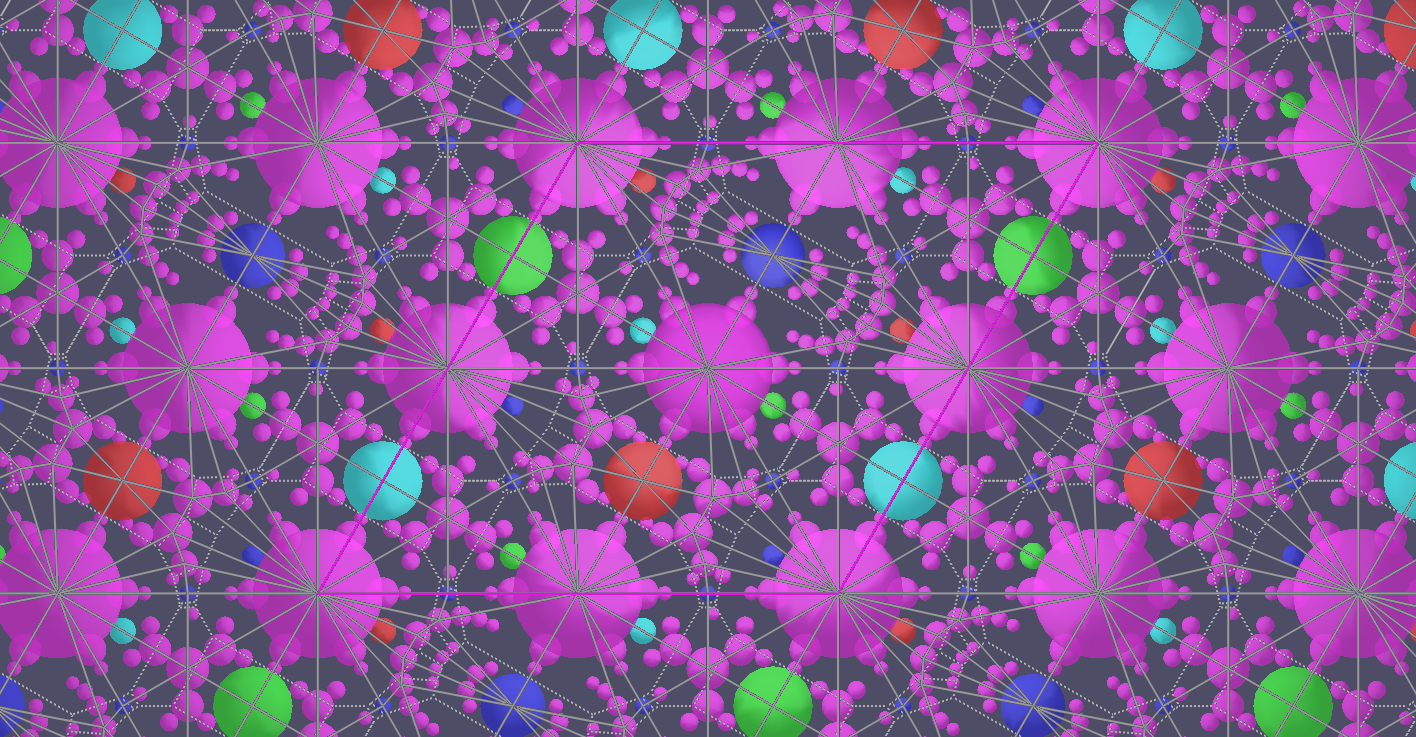}
\caption{ $4$-cusp maximal horoball packing of $\mathbb{H}^3$ for \texttt{otet20\_00097} (\texttt{K\_hlgy\_lk}[167])  (picture obtained from SnapPy \cite{snappy}).}
\label{fullhoronosym_167}
\end{figure}

\begin{figure}
\centering 
\captionsetup{justification=centering}
\includegraphics[scale=.1]{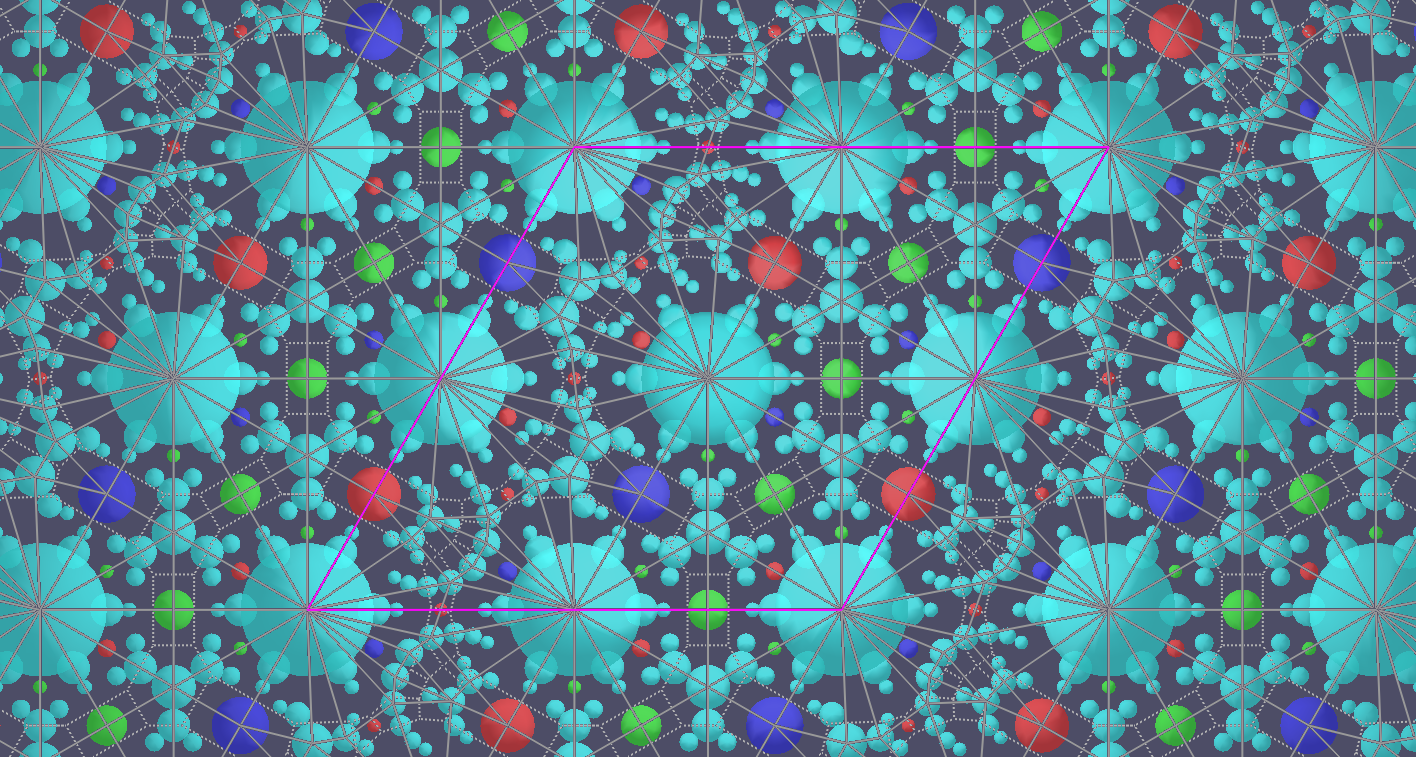}
\caption{ $3$-cusp maximal horoball packing of $\mathbb{H}^3$ for \texttt{otet20\_00102} (\texttt{K\_hlgy\_lk}[172])  (picture obtained from SnapPy \cite{snappy}).}
\label{fullhoronosym_172}
\end{figure}

\begin{figure}
\centering 
\captionsetup{justification=centering}
\includegraphics[scale=.1]{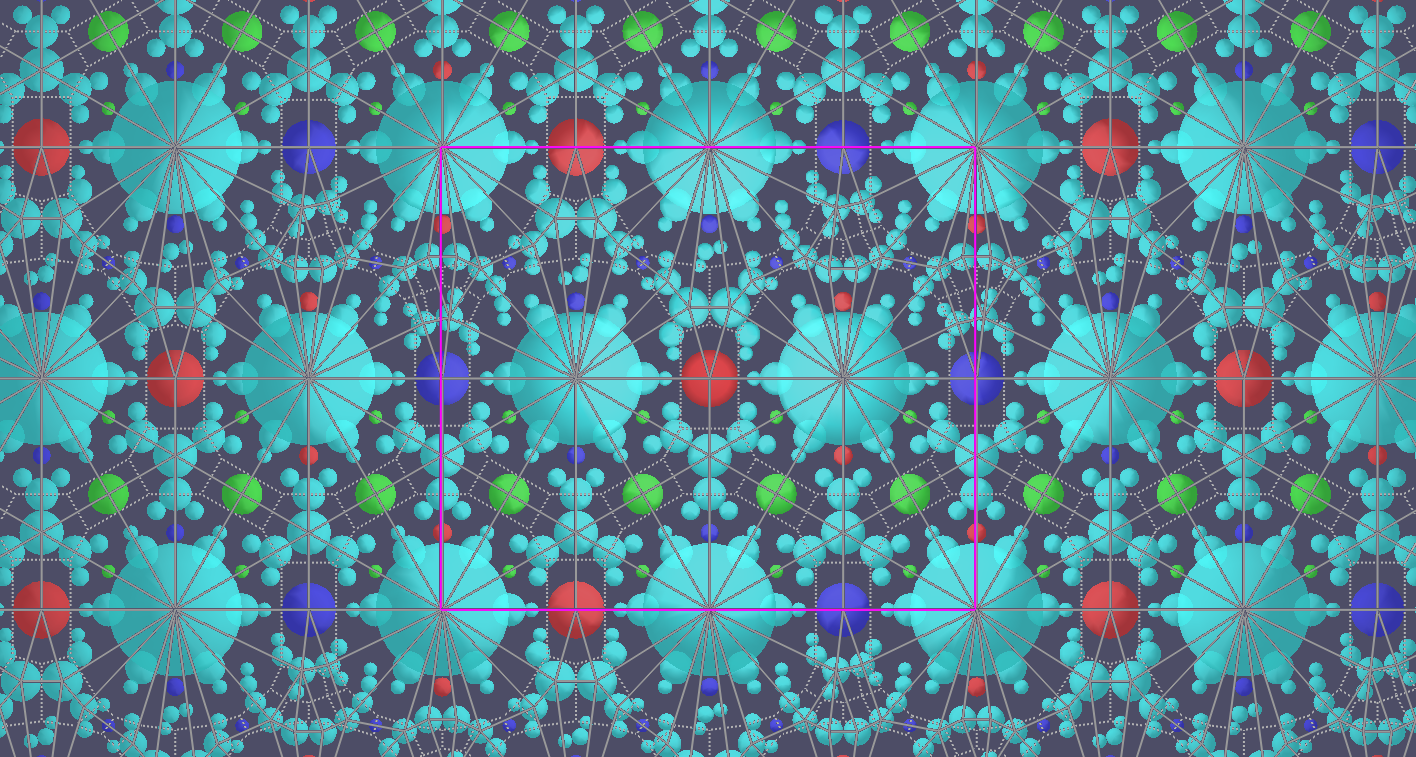}
\caption{ $3$-cusp maximal horoball packing of $\mathbb{H}^3$ for \texttt{otet20\_00131} (\texttt{K\_hlgy\_lk}[175])  (picture obtained from SnapPy \cite{snappy}).}
\label{fullhoronosym_175}
\end{figure}

\begin{figure}
\centering 
\captionsetup{justification=centering}
\includegraphics[scale=.1]{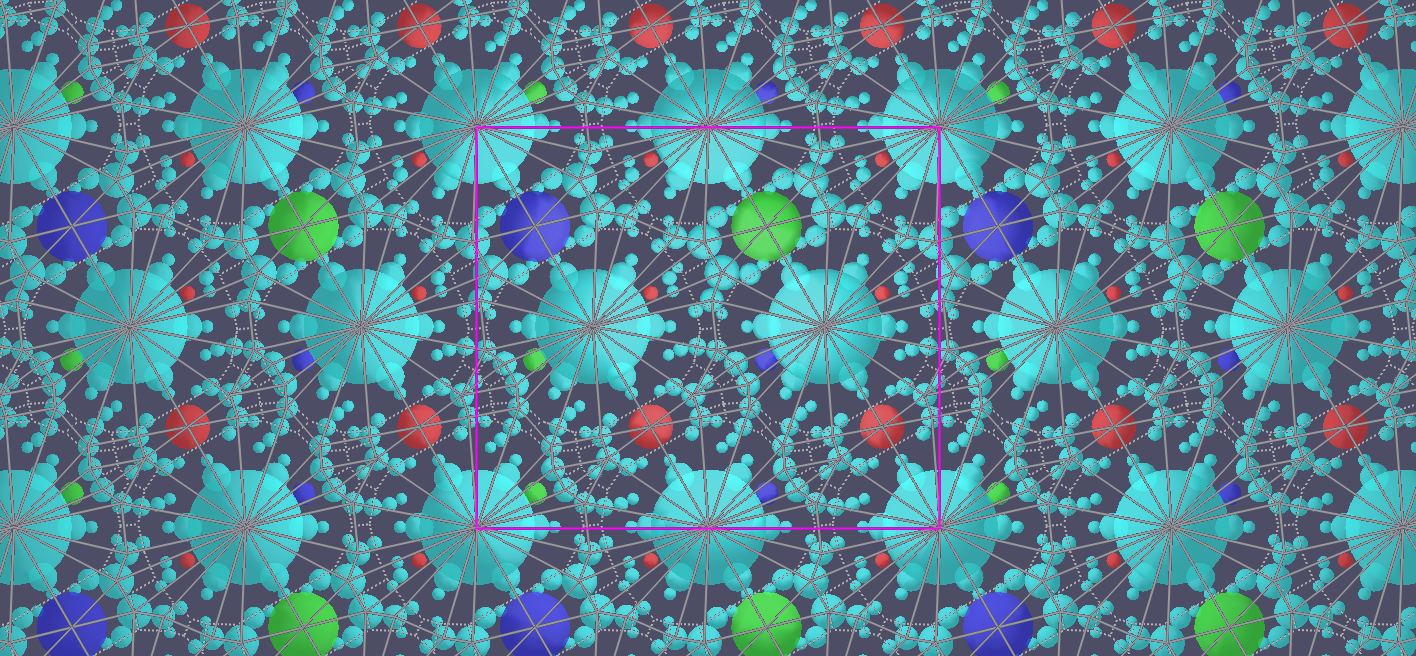}
\caption{ $3$-cusp maximal horoball packing of $\mathbb{H}^3$ for \texttt{otet20\_00770} (\texttt{K\_hlgy\_lk}[301])  (picture obtained from SnapPy \cite{snappy}).}
\label{fullhoronosym_301}
\end{figure}

\begin{figure}
\centering 
\captionsetup{justification=centering}
\includegraphics[scale=.1]{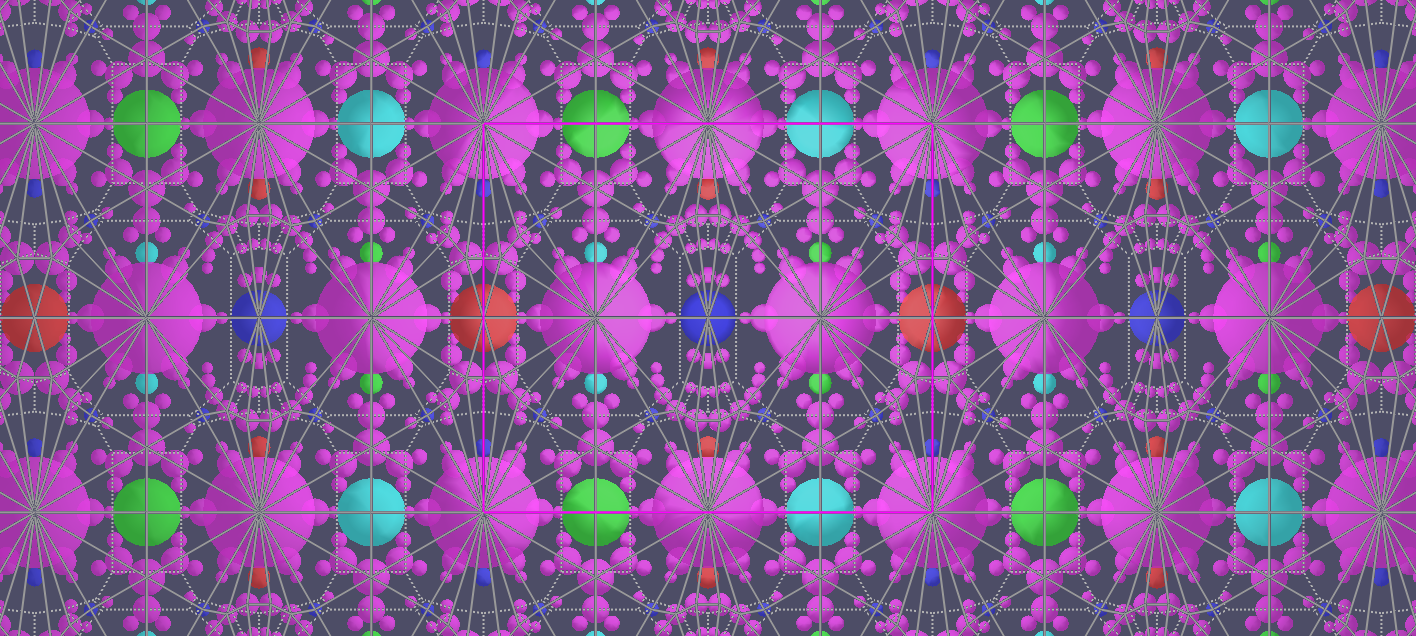}
\caption{ $4$-cusp maximal horoball packing of $\mathbb{H}^3$ for \texttt{otet20\_01405} (\texttt{K\_hlgy\_lk}[331])  (picture obtained from SnapPy \cite{snappy}).}
\label{fullhoronosym_331}
\end{figure}

\begin{figure}
\centering 
\captionsetup{justification=centering}
\includegraphics[scale=.1]{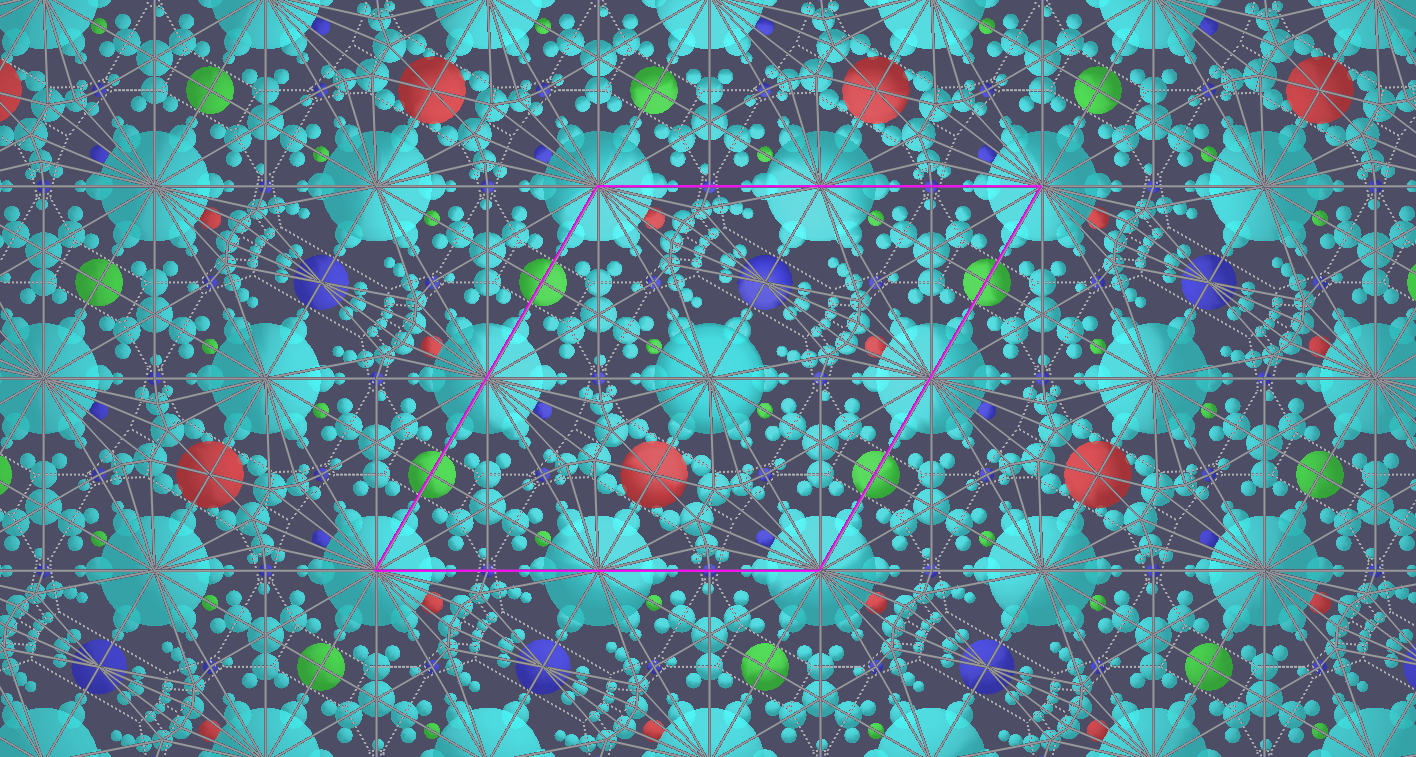}
\caption{ $3$-cusp maximal horoball packing of $\mathbb{H}^3$ for \texttt{otet20\_01414} (\texttt{K\_hlgy\_lk}[333])  (picture obtained from SnapPy \cite{snappy}).}
\label{fullhoronosym_333}
\end{figure}

\begin{figure}
\centering 
\captionsetup{justification=centering}
\includegraphics[scale=.1]{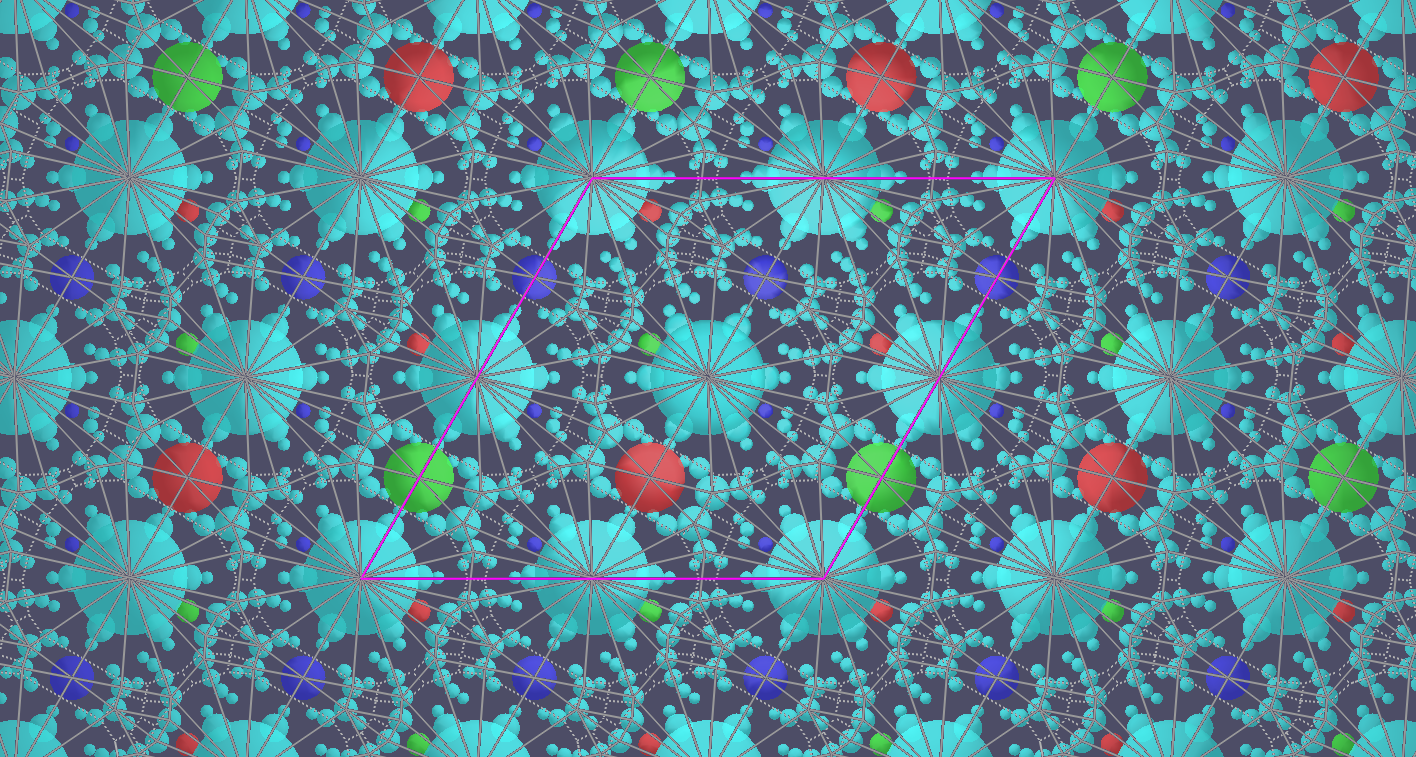}
\caption{ $3$-cusp maximal horoball packing of $\mathbb{H}^3$ for \texttt{otet20\_01431} (\texttt{K\_hlgy\_lk}[334])  (picture obtained from SnapPy \cite{snappy}).}
\label{fullhoronosym_334}
\end{figure}

\begin{figure}
\centering 
\captionsetup{justification=centering}
\includegraphics[scale=.1]{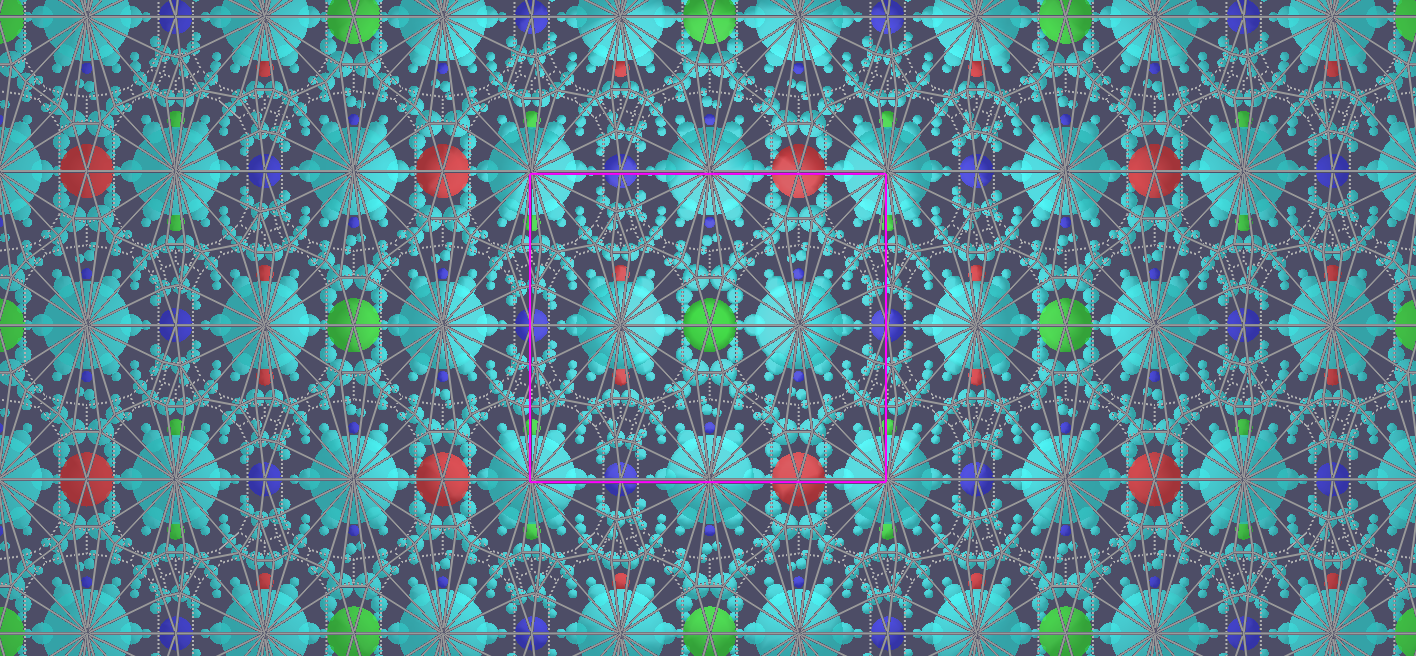}
\caption{ $3$-cusp maximal horoball packing of $\mathbb{H}^3$ for \texttt{otet20\_01433} (\texttt{K\_hlgy\_lk}[335])  (picture obtained from SnapPy \cite{snappy}).}
\label{fullhoronosym_335}
\end{figure}

\begin{figure}
\centering 
\captionsetup{justification=centering}
\includegraphics[scale=.1]{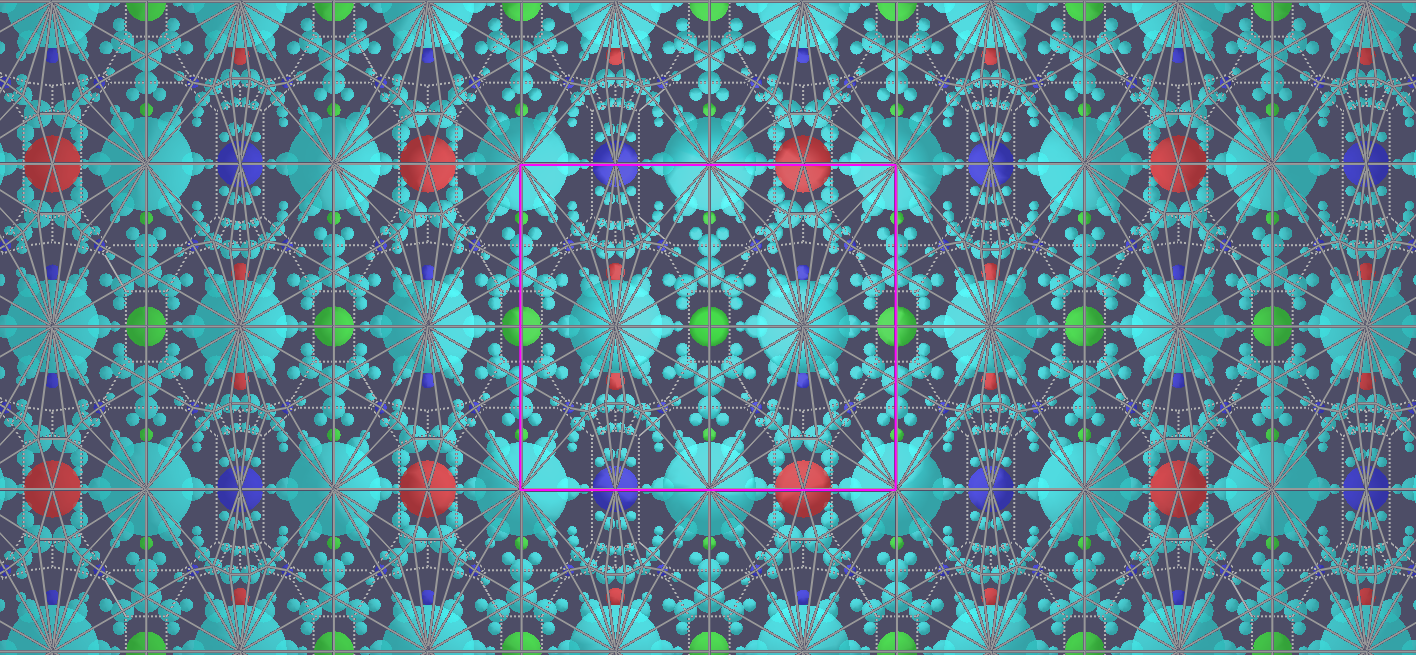}
\caption{ $3$-cusp maximal horoball packing of $\mathbb{H}^3$ for \texttt{otet20\_01438} (\texttt{K\_hlgy\_lk}[336])  (picture obtained from SnapPy \cite{snappy}).}
\label{fullhoronosym_336}
\end{figure}

\subsection{Utilities \texttt{SigToSeq.py} and \texttt{TestForCovers.py} from \cite{orbcenpract}}\label{censusutilities}

In this subsection, we will review the work in \cite{orbcentheory} and \cite{orbcenpract} and describe how the utilities \texttt{SigToSeq.py} and \texttt{TestForCovers.py} from \cite{orbcenpract} work. These utilities are available at \cite{orbtricode}. \cite{orbcentheory} and \cite{orbcenpract} investigate the commensurability class of the figure eight knot complement and this class is denoted there as $\mathcal{C}_3$. Since the figure-eight knot complement covers $\mathbb{H}^3/\operatorname{PGL_2}(O_3)$, $\mathcal{C}_3$ contains $\mathbb{H}^3/\operatorname{PGL_2}(O_3)$ as well. To understand  $\mathcal{C}_3$, \cite{orbcentheory} and \cite{orbcenpract}  look into a sub-collection $\mathcal{C}_{main}$ of $\mathcal{C}_3$ which is defined to be the collection of the finite covers of $\mathbb{H}^3/\operatorname{PGL_2}(O_3)$. Left of Figure \ref{orbtet} shows the fundamental domain $\Delta$ of the action of $\operatorname{PGL_2}(O_3)$ on $\mathbb{H}^3$ which is a tetrahedron with one ideal vertex $v$ at $\infty$ and three finite vertices $f_0$, $f_1$ and $e$ lying above respectively $\frac{1}{2}-\frac{i}{2\sqrt{3}}$, $\frac{1}{2}+\frac{i}{2\sqrt{3}}$ and $0$ on the unit sphere centered at the origin. $\Delta$ can be obtained by gluing two simplices of the barycentric subdivision of the regular ideal tetrahedron (see the picture in the right of Figure \ref{orbtet}) along their faces which are contained in the faces of the regular ideal tetrahedra. Note that the volume of $\Delta$ is $\frac{v_0}{12}$ where $v_0$ is the volume of the regular ideal tetrahedron. Denote the faces of $\Delta$ by their opposite vertices. Let $\rho_v$ (respectively, $\rho_e$ and $\rho_f$) denote the rotation of angle $\pi$ (respectively, $\pi$ and $\frac{\pi}{3}$) around the axis which joins $e$ to $m$ (respectively, $v$ to $m$, and $v$ to $e$). Now, gluing face $v$ to itself via $\rho_v$, face $f_0$ and $f_1$ to each other via $\rho_f$ and face $e$ to itself via $\rho_e$ gives $\mathbb{H}^3/\operatorname{PGL_2}(O_3)$ an \textit{orbifold triangulation}. Since members of  $\mathcal{C}_{main}$ are finite covers of $\mathbb{H}^3/\operatorname{PGL_2}(O_3)$, one can get  an orbifold triangulation of any member of $\mathcal{C}_{main}$ into tetrahedra each isometric to $\Delta$ by lifting the (orbifold) triangulation of $\mathbb{H}^3/\operatorname{PGL_2}(O_3)$. Suppose $O \in \mathcal{C}_{main}$ and there are $n$ tetrahedra $\Delta^O_0, \ldots \Delta^O_{n-1}$ in such triangulation of $O$. Then for each $i \in \{0,1, \ldots, n-1\}$, the $v$ face of $\Delta^O_i$ is glued to the $v$ face of $\Delta^O_{v(i)}$ for some $v(i) \in  \{0,1, \ldots, n-1\}$, the $f_0$ (respectively, $f_1$) face of $\Delta^O_i$ is glued to the $f_1$ (respectively $f_0$) face of $\Delta^O_{f_0(i)}$ (respectively $\Delta^O_{f_1(i)}$) for some $f_0(i)$ (respectively, $f_1(i)$) $\in \{0,1, \ldots, n-1\}$ and the $e$ face of $\Delta^O_i$ is glued to the $e$ face of $\Delta^O_{e(i)}$ for some $e(i) \in  \{0,1, \ldots, n-1\}$. The \textit{destination sequence} $\mathcal{S}^O$ of this triangulation is the finite sequence of $4n$ numbers from $\{0,1, \ldots, n-1\}$ defined as 
\begin{align*}
\mathcal{S}^O(4i)=v(i)\\
 \mathcal{S}^O(4i+1)=f_0(i)\\
 \mathcal{S}^O(4i+2)=f_1(i)\\
 \mathcal{S}^O(4i+3)=e(i)
\end{align*}
 for each $i \in \{0,1, \ldots, n-1\}$. 
The gluing maps of this triangulation induce a partition on the collection of all ideal vertices in this triangulation and the equivalence classes in this partition have a one-to-one correspondence to the cusps of $O$. Since each $\Delta^O_i$ has only one ideal vertex, each cusp of $O$ can be written uniquely as the collection of the indices of the tetrahedra $\Delta^O_i$'s containing the ideal vertices in that class.

\begin{figure}
\centering 
\captionsetup{justification=centering}
\begin{tikzpicture}[scale=2.5]
\draw [-](.5,0)--(.3,.25)--(.5, .803);
\draw [dotted](.5, .803)--(.73, .3)--(.5,0);

\draw [gray=!60] (-.3,0)-- (1.3,0);
\draw [] (.5, .792) -- (.5,-.792);
\draw [gray=!60] (.5,.792)  -- (1.3,0);
\draw [gray!60] (-.3,0) -- (.5, -.792);
\draw [gray!60] (1.3,0)  -- (.5,-.792);
\draw [gray!60] (.5,.792)  -- (.5,-.792);
\draw [gray!60] (.5,.792)  -- (-.3,0);

\draw [-](.5, .792)--(.5,0);
\draw [dashed](.3, .25)--(.73,.3);
\draw[fill] (.3, .25) circle (1.7pt);
\draw[fill] (.5, 0) circle (1.7pt);
\draw[fill=black!50] (.73, .3) circle (1.7pt);

\draw[fill=white] (-.3,0) circle (1.7pt);
\draw[fill=white] (1.3,0) circle (1.7pt);
\draw[fill=white] (.5,.792) circle (1.7pt);
\draw[fill=white] (.5,-.792) circle (1.7pt);

\begin{scope}[xshift=-2.5cm, yshift=-.6cm, scale=1.25]

\draw[fill] (.3, .25) circle (1.5pt);
\draw[fill] (.5, 0) circle (1.5pt);
\draw[fill] (1.24, .51) circle (1.5pt);

\draw [-](.5, .803)--(.5,0);
\draw [dashed](.3, .25)--(.77,.38)--(1.24, .51);
\draw [dotted] (.5, .803)--(.77,.38);
\draw [dotted] (.5,0)--(.77, .38);
\draw (.5,.803)--(.3,.25);
\draw (.3, .25)--(.5,0);
\draw (.5, .803)--(1.24, .51);
\draw (.5,0)--(1.24, .51);
\draw[fill=black!50] (.77, .38) circle (1.5pt);
\draw [fill=white](.5, .803) circle (1.5pt);

\node at (.5, .92){$v$};
\node at (.5, .92){$v$};
\node at (.5, -.13){$e$};
\node at (.79, .48){$m$};

\node at (.17, .25){$f_1$};
\node at (1.37, .51){$f_0$};

\end{scope}
\end{tikzpicture}
\caption{Left: Fundamental domain $\Delta$ of $\operatorname{PGL}_2(O_3)$: Tetrahedron with ideal vertex $v$ and finite vertices $f_0$, $f_1$ and $e$; Right: $\Delta$ is the double of a simplex of the barycentric subdivision of a regular ideal tetrahedron.}
\label{orbtet}
\end{figure}
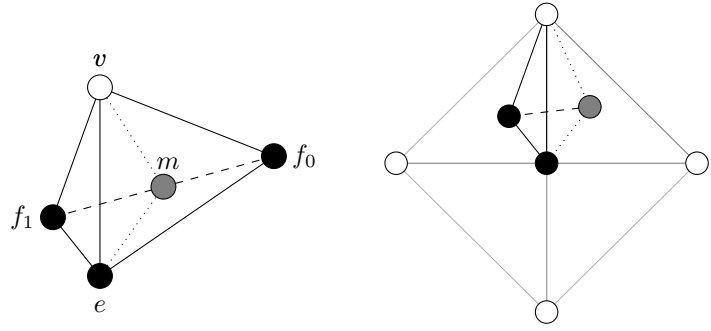

One can use the utility \texttt{TestForCovers.py} of \cite{orbcenpract} (code available at \cite{orbtricode}) to determine the cusps of an orbifold triangulation of an element of $C_{main}$ and how a triangulation and orientation preserving cover between two such orbifold triangulations behave. The \texttt{CuspSeqs} function of \texttt{TestForCovers.py} represents the cusps of an orbifold triangulation of an element of $\mathcal{C}_{main}$ in adherence to the observation above. If we turn the destination sequence $\mathcal{S}^O$ into a Python list \texttt{dseq}, then  \texttt{CuspSeqs(dseq)} will return a (Python) list each of whose members is itself a (Python) list of tetrahedra indices representing a cusp of $O$.

The \texttt{Covers} function of the utility \texttt{TestForCovers.py} works in the following principle. Let $\tilde{O}, O \in \mathcal{C}_{main}$. 
Let $\Delta_0^{\tilde{O}}, \ldots, \Delta_{n-1}^{\tilde{O}}$ be the tetrahedra in the orbifold triangulation of $\tilde{O}$ and $\Delta^O_0, \ldots \Delta^O_{m-1}$ the tetrahedra in the orbifold triangulation of $O$ where $n$ and $m$ are non-negative integers. If $\Phi: \tilde{O}\to O$ is an orientation covering map of degree $d$ which preserves these triangulations then each $\Delta^O_j$ has $d$ many $\Delta_i^{\tilde{O}}$'s in its pre-image. Given any $i \in \{0, \ldots, n-1\}$, we will use $\Phi(i)$ to denote the index of the tetrahedra $\Phi(\Delta_i^{\tilde{O}})$. Note that $\Phi(i) \in \{0, \ldots, m-1\}$ and 
\begin{align*}
\Phi( \mathcal{S}^{\tilde{O}}(4i+k))=\mathcal{S}^O\left(4\Phi(i)+k)\right)
\end{align*}
where $i \in \{0, \dots, n-1\}$ and $k \in \{0,1,2,3\}$. 

If we convert $\mathcal{S}^{\tilde{O}}$ and $\mathcal{S}^O$ into Python lists \texttt{dsequp} and  \texttt{dseqdown} respectively, then the function \texttt{Covers(dsequp,dseqdown)} would return a (Python) list each of whose members is a Python list of $n$ entries of tetrahedra indices from $\{0, \ldots, m-1\}$ representing a covering map from $\tilde{O}$ onto $O$, say $\Phi$ whose $i$-th entry is $\Phi(i)$. Suppose we choose a Python list, say \texttt{L\_cover}, representing a cover of $O$ by $\tilde{O}$. Let $c_0, \ldots, c_{p-1}$ be the equivalence classes of the partition of $\{0, \ldots, n-1\}$ each representing a cusp of $\tilde{O}$. For $t\in \{0, \ldots, p-1\}$, we will use $\Phi(c_t)$ to denote the cusp of $O$ that is $\{\Phi(i): i \in c_t\}$. Let \texttt{L\_t} be the Python list representing $c_t$ that we get from the function \texttt{CuspSeqs(dsequp)} for $t\in \{0, \ldots, p-1\}$. Then, the function \texttt{CuspCovers(\texttt{L\_cover},CuspSeqs(dsequp),CuspSeqs(dseqdown))} returns the Python list (of lists) \texttt{[[L\_0, L\_image\_0], \ldots, [L\_{p-1}, L\_image\_{p-1}]]} where \texttt{L\_image\_{t}} is the Python list representing $\Phi(c_t)$ that we get from the function \texttt{CuspSeqs(dseqdown)} for $t\in \{0, \ldots, p-1\}$. 

The utility \texttt{SigToSeq.py} developed in \cite{orbcenpract} (code available at \cite{orbtricode}) when run inside Regina \cite{regina}, can compute the destination sequence (i.e. an orbifold triangulation) of a tetrahedral manifold given its tetrahedral triangulation information. We first describe how the destination sequence is computed in \texttt{SigToSeq.py} given a tetrahedral triangulation. Suppose $M$ is a tetrahedral manifold with a tetrahedral triangulation $\mathcal{T}^M$ into (regular ideal) tetrahedra $T^M_0, \ldots, T^M_{r-1}$. Let $i \in \{0, \ldots, r-1\}$. The barycentric subdivision of $T^M_i$ gives rise to $24$ barycentric pieces, each a tetrahedron. Each such piece has one ideal vertex (of $T^M_i$) and three finite vertices - the barycenter of $T^M_i$, an \textit{edge vertex} that is the barycenter of an edge of $T^M_i$ and a \textit{face vertex} that is the barycenter of a face of $T^M_i$. We can then denote such a piece by the $4$-tuple $(i,j,k, l)$ where $j\in \{0,1,2,3\}$ is the index of the face of $T^M_i$ containing the finite face vertex of the piece, $k\in \{0,1,2,3\}$ is the index of the ideal vertex of $T^M_i$ that is the only ideal vertex of the piece and $l\in \{0,1,2,3\}$ is the index of the ideal vertex of $T^M_i$ such that the barycenter of the edge of $T^M_i$ joining $k$ and $l$ is the finite edge vertex of the piece.

Now for a given $i \in \{0, \ldots, r-1\}$, the $j$-face of  $T^M_i$ is glued to a unique $j'$-face of a unique tetrahedron  $T^M_{i'}$ of $\mathcal{T}^M$ where $i' \in \{0, \ldots, r-1\}$ and $j' \in \{0,1,2,3\}$. This gluing also sends ideal vertices $k$ and $l$ of  $T^M_i$ to unique ideal vertices $k'$ and $l'$ of  $T^M_{i'}$ since the $j$-face contains the ideal vertices $k$ and $l$. The gluing of the pieces $(i,j,k, l)$ and $(i',j',k', l')$ along faces $j$ and $j'$ makes up a tetrahedron in the orbifold triangulation $\mathcal{O}^M$ of $M$ into $12r$ many tetrahedron of the form shown in the left of Figure \ref{orbtet}. So, each of such $12r$ tetrahedra can be written as $\left\{(i,j,k, l),(i',j',k', l')\right\}$. Note that the tetrahedron (i.e. the set of two $4$ tuples) whose $v$ face (respectively, $f_1$ face, $f_0$ face and $e$ face) is glued to the $v$ face (respectively, $f_0$ face, $f_1$ face and $e$ face) of $\left\{(i,j,k, l),(i',j',k', l')\right\}$ is automatically determined by the barycentric subdivisions of $T^M_1, \ldots, T^M_r$. Now choosing one tetrahedron $\left\{(i,j,k, l),(i',j',k', l')\right\}$ and labelling it as the $0$-th tetrahedron, we will look at the tetrahedron whose $v$ face is glued to the $v$ face of this $0$-th tetrahedron $\left\{(i,j,k, l),(i',j',k', l')\right\}$ and if it is a new tetrahedron, we label it as $1$. We continue the same process with the tetrahedra that are glued to the $f_0$ face, the $f_1$ face and the $e$ face of $0$-th tetrahedron  $\left\{(i,j,k, l),(i',j',k', l')\right\}$, i.e., if they are new tetrahedra, we label them the next non-negative integer(s). We repeat the same process with tetrahedron with label $1$, $2$ and so on. If at any stage, after completing the process for the tetrahedron with label $s$, we don't have a tetrahedron with label $s+1$ from our process of labelling so far, we choose one unlabelled tetrahedron and label it $s+1$ and continue the process. This process itself produces the destination sequence of $M$ corresponding to the orbifold triangulation $\mathcal{O}^M$ of $M$.  

Given a tetrahedral manifold $M$, if we know its tetrahedral Regina triangulation, say \texttt{t}, then the function \texttt{Des\_seq(t)} of \texttt{SigToSeq.py} returns a Python list \texttt{dseq} representing a destination sequence of $M$ corresponding to the obtained orbifold triangulation of $M$ described above. 

We note that each ideal vertex of index $k\in \{0,1,2,3\}$ of a $T^M_i$ corresponds uniquely to a cusp of $M$ and so we can partition the set 
$$\left\{(i,k): k \in \{0,1,2,3\}, i \in \{0, \ldots, r-1\}\right\}$$ into equivalence classes such that each class uniquely represents a cusp of $\mathcal{T}^M$. Now, if \texttt{L} is an entry of \texttt{CuspSeqs(dseq)}, i.e., a Python list representing a cusp of \texttt{dseq}, then \texttt{L[0]} is the index of some tetrahedron in $\mathcal{O}^M$. Suppose this index is $s$ and the $s$-th tetrahedron in $\mathcal{O}^M$ is $\left\{(i,j,k, l),(i',j',k', l')\right\}$. Then the equivalence class of both $(i,k)$ and $(i',k')$ represents a cusp of $\mathcal{T}^M$. \texttt{SigToSeq.py} also includes a function \texttt{Cusp\_info} which, with \texttt{t} as the input, returns a list whose $s$-th element is $(s, (i,k),(i',k'))$ (the order of $(i,k)$ anxd $(i',k')$ might be switched).

\subsection{Members of $\mathcal{E}_1$ and $\mathcal{E}_4$}\label{E1E4}

In this subsection, we are going to use Proposition \ref{coversmallestmulticusp} to conclude that when $(M,i) \in \mathcal{E}_1\cup \mathcal{E}_4$, if $\mathcal{F}$ is a family of hyperbolic knot complements obtained from filling all cusps of $M$ but cusp $i$ that geometrically converges to $M$, then $\mathcal{F}$ can have at-most finitely elements with hidden symmetries. Note that existence of $\mathcal{F}$ implies that $M$ is a link complement. 

We first recall the two cusped orbifold $O_{(2,3,6),(2,2,2,2)}$ with one $(2,3,6)$ cusp and one $(2,2,2,2)$ cusp from Section \ref{coversection}. We will use the utilities \texttt{TestForCovers.py} and \texttt{SigToSeq.py} from \cite{orbcenpract} (see codes at \cite{orbtricode}) to show that for each $(M,i) \in \mathcal{E}_1\cup \mathcal{E}_4$, there is a covering map  $\phi_{M,i}: M \to O_{(2,3,6), (2,2,2,2)}$ such that there is only one cusp $c$ of $M$ that maps to the $(2,3,6)$-cusp of $O_{(2,3,6), (2,2,2,2)}$ via  $\phi_{M,i}$ and that $c$ is symmetric to the $i$-th cusp of $M$.

From \cite{orbcenpract}, we see that the orbifold triangulation of $O_{(2,3,6), (2,2,2,2)}$ has $10$ tetrahedra labelled $0, \ldots, 9$ and the corresponding destination sequence as a Python list is 
$$\texttt{[0,0,0,1,2,3,4,0,1,5,6,2,6,4,1,4,5,1,3,3,4,6,2,7,3,2,5,8,8,8,9,5,7,9,7,6,9,7,8,9]}$$
and that \texttt{[0,1,3,4]} represents the $(2,3,6)$-cusp of $O_{(2,3,6), (2,2,2,2)}$. We will use \texttt{dseqdown} to denote the above destination sequence of $O_{(2,3,6), (2,2,2,2)}$. 

Our goal here is to obtain orbifold triangulations for the manifolds appearing in $\mathcal{E}_1 \cup \mathcal{E}_4$ from their regular ideal triangulations using \texttt{SigToSeq.py} of \cite{orbcenpract}. Now, one can see the isomorphism signatures of all the regular ideal tetrahedral decompositions of all the orientable tetrahedral manifolds of the Fominykh-Garoufalidis-Goerner-Tarkaev-Vesnin census \cite{FGGTV} from the ancillary file \cite{isosigcensus} of their paper \cite{FGGTV}. We should also note from this ancillary file \cite{isosigcensus} of \cite{FGGTV} that some tetrahedral manifolds have more than one (regular ideal) tetrahedral triangulations and hence, more than one corresponding tetrahedral isomorphism signatures.

We first write a code \texttt{Iso\_sig.py}, which we run in \texttt{regina-python}\footnote{When \texttt{.py} files are run via \texttt{regina-python}, \texttt{regina} is imported as a Python module. See \url{https://regina-normal.github.io/docs/man-regina-python.html}.}, to collect all the isomorphism signatures of all the regular ideal tetrahedral decompositions of the manifolds appearing in the members of $\mathcal{E}$ in the \texttt{Iso\_sigs.json} file. Both \texttt{Iso\_sig.py} and \texttt{Iso\_sigs.json} can be accessed from \cite{tetra_code}. Inside \texttt{Iso\_sig.py}, we imported \texttt{regina} and \texttt{json} modules, the \texttt{children\_iterator} function from (the ancillary file) \texttt{example.py} \cite{examplecensus} of \cite{FGGTV}, and the regina packet $$\texttt{tetrahedralOrientableCuspedCensus.rga}$$ \cite{censusreginapacket} of \cite{FGGTV} containing (information on) all regular ideal tetrahedral decompositions in the orientable census of \cite{FGGTV}. We also imported the \texttt{Excep\_tuple.json} file (can be found in \cite{tetra_code}) in \texttt{Iso\_sig.py}. In \texttt{Iso\_sig.py}, we check if the default isomorphism signatures from (the last entries of the $4$-tuples in) \texttt{Excep\_tuple.json} appear as the tetrahedral isomorphism signatures of manifolds in $$\texttt{tetrahedralOrientableCuspedCensus.rga}$$ and when it does, we collect all the tetrahedral isomorphism signatures corresponding to all the regular ideal triangulations from $$\texttt{tetrahedralOrientableCuspedCensus.rga}$$ of those manifolds. Each element of the list exported as the \texttt{Iso\_sigs.json} file is a list itself containing all the tetrahedral isomorphism signatures corresponding to all the regular ideal triangulations of a member $(M,i)$ of $\mathcal{E}$. 

We now discuss how we split $\mathcal{E}$ into $\mathcal{E}_1$,  $\mathcal{E}_2$,  $\mathcal{E}_3$ and  $\mathcal{E}_4$. We will define $\mathcal{E}_1$ and  $\mathcal{E}_4$. Then, $\mathcal{E}_2$ and  $\mathcal{E}_3$ are distinguished by the special properties of their corresponding horoball packings as already seen in the last two paragraphs of Subsection \ref{snappyrun}. Let $\mathcal{E}_{\text{single}}$ be defined as the set of all pairs $(M,i)$ in $\mathcal{E}$ such that $ \lvert \mathcal{C}_{\mathcal{E}}(M) \rvert =1$. Let $\mathcal{E}_{\text{good cover}}$ be the set of all pairs $(M,i)$ in $\mathcal{E}$ such that there exists a triangulation and orientation preserving (orbifold) cover from $M$ to $O_{(2,3,6),(2,2,2,2)}$ mapping only a single cusp of $M$ to the $(2,3,6)$ cusp of $O_{(2,3,6),(2,2,2,2)}$, where the orbifold triangulation of $M$ is obtained from one of its regular ideal tetrahedral triangulation (corresponding to an isomorphism signature from \texttt{Iso\_sigs.json}) via \texttt{SigToSeq.py} of \cite{orbcenpract}.  
Let 
$$\mathcal{E}'_1\coloneq \mathcal{E}_{\text{good cover}} \cap \mathcal{E}_{\text{single}} \text{ and } \mathcal{E}'_4\coloneq \mathcal{E}_{\text{good cover}}-\mathcal{E}'_1.$$ 
We finally define,
\begin{align*}
 \mathcal{E}_1&=\mathcal{E}'_1 \cup \left \{(\texttt{otet20\_01414}, 2), (\texttt{otet20\_01438}, 2)\right\}, \text{ and, }\\
  \mathcal{E}_4&=\mathcal{E}'_4 -\left \{(\texttt{otet20\_01414}, 2), (\texttt{otet20\_01414}, 3), (\texttt{otet20\_01438}, 2), (\texttt{otet20\_01438}, 3)\right \}.
 \end{align*}
 (Note that both (\texttt{otet20\_01414}, 3) and (\texttt{otet20\_01438}, 3) belong to $\mathcal{E}_3$.) 

We write \texttt{E\_single.py} and \texttt{E\_good\_cover.py} to obtain the elements of respectively $\mathcal{E}_{\text{single}}$ and $\mathcal{E}_{\text{good cover}}$ and export them as respectively \texttt{E\_single.json} and \texttt{E\_good\_cover.json} files. All of \texttt{E\_single.py}, \texttt{E\_good\_cover.py}, \texttt{E\_single.json} and \texttt{E\_good\_cover.json} are available at \cite{tetra_code}. The pseudocode for \texttt{E\_single.py} is discussed in Algorithm \ref{Esinglealg} and the pseudocode for \texttt{E\_good\_cover.py} in Algorithm \ref{Egoodalg}.

 \begin{algorithm}
  \caption{Pseudocode for obtaining the Python list representing $\mathcal{E}_{\text{single}}$}

 \begin{algorithmic}

 \State Get the \texttt{json} file \texttt{Excep\_tuple}
 \State 
 \Function {All\_exep\_cusp}{Integer $j$}
 \State Exceptional\_cusp$\gets$ empty list
 
\For {(index, manifold name, cusp number, isomorphism signature) $\in$  \texttt{Excep\_tuple}}
\If{$j$ is index}
\State  add cusp number to Exceptional\_cusp
\EndIf
\EndFor
\State \Return Exceptional\_cusp
 \EndFunction
 
\State

\State E\_indices $\gets \{ \}$
\State
\For {(index, manifold name, cusp number, isomorphism signature) $\in$  \texttt{Excep\_tuple}}
\State add index to E\_indices
\EndFor
\State 
\State E\_single $\gets$ empty list
\State 
\For {$j \in$ E\_indices}
\If {All\_excep\_cusp($j$) has only one element}
\State add $(j, \text{the only member of All\_excep\_cusp}(j))$ to E\_single
 
\EndIf
\EndFor
\State 
\State Export E\_single as the \texttt{E\_single.json} file

 \end{algorithmic}
 \label{Esinglealg}
 \end{algorithm}

 \begin{algorithm}
 \caption{Pseudocode for obtaining the Python list representing $\mathcal{E}_{\text{good cover}}$}

\begin{algorithmic}
 \State Get the \texttt{json} files \texttt{Excep\_tuple}, \texttt{Iso\_sigs} and the list \texttt{dseqdown}
 \State

\State  E\_good\_cover $\gets $ empty list
\State 
\For {$j \in \{0, \dots, 85\}$}
\For {sig in the $j$-th member of \texttt{Iso\_sigs}} 
 \State \texttt{dsequp} $=$ the destination sequence list for the Regina triangulation corresponding to sig 
 \State Cover\_list$=$list of triangulation and orientation preserving covers from \texttt{dsequp} to \texttt{dseqdown}
 \If {Cover\_list is non-empty}
\For {$c$ in Cover\_list}
\State Rigid\_preimage$=$the list of all cusps of \texttt{dsequp} mapping to \texttt{[0,1,3,4]} via $c$
\If {Rigid\_preimage has only one element}
\State add $j$-th member of \texttt{Excep\_tuple} to E\_good\_cover
\EndIf

\EndFor

\EndIf

\EndFor

\EndFor
\State
\State Export  E\_good\_cover as the \texttt{E\_good\_cover.json} file

 \end{algorithmic}
 \label{Egoodalg}
 \end{algorithm}
 \begin{remark} We make a note of following. 
 \begin{itemize}
 \item We imported the \texttt{json} module in \texttt{E\_single.py}, and the \texttt{json} module, the \texttt{Des\_seq} function from \texttt{SigToSeq.py} of \cite{orbcenpract} and the functions \texttt{CuspSeqs}, \texttt{Covers}  and \texttt{CuspCovers} from \texttt{TestForCovers} of  \cite{orbcenpract} in \texttt{E\_good\_cover.py}. 
 \item  \texttt{E\_good\_cover.py} needs to be run in \texttt{regina-python}\footnote{See \url{https://regina-normal.github.io/docs/man-regina-python.html}.}.
 \item We use \texttt{Des\_seq} function from \texttt{SigToSeq.py} of  \cite{orbcenpract} (and \texttt{Triangulation3.fromIsoSig} command from Regina \cite{regina}) to obtain \texttt{dsequp} in Algorithm \ref{Egoodalg}. In particular, the destination sequence mentioned in the definition of \texttt{dsequp} in Algorithm \ref{Egoodalg} is the orbifold destination sequence from \cite{orbcenpract} (and \cite{orbcentheory}). 
 \item We use the \texttt{Covers} function from \texttt{TestForCovers.py} of \cite{orbcenpract} to get Cover\_list in Algorithm \ref{Egoodalg} from \texttt{dsequp} and \texttt{dseqdown}. 
 \item We use  the \texttt{CuspCovers} function from  \texttt{TestForCovers.py} of \cite{orbcenpract} to obtain Rigid\_preimage in Algorithm \ref{Egoodalg} . 
 
 \end{itemize}
 \end{remark}

Before we move on to the $\mathcal{E}_1$ case, we point out that for each $(M,i)$ in $\mathcal{E}_1$, $\mathcal{E}_4$ and $\mathcal{E}_2\cup \mathcal{E}_3$, \texttt{E\_split.py} file in \cite{tetra_code} prints out $(j, \text{namestring}, i)$ where $j$ is the index of $M$ in \texttt{K\_hlgy\_lk} and \text{namestring} is the name of $M$ (in the notation of \cite{FGGTV}) as a Python string. Inside \texttt{E\_split.py}, we imported the \texttt{json} module, and the files \texttt{Excep\_tuple.json}, \texttt{E\_single.json} and \texttt{E\_good\_cover.json}.

Let $(M,i)\in \mathcal{E}_1$. By the definition of $\mathcal{E}_1$, there is a covering $\phi_{M,i}: M \to O_{(2,3,6), (2,2,2,2)}$ so that there is only one cusp of $M$, say $c_M$, that maps to the $(2,3,6)$ cusp of $O_{(2,3,6), (2,2,2,2)}$ via $\phi_{M,i}$. Note from Appendix \ref{allEi} that the volume of $M$ is $20v_0$. Now, let $\mathcal{F}$ be a family of hyperbolic knot complements obtained from Dehn filling all cusps of $M$ but cusp $i$ and geometrically converging to $M$. Then, $M$ is a link complement as well. So, the hypothesis of Proposition \ref{codeprop20v0} holds. If $(M,i)\in \mathcal{E}'_1$, then $\mathcal{C}_{\mathcal{E}}(M)=\{i\}$, i.e., the only cusp in $\mathcal{C}(M)$ for which the \texttt{Free\_rot\_strng} function of \texttt{TwoCentroids.py} returns \texttt{True} is $i$. Hence, Proposition \ref{codeprop20v0} implies that $c_M$ is symmetric to $i$. Now, assume $(M,i)\in \mathcal{E}_1-\mathcal{E}'_1$. Then $i=2$ and $\mathcal{C}_{\mathcal{E}}(M)=\{2,3\}$. But $(M,3) \in \mathcal{E}_3$. So, by Remark \ref{codeprop20v0remark}, $c_M$ cannot be symmetric to $3$. Since Proposition \ref{codeprop20v0} implies that $c_M$ is symmetric to an element of $\mathcal{C}_{\mathcal{E}}(M)$, we conclude that $c_M$ is symmetric to $i=2$ in this situation as well. We now apply Proposition \ref{coversmallestmulticusp} to see that $\mathcal{F}$ can have at-most finitely many elements with hidden symmetries.

The last case that remains to be considered is when $(M,i) \in \mathcal{E}_4$. This case would require careful analysis of each manifold $M$ for which $(M,i) \in \mathcal{E}_4$ for some $i$, since for each such manifold $M$ there are multiple $i$'s for which $(M,i) \in \mathcal{E}_4$.  So, we will have to make sure that for each $(M,i) \in \mathcal{E}_4$, there is a covering map $\phi_{(M,i)}: M \to O_{(2,3,6), (2,2,2,2)}$ so that the only cusp of $M$ that maps to the $(2,3,6)$-cusp of $O_{(2,3,6), (2,2,2,2)}$ is the $i$-cusp. Let $\mathcal{M}_{\mathcal{E}_4}$ denote the collection of all $M$ so that $(M,i) \in \mathcal{E}_4$ for some $i$. For $M \in \mathcal{M}_{\mathcal{E}_4}$, we define, 
$$\mathcal{C}_{\mathcal{E}_4}(M)=\{ i \in \mathcal{C}(M): (M,i) \in \mathcal{E}_4\}.$$

From Appendix \ref{E4}, we see that, 
\begin{align*}
&\mathcal{M}_{\mathcal{E}_4}=\left \{ \texttt{otet20\_00049},\texttt{otet20\_00063}, \texttt{otet20\_00462},\texttt{otet20\_00474},\texttt{otet20\_00577}   \right\},\nonumber \\
&\mathcal{C}_{\mathcal{E}_4}(\texttt{otet20\_00049})=\{0,1\}, \nonumber \\
&\mathcal{C}_{\mathcal{E}_4}(\texttt{otet20\_00063})=\{0,2\},  \nonumber \\
&\mathcal{C}_{\mathcal{E}_4}(\texttt{otet20\_00462})=\{1,2\}, \nonumber \\
&\mathcal{C}_{\mathcal{E}_4}(\texttt{otet20\_00474})=\{0,1,2\},  \nonumber \\
&\mathcal{C}_{\mathcal{E}_4}(\texttt{otet20\_00577})=\{0,3\}. \nonumber
\end{align*}
Before delving into each individual case from $\mathcal{E}_4$, we describe the general procedure that we follow to investigate the members of $\mathcal{E}_4$. Let $M \in \mathcal{M}_{\mathcal{E}_4}$. We will use the isomorphism signature(s) of $M$ corresponding to its tetrahedral triangulations from the ancillary file \cite{isosigcensus} of \cite{FGGTV} to get the triangulation data from SnapPy \cite{snappy}. If \texttt{sig} is the string representing such an isomorphism signature, then the command \texttt{Manifold(sig).\_to\_string()} in SnapPy \cite{snappy} returns a string, say \texttt{s}, which consists of the corresponding tetrahedral triangulation information. (In \texttt{E4\_all\_triangulations.py} from \cite{tetra_code}, we collect these strings of triangulation data of all the tetrahedral isomorphism signatures of all $5$ manifolds of $\mathcal{M}_{\mathcal{E}_4}$ from the ancillary file \cite{isosigcensus} of \cite{FGGTV} in a Python dictionary and export this dictionary as the \texttt{E4\_snappy\_triangulations.json} file, also available in \cite{tetra_code}. We had to import both \texttt{snappy} and \texttt{json} modules in \texttt{E4\_all\_triangulations.py}.

Now, the command \texttt{Triangulation3(s)} in Regina \cite{regina} would produce a Regina tetrahedral triangulation, say \texttt{t}. Now, as in the $\mathcal{E}_1$ case, we will use the function \texttt{Des\_seq(t)} from \texttt{SigToSeq.py} of  \cite{orbcenpract} (available from \cite{orbtricode}) to get the destination sequence \texttt{dsequp} of the orbifold triangulation of $M$ and the functions \texttt{Covers} and \texttt{CuspCovers} from \texttt{TestForCovers.py} of  \cite{orbcenpract} (available from \cite{orbtricode}) to investigate the triangulation and orientation preserving covers of \texttt{dseqdown} (i.e. the destination sequence of $O_{(2,3,6),(2,2,2,2)}$) by \texttt{dsequp} and how they act on the cusps. From the output of \texttt{CuspCovers}, we will be able to see which cusps of \texttt{dsequp} map to \texttt{[0,1,3,4]} (i.e. the $(2,3,6)$ cusp of  $O_{(2,3,6),(2,2,2,2)}$). The cusps in the pre-image of \texttt{[0,1,3,4]} are Python lists of indices of tetrahedra in the orbifold triangulation of $M$ corresponding to the destination sequence \texttt{dsequp}. If the pre-image of  \texttt{[0,1,3,4]} has only one such Python list, say \texttt{L}, we look at which tetrahedra indices (of the orbifold triangulation of $M$ corresponding to \texttt{dsequp}) \texttt{L} contains. Given any such tetrahedron index in \texttt{L}, we then can use the \texttt{Cusp\_info} function from \texttt{SigToSeq.py} of  \cite{orbcenpract}  to get one corresponding pair $(i, k)$ where $k$ is the index of an ideal vertex of the $i$-th tetrahedron of the tetrahedral triangulation \texttt{t}. Now, the information given by the Regina command \texttt{ t.boundaryComponents()} would tell us the label of the cusp corresponding to $(i, k)$ in \texttt{t} and so in \texttt{Manifold(sig)} in SnapPy \cite{snappy} as the Regina triangulation \texttt{t} is lifted from the SnapPy triangulation of \texttt{Manifold(sig)} by feeding the string \texttt{s} containing the SnapPy triangulation data of \texttt{Manifold(sig)} to Regina. It is crucial that we take all these extra steps since we want to make sure that for each $(M,i) \in \mathcal{E}_4$, we have a covering  $\phi_{(M,i)}: M \to O_{(2,3,6), (2,2,2,2)}$ such that the only cusp of $M$ that maps to the $(2,3,6)$-cusp of $O_{(2,3,6), (2,2,2,2)}$ is the $i$-cusp and the fact that the label of the cusps change if we feed a different tetrahedral triangulation of $M$ (via a different isomorphism signature) in SnapPy \cite{snappy}. 

It is important to note that for each of \texttt{Manifold(`otet20\_00049')}, \texttt{Manifold(`otet20\_00063')}, \texttt{Manifold(`otet20\_00462')}, \texttt{Manifold(`otet20\_00474')} and \texttt{Manifold(`otet20\_00577')} SnapPy \cite{snappy} produces a default tetrahedral triangulation. To get the other tetrahedral triangulations (if any), we can use the command \texttt{Manifold(sig)} where \texttt{sig} is the string representing the isomorphism signature of that other tetrahedral triangulation. We are now ready to discuss the individual cases corresponding to $\mathcal{E}_4$. The computations below can be seen by running \texttt{E4\_computations.py} from \cite{tetra_code} in \texttt{regina-python}\footnote{refer to \url{https://regina-normal.github.io/docs/man-regina-python.html}.}. (We remark that in \texttt{E4\_computations.py}, we imported the \texttt{regina} and \texttt{json} modules, the functions \texttt{Des\_seq}, \texttt{Cusp\_info}, \texttt{Mfd\_cusp\_index} from \texttt{SigToSeq.py} of \cite{orbcenpract}, and the functions \texttt{CuspSeqs}, \texttt{Covers}, \texttt{CuspCovers} from \texttt{TestForCovers.py} of \cite{orbcenpract}. We also imported the file \texttt{E4\_snappy\_triangulations.json} from \cite{tetra_code} inside the Python file \texttt{E4\_computations.py} to get hold of the SnapPy triangulation data for all the tetrahedral isomorphism signatures of all manifolds from $\mathcal{M}_{\mathcal{E}_4}$.)

For \texttt{otet20\_00049}, the orbifold triangulation (or the destination sequence) obtained from the default tetrahedral triangulation corresponding to the isomorphism signature 
$$\texttt{uLLLLvzPPzLQQQccefemlkokqsqrtqsrtrstiitdipattiuapqlqpqphi}$$
has two triangulation and orientation preserving covers of $O_{(2,3,6), (2,2,2,2)}$. In the first one, only cusp $1$ of \texttt{otet20\_00049} covers the $(2,3,6)$ cusp of  $O_{(2,3,6), (2,2,2,2)}$ and in the second, only cusp $0$ of \texttt{otet20\_00049} covers the $(2,3,6)$ cusp of  $O_{(2,3,6), (2,2,2,2)}$. 

For \texttt{otet20\_00063}, the orbifold triangulation (or the destination sequence) obtained from the default tetrahedral triangulation, say $\mathcal{T}_{161}^0$ (we use $161$ in the suffix as the index of \texttt{otet20\_00063} in \texttt{K\_hlgy\_lk} is $161$) corresponding to the isomorphism signature 
$$\texttt{uLLLLvzQMzPPQPccefemllkkkppqprsqstttiitdimpamtiaplttdhxeh} $$
has two triangulation and orientation preserving covers of $O_{(2,3,6), (2,2,2,2)}$. In one of such covers only cusp $0$ of \texttt{otet20\_00063} covers the $(2,3,6)$ cusp of  $O_{(2,3,6), (2,2,2,2)}$.  We also note from the ancillary file \cite{isosigcensus} of \cite{FGGTV}
that \texttt{otet20\_00063} also has a second tetrahedral decomposition, say $\mathcal{T}_{161}^1$, corresponding to the isomorphism signature 
$$\text{\texttt{uLLPLvAALQMvPQcceefejijimnomnpsttstsiiatdpaxteqahiaehdqid}}.$$
We further see that the orbifold triangulation (or the destination sequence) obtained from this tetrahedral triangulation $\mathcal{T}_{161}^1$ has two triangulation and orientation preserving covers of $O_{(2,3,6), (2,2,2,2)}$. In each such case, only one cusp of the tetrahedral decomposition $\mathcal{T}_{161}^1$ covers the $(2,3,6)$ cusp of  $O_{(2,3,6), (2,2,2,2)}$. These are cusps $1$ and $0$ of $\mathcal{T}_{161}^1$ (in SnapPy \cite{snappy}). Here we need to note that the $1$ and $0$ cusps of $\mathcal{T}_{161}^1$ might not be the $1$ and $0$ cusps of $\mathcal{T}_{161}^0$ as cusp indexing may change if we take a different triangulation in SnapPy. We will now consider the symmetry equivalence on the set of the cusps, i.e., two cusps belong to the same equivalence class if and only if they are symmetric to each other. One can check from SnapPy that the partition on the set of cusps corresponding to this symmetry equivalence for $\mathcal{T}_{161}^0$ and $\mathcal{T}_{161}^1$ are respectively $\left \{ \{0\}, \{1,3\}, \{2,4\}\right \}$ and $\left \{ \{0,1\}, \{2,4\}, \{3\}\right \}$. 
So, cusp $0$ of $\mathcal{T}_{161}^0$ is cusp $3$ of $\mathcal{T}_{161}^1$. Now one can also check from SnapPy that the maximal volume of cusp $1$ and $2$ of $\mathcal{T}_{161}^0$ are $3.4641$ and $6.9282$ respectively where as the maximal volume of cusp $0$ and $2$ of $\mathcal{T}_{161}^1$ are $6.9282$ and $3.4641$ respectively. Since the maximal volume of a cusp is invariant under the action of an isometry, we can conclude that cusps $0$ and $1$ of $\mathcal{T}_{161}^1$ are cusps $2$ and $4$ of $\mathcal{T}_{161}^0$ (may be not in the same order). So, there is a covering of \texttt{otet20\_00063} onto $O_{(2,3,6), (2,2,2,2)}$ such that only cusp $2$ of \texttt{otet20\_00063} covers the $(2,3,6)$ cusp of  $O_{(2,3,6), (2,2,2,2)}$. 

 \texttt{otet20\_00462} also has two tetrahedral triangulations. We will consider the one that corresponds to the isomorphism signature $$\texttt{uvvLLMvvQAQQQPcfghfnrponsoprsqosqtttaaulaqxxluaulqahluluq}$$ and not the default triangulation of \texttt{otet20\_00462}. Let us denote this triangulation by  $\mathcal{T}_{221}^1$. We can see that the orbifold triangulation (or the destination sequence) obtained from  $\mathcal{T}_{221}^1$ has two triangulation and orientation preserving covers of $O_{(2,3,6), (2,2,2,2)}$. In the first cover, only cusp $1$ of $\mathcal{T}_{221}^1$ covers the $(2,3,6)$ cusp of  $O_{(2,3,6), (2,2,2,2)}$ and in the second, only cusp $0$ of $\mathcal{T}_{221}^1$ covers the $(2,3,6)$ cusp of  $O_{(2,3,6), (2,2,2,2)}$. Now, one can check from SnapPy that no cusp of \texttt{otet20\_00462} is symmetric to another cusp of \texttt{otet20\_00462}. Furthermore, the maximal volumes of cusp $0$, $1$ and $2$ of the default triangulation $\mathcal{T}_{221}^0$ are respectively $13.8564$, $6.9282$ and $6.9282$ where as the maximal volumes of cusp $0$, $1$ and $2$ of $\mathcal{T}_{221}^1$ are respectively $6.9282$, $6.9282$ and $13.8564$. So, cusps $1$ and $0$ of $\mathcal{T}_{221}^1$ are cusps $1$ and $2$ of (the default triangulation) \texttt{otet20\_00462} (may not be in the same order). So, for each $i \in \{1,2\}$, there is a covering map from \texttt{otet20\_00462} onto $O_{(2,3,6), (2,2,2,2)}$ such that only cusp $i$ covers the $(2,3,6)$ cusp of  $O_{(2,3,6), (2,2,2,2)}$.

 For \texttt{otet20\_00577}, the orbifold triangulation (or the destination sequence) obtained from the default tetrahedral triangulation, denoted as $\mathcal{T}_{256}^0$, corresponding to the isomorphism signature 
 $$\texttt{uLLzvvvPQQwPQQcceenqllnpomotssrsrtrtiiamqamxaqxippaihhiia}$$
 has two triangulation and orientation preserving covers of $O_{(2,3,6), (2,2,2,2)}$. In the first one, only cusp $1$ of \texttt{otet20\_00577} covers the $(2,3,6)$ cusp of $O_{(2,3,6), (2,2,2,2)}$ and in the second, only cusp $0$ of \texttt{otet20\_00577} covers the $(2,3,6)$ cusp of  $O_{(2,3,6), (2,2,2,2)}$. But it can be checked from SnapPy that cusps $0$ and $1$ of \texttt{otet20\_00577} are symmetric. So, we consider the second tetrahedral triangulation $\mathcal{T}_{256}^1$ corresponding to the isomorphism signature 
 $$\texttt{uvvLLAwLMMQAPQchfgkhomnmqpooqrststtseeauauappptahqauaxqex}.$$
 The orbifold triangulation (or the destination sequence) obtained from $\mathcal{T}_{256}^1$  has two triangulation and orientation preserving covers of $O_{(2,3,6), (2,2,2,2)}$ and in one of such covers, the only cusp of $\mathcal{T}_{256}^1$ that covers the $(2,3,6)$ cusp of $O_{(2,3,6), (2,2,2,2)}$ is cusp $2$ of $\mathcal{T}_{256}^1$. From SnapPy we can check that the partition on the set of cusps corresponding to the symmetry equivalence for $\mathcal{T}_{256}^0$ and $\mathcal{T}_{256}^1$ are respectively $\left \{\{0,1\}, \{2,4\}, \{3\} \right \}$ and $\left \{ \{0,3\},  \{1,4\}, \{2\}\right \}$. So, cusp $2$ of $\mathcal{T}_{256}^1$ is cusp $3$ of  $\mathcal{T}_{256}^0$. Hence, there is a covering from  \texttt{otet20\_00577} onto $O_{(2,3,6), (2,2,2,2)}$ such that only cusp $3$ maps to the $(2,3,6)$ cusp of $O_{(2,3,6), (2,2,2,2)}$.

\texttt{otet20\_00474} also has two tetrahedral triangulations. The default triangulation that SnapPy uses for \texttt{otet20\_00474}, which we will denote by $\mathcal{T}_{222}^0$, corresponds to the isomorphism signature 
$$\texttt{uLLvvvLAPQQAPQccennmlllnrorqpqttssstiilxhhtxapaaqaipaqiaa}.$$
There are two triangulation and orientation preserving covers of the orbifold triangulation (or the destination sequence) obtained from $\mathcal{T}_{222}^0$ onto $O_{(2,3,6), (2,2,2,2)}$. In the first such cover, the only cusp of $\mathcal{T}_{222}^0$ that covers the  $(2,3,6)$ cusp of  $O_{(2,3,6), (2,2,2,2)}$ is cusp $1$ of $\mathcal{T}_{222}^0$ and in the second cover, the only cusp of $\mathcal{T}_{222}^0$ that covers the $(2,3,6)$ cusp of  $O_{(2,3,6), (2,2,2,2)}$ is cusp $0$.  The second tetrahedral triangulation of \texttt{otet20\_00474}, which we will denote by $\mathcal{T}_{222}^1$, corresponds to the isomorphism signature
$$\texttt{uvvLLPwLPMQAPQcgjilgmnmnpopqrptststsmataptdammaaatpupqqux}.$$
The orbifold triangulation (or the destination sequence) obtained from $\mathcal{T}_{222}^1$ has two triangulation and orientation preserving covers of $O_{(2,3,6), (2,2,2,2)}$. In one such cover, the only cusp of $\mathcal{T}_{222}^1$ that covers the $(2,3,6)$ cusp of $O_{(2,3,6), (2,2,2,2)}$ is cusp $1$ of $\mathcal{T}_{222}^1$. We will now argue that cusp $1$ of $\mathcal{T}_{222}^1$ is cusp $2$ of $\mathcal{T}_{222}^0$. First we note from SnapPy that the partition on the cusps corresponding to the symmetry equivalence for $\mathcal{T}_{222}^0$ and $\mathcal{T}_{222}^1$ are 
respectively $\left \{\{0\}, \{1\}, \{2\}, \{3,4\} \right \}$ and $\left \{\{0\}, \{1\}, \{2,3\}, \{4\} \right \}$. So, cusps $0$, $1$ and $2$ of $\mathcal{T}_{222}^0$ are cusps $0$, $1$ and $4$ of $\mathcal{T}_{222}^1$ (may not be in the same order). From SnapPy we can see that the symmetry group of the $c$-maximal horoball packing of $\mathbb{H}^3$ where $c$ is either of cusps $0$, $1$ and $4$ of $\mathcal{T}_{222}^1$ contains $W_{(2,3,6)}$ as a subgroup. We can further check from SnapPy that the maximal cusp volume for each of $0$, $1$ and $4$ cusps of $\mathcal{T}_{222}^1$ (and so for cusps $0$, $1$ and $2$ of  $\mathcal{T}_{222}^0$ as well) is 6.9282. Let $\mathcal{H}^0_{i,j,k}$ (respectively, $\mathcal{H}^1_{i,j,k}$) denote the horoball packing of $\mathbb{H}^3$ that one obtain by first maximizing the $i$-horoballs of $\mathcal{T}_{222}^0$ (respectively, $\mathcal{T}_{222}^1$) equivariantly until one of them touches another such $i$-horoball, then maximizing the $j$-horoballs of $\mathcal{T}_{222}^0$ (respectively, $\mathcal{T}_{222}^1$) equivariantly until one of them touches another such $i$- or $j$-horoball and finally maximizing the $k$-horoballs of $\mathcal{T}_{222}^0$ (respectively, $\mathcal{T}_{222}^1$)  equivariantly until one of them touches another such $i$- or $j$- or $k$-horoball. Let $V^0_{i,j,k}$ (respectively, $V^1_{i,j,k}$) denote the set (with three elements) consisting of the volumes of cusp $i$, cusp $j$ and cusp $k$ in $\mathcal{H}^0_{i,j,k}$ (respectively, $\mathcal{H}^1_{i,j,k}$). From SnapPy \cite{snappy} we see that 
\begin{align*}
&V^0_{0,1,2}=V^0_{0,2,1}=V^0_{1,0,2}=V^0_{1,2,0}=\left \{6.9282, 6.9282, 1.7321\right \}\\
&V^0_{2,0,1}=V^0_{2,1,0}=\{6.9282, 1.7321, 1.7321\}\\
&V^1_{0,4,1}=V^1_{0,1,4}=V^1_{4,0,1}=V^1_{4,1,0}=\{6.9282, 6.9282, 1.7321\}\\
&V^1_{1,0,4}=V^1_{1,4,0}=\{6.9282, 1.7321, 1.7321\}. 
\end{align*}
If $f$ is an isometry from $\mathcal{T}_{222}^0$ onto $\mathcal{T}_{222}^1$, then, $V^0_{i,j,k}=V^1_{f(i),f(j),f(k)}$. So, cusp $1$ of $\mathcal{T}_{222}^1$ is cusp $2$ of $\mathcal{T}_{222}^0$. So, for each $i \in \{0,1,2\}$, there is a covering from \texttt{otet20\_00474} onto $O_{(2,3,6), (2,2,2,2)}$ such that only cusp $i$ of \texttt{otet20\_00474} maps to the $(2,3,6)$ cusp of $O_{(2,3,6), (2,2,2,2)}$. 

So, Proposition \ref{coversmallestmulticusp} and the discussion above shows that for each $(M,i) \in \mathcal{E}_4$, a family of hyperbolic knot complements obtained from Dehn filling all cusps of $M$ but cusp $i$ and geometrically converging to $M$ can have at-most finitely many elements with hidden symmetries.

\subsection{Conclusion of Section \ref{tlcom}} Together the discussion in Subsection \ref{snappyrun} and \ref{E1E4} prove the following theorem. 
\begin{theorem}\label{tetracompthm}
\tetcompthm 
\end{theorem}
As an immediate corollary we can see the following. 
\begin{corollary}\label{cenpaperthm}
\censuspaperlinks 
\end{corollary}

\begin{remark}\label{622bergebefore}
We point out that both $6^2_2$ from \cite[Appendix C]{Rolfsen} and the link corresponding to the Berge manifold belong to this explicit list of $25$ links from \cite{FGGTV} (see the second and the third picture in \cite[Figure 3]{FGGTV}). We already know the $6^2_2$ case from \cite[Corollary 3.4]{CDM} (or, \cite[Theorem 1.6]{CDHMMWIMRN}) and the Berge manifold case from \cite[Proof of Theorem 6.1]{Hoff_smallknot}(or, \cite[Proof of Theorem 1.1]{Hoff_Comm}).
\end{remark}

\section{Two infinite families of tetrahedral links}\label{cycliccovers}
In this section, we analyze two infinite families of tetrahedral links. In the first family, the complement of each link is a cyclic cover of the Berge manifold. In the second, the complement of each link is a cyclic cover of $\mathbb{S}^3-6^2_2$. Both the Berge manifold and $\mathbb{S}^3-6^2_2$ belong to the (orientable) tetrahedral census of \cite{FGGTV}. 
From SnapPy \cite{snappy} (or, the second and the third picture in \cite[Figure 3]{FGGTV}), one can check that the Berge manifold is identified as \texttt{otet04\_00000} and $\mathbb{S}^3-6^2_2$ as \texttt{otet04\_00001} in the notation of \cite{FGGTV}. In particular, they both have four regular ideal tetrahedra in their tetrahedral decompositions. As noted in Remark \ref{622bergebefore}, both the Berge manifold and $\mathbb{S}^3-6^2_2$ have been studied before (see \cite{Hoff_smallknot}, \cite{Hoff_Comm}, \cite{CDM}, \cite{CDHMMWIMRN}). Putting together Theorem \ref{Bergecyclic} and Theorem \ref{622cyclicresult} of this section one can see that for each link $L$ in these two families, a family of hyperbolic knot complements obtained by Dehn filling all but one component of $L$ and geometrically converging to $\mathbb{S}^3-L$ can have at-most finitely many elements with hidden symmetries. 

\subsection{Cyclic covers of the Berge manifold}

Let $B^{0,n}$ denote the link shown in Figure \ref{bergecover}. Observe that the complement of $B^{0,1}$ is the Berge manifold (see \cite[Figure 1]{Hoff_Comm}) and that $\mathbb{S}^3-B^{0,n}$ is a degree $n$ cyclic cover of the Berge manifold. We have the following theorem.

\begin{figure}
\centering 
\captionsetup{justification=centering}
\includegraphics[scale=.22]{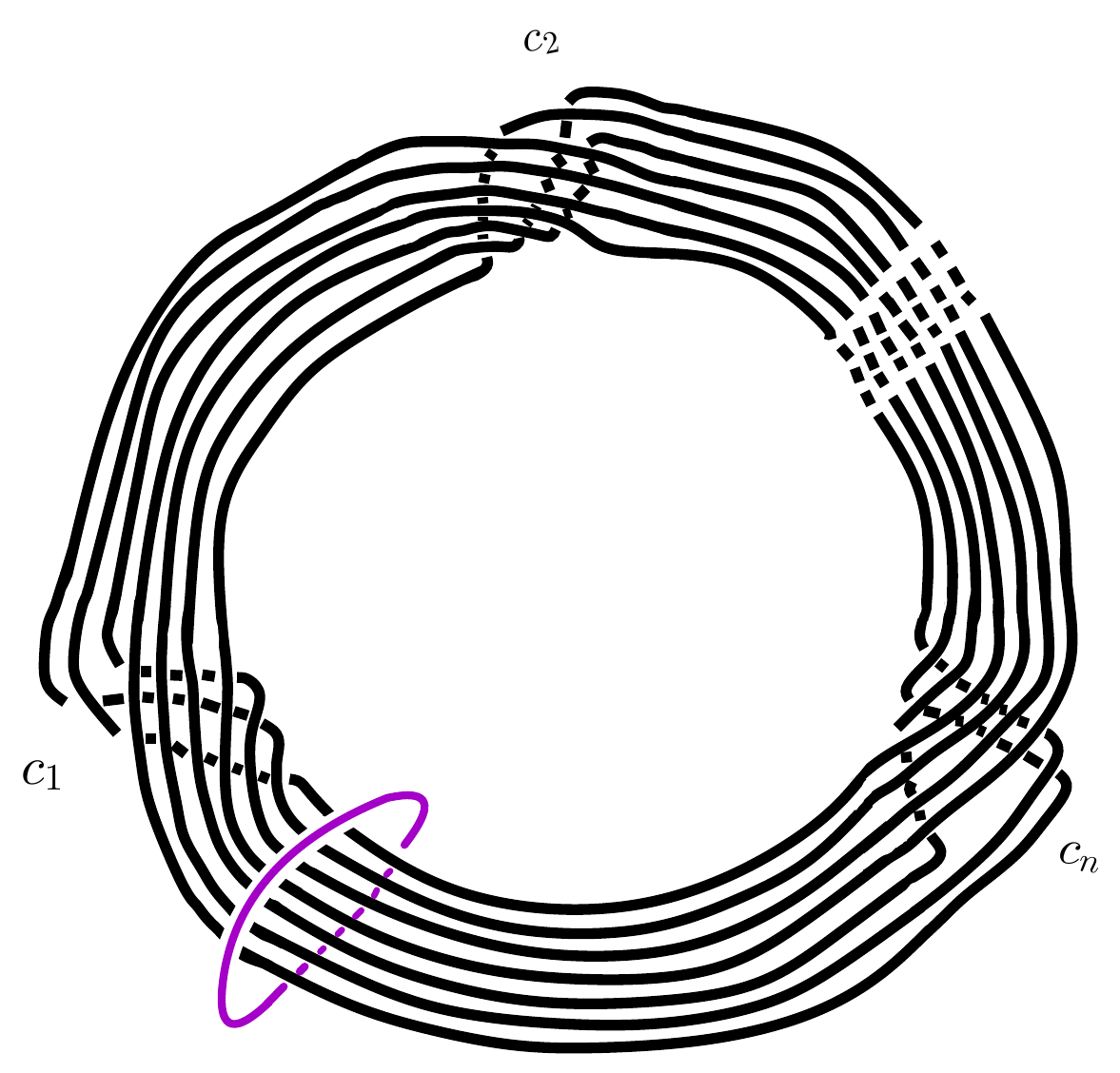}
\caption{The link $B^{0,n}$ (cf. \cite[Figure 1]{Hoff_Comm}).}
\label{bergecover}
\end{figure}

\begin{figure}
\centering 
\captionsetup{justification=centering}
\includegraphics[scale=.3]{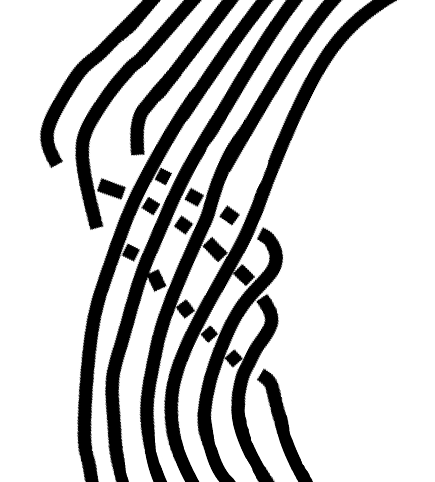}
\caption{Braid $b_7$ corresponding to the cycle $(1472536)$.}
\label{b7braid}
\end{figure}

\begin{figure}
\centering 
\captionsetup{justification=centering}
 \includegraphics[scale=.22]{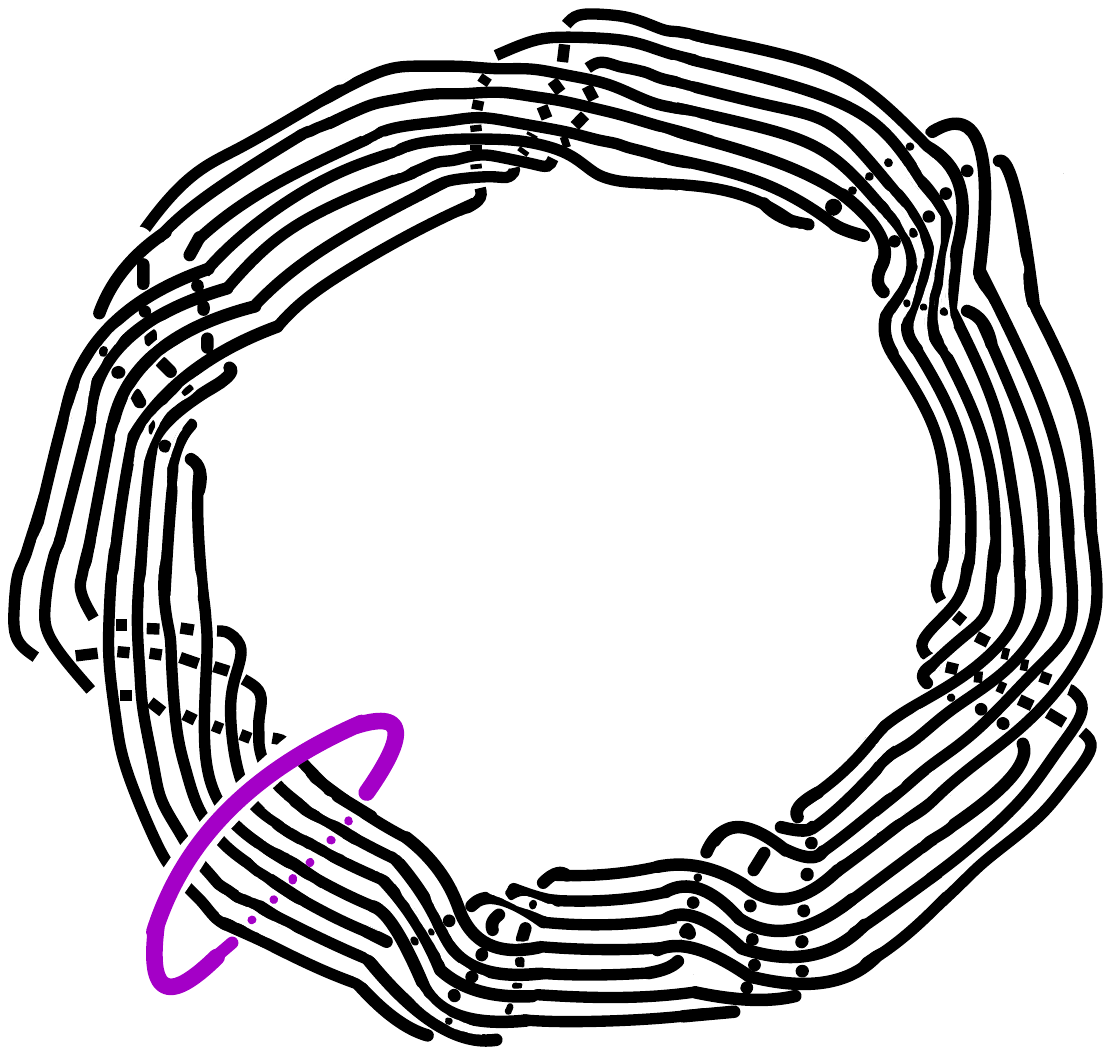}
 \caption{The $8$-component link $B^{0,7}$.}
 \label{berge7fold}
\end{figure}

\begin{theorem} \label{Bergecyclic}
\bergecyclic
\end{theorem}
\begin{proof}
We first write $B^{0,n}$ as $L \sqcup K$ where $K$ is the violet component and $L$ the union of the rest of the components of $B^{0,n}$ (see Figure \ref{bergecover}). We observe that $L$ is obtained by joining the end points of the braid $b_7^n$ that is produced by concatenating the braid $b_7$ on $7$ strands shown in Figure \ref{b7braid} $n$-times. Since $b_7$ has $7$ end points, $b_7$ belongs to the Braid group $\mathcal{B}_7$ on $7$ generators (see \cite[Subsection 1.2.4]{KasTur} for details on $\mathcal{B}_m$, where $m$ is a natural number). 

For a natural number $m$, we will denote the symmetric group on $m$ elements by $S_m$. By $\rho_m$, we will denote the group homomorphism $\rho_m: \mathcal{B}_m \rightarrow S_m$ which sends each element of $\mathcal{B}_m$ to its associated permutation on the end points of its strands counted from left to right counterclockwise (see Subsection 1.1.2 and the paragraph before \cite[Corollary 1.14]{KasTur} for more details about this homomorphism). 

Consider $\rho_7: \mathcal{B}_7 \rightarrow S_7$. One can then see that $\rho_7(b_7)$ is the $7$-cycle $(1472536)$. So, $\rho_7(b_7^n)$ is identity when $n$ is a multiple of $7$ and a $7$-cycle otherwise. So, $L$ has seven components when $n$ is a multiple of $7$ and only one component otherwise. Also note that $\mathbb{S}^3-B^{0,7k}$ is a cyclic cover of degree $k$ of $\mathbb{S}^3-B^{0,7}$ for each $k \in \mathbb{N}$. This implies  
\begin{enumerate}
\item $B^{0,n}$ is a two-component link when $n$ is not divisible by $7$, and, 
\item for all $k \in \mathbb{N}$, $B^{0,7k}$ is an eight-component link and its complement is a cyclic cover of degree $k$ of the complement of the eight-component link $B^{0,7}$ shown in Figure \ref{berge7fold}.
\end{enumerate}

\begin{figure}
\centering 
\captionsetup{justification=centering}
\includegraphics[scale=.12]{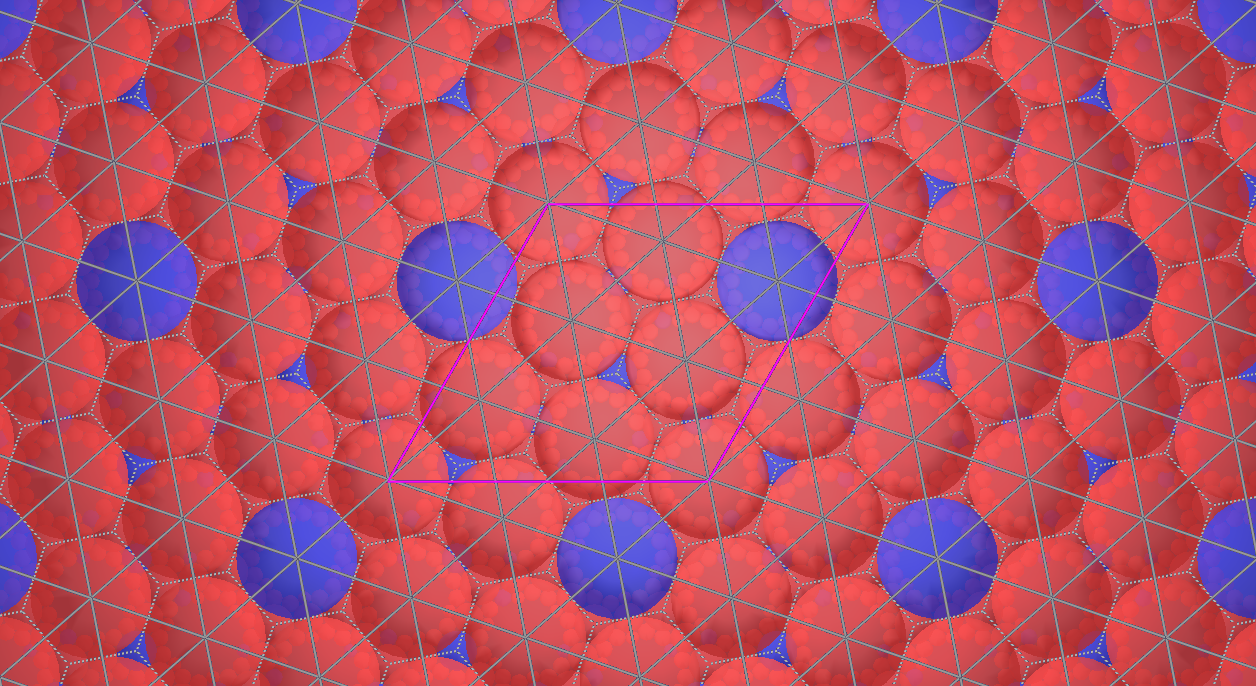} 
\caption{$(0, \texttt{m202})$-maximal horoball packing of $\mathbb{H}^3$ (picture obtained from SnapPy \cite{snappy}).}
\label{Bergeredhp}

\end{figure}

Let $n \in \mathbb{N}$. Suppose $c$ is a cusp of $B^{0,n}$. If $7$ does not divide $n$, then $c$ maps to a cusp $c_{B^{0,1}}$ of the Berge manifold such that the $c$-circle packing of $\mathbb{C}$ is the $c_{B^{0,1}}$-circle packing of $\mathbb{C}$ as the only cusp of $B^{0,n}$ that maps to $c_{B^{0,1}}$ is $c$. But since the cusps of the Berge manifold are symmetric (which one can check from SnapPy\cite{snappy}), the $c$-circle packing (which is also the $c_{B^{0,1}}$-circle packing) is the red cusp circle packing of $\mathbb{C}$ for the census manifold \texttt{m202} (i.e. the Berge manifold) in SnapPy \cite{snappy}.

On the other hand, when $n=7k$ for some $k \in \mathbb{N}$, $c$ maps to a cusp $c_{B^{0,7}}$ of the link $B^{0,7}$. So, the $c$-circle packing of $\mathbb{C}$ is the $c_{B^{0,7}}$-circle packing of $\mathbb{C}$ since the only cusp of $B^{0,n}$ that maps to $c_{B^{0,7}}$ is $c$. Now, one can check from SnapPy \cite{snappy} that any two cusps of $B^{0,7}$ are symmetric to each other. So, the $c$-circle packing (which is the $c_{B^{0,7}}$-circle packing) of $\mathbb{C}$ is the violet cusp circle packing of $\mathbb{C}$ for the violet cusp of $B^{0,7}$ shown in Figure \ref{berge7fold}. Now, the only cusp of $B^{0,7}$ which maps to the cusp $c_U$ corresponding to the unknotted component of the link $B^{0,1}$ ($\mathbb{S}^3- B^{0,1}$ is the Berge manifold) is the violet cusp in Figure  \ref{berge7fold}. So, the violet cusp circle packing of $\mathbb{C}$ (for $B^{0,7}$) is the $c_U$-circle packing of $\mathbb{C}$. Since the cusps of the Berge manifold are symmetric, this means that the $c$-circle packing of $\mathbb{C}$ is the red cusp circle packing of $\mathbb{C}$ for the census manifold \texttt{m202} in SnapPy \cite{snappy}.

So, Theorem \ref{hexlatt} implies that to conclude the theorem, it is enough to show that the red cusp circle packing of $\mathbb{C}$ for  SnapPy \cite{snappy} manifold \texttt{m202}, which can be seen from Figure \ref{Bergeredhp}, has no order $3$ rotational symmetry that does not fix the center of a blue horoball. We denote this circle packing by $\mathcal{H}$. It should be pointed out that in Figure \ref{Bergeredhp} the eye at $\infty$ contains a red horoball. Let us denote the cusp parallelogram drawn in pink in Figure \ref{Bergeredhp} as $P$. For the sake of contradiction, if we assume that $\mathcal{H}$ has an order $3$ rotational symmetry that does not fix the center of a blue horoball, then, by Fact \ref{cuspparasym}, $\mathcal{H}$ has an order $3$ rotational symmetry $r_{3,\mathbf{p}}$ such that $\mathbf{p} \in P$ and $\mathbf{p}$ is not the horocenter of a blue horoball. 

Consider the hexagonal triangulation $\mathcal{T}$ of the complex plane $\mathbb{C}$ shown in Figure \ref{Bergeredhp}. Let $L$ be a line from $\mathcal{T}$ such that $\mathbf{p} \notin L$. Since $r^{\pm 1}_{3,\mathbf{p}}$ is a symmetry of $\mathcal{H}$, $r^{\pm1}_{3,\mathbf{p}}$ is also a symmetry of $\mathcal{T}$. So, $r_{3,\mathbf{p}}(L)$ and $r^2_{3,\mathbf{p}}(L)$ are also lines from $\mathcal{T}$. Let $\mathbf{p}_i$ denote the point of intersection of $r^{i-1}_{3,\mathbf{p}}(L)$ and $r^i_{3,\mathbf{p}}(L)$ for $i \in \{1,2,3\}$. Note that $\mathbf{p}_i$ is a vertex of $\mathcal{T}$ for each $i=1,2,3$. Then $r_{3,\mathbf{p}}(\mathbf{p}_i)=\mathbf{p}_{i+1 (mod\; 3)}$. Since $\mathbf{p} \notin L$, by Fact \ref{equilatfact}, we see that $\mathbf{p}_1$, $\mathbf{p}_2$ and $\mathbf{p}_3$ are the three distinct vertices of the equilateral triangle $\Delta$ bounded by $L$, $r_{3,\mathbf{p}}(L)$ and $r^2_{3,\mathbf{p}}(L)$. So, $\mathbf{p}$ is the centroid of $\Delta$. Since the side length of a $2$-simplex of $\mathcal{T}$ is $1$, we see that the side length of $\Delta$ is a natural number $n$. So, the centroid $\mathbf{p}$ of the equilateral triangle $\Delta$ is the centeroid of a $2$-simplex of $\mathcal{T}$ if $n$ is not a multiple of $3$ and its a vertex of $\mathcal{T}$ if $n$ is a multiple of $3$. 

Now, amongst all the vertices and the centroids of the $2$-simplices of $\mathcal{T}$ lying in $P$, the only points which could be the fixed points of order $3$ symmetries of $\mathcal{H}$ are actually the horocenters of the blue horoballs. This contradicts that $\mathbf{p}$ can not be the horocenter of a blue horoball. So, we are done.
\end{proof}

\subsection{Cyclic covers of the $6^2_2$ link complement}
Let $(6^2_2)^{0,n}$ denote the link in Figure \ref{622cycliccoverpic}. Note that $(6^2_2)^{0,1}$ is the link $6^2_2$ from \cite[Appendix C]{Rolfsen} (see the right picture in \cite[Figure 11]{CDHMMWIMRN}). Observe that $\mathbb{S}^3-(6^2_2)^{0,n}$ is a degree $n$ cyclic cover of $\mathbb{S}^3-6^2_2$. We also note that SnapPy \cite{snappy} identifies $(6^2_2)^{0,3}$ as \texttt{L12n2208} of the Hoste-Thistlewaite census. We can show the following.

\begin{theorem}\label{622cyclicresult}
\sixtwotwocyclic
\end{theorem}

\begin{proof}
We note from Figure \ref{622cycliccoverpic} that $(6^2_2)^{0,n}=L\sqcup K$ where $K$ is the blue component and $L$ is the link obtained by joining the ends of a braid $b_3^n$ on $3$ strands so that $\rho_3(b_3)$ is the $3$-cycle $(132)$. So, arguing similarly as in the proof of Theorem \ref{Bergecyclic}, we can show that 

\begin{enumerate}
\item $(6^2_2)^{0,n}$ is a two-component link when $n$ is not divisible by $3$, and, 
\item for all $k \in \mathbb{N}$, $(6^2_2)^{0,3k}$ is a four-component link and its complement is a cyclic cover of degree $k$ of the complement of the four component link $(6^2_2)^{0,3}=\texttt{L12n2208}$. 
\end{enumerate}

\begin{figure}
\centering 
\captionsetup{justification=centering}
\includegraphics[scale=.2]{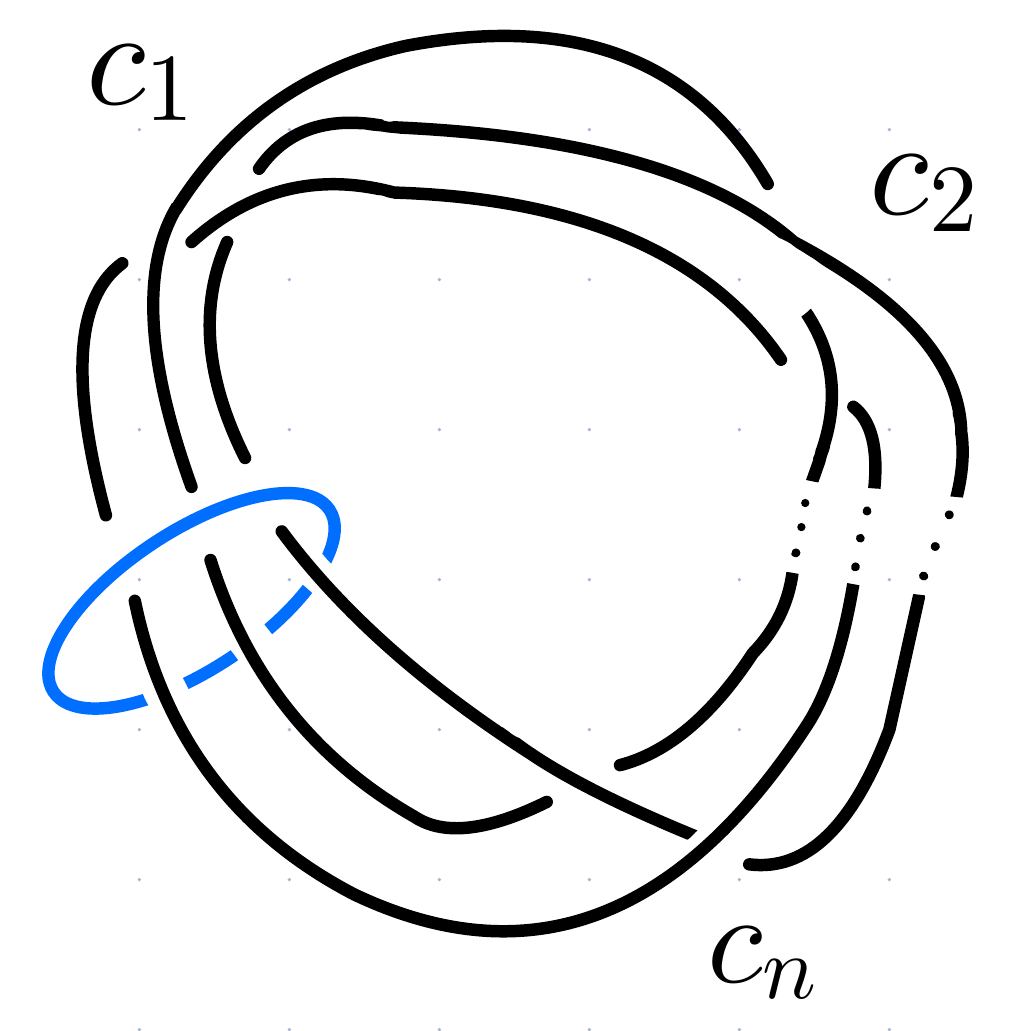}
\caption{The link $(6^2_2)^{0,n}$.}
\label{622cycliccoverpic}
\end{figure}

 Assume for contradiction that there exists a natural number $n$ such that there is an infinite family $\mathcal{F}$ of hyperbolic knot complements with hidden symmetries obtained by Dehn filling all but a fixed cusp of $(6^2_2)^{0,n}$ and geometrically converging to $\mathbb{S}^3-(6^2_2)^{0,n}$. Let us denote the un-filled cusp of $(6^2_2)^{0,n}$ as $c_{uf}$. Then Proposition \ref{hscharac}, Theorem \ref{hoff_cusptype} and Theorem \ref{CDHMMWorb} together imply that $\mathbb{S}^3-(6^2_2)^{0,n}$ covers an orbifold $O$ with at-least one smooth cusp and exactly one rigid cusp $c_{rigid}$ that is covered only by cusp $c_{uf}$ of $\mathbb{S}^3-(6^2_2)^{0,n}$ so that the conditions laid out in Theorem \ref{CDHMMWorb} are satisfied. Let $\mathcal{H}_{rigid}$ denote the $c_{rigid}$-maximal horoball packing of $\mathbb{H}^3$.

 If $3$ divides $n$, then we have seen that $(6^2_2)^{0,n}$ is a four component link whose complement covers the four cusped manifold $\mathbb{S}^3-\texttt{L12n2208}$. So, $c_{uf}$ is the only cusp of  $\mathbb{S}^3-(6^2_2)^{0,n}$ that maps to its image cusp in $\mathbb{S}^3-\texttt{L12n2208}$. This means that $c_{uf}$-maximal horoball packing of $\mathbb{H}^3$ is same as the maximal horoball packing of $\mathbb{H}^3$ that projects down to the cusp of $\mathbb{S}^3-\texttt{L12n2208}$ where $c_{uf}$ maps to. We can check from SnapPy \cite{snappy} that the cusps of \texttt{L12n2208} are symmetric. So, we can see that $c_{uf}$-maximal horoball packing, which is same as $\mathcal{H}_{rigid}$, is isometric to the red cusp maximal horoball packing $\mathcal{H}_{red}$ of $\mathbb{H}^3$ for \texttt{L12n2208}. $\mathcal{H}_{red}$ can be seen in Figure \ref{L12n2208red} obtained from SnapPy \cite{snappy}. The eye at $\infty$ in Figure \ref{L12n2208red} contains a red horoball. 
 
 Now, for $n=3$, the cusp corresponding to the blue component in Figure \ref{622cycliccoverpic} is the only cusp of $\mathbb{S}^3-(6^2_2)^{0,n}$ that maps to its image cusp in $\mathbb{S}^3-6^2_2$. So, since the cusps of $\mathbb{S}^3-6^2_2$ are symmetric, which can be checked from SnapPy \cite{snappy}, and so are the cusps of $\mathbb{S}^3-\texttt{L12n2208}$, we conclude that $\mathcal{H}_{red}$ is also isometric to the red cusp maximal horoball packing of $\mathbb{H}^3$ for $6^2_2$ from SnapPy \cite{snappy}. Let us denote the red cusp maximal horoball packing of $\mathbb{H}^3$ for $6^2_2$ as $\mathcal{H}$. One can see $\mathcal{H}$ from Figure \ref{622red}, where the horoball at eye at $\infty$ is a red horoball.

On the other hand, if $n$ is not divisible by $3$, then $(6^2_2)^{0,n}$ has two components and therefore, $O$ has two cusps. So, since cusps of $6^2_2$ are symmetric (as checked from SnapPy \cite{snappy}), $c_{uf}$-maximal horoball packing is isometric to  $\mathcal{H}$, which we know to be isometric to $\mathcal{H}_{red}$,  and consequently, $\mathcal{H}_{rigid}$ is isometric to $\mathcal{H}_{red}$ even in this case. 
\begin{cl}
$O$ has exactly two cusps - one $(3,3,3)$ cusp and one smooth cusp. 
\end{cl}
\begin{proof}
Now, using Theorem \ref{hoff_cusptype} and the properties of $O$ laid out in Theorem \ref{CDHMMWorb}, proof of Theorem \ref{2comp_sym} shows that $c_{rigid}$ is $(3,3,3)$ or $(2,3,6)$. 

We first consider the case when $3$ divides $n$. Observe from Figure \ref{L12n2208red} that all order $6$ symmetries of the $\mathcal{H}_{red}$-circle packing fix the horocenter of a non-red horoball. We note that the isometry of $\mathbb{H}^3$ mentioned above that sends $\mathcal{H}_{red}$ isometrically to the $c_{uf}$-maximal horoball packing of $\mathbb{H}^3$ would send the horocenters of these non-red horoballs to the horocenters of the non $c_{uf}$-horoballs. So, all order $6$ rotational symmetries of the $\mathcal{H}_{rigid}$-circle packing of $\mathbb{C}$ fix the center of a non $c_{uf}$-horoball. Since all cusps of $O$ other than $c_{rigid}$ are smooth, there is no order $6$ symmetry of the $\mathcal{H}_{rigid}$-circle packing that belongs to the peripheral subgroup of $\pi_1^{Orb}(O)$ corresponding to $c_{rigid}$. So, $c_{rigid}$ is $(3,3,3)$. 

Now, since $3 \mid n$, both \texttt{L12n2208} and $(6^2_2)^{0,n}$ have the same number of components and so, each cusp of $\mathbb{S}^3-\texttt{L12n2208}$ is covered by exactly one cusp of $\mathbb{S}^3-(6^2_2)^{0,n}$. Consequently, any $(\mathbb{S}^3-\texttt{L12n2208})$-horoball packing of $\mathbb{H}^3$ is a $(\mathbb{S}^3-(6^2_2)^{0,n})$-horoball packing of $\mathbb{H}^3$. Now, Figure \ref{L12n2208red} shows that any order $3$ symmetry of $\mathcal{H}_{red}$-circle packing from (the isomorphic copy of) $\pi_1^{Orb}(c_{rigid})<\pi_1^{Orb}(O)$ identifies the other three cusps of \texttt{L12n2208}. So, all the filled cusps of $(6^2_2)^{0,n}$ map to a single smooth cusp of $O$. So, when $3$ divides $n$, $O$ has two cusps - one $(3,3,3)$-cusp and one smooth cusp. 

Now, consider the case when $3$ does not divide $n$. We can see from Figure \ref{622red} that all order $6$ symmetries of $\mathcal{H}$ fix the horocenters of blue horoballs. The isometry of $\mathbb{H}^3$ mentioned above that sends $\mathcal{H}$ isometrically to the $c_{uf}$-maximal horoball packing of $\mathbb{H}^3$ sends these blue horocenters to the horocenters of the non $c_{uf}$-horoballs. So, arguing similarly as in first case, we can conclude that $c_{rigid}$ is $(3,3,3)$. We already know that $O$ has two cusps in this case. So, the claim holds when $3$ does not divide $n$ as well.

\end{proof}

 \begin{figure}
 \centering 
\captionsetup{justification=centering}
\includegraphics[scale=.13]{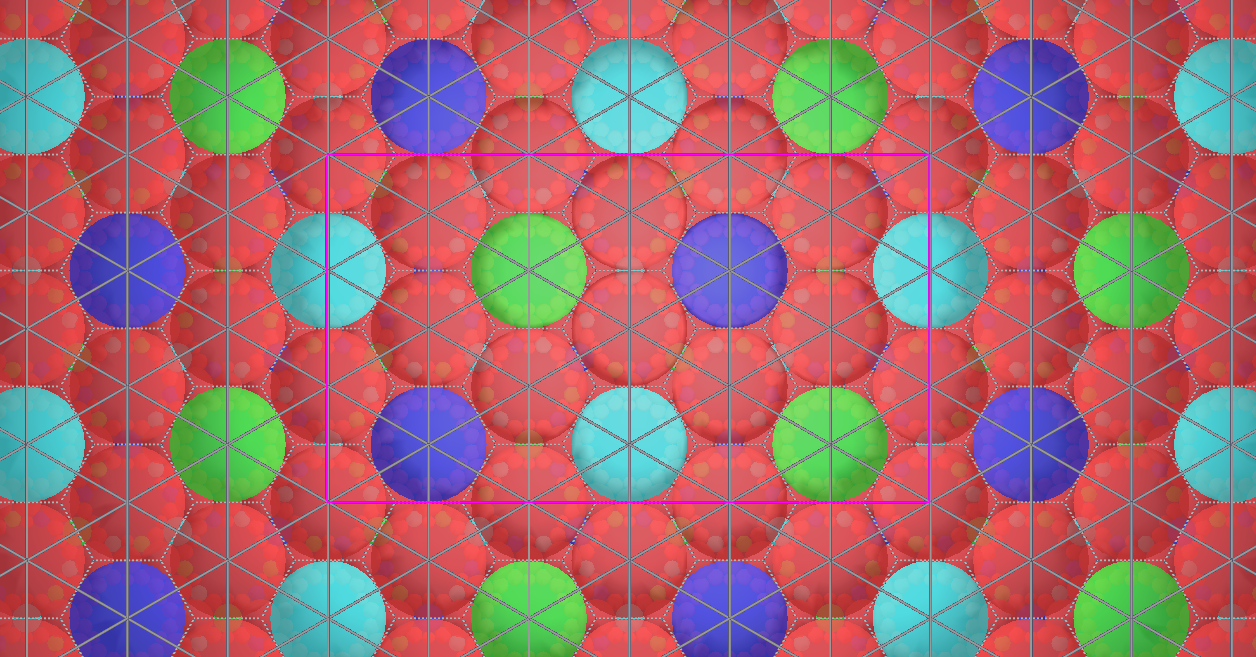}
\caption{$(0, \mathbb{S}^3-\texttt{L12n2208})$-maximal horoball packing of $\mathbb{H}^3$ (picture obtained from SnapPy \cite{snappy}).}
\label{L12n2208red}
\end{figure}

 \begin{figure}
 \centering 
\captionsetup{justification=centering}
\includegraphics[scale=.12]{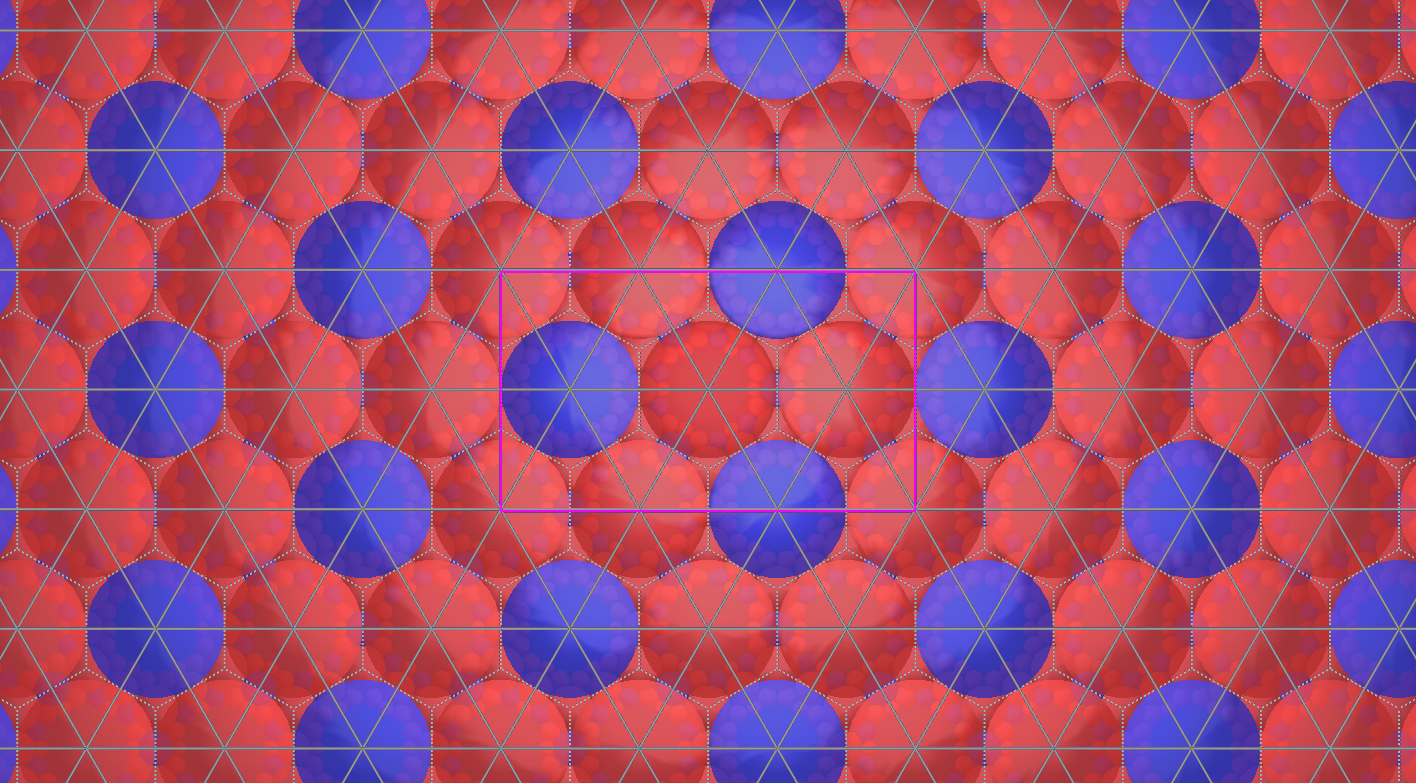}
\caption{$(0, \mathbb{S}^3-6^2_2)$-maximal horoball packing of $\mathbb{H}^3$ (picture obtained from SnapPy \cite{snappy}).}
\label{622red}
\end{figure}

Let $S_{uf}$ be the set of all parabolic fixed points of $\pi_1\left(\mathbb{S}^3-{(6^2_2)}^{0,n}\right)$ corresponding to cusp $c_{uf}$ and $S_{filled}$ the set of the rest of the parabolic fixed points of $\pi_1\left(\mathbb{S}^3-{(6^2_2)}^{0,n} \right)$ (i.e. the parabolic fixed points corresponding  to the cusps of $\mathbb{S}^3-(6^2_2)^{0,n}$ other than $c_{uf}$). 

Let $M_0$ denote $\mathbb{S}^3-6^2_2$ when $3 \nmid n$ and $\mathbb{S}^3-\texttt{L12n2208}$ when $ 3 \mid n$.  From Figure \ref{L12n2208red} (respectively, Figure \ref{622red}) we can see that for any non-red cusp $c_{\text{non-red}}$ of $\mathbb{S}^3-\texttt{L12n2208}$ (respectively, $\mathbb{S}^3-6^2_2$), $\mathcal{H}_{red}$ (respectively,   $\mathcal{H}$) has an order $6$ rotational symmetry $\gamma_{6,red}$ such that the axis of $\gamma_{6,red}$ joins the red cusp at $\infty$ and $c_{\text{non-red}}$ at the horocenter of a $c_{\text{non-red}}$-horoball. Since any two cusps of $M_0$ are exchanged by a self-isometry of $M_0$ fixing the other two cusps of $M_0$ (if any), which can be checked from SnapPy \cite{snappy}, we can conclude that $\mathcal{H}_{rigid}$ has an order $6$ rotational symmetry $\gamma_{6}$ such that $\gamma_{6}$ has one boundary fixed point in $S_{uf}$ and the other in $S_{filled}$. 
 
Let $c'_{uf}$ be the cusp of $M_0$ to where cusp $c_{uf}$ of $\mathbb{S}^3-(6^2_2)^{0,n}$ is mapped.

\begin{cl}
There exists an orbifold $O_1$ of volume $\frac{2v_0}{3}$ which has two cusps - both of the type $(2,3,6)$ such that $O_1$ is covered by both $O$ and $M_0$. 
\end{cl}
\begin{proof}
We note that $\mathcal{H}_{rigid}$ is same as the $c'_{uf}$-maximal and the $c_{uf}$-maximal horoball packings of $\mathbb{H}^3$. So, all of $\pi_1^{Orb}(O)$, $\pi_1(M_0)$ and $\gamma_6$ act as symmetries of $\mathcal{H}_{rigid}$. As in the proof of Proposition \ref{coversmallestmulticusp}, \cite[Lemma 2.1 and Proof of Lemma 2.2]{GoHeHo} then implies that $$\Gamma_1=\left \langle \pi_1^{Orb}(O), \gamma_6, \pi_1(M_0)\right \rangle $$ 
is discrete. Let $O_1=\mathbb{H}^3/\Gamma_1$. Now, $S_{uf}$ is invariant under the action of $\gamma_6$ as well as the elements of $\pi_1^{Orb}(O)$ and  $\pi_1(M_0)$. Consequently, it is invariant under the action of $\Gamma_1$.  It should be noted that  $\mathbb{S}^3-(6^2_2)^{0,n}$ covers $O_1$ and therefore, the set of all parabolic fixed points of $\pi_1\left(\mathbb{S}^3-(6^2_2)^{0,n}\right)$ is same as that of $\Gamma_1$.  So, $S_{filled}$ is invariant under the action of $\Gamma_1$ as well. This implies that $O_1$ has at-least two cusps. On the other hand, since $O_1$ is covered by the two cusped orbifold $O$, it has at-most two cusps. Hence, $O_1$ has exactly two cusps such that cusp $c_{uf}$ of $\mathbb{S}^3-(6^2_2)^{0,n}$, cusp $c'_{uf}$ of $M_0$ and cusp $c_{rigid}$ of $O$ all map to the same cusp of $O_1$. Call this cusp of $O_1$ as $c_1$ and the other cusp as $c_2$. It should be pointed out that the parabolic fixed points of $\Gamma_1$ corresponding to cusp $c_1$ is $S_{uf}$ and so all the filled cusps of $\mathbb{S}^3-(6^2_2)^{0,n}$ (i.e. the cusps corresponding to $S_{filled}$) map to $c_2$. We now conclude that both $c_1$ and $c_2$ are $(2,3,6)$ cusps as the order $6$ elliptic element $\gamma_6$ of $\Gamma_1$ fixes parabolic fixed points of $\Gamma_1$ corresponding to both $c_1$ and $c_2$.

When $ 3 \mid n$, we note that since the three torus cusps of $M_0$ other than $c'_{uf}$ map to the $(2,3,6)$ cusp $c_2$, $[\Gamma_1:\pi_1(M_0)]$ is a multiple of $6$ and is at-least $18$. So, the volume of $O_1$ is $\frac{12v_o}{6m}=\frac{2v_o}{m}$ where $m \ge 3$.  

On the other hand, when  $ 3 \nmid n$, the torus cusp of $M_0$ which is not $c'_{uf}$ maps to the $(2,3,6)$ cusp $c_2$ and so $[\Gamma_1:\pi_1(M_0)]=6m$ for some positive integer $m$. So, volume of $O_1$ is $\frac{4v_o}{6m}=\frac{2v_o}{3m}$ where $m \ge 1$. 

So, the volume of $O_1$ is less than or equal to $ \frac{2v_o}{3}$ in both cases. If the volume of $O_1$ is strictly less than $\frac{2v_o}{3}$, then, the two volume formulae above imply that the volume of $O_1$ is less than or equal to $\frac{v_0}{2}$. But, since $O_1$ has two $(2,3,6)$ cusps and its volume is of the form $\frac{2v_0}{k}$ where $k$ is a positive integer, we can apply \cite[Lemma 2.2]{Adams_multi} to conclude that $O_1$ is one of the unique orbifolds with two $(2,3,6)$ cusps of volume $\frac{v_0}{3}$ and $\frac{v_0}{2}$ given in \cite{Adams_multi}.  But, $O_1$ is covered by $O$, which is a contradiction to Lemma \ref{coveradams}. So, the volume of $O_1$ is exactly $\frac{2v_o}{3}$.

\end{proof}

Recall from the proof of Lemma \ref{coveradams} that $ O_{\frac{v_0}{3}}$ denotes the unique orbifold of volume $\frac{v_0}{3}$ with two $(2,3,6)$ cusps given in Adams \cite{Adams_multi}. 

\begin{lemma}\label{index36cover}
For each cusp $c$ of $M_0$, there exists a covering $\phi_{c}: M_0 \rightarrow O_{\frac{v_0}{3}}$ such that $c$ is the only cusp of $M_0$ that maps to its image cusp in $ O_{\frac{v_0}{3}}$ via $\phi_{c}$.

\end{lemma}
\begin{proof}
We will break down the proof into two cases - Case 1: when $3 \nmid n$ and Case 2: when $3 \mid n$.

\textit{Case 1}: $M_0$ in this case is $\mathbb{S}^3-6^2_2$. SnapPy \cite{snappy} tells us that $\mathbb{S}^3-6^2_2$ can be identified as \texttt{otet04\_00001} in the notation of \cite{FGGTV}. We note from the ancillary file \cite{isosigcensus} of \cite{FGGTV} that the isomorphism signature of the default tetrahedral triangulation of \texttt{otet04\_00001} is \texttt{eLMkbcddddedde}.
$O_{\frac{v_0}{3}}$ is identified in \cite{orbcenpract} as $O^4$ and the orbifold destination sequence of $O^4$ in \cite{orbcenpract} is given to be 
$$\texttt{[0,1,1,0, 1,0,0,2, 3,2,2,1, 2,3,3,3]}.$$ We will denote this Python list as \texttt{des\_seq\_O4}. Following what we have done in Subsection \ref{censusutilities}, we now use the command \texttt{Triangulation3.fromIsoSig} from Regina \cite{regina} and the utility \texttt{SigToSeq.py} of \cite{orbcenpract} (available from \cite{orbtricode}) to get an orbifold destination sequence \texttt{des\_seq\_622}  of \texttt{otet04\_00001} and check via \texttt{TestForCovers.py} of \cite{orbcenpract} (available from \cite{orbtricode}) that \texttt{des\_seq\_622} covers \texttt{des\_seq\_O4}. Now one can check from SnapPy \cite{snappy} that there is an orientation preserving self-isometry of $\mathbb{S}^3-6^2_2$ exchanging its cusps. So, since both $\mathbb{S}^3-6^2_2$ and $O_{\frac{v_0}{3}}$ have two cusps, the lemma follows when $3 \nmid n$.

\textit{Case 2}: $M_0$ in this case is $\mathbb{S}^3-\texttt{L12n2208}$. We note from SnapPy \cite{snappy} that the link complement $\mathbb{S}^3-\texttt{L12n2208}$ is identified as \texttt{otet12\_00009} of the tetrahedral census \cite{FGGTV}. We can now see from the ancillary file \cite{isosigcensus} of \cite{FGGTV} that the canonical triangulation of $\texttt{otet12\_00009}$ (different from its default triangulation) is a tetrahedral triangulation corresponding to the isomorphism signature 
$$\texttt{mvvLPQwQQfhgffijlklklkaaaaaaaaaaaaa}. $$

As in Case 1, we use the command \texttt{Triangulation3.fromIsoSig} from Regina \cite{regina} and \texttt{SigToSeq.py} of \cite{orbcenpract} from \cite{orbtricode} to get an orbifold destination sequence \texttt{des\_seq\_L12n2208} corresponding to the above isomorphism signature of  $\mathbb{S}^3-\texttt{L12n2208}$. Using the functions \texttt{Covers}, \texttt{CuspSeqs()} and \texttt{CuspCovers()} of \cite{orbcenpract} 's \texttt{TestForCovers.py} from \cite{orbtricode}, we can see that there are $4$ triangulation preserving and orientation preserving covers of \texttt{des\_seq\_O4} by \texttt{des\_seq\_L12n2208} and that for each cusp $c$ of $\mathbb{S}^3-\texttt{L12n2208}$ there is precisely one such cover which sends only cusp $c$ to its image cusp in $O_{\frac{v_0}{3}}$ and the other three cusps of $\mathbb{S}^3-\texttt{L12n2208}$ to the other cusp of $O_{\frac{v_0}{3}}$. 

(We remark that the computations for both cases can be seen from running the Python file \texttt{6220n\_computation.py} from \cite{tetra_code} in \texttt{regina-python}. In \texttt{6220n\_computation.py}, we imported the \texttt{regina} module, the function \texttt{Des\_seq} from \texttt{SigToSeq.py} of \cite{orbcenpract} and the functions \texttt{CuspSeqs}, \texttt{Covers} and \texttt{CuspCovers} from \texttt{TestForCovers.py} of \cite{orbcenpract}). 

\end{proof}

Since the $c'_{uf}$-maximal horoball packing of $\mathbb{H}^3$ is $\mathcal{H}_{rigid}$, Lemma \ref{index36cover} implies that $\pi_1(M_0)$ is contained a Kleinian group $\Gamma_2$ isomorphic to $\pi_1^{Orb}(O_{\frac{v_0}{3}})$ such that the elements of $\Gamma_2$ are symmetries of $\mathcal{H}_{rigid}$ and there are two orbits of parabolic fixed points of $\Gamma_2$: $S_{uf}$ and $S_{filled}$. So, we can again use \cite[Lemma 2.1 and Proof of Lemma 2.2]{GoHeHo} to see that the group $$\Gamma_3=\left \langle \Gamma_1, \Gamma_2 \right \rangle $$ is a discrete subgroup of $\operatorname{PSL}(2, \mathbb{C})$. So, $O_3=\mathbb{H}^3/\Gamma_3$ is a hyperbolic orbifold with two $(2,3,6)$ cusps. Now, by Lemma \ref{isovertex}, the isotropy graph of $O$ has a finite vertex of type $(2,3,3)$ or $(2,3,4)$ or $(2,3,5)$. On the other hand, the isotropy graph of $O_{\frac{v_0}{3}}$ does not (see Figure \ref{O4andO6_1isopic}). Since $O_1$ is covered by $O$ and neither of $A_4$, $S_4$ and $A_5$ is a subgroup of a finite dihedral group, the isotropy graph of any orbifold that $O_1$ covers also has a finite vertex of type $(2,3,3)$ or $(2,3,4)$ or $(2,3,5)$. This implies that $O_{\frac{v_0}{3}}$ is not covered by $O_1$. So, $O_3$ is covered by $O_{\frac{v_0}{3}}$ by an index greater than or equal to $2$. So, $O_3$ is an orbifold with two $(2,3,6)$ cusps with volume less than $\frac{v_0}{3}$ which contradicts \cite[Lemma 2.2] {Adams_multi}.

So, our assumption was wrong and no family of hyperbolic knot complements obtained by Dehn filling all but a fixed cusp of $\mathbb{S}^3-(6^2_2)^{0,n}$ and geometrically converging to  $\mathbb{S}^3-(6^2_2)^{0,n}$ can have infinitely many elements with hidden symmetries.

\end{proof}

\appendix

\section{Members of $\mathcal{E}_1$, $\mathcal{E}_2$, $\mathcal{E}_3$ and $\mathcal{E}_4$}\label{allEi}
For each of $\mathcal{E}_1$, $\mathcal{E}_2$, $\mathcal{E}_3$ and $\mathcal{E}_4$, we describe its elements $(M,i)$ in the following format: $(j,\text{namestring}, i)$ where $j$ is the index of $M$ in \texttt{K\_hlgy\_lk} and \text{namestring} is the name of $M$ as a Python string in the notation of \cite{FGGTV}. 

\subsection{Members of $\mathcal{E}_1$}
\begin{enumerate}
\item (157, \texttt{`otet20\_00059'}, 0)
\item (158, \texttt{`otet20\_00060'}, 0)
\item (159, \texttt{`otet20\_00061'}, 0)
\item (166, \texttt{`otet20\_00090'}, 1)
\item (168, \texttt{`otet20\_00098'}, 3)
\item (183, \texttt{`otet20\_00158'}, 1)
\item (195, \texttt{`otet20\_00363'}, 0)
\item (214, \texttt{`otet20\_00429'}, 0)
\item (220, \texttt{`otet20\_00461'}, 1)
\item (223, \texttt{`otet20\_00477'}, 1)
\item (242, \texttt{`otet20\_00540'}, 0)
\item (243, \texttt{`otet20\_00541'}, 0)
\item (244, \texttt{`otet20\_00542'}, 0)
\item (248, \texttt{`otet20\_00552'}, 0)
\item (250, \texttt{`otet20\_00565'}, 4)
\item (252, \texttt{`otet20\_00567'}, 2)
\item (257, \texttt{`otet20\_00581'}, 0)
\item (308, \texttt{`otet20\_01296'}, 0)
\item (309, \texttt{`otet20\_01301'}, 2)
\item (310, \texttt{`otet20\_01302'}, 2)
\item (311, \texttt{`otet20\_01330'}, 2)
\item (312, \texttt{`otet20\_01335'}, 2)
\item (322, \texttt{`otet20\_01385'}, 2)
\item (324, \texttt{`otet20\_01388'}, 2)
\item (325, \texttt{`otet20\_01389'}, 2)
\item (326, \texttt{`otet20\_01390'}, 1)
\item (327, \texttt{`otet20\_01391'}, 2)
\item (329, \texttt{`otet20\_01396'}, 0)
\item (330, \texttt{`otet20\_01402'}, 0)
\item (333, \texttt{`otet20\_01414'}, 2)
\item (336, \texttt{`otet20\_01438'}, 2)
\item (338, \texttt{`otet20\_01488'}, 0)
\item (343, \texttt{`otet20\_01500'}, 0)
\item (355, \texttt{`otet20\_01662'}, 3)
\end{enumerate}

\subsection{Members of $\mathcal{E}_2$}
\begin{enumerate}
\item (3, \texttt{`otet08\_00002'}, 0)

\item (3, \texttt{`otet08\_00002'}, 1)
\item (4, \texttt{`otet08\_00003'}, 0)
\item (22, \texttt{`otet12\_00009'}, 0)
\item (58, \texttt{`otet16\_00013'}, 0)

\item (58, \texttt{`otet16\_00013'}, 1)
\item (74, \texttt{`otet16\_00058'}, 0)
\item (74, \texttt{`otet16\_00058'}, 1)

\item (79, \texttt{`otet16\_00090'}, 0)
\item (98, \texttt{`otet18\_00028'}, 0)
\item (98, \texttt{`otet18\_00028'}, 1)

\item (127, \texttt{`otet18\_00104'}, 0)
\item (145, \texttt{`otet18\_00171'}, 0)
\item (219, \texttt{`otet20\_00443'}, 0)
\item (219, \texttt{`otet20\_00443'}, 1)

\item (245, \texttt{`otet20\_00543'}, 0)
\item (299, \texttt{`otet20\_00762'}, 0)
\item (307, \texttt{`otet20\_00793'}, 0)
\item (425, \texttt{`otet22\_00130'}, 0)
\item (633, \texttt{`otet24\_00259'}, 0)

\item (633, \texttt{`otet24\_00259'}, 1)
\item (634, \texttt{`otet24\_00260'}, 0)
\item (634, \texttt{`otet24\_00260'}, 1)

\item (636, \texttt{`otet24\_00263'}, 0)
\item (645, \texttt{`otet24\_00290'}, 0)
\item (645, \texttt{`otet24\_00290'}, 1)

\item (681, \texttt{`otet24\_00396'}, 0)
\item (682, \texttt{`otet24\_00398'}, 0)
\item (683, \texttt{`otet24\_00399'}, 0)
\item (684, \texttt{`otet24\_00401'}, 0)
\item (821, \texttt{`otet24\_00979'}, 0)
\end{enumerate}

\subsection{Members of $\mathcal{E}_3$}
\begin{enumerate}
\item (160, \texttt{`otet20\_00062'}, 2)
\item (167, \texttt{`otet20\_00097'}, 4)
\item (172, \texttt{`otet20\_00102'}, 3)
\item (175, \texttt{`otet20\_00131'}, 3)
\item (301, \texttt{`otet20\_00770'}, 3)
\item (331, \texttt{`otet20\_01405'}, 4)
\item (333, \texttt{`otet20\_01414'}, 3)
\item (334, \texttt{`otet20\_01431'}, 3)
\item (335, \texttt{`otet20\_01433'}, 3)
\item (336, \texttt{`otet20\_01438'}, 3)
\end{enumerate}

\subsection{Members of $\mathcal{E}_4$}\label{E4}
\begin{enumerate}
\item (156, \texttt{`otet20\_00049'}, 0)
\item (156, \texttt{`otet20\_00049'}, 1)

\item (161, \texttt{`otet20\_00063'}, 0)
\item (161, \texttt{`otet20\_00063'}, 2)
\item (221, \texttt{`otet20\_00462'}, 1)

\item (221, \texttt{`otet20\_00462'}, 2)
\item (222, \texttt{`otet20\_00474'}, 0)
\item (222, \texttt{`otet20\_00474'}, 1)
\item (222, \texttt{`otet20\_00474'}, 2)
\item (256, \texttt{`otet20\_00577'}, 0)

\item (256, \texttt{`otet20\_00577'}, 3)
\end{enumerate}

\bibliography{tetrahedral.bib}
\bibliographystyle{plain}

\end{document}